\documentclass[12pt]{amsart}

\usepackage{amssymb}
\usepackage{array}
\usepackage{booktabs}
\usepackage{mdwtab}
\usepackage[T1]{fontenc}
\usepackage[utf8]{inputenc}
\usepackage{libertine}
\usepackage[libertine]{newtxmath}
\usepackage[scaled=0.96]{zi4}
\usepackage{hyphenat}
\usepackage{enumitem}
\usepackage{xcolor}
\definecolor{medium-blue}{rgb}{0,0,0.65}
\usepackage{graphicx}
\usepackage[margin=1in]{geometry}
\usepackage{url}
\usepackage[textsize=small]{todonotes}
\usepackage{makecell} 
\usepackage[tableposition=below]{caption}

\usepackage{silence}
\usepackage{tikz}
\usetikzlibrary{shapes}
\usetikzlibrary{calc}
\usetikzlibrary{knots}
\usepackage{braids}

\usepackage{hyperref}
\hypersetup{bookmarks=true, bookmarksopen=true,%
    bookmarksdepth=3,bookmarksopenlevel=2,%
    colorlinks=true,%
    linkcolor=medium-blue,%
    citecolor=medium-blue,%
    filecolor=blue,%
    menucolor=blue,%
    urlcolor=medium-blue,%
    hypertexnames=false}
\hypersetup{pdftitle={Towards the quantum exceptional series}}
\hypersetup{pdfauthor={Kim Morrison, Noah Snyder, and Dylan P. Thurston}}

\newcommand{\RR}{\mathbb R}
\newcommand{\CC}{\mathbb C}
\newcommand{\ZZ}{\mathbb Z}
\newcommand{\QQ}{\mathbb Q}

\newcommand{\HH}{\mathbb H}

\newcommand{\fg}{{\mathfrak g}}
\newcommand{\co}{\colon\thinspace}

\newcommand{\abs}[1]{{\lvert #1 \rvert}}

\DeclareMathOperator{\End}{End}

\DeclareMathOperator{\Sym}{Sym}
\DeclareMathOperator{\tr}{tr}
\DeclareMathOperator{\Hom}{Hom}

\DeclareMathOperator{\SO}{\mathit{SO}}
\DeclareMathOperator{\PSO}{\mathit{PSO}}

\DeclareMathOperator{\PSL}{\mathit{PSL}}
\DeclareMathOperator{\SOSp}{\mathit{SOSp}}

\DeclareMathOperator{\Rep}{\mathsf{Rep}}
\newcommand{\Vect}{\mathrm{Vec}}

\newcommand{\arXiv}[1]{\href{http://arxiv.org/abs/#1}{\tt \nolinkurl{arXiv:#1}}}
\newcommand{\doi}[1]{\href{http://dx.doi.org/#1}{{\tt \nolinkurl{DOI:#1}}}}
\newcommand{\oeis}[1]{\href{http://oeis.org/#1}{{#1}}}

\theoremstyle{plain}
\newtheorem{theorem}{Theorem}
\newtheorem{proposition}{Proposition}
\numberwithin{proposition}{section}
\newtheorem{lemma}[proposition]{Lemma}
\newtheorem{corollary}[proposition]{Corollary}

\newtheorem{conjecture}[proposition]{Conjecture}

\theoremstyle{definition}
\newtheorem{definition}[proposition]{Definition}

\newtheorem{convention}[proposition]{Convention}
\newtheorem{warning}[proposition]{Warning}

\theoremstyle{remark}
\newtheorem{example}[proposition]{Example}
\newtheorem{remark}[proposition]{Remark}

\DeclareMathOperator{\eval}{eval}
\DeclareMathOperator{\Qeval}{Qeval}
\newcommand{\QExc}{\mathsf{QExc}}

\DeclareMathOperator{\Tw}{Tw}
\DeclareMathOperator{\HTw}{HTw}
\DeclareMathOperator{\Fr}{Fr}

\DeclareMathOperator{\fork}{fork}
\DeclareMathOperator{\fuse}{fuse}

\makeatletter
\newcommand\mi@kern[1]{%
  \settowidth\@tempdima{$\mi@obj^{#1}$}
  \kern-\@tempdima
  #1
  \settowidth\@tempdima{$\mi@obj$}
  \kern\@tempdima
}

\newtoks\mi@toksp
\newtoks\mi@toksb
\DeclareRobustCommand{\manyindices}[5]{
  \def\mi@obj{#5}
  \mi@toksp\expandafter{\mi@kern{#2}}
  \mi@toksb\expandafter{\mi@kern{#1}}
  \@mathmeasure4\textstyle{#5_{#1}^{#2}}
  \@mathmeasure6\textstyle{#5_{#3}^{#4}}
  \dimen0-\wd6 \advance\dimen0\wd4
  \@mathmeasure8\textstyle{\hphantom{{}_{#1}^{#2}}#5^{\the\mi@toksp#4}_{\the\mi@toksb#3}}
  \hbox to \dimen0{}{\kern-\dimen0\box8}
}
\makeatother 

\def\semicolon{;}
\def\applytolist#1{
    \expandafter\def\csname multi#1\endcsname##1{
        \def\multiack{##1}\ifx\multiack\semicolon
            \def\next{\relax}
        \else
            \csname #1\endcsname{##1}
            \def\next{\csname multi#1\endcsname}
        \fi
        \next}
    \csname multi#1\endcsname}

\def\calc#1{\expandafter\def\csname c#1\endcsname{{\mathcal #1}}}
\applytolist{calc}QWERTYUIOPLKJHGFDSAZXCVBNM;
\def\bbc#1{\expandafter\def\csname bb#1\endcsname{{\mathbb #1}}}
\applytolist{bbc}QWERTYUIOPLKJHGFDSAZXCVBNM;
\def\bfc#1{\expandafter\def\csname bf#1\endcsname{{\mathbf #1}}}
\applytolist{bfc}QWERTYUIOPLKJHGFDSAZXCVBNM;

\DeclareMathAlphabet{\mathbbold}{U}{bbold}{m}{n}
\newcommand{\bbOne}{\mathbbold{1}}
\newcommand{\bfOne}{\mathbf{1}}

\graphicspath{{diagrams/}}

\newcommand{\mathfig}[2]{{\hspace{-3pt}\begin{array}{c}%
  \raisebox{-2.5pt}{\includegraphics[width=#1\textwidth]{#2}}%
\end{array}\hspace{-3pt}}}

\newread\testin

\def\mathcenter#1{%
  \vcenter{\hbox{$#1$}}%
}

\newcommand{\diagram}[2]{\mathfig{#1}{graphs/urn_sha1_#2.pdf}}

\usepackage{fp}
\usepackage{tikz}
\usetikzlibrary{matrix}
\usetikzlibrary{arrows,backgrounds,patterns,scopes,external,hobby,
    decorations.pathreplacing,
    decorations.pathmorphing
}
\usepackage{tikzit}

\newlength{\fuzzwidth}
\newlength{\arrowlength}
\newlength{\arrowwidth}
\newlength{\pointrad}
\newlength{\linewid}
\newlength{\circlerad}
\newlength{\smcirclerad}
\newcommand{\fuzzcolor}{black!25}
\newcommand{\arrowcolor}{black!25}
\newcommand{\covercolor}{black!0}
\newcommand{\graycolor}{black!55}
\newcommand{\graylightcolor}{black!40}

\newcommand{\coverwidthfuzz}{6pt}
\newcommand{\coverwidth}{3.5pt}
\newcommand{\coverwidththin}{3.25pt}
\newcommand{\coverwidththick}{3.75pt}

\newlength{\linewidthin}
\newlength{\linewidthick}
\tikzset{cdlabel/.style={execute at begin node=$\scriptstyle,execute at end node=$}}

\tikzset{use Hobby shortcut}
\tikzset{
	coverline/.style={
	preaction={draw,line width=\coverwidth,\covercolor}}, 
	coverlinethin/.style={
	preaction={draw,line width=\coverwidththin,\covercolor}}, 
	coverlinethick/.style={
	preaction={draw,line width=\coverwidththick,\covercolor}}, 
	coverlineleft/.style={
	preaction={draw,line width=\coverwidthfuzz,\covercolor,decorate,decoration={curveto,amplitude=0,raise=.35*\fuzzwidth}}}, 
coverlinelefttail/.style={
	preaction={draw,line width=\coverwidthfuzz,\covercolor,decorate,decoration={curveto,amplitude=0,raise=.35*\fuzzwidth,pre=moveto,pre length=2pt}}}, 
        fuzzlefttail/.style={
        preaction={draw,line width=\fuzzwidth,\fuzzcolor,decorate,decoration={curveto,pre=moveto,pre length=2pt,amplitude=0,raise=.5*\fuzzwidth}}}, 
        linestylethin/.style={line width=\linewidthin},
        linestylethick/.style={line width=\linewidthick},
        linestylegray/.style={line width=\linewid,\graycolor},
        linestylegraylight/.style={line width=\linewid,\graylightcolor}
}
\tikzset{
        fuzzright/.style={
        preaction={draw,line width=\fuzzwidth,\fuzzcolor,decorate,decoration={curveto,amplitude=0,raise=-.5*\fuzzwidth}}},
        fuzzleft/.style={
        preaction={draw,line width=\fuzzwidth,\fuzzcolor,decorate,decoration={curveto,amplitude=0,raise=.5*\fuzzwidth}}},
        fuzzrightpre/.style={ 
        preaction={draw,line width=2pt,\fuzzcolor,decorate,decoration={curveto,amplitude=0,raise=-1pt,pre=moveto,pre length=12pt}}},
        fuzzleftpre/.style={ 
        preaction={draw,line width=2pt,\fuzzcolor,decorate,decoration={curveto,post=moveto,post length=32pt,amplitude=0,raise=1pt}}},        
        outstyle/.style={\arrowcolor, line width=\arrowwidth},
        linestyle/.style={line width=\linewid}
}

\newcommand{\newfig}[3]{
  \newsavebox{#2}
  \savebox{#2}{#3}
  \newcommand{#1}{\usebox{#2}}
}

\newcommand{\smallfig}[1]{\mathcenter{\scalebox{0.5}{#1}}}

\newfig{\twistedsquarehor}{\twistedsquarehorbox}{
\begin{tikzpicture}[baseline=-.5ex,scale=.8]
\draw (45:.8cm) -- (-135:.8cm);
\draw[line width=1mm,white,double=black] (-45:.8cm) -- (135:.8cm);
\draw (45:.5cm) .. controls +(-75:.3cm) and +(75:.3cm) .. (-45:.5cm);
\draw (135:.5cm) .. controls +(-105:.3cm) and +(105:.3cm) .. (-135:.5cm);
\end{tikzpicture}}

\newfig{\twistedsquarever}{\twistedsquareverbox}{
\begin{tikzpicture}[baseline=-.5ex,scale=.8]
\draw (45:.8cm) -- (-135:.8cm);
\draw[line width=1mm,white,double=black] (-45:.8cm) -- (135:.8cm);
\draw (45:.5cm) .. controls +(165:.3cm) and +(15:.3cm) .. (135:.5cm);
\draw (-45:.5cm) .. controls +(-165:.3cm) and +(-15:.3cm) .. (-135:.5cm);
\end{tikzpicture}}

\newcommand{\ngon}[2][0]{
\begin{tikzpicture}[baseline=-0.5ex,scale=0.8]
\foreach \x in {1, ..., #2}
	\draw (360*\x/#2+#1:.8cm)--(360*\x/#2+#1:.5cm);
\foreach \x in {1, ..., #2}
	\draw (360*\x/#2+#1:.5cm) .. controls +(360*\x/#2+120+#1:.3cm) and +(360*\x/#2+360/#2-120+#1:.3cm) .. (360*\x/#2+360/#2+#1:.5cm);
\end{tikzpicture}
}

\newcommand{\nvertex}[2][0]{
\begin{tikzpicture}[baseline=-0.5ex,scale=0.8]
\foreach \x in {1, ..., #2}
	\draw (360*\x/#2+#1:.8cm)--(0,0);
\end{tikzpicture}
}

\newfig{\twogon}{\twogonbox}{\ngon{2}}
\newfig{\threegon}{\threegonbox}{\ngon[-90]{3}}
\newfig{\threevertex}{\threevertexbox}{\nvertex[-90]{3}}
\newfig{\fourgon}{\fourgonbox}{\ngon[45]{4}}

\newfig{\upsidedownthreevertex}{\upsidedownthreevertexbox}{%
\begin{tikzpicture}[baseline=-0.5ex,scale=0.8]
	\draw (330:.8cm)--(0,0);
	\draw (90:.8cm)--(0,0);
	\draw (210:.8cm)--(0,0);
\end{tikzpicture}}

\newfig{\unknot}{\unknotbox}{%
\begin{tikzpicture}[baseline=-0.5ex,scale=0.8]
  \draw (0,0) circle (.6cm);
\end{tikzpicture}}

\newfig{\drawI}{\drawIbox}{\begin{tikzpicture}[baseline=-0.5ex,scale=0.8]
 	\draw (0,.2) -- (45:.8cm);
 	\draw (0,.2) -- (135:.8cm);
	\draw (0,.2) -- (0,-.2);
 	\draw (0,-.2) -- (-45:.8cm);
 	\draw (0,-.2) -- (-135:.8cm);
\end{tikzpicture}}

\newfig{\drawH}{\drawHbox}{\begin{tikzpicture}[baseline=-0.5ex,rotate=90,scale=0.8]
 	\draw (0,.2) -- (45:.8cm);
 	\draw (0,.2) -- (135:.8cm);
	\draw (0,.2) -- (0,-.2);
 	\draw (0,-.2) -- (-45:.8cm);
 	\draw (0,-.2) -- (-135:.8cm);
\end{tikzpicture}}

\newfig{\drawdottedH}{\drawdottedHbox}{\begin{tikzpicture}[baseline=-0.5ex,scale=0.8]
        \draw (45:.8cm) -- (-45:.8cm);
        \draw (135:.8cm) -- (-135:.8cm);
        \draw[dotted] (-0.57cm,0) -- (0.57cm, 0);
\end{tikzpicture}}

\newfig{\onestrandid}{\onestrandidbox}{\begin{tikzpicture}[baseline=-0.5ex,scale=0.8]
	\draw (-.8cm,0)--(.8cm,0);
\end{tikzpicture}}

\newfig{\vertonestrandid}{\vertonestrandidbox}{\begin{tikzpicture}[rotate=90,baseline=-0.5ex,scale=0.8]
	\draw (-.8cm,0)--(.8cm,0);
\end{tikzpicture}}

\newfig{\drawcup}{\drawcupbox}{\begin{tikzpicture}[baseline=-0.5ex,scale=0.8]
	\draw (45:.8cm) to [curve through=(90:0cm)] (135:.8cm);
\end{tikzpicture}}

\newfig{\drawcap}{\drawcapbox}{\begin{tikzpicture}[baseline=-0.5ex,scale=0.8]
	\draw (-45:.8cm) to [curve through=(-90:0cm)] (-135:.8cm);
\end{tikzpicture}}

\newfig{\drawbubblecap}{\drawbubblecapbox}{\begin{tikzpicture}[baseline=-0.5ex,scale=0.8]
	\draw (180:.8cm) to [curve through = (135: .8cm)] (120:.8cm);
	\draw (0:.8cm) to [curve through = (45: .8cm)] (60:.8cm);
	\draw (120:.8cm) to [curve through = (90:1cm)] (60:.8cm);
	\draw (120:.8cm) to [curve through = (90:.4cm)] (60:.8cm);
\end{tikzpicture}}

\newfig{\cupcap}{\cupcapbox}{\begin{tikzpicture}[baseline=-0.5ex,scale=0.8]
	\draw (45:.8cm) to [curve through=(90:.3cm)] (135:.8cm);
	\draw (-45:.8cm) to [curve through=(-90:.3cm)] (-135:.8cm);
\end{tikzpicture}}

\newfig{\twostrandid}{\twostrandidbox}{\begin{tikzpicture}[baseline=-0.5ex,rotate=90,scale=0.8]
	\draw (45:.8cm) to [curve through=(90:.3cm)] (135:.8cm);
	\draw (-45:.8cm) to [curve through=(-90:.3cm)] (-135:.8cm);
\end{tikzpicture}}

\newfig{\symcross}{\symcrossbox}{\begin{tikzpicture}[baseline=-0.5ex,scale=0.8]
	\draw (45:.8cm) -- (-135:.8cm);
	\draw (-45:.8cm) -- (135:.8cm);
\end{tikzpicture}}

\newfig{\braidcross}{\braidcrossbox}{\begin{tikzpicture}[baseline=-0.5ex,scale=0.8]
	\draw (45:.8cm) -- (-135:.8cm);
	\draw[line width=1mm,white,double=black] (-45:.8cm) -- (135:.8cm);
\end{tikzpicture}}

\newfig{\invbraidcross}{\invbraidcrossbox}{\begin{tikzpicture}[baseline=-0.5ex,scale=0.8]
	\draw (-45:.8cm) -- (135:.8cm);
	\draw[line width=1mm,white,double=black] (45:.8cm) -- (-135:.8cm);
\end{tikzpicture}}

\newfig{\drawcrossX}{\drawcrossXbox}{\begin{tikzpicture}[baseline=-0.5ex,scale=0.8,rotate=90]
        \draw ([out angle=-150].2,-.3) ..([in angle=-70]135:.8cm);
 	\draw[line width=2mm,white]
              ([out angle=150].2,.3) .. ([in angle=70]-135:.8cm);
 	\draw ([out angle=150].2,.3) .. ([in angle=70]-135:.8cm);
 	\draw ([out angle=30].2,.3) .. (45:.8cm);
	\draw (.2,.3) -- (.2,-.3);
 	\draw ([out angle=-30].2,-.3) .. (-45:.8cm);
\end{tikzpicture}}

\newfig{\drawinvcrossX}{\drawinvcrossXbox}{\begin{tikzpicture}[baseline=-0.5ex,scale=0.8,rotate=90]
 	\draw ([out angle=150].2,.3) .. ([in angle=70]-135:.8cm);
        \draw[line width=2mm,white] ([out angle=-150].2,-.3) ..([in angle=-70]135:.8cm);
        \draw([out angle=-150].2,-.3) ..([in angle=-70]135:.8cm);
	\draw ([out angle=30].2,.3) .. (45:.8cm);
	\draw (.2,.3) -- (.2,-.3);
 	\draw ([out angle=-30].2,-.3) .. (-45:.8cm);
\end{tikzpicture}}

\newfig{\drawsymcrossX}{\drawsymcrossXbox}{\begin{tikzpicture}[baseline=-0.5ex,scale=0.8,rotate=90]
 	\draw ([out angle=30].2,.3) .. (45:.8cm);
	\draw (.2,.3) -- (.2,-.3);
 	\draw ([out angle=-30].2,-.3) .. (-45:.8cm);
        \draw ([out angle=-150].2,-.3) ..([in angle=-70]135:.8cm);
 	\draw
              ([out angle=150].2,.3) .. ([in angle=70]-135:.8cm);
\end{tikzpicture}}

\newfig{\drawsyminvcrossX}{\drawsyminvcrossXbox}{\begin{tikzpicture}[baseline=-0.5ex,scale=0.8,rotate=-90]
 	\draw ([out angle=30].2,.3) .. (45:.8cm);
	\draw (.2,.3) -- (.2,-.3);
 	\draw ([out angle=-30].2,-.3) .. (-45:.8cm);
        \draw ([out angle=-150].2,-.3) ..([in angle=-70]135:.8cm);
 	\draw
              ([out angle=150].2,.3) .. ([in angle=70]-135:.8cm);
\end{tikzpicture}}

\newfig{\twist}{\twistbox}{%
  \begin{tikzpicture}[baseline=-0.5ex,scale=0.8]
    \draw[line width=1mm,white,double=black]
       ([out angle=-15]135:.8cm)..([blank=soft]-60:.4cm)..(-120:.4cm)..([in angle=-165]45:.8cm);
    \draw[line width=1mm,white,double=black,use previous Hobby path={invert soft blanks,disjoint}];
  \end{tikzpicture}}

\newfig{\symtwist}{\symtwistbox}{%
  \begin{tikzpicture}[baseline=-0.5ex,scale=0.8]
    \draw
       ([out angle=-15]135:.8cm)..([blank=soft]-60:.4cm)..(-120:.4cm)..([in angle=-165]45:.8cm);
   \draw[use previous Hobby path={invert soft blanks,disjoint}];
  \end{tikzpicture}}

\newfig{\twistvertex}{\twistvertexbox}{%
\begin{tikzpicture}[baseline=-0.5ex,scale=0.8]
  \draw (0,-0.4)--(0,-0.8);
  \draw ([out angle=-170]30:0.8cm)..([in angle=160](0,-0.4);
  \draw[line width=1mm,white,double=black]
        ([out angle=-10]150:0.8cm)..([in angle=20](0,-0.4);
\end{tikzpicture}}

\newfig{\symtwistvertex}{\symtwistvertexbox}{%
\begin{tikzpicture}[baseline=-0.5ex,scale=0.8]
  \draw (0,-0.4)--(0,-0.8);
  \draw ([out angle=-170]30:0.8cm)..([in angle=160](0,-0.4);
  \draw
        ([out angle=-10]150:0.8cm)..([in angle=20](0,-0.4);
\end{tikzpicture}}

\newfig{\loopvertex}{\loopvertexbox}{%
\begin{tikzpicture}[baseline=-0.5ex,scale=0.8]
  \draw (0,-0.5)--(0,-0.8);
  \draw ([out angle=150]0,-0.5)..(-0.02,0.6)..(0.02,0.6)..([in angle=30]0,-0.5);
\end{tikzpicture}}

\newfig{\idtangle}{\idtanglebox}{%
\begin{tikzpicture}[baseline=-0.5ex,scale=0.8]
\draw (0,0) node {$D$};
\draw[densely dashed] (0,0) circle (.5cm);
\draw (45:.5cm) -- (45:0.8cm);
\draw (135:.5cm) -- (135:0.8cm);
\draw (-45:.5cm) -- (-45:0.8cm);
\draw (-135:.5cm) -- (-135:0.8cm);
\end{tikzpicture}}
    
\newfig{\Rcycle}{\Rcyclebox}{%
\begin{tikzpicture}[baseline=-0.5ex,scale=0.8]
\draw (0,0) node {$D$};
\draw[densely dashed] (0,0) circle (.4cm);
\draw (45:.4cm) -- (45:1.1cm);
\draw ([out angle=-70]-45:.4cm) .. (-90:.7cm) .. ([in angle=0]-135:1.1cm);
\draw ([out angle=-160]-135:.4cm) .. (180:.7cm) .. ([in angle=-90]135:1.1cm);
\draw[line width=0.8mm,white,double=black]
      ([out angle=90]135:.4cm) .. (45:.7cm) .. ([in angle=90]-45:1.1cm);
\end{tikzpicture}}

\newfig{\ladder}{\ladderbox}{%
\begin{tikzpicture}[baseline=-0.5ex,scale=0.8]
\draw (0,0) node {$D$};
\draw[densely dashed] (0,0) circle (.5cm);
\draw (135:.5cm) -- (135:0.9cm);
\draw (-135:.5cm) -- (-135:0.9cm);
\draw ([out angle=45]45:.5cm) .. ([in angle=150]0.8,0.45);
\draw ([out angle=-45]-45:.5cm) .. ([in angle=-150]0.8,-0.45);
\draw (0.8,0.45)--(0.8,-0.45);
\draw ([out angle=30]0.8,0.45) .. ([in angle=-135]1.2,0.636);
\draw ([out angle=-30]0.8,-0.45) .. ([in angle=135]1.2,-0.636);
\end{tikzpicture}}

\newfig{\sqsq}{\sqsqbox}{%
\begin{tikzpicture}[scale=.15]
\draw (0,0) -- (2,0) -- (2,1) -- (0,1) -- cycle (1,0) -- (1,1);
\end{tikzpicture}}

\newfig{\overviolin}{\overviolinbox}{%
\begin{tikzpicture}[baseline=0]
	\coordinate (E1) at (45:.8cm);
	\coordinate (E2) at (225:.8cm);
	\coordinate (E3) at (270:.8cm);
	\coordinate (P1) at (0,0);
	\coordinate (P2) at (120:.6cm);
	\coordinate (P3) at (.4cm,.2cm);
	\coordinate (C1) at (100:.9cm);
	\coordinate (C2) at (135:.5cm);
	\coordinate (C3) at (.8,.1);
	\draw (E1) -- (P1) -- (E3);
	\draw (P1) .. controls  (C1) .. (P2);
	\draw[white, line width=4pt] (P2) .. controls  (C2) .. (P3);	
	\draw[white, line width=4pt] (P3) .. controls (C3) .. (E2);
	\draw (P2) .. controls  (C2) .. (P3);	
	\draw (P3) .. controls (C3) .. (E2);
\end{tikzpicture}}

\newfig{\rotatedtrivalent}{\rotatedtrivalentbox}{%
\begin{tikzpicture}[baseline=0]
	\coordinate (E1) at (45:.8cm);
	\coordinate (E2) at (225:.8cm);
	\coordinate (E3) at (270:.8cm);
	\coordinate (P0) at (0,0);
	\coordinate (C1) at (-.1,.7);
	\coordinate (C2) at (-.3,-.2);
	\coordinate (C3) at (.2,-.2);
	\draw (P0) .. controls (C1) .. (E2);
	\draw (P0) .. controls (C2) .. (E3);
	\draw (P0) .. controls (C3) .. (E1);
\end{tikzpicture}}

\newcommand{\pentagon}{%
\begin{tikzpicture}[baseline=0, use Hobby shortcut]
	\coordinate (E1) at (90-72:.8cm);
	\coordinate (E2) at (90:.8cm);
	\coordinate (E3) at (90+72:.8cm);
	\coordinate (E4) at (90+2*72:.8cm);
	\coordinate (E5) at (90+3*72:.8cm);
	\coordinate (F1) at (90-72:.4cm);
	\coordinate (F2) at (90:.4cm);
	\coordinate (F3) at (90+72:.4cm);
	\coordinate (F4) at (90+2*72:.4cm);
	\coordinate (F5) at (90+3*72:.4cm);
        \draw (F1)--(F2)--(F3)--(F4)--(F5)--cycle;
        \draw (E1)--(F1);
        \draw (E2)--(F2);
        \draw (E3)--(F3);
        \draw (E4)--(F4);
        \draw (E5)--(F5);
\end{tikzpicture}}

\newcommand{\pentasquare}{%
\begin{tikzpicture}[baseline=0, use Hobby shortcut]
	\coordinate (E1) at (90-72:.8cm);
	\coordinate (E2) at (90:.8cm);
	\coordinate (E3) at (90+72:.8cm);
	\coordinate (E4) at (90+2*72:.8cm);
	\coordinate (E5) at (90+3*72:.8cm);
	\coordinate (F1) at (90-72:.4cm);
	\coordinate (F2) at (90:.4cm);
	\coordinate (F3) at (90+72:.4cm);
	\coordinate (F4) at ($(90+2*72:.3cm)+(0,0.1cm)$);
	\coordinate (F5) at ($(90+3*72:.3cm)+(0,0.1cm)$);
        \coordinate (G4) at ($(F4)-(0,.3cm)$);
        \coordinate (G5) at ($(F5)-(0,.3cm)$);
        \draw (F1)--(F2)--(F3)--(F4)--(F5)--cycle;
        \draw (E1)--(F1);
        \draw (E2)--(F2);
        \draw (E3)--(F3);
        \draw (E4)--(G4)--(F4);
        \draw (E5)--(G5)--(F5);
        \draw (G4)--(G5);
\end{tikzpicture}}

\newcommand{\twisttreebase}{%
	\coordinate (E1) at (90-72:.8cm);
	\coordinate (E2) at (90:.8cm);
	\coordinate (E3) at (90+72:.8cm);
	\coordinate (E4) at (90+2*72:.8cm);
	\coordinate (E5) at (90+3*72:.8cm);
	\coordinate (F1) at (90-72:.4cm);
	\coordinate (F2) at (90:.4cm);
	\coordinate (F3) at (90+72:.4cm);
        \draw (E4)--(F1);
	\draw[white, line width = 4pt] (E5) -- (F3);
        \draw (E5) -- (F3);
        \draw (F1)--(F2)--(F3);
        \draw (E1)--(F1);
        \draw (E2)--(F2);
        \draw (E3)--(F3);
}

\newcommand{\twisttreefourfive}{%
\begin{tikzpicture}[baseline=0, use Hobby shortcut]
\twisttreebase
\end{tikzpicture}}

\newcommand{\twisttreefivefour}{%
\begin{tikzpicture}[baseline=0, use Hobby shortcut,xscale=-1]
\twisttreebase
\end{tikzpicture}}

\newcommand{\twisttreefourthree}{%
\begin{tikzpicture}[baseline=0, use Hobby shortcut,rotate=-72,xscale=-1]
\twisttreebase
\end{tikzpicture}}

\newcommand{\casimirbubble}{%
\begin{tikzpicture}[baseline=-.5ex,scale=0.8]
\draw[->]  (0,-1) -- (0,-.75);
\draw[->]  (0,-.75) -- (0,0);
\draw[->] (0,0) -- (0,.87);
\draw  (0,.87) -- (0,1);
\draw (0,-.5) .. controls (-.5,-.25) and (-.5,.25) .. (0,.5) ;
\end{tikzpicture}}

\newcommand{\orientedonestrandid}{%
\begin{tikzpicture}[baseline=-.5ex,scale=0.8]
\draw[->]  (0,-1) -- (0,0);
\draw  (0,0) -- (0,1);
\end{tikzpicture}}

\newcommand{\ladderaa}{%
\begin{tikzpicture}[baseline=-1.5cm]
		\node[draw, rectangle] (coupon) at (-3.75, -1.5) {$\pi_{V_\mu}$};
		\node  (SW) at (-4, -2.5) {};
		\node  (SE) at (-3.5, -2.5) {};
		\node  (NW) at (-4, -1) {};
		\node  (NE) at (-3.5, -1) {};
		\node  (NofNW) at (-4, 0) {};
		\node  (NofNE) at (-3.5, 0) {};
		\node  (NNofNW) at (-4, .5) {};
		\node  (NNofNE) at (-3.5, .5) {};
		\draw (SW.center) to (coupon.230);
		\draw (SE.center) to (coupon.310);
		\draw (coupon.130) to (NW.center);
		\draw (coupon.50) to (NE.center);
		\draw (NW.center) to (NofNW.center);
		\draw (NE.center) to (NofNE.center);
		\draw [bend left=75, looseness=1] (NW.center) to (NofNW.center);
		\draw (NofNW.center) to (NNofNW.center);
		\draw (NofNE.center) to (NNofNE.center);
\end{tikzpicture}}

\newcommand{\ladderab}{%
\begin{tikzpicture}[baseline=-1.5cm]
		\node[draw, rectangle] (coupon) at (-3.75, -1.5) {$\pi_{V_\mu}$};
		\node  (SW) at (-4, -2.5) {};
		\node  (SE) at (-3.5, -2.5) {};
		\node  (NW) at (-4, -1) {};
		\node  (NE) at (-3.5, -1) {};
		\node  (NofNW) at (-4, 0) {};
		\node  (NofNE) at (-3.5, 0) {};
		\node  (NNofNW) at (-4, .5) {};
		\node  (NNofNE) at (-3.5, .5) {};
		\draw (SW.center) to (coupon.230);
		\draw (SE.center) to (coupon.310);
		\draw (coupon.130) to (NW.center);
		\draw (coupon.50) to (NE.center);
		\draw (NW.center) to (NofNW.center);
		\draw (NE.center) to (NofNE.center);
		\draw [bend left=70, looseness=2] (NW.center) to (NofNE.center);
		\draw (NofNW.center) to (NNofNW.center);
		\draw (NofNE.center) to (NNofNE.center);
\end{tikzpicture}}

\newcommand{\ladderba}{%
\begin{tikzpicture}[baseline=-1.5cm]
		\node[draw, rectangle] (coupon) at (-3.75, -1.5) {$\pi_{V_\mu}$};
		\node  (SW) at (-4, -2.5) {};
		\node  (SE) at (-3.5, -2.5) {};
		\node  (NW) at (-4, -1) {};
		\node  (NE) at (-3.5, -1) {};
		\node  (NofNW) at (-4, 0) {};
		\node  (NofNE) at (-3.5, 0) {};
		\node  (NNofNW) at (-4, .5) {};
		\node  (NNofNE) at (-3.5, .5) {};
		\draw (SW.center) to (coupon.230);
		\draw (SE.center) to (coupon.310);
		\draw (coupon.130) to (NW.center);
		\draw (coupon.50) to (NE.center);
		\draw (NW.center) to (NofNW.center);
		\draw (NE.center) to (NofNE.center);
		\draw [bend left=70, looseness=2] (NE.center) to (NofNW.center);
		\draw (NofNW.center) to (NNofNW.center);
		\draw (NofNE.center) to (NNofNE.center);
\end{tikzpicture}}

\newcommand{\ladderbb}{%
\begin{tikzpicture}[baseline=-1.5cm]
		\node[draw, rectangle] (coupon) at (-3.75, -1.5) {$\pi_{V_\mu}$};
		\node  (SW) at (-4, -2.5) {};
		\node  (SE) at (-3.5, -2.5) {};
		\node  (NW) at (-4, -1) {};
		\node  (NE) at (-3.5, -1) {};
		\node  (NofNW) at (-4, 0) {};
		\node  (NofNE) at (-3.5, 0) {};
		\node  (NNofNW) at (-4, .5) {};
		\node  (NNofNE) at (-3.5, .5) {};
		\draw (SW.center) to (coupon.230);
		\draw (SE.center) to (coupon.310);
		\draw (coupon.130) to (NW.center);
		\draw (coupon.50) to (NE.center);
		\draw (NW.center) to (NofNW.center);
		\draw (NE.center) to (NofNE.center);
		\draw [bend left=90, looseness=2.5] (NE.center) to (NofNE.center);
		\draw (NofNW.center) to (NNofNW.center);
		\draw (NofNE.center) to (NNofNE.center);
\end{tikzpicture}}

\newcommand{\unitcap}{%
\begin{tikzpicture}[baseline=.5cm]
	\node (W) at (-.5, 0) {$\bullet$};
	\node (E) at (.5, 0) {$\bullet$};
	\draw [bend left=90, looseness=2] (W.center) to (E.center);
\end{tikzpicture}}

\newcommand{\twocap}{%
\begin{tikzpicture}[baseline= .5cm]
	\node (1) at (-1, 0) {};
	\node (2) at (-.25, 0) {};
	\node (3) at (.25, 0) {};
	\node (4) at (1, 0) {};
	\draw [bend left=90, looseness=2] (1) to (2);
	\draw [bend left=90, looseness=2] (3) to (4);
\end{tikzpicture}}

\newcommand{\HurwitzTerma}{%
\begin{tikzpicture}[baseline=.5cm]
	\node (a) at (-1.5, 0) {};
	\node (b) at (-.5, 0) {};
	\node (c) at (.5, 0) {};
	\node (d) at (1.5, 0) {};
	\node [draw, rectangle] (1) at (-.5,1) {$\mu$};
	\node [draw, rectangle] (2) at (.5,1) {$\mu$};
	\draw (a) to (1.250);
	\draw (c) to (1.290);
	\draw (b) to (2.250);
	\draw (d) to (2.290);
	\draw [bend left=90, looseness=2] (1) to (2);
\end{tikzpicture}}

\newcommand{\HurwitzTermb}{%
\begin{tikzpicture}[baseline=.5cm]
	\node (a) at (-1.5, 0) {};
	\node (b) at (-.5, 0) {};
	\node (c) at (.5, 0) {};
	\node (d) at (1.5, 0) {};
	\node [draw, rectangle] (1) at (-.5,1) {$\mu$};
	\node [draw, rectangle] (2) at (.5,1) {$\mu$};
	\draw (a) to (1.250);
	\draw (d) to (1.290);
	\draw (b) to (2.250);
	\draw (c) to (2.290);
	\draw [bend left=90, looseness=2] (1) to (2);
\end{tikzpicture}}

\newcommand{\TracialJordanAxiom}{%
\begin{tikzpicture}[baseline=2cm]
	\node [text height=1.5ex] (a) at (-.5,0) {$a$};
	\node [text height=1.5ex] (b) at (0,0) {$b$};
	\node [text height=1.5ex] (c) at (.5,0) {$c$};
	\node [draw, rectangle] (sym) at  (0,1) {$\; \Sym^3 \;$};
	\node [draw, rectangle] (wedge) at (-.25,3) {$\; \wedge^2 \;$};
	\node [text height=1.5ex] (y) at (-.5,4) {$y$};
	\node (output) at (0,4) {};
	\node (r) at (-.5, 2) {};
	\node (s) at (0,2) {};
	\node (t) at (.25,1.75) {};
	\draw (a) to (sym.216);
	\draw (b) to (sym.270);
	\draw (c) to (sym.324);
	\draw (r.center) to (sym.144);
	\draw (r.center) to (s.center);
	\draw (s.center) to (t.center);
	\draw (sym.90) to (t.center);
	\draw (sym.36) to (t.center);
	\draw (r.center) to (wedge.233);
	\draw (s.center) to (wedge.307);
	\draw (wedge.127) to (y);
	\draw (wedge.53) to (output);
\end{tikzpicture}}

\newcommand{\TracialJordanAxiomLonga}{%
\begin{tikzpicture}[baseline=2cm]
	\node [text height=1.5ex] (a) at (-.5,0) {$a$};
	\node [text height=1.5ex] (b) at (0,0) {$b$};
	\node [text height=1.5ex] (c) at (.5,0) {$c$};
	\node [draw, rectangle] (sym) at  (0,1) {$\; \Sym^3 \;$};
	\node [text height=1.5ex] (y) at (-.5,4) {$y$};
	\node (output) at (0,4) {};
	\node (r) at (-.5, 2) {};
	\node (s) at (0,2) {};
	\node (t) at (.25,1.75) {};
	\draw (a) to (sym.216);
	\draw (b) to (sym.270);
	\draw (c) to (sym.324);
	\draw (r.center) to (sym.144);
	\draw (r.center) to (s.center);
	\draw (s.center) to (t.center);
	\draw (sym.90) to (t.center);
	\draw (sym.36) to (t.center);
	\draw (r.center) to (y);
	\draw (s.center) to (output);
\end{tikzpicture}}

\newcommand{\TracialJordanAxiomLongb}{%
\begin{tikzpicture}[baseline=2cm]
	\node [text height=1.5ex] (a) at (-.5,0) {$a$};
	\node [text height=1.5ex] (b) at (0,0) {$b$};
	\node [text height=1.5ex] (c) at (.5,0) {$c$};
	\node [draw, rectangle] (sym) at  (0,1) {$\; \Sym^3 \;$};
	\node [text height=1.5ex] (y) at (-.5,4) {$y$};
	\node (output) at (0,4) {};
	\node (r) at (-.5, 2) {};
	\node (s) at (0,2) {};
	\node (t) at (.25,1.75) {};
	\draw (a) to (sym.216);
	\draw (b) to (sym.270);
	\draw (c) to (sym.324);
	\draw (r.center) to (sym.144);
	\draw (r.center) to (s.center);
	\draw (s.center) to (t.center);
	\draw (sym.90) to (t.center);
	\draw (sym.36) to (t.center);
	\draw (r.center) to (output);
	\draw (s.center) to (y);
\end{tikzpicture}}

\newcommand{\TracelessJordanAxioma}{%
\begin{tikzpicture}[baseline=2cm]
	\node (a) at (-.5,0) {};
	\node (b) at (0,0) {};
	\node (c) at (.5,0) {};
	\node [draw, rectangle] (sym) at  (0,1) {$\; \Sym^3 \;$};
	\node [draw, rectangle] (wedge) at (-.25,3) {$\; \wedge^2 \;$};
	\node (y) at (-.5,4) {};
	\node (output) at (0,4) {};
	\node (r) at (-.5, 2) {};
	\node (s) at (0,2) {};
	\node (t) at (.25,1.75) {};
	\draw (a) to (sym.216);
	\draw (b) to (sym.270);
	\draw (c) to (sym.324);
	\draw (r.center) to (sym.144);
	\draw (r.center) to (s.center);
	\draw (s.center) to (t.center);
	\draw (sym.90) to (t.center);
	\draw (sym.36) to (t.center);
	\draw (r.center) to (wedge.233);
	\draw (s.center) to (wedge.307);
	\draw (wedge.127) to (y);
	\draw (wedge.53) to (output);
\end{tikzpicture}}

\newcommand{\TracelessJordanAxiomb}{%
\begin{tikzpicture}[baseline=2cm]
	\node (a) at (-.5,0) {};
	\node (b) at (0,0) {};
	\node (c) at (.5,0) {};
	\node [draw, rectangle] (sym) at  (0,1) {$\; \Sym^3 \;$};
	\node [draw, rectangle] (wedge) at (-.25,3) {$\; \wedge^2 \;$};
	\node (y) at (-.5,4) {};
	\node (output) at (0,4) {};
	\node (s) at (0,2) {};
	\draw (a) to (sym.216);
	\draw (b) to (sym.270);
	\draw (c) to (sym.324);
	\draw (wedge.233) to (sym.144);
	\draw (sym.90) to (s.center);
	\draw (sym.36) to (s.center);
	\draw (s.center) to (wedge.307);
	\draw (wedge.125) to (y);
	\draw (wedge.53) to (output);
\end{tikzpicture}}

\newcommand{\TracialThreefoldCHa}{%
\begin{tikzpicture}[baseline=1.5cm]
	\node (a) at (-.5,0) {};
	\node (b) at (0,0) {};
	\node (c) at (.5,0) {};
	\node [draw, rectangle] (sym) at  (0,1) {$\; \Sym^3 \;$};
	\node (d) at (0,3) {};
	\node (s) at (-.25,1.75) {};
	\node (t) at (0,2.25) {};
	\draw (a) to (sym.216);
	\draw (b) to (sym.270);
	\draw (c) to (sym.324);
	\draw (sym.144) to (s.center);
	\draw (sym.90) to (s.center);
	\draw (s.center) to (t.center);
	\draw (sym.35) to (t.center);
	\draw (t.center) to (d);
\end{tikzpicture}}

\newcommand{\TracialThreefoldCHb}{%
\begin{tikzpicture}[baseline=1.5cm]
	\node (a) at (-.5,0) {};
	\node (b) at (0,0) {};
	\node (c) at (.5,0) {};
	\node [draw, rectangle] (sym) at  (0,1) {$\; \Sym^3 \;$};
	\node (d) at (0,3) {};
	\node (s) at (-.25,2) {};
	\node (u) at (.5,2) {$\bullet$};
	\draw (a) to (sym.216);
	\draw (b) to (sym.270);
	\draw (c) to (sym.324);
	\draw (sym.144) to (s.center);
	\draw (sym.90) to (s.center);
	\draw (s.center) to (d);
	\draw (sym.36) to (u.center);
\end{tikzpicture}}

\newcommand{\TracialThreefoldCHc}{%
\begin{tikzpicture}[baseline=1.5cm]
	\node (a) at (-.5,0) {};
	\node (b) at (0,0) {};
	\node (c) at (.5,0) {};
	\node [draw, rectangle] (sym) at  (0,1) {$\; \Sym^3 \;$};
	\node (d) at (0,3) {};
	\node (u2) at (0,2) {$\bullet$};
	\node (u3) at (.5,2) {$\bullet$};
	\draw (a) to (sym.216);
	\draw (b) to (sym.270);
	\draw (c) to (sym.324);
	\draw (sym.144) to (d);
	\draw (sym.90) to (u2.center);
	\draw (sym.36) to (u3.center);	
\end{tikzpicture}}

\newcommand{\TracialThreefoldCHd}{%
\begin{tikzpicture}[baseline=1.5cm]
	\node (a) at (-.5,0) {};
	\node (b) at (0,0) {};
	\node (c) at (.5,0) {};
	\node [draw, rectangle] (sym) at  (0,1) {$\; \Sym^3 \;$};
	\node (d) at (0,3) {};
	\draw (a) to (sym.216);
	\draw (b) to (sym.270);
	\draw (c) to (sym.324);
	\draw (sym.144) to (d);
	\draw [bend left=90, looseness=2] (sym.90) to (sym.36);
\end{tikzpicture}}

\newcommand{\TracialThreefoldCHe}{%
\begin{tikzpicture}[baseline=1.5cm]
	\node (a) at (-.5,0) {};
	\node (b) at (0,0) {};
	\node (c) at (.5,0) {};
	\node [draw, rectangle] (sym) at  (0,1) {$\; \Sym^3 \;$};
	\node (d) at (0,3) {};
	\node (u) at (0,2.5) {$\bullet$};
	\node (s) at (0,2) {};
	\draw (a) to (sym.216);
	\draw (b) to (sym.270);
	\draw (c) to (sym.324);
	\draw (sym.144) to (s.center);
	\draw (sym.90) to (s.center);
	\draw (sym.36) to (s.center);
	\draw (u.center) to (d);	
\end{tikzpicture}}

\newcommand{\TracialThreefoldCHf}{%
\begin{tikzpicture}[baseline=1.5cm]
	\node (a) at (-.5,0) {};
	\node (b) at (0,0) {};
	\node (c) at (.5,0) {};
	\node [draw, rectangle] (sym) at  (0,1) {$\; \Sym^3 \;$};
	\node (d) at (0,3) {};
	\node (u1) at (-.5,1.75) {$\bullet$};
	\node (u2) at (0,2.25) {$\bullet$};	
	\draw (a) to (sym.216);
	\draw (b) to (sym.270);
	\draw (c) to (sym.324);
	\draw (sym.144) to (u1.center);
	\draw (u2.center) to (d);	
	\draw [bend left=90, looseness=2] (sym.90) to (sym.36);
\end{tikzpicture}}

\newcommand{\TracialThreefoldCHg}{%
\begin{tikzpicture}[baseline=1.5cm]
	\node (a) at (-.5,0) {};
	\node (b) at (0,0) {};
	\node (c) at (.5,0) {};
	\node [draw, rectangle] (sym) at  (0,1) {$\; \Sym^3 \;$};
	\node (d) at (0,3) {};
	\node (u1) at (-.5,1.75) {$\bullet$};
	\node (u2) at (0,1.75) {$\bullet$};
	\node (u3) at (.5,1.75) {$\bullet$};
	\node (u4) at (0,2.25) {$\bullet$};
	\draw (a) to (sym.216);
	\draw (b) to (sym.270);
	\draw (c) to (sym.324);
	\draw (sym.144) to (u1.center);
	\draw (sym.90) to (u2.center);
	\draw (sym.36) to (u3.center);
	\draw (u4.center) to (d);	
\end{tikzpicture}}

\newcommand{\firstthreebox}{%
\begin{tikzpicture}[baseline=-0.5ex,scale=0.8]
	\node (a) at (150:.8cm) {};
	\node (b) at (30:.8cm) {};
	\node (c) at (270:.8cm) {};
	\node (a') at (150:.2cm) {$\bullet$};
	\node (b') at (30:.2cm) {$\bullet$};
	\node (c') at (270:.2cm) {$\bullet$};
	\draw (a.center) to (a'.center);
	\draw (b.center) to (b'.center);
	\draw (c.center) to (c'.center);
\end{tikzpicture}}

\newcommand{\secondthreebox}{%
\begin{tikzpicture}[baseline=-0.5ex,scale=0.8]
	\node (a) at (150:.8cm) {};
	\node (b) at (30:.8cm) {};
	\node (c) at (270:.8cm) {};
	\node (a') at (150:.2cm) {$\bullet$};
	\draw (a.center) to (a'.center);
	\draw [bend right=45, looseness=.5] (b.center) to (c.center);
\end{tikzpicture}}

\newcommand{\thirdthreebox}{%
\begin{tikzpicture}[baseline=-0.5ex,scale=0.8]
	\node (a) at (150:.8cm) {};
	\node (b) at (30:.8cm) {};
	\node (c) at (270:.8cm) {};
	\node (b') at (30:.2cm) {$\bullet$};
	\draw [bend left=45, looseness=.5] (a.center) to (c.center);
	\draw (b.center) to (b'.center);
\end{tikzpicture}}

\newcommand{\fourththreebox}{%
\begin{tikzpicture}[baseline=-0.5ex,scale=0.8]
	\node (a) at (150:.8cm) {};
	\node (b) at (30:.8cm) {};
	\node (c) at (270:.8cm) {};
	\node (c') at (270:.2cm) {$\bullet$};
	\draw [bend right=45, looseness=1] (a.center) to (b.center);
	\draw (c.center) to (c'.center);
\end{tikzpicture}}

\newcommand{\dogbone}{%
\begin{tikzpicture}[baseline=0cm]
	\node (a) at (-.5,0) {$\bullet$};
	\node (b) at (.5,0) {$\bullet$};
	\draw (a.center) to (b.center);
\end{tikzpicture}}

\newcommand{\onevalent}{%
\begin{tikzpicture}[baseline=0cm]
	\node (a) at (0,-.5) {};
	\node (b) at (0,.5) {$\bullet$};
	\draw (a.center) to (b.center);
\end{tikzpicture}}

\newcommand{\leftunitvertex}{%
\begin{tikzpicture}[baseline=0cm]
	\node (a) at (-.5,.5) {$\bullet$};
	\node (b) at (.5,.5) {};
	\node (c) at (0,-.5) {};
	\node (d) at (0,0) {};
	\draw (a.center) to (d.center);
	\draw (b.center) to (d.center);
	\draw (c.center) to (d.center);
\end{tikzpicture}}

\newcommand{\rightunitvertex}{%
\begin{tikzpicture}[baseline=0cm]
	\node (a) at (-.5,.5) {};
	\node (b) at (.5,.5) {$\bullet$};
	\node (c) at (0,-.5) {};
	\node (d) at (0,0) {};
	\draw (a.center) to (d.center);
	\draw (b.center) to (d.center);
	\draw (c.center) to (d.center);
\end{tikzpicture}}

\newcommand{\TracelessJordanProofa}{%
\begin{tikzpicture}[baseline=2cm]
	\node (a) at (-.5,0) {};
	\node (b) at (0,0) {};
	\node (c) at (.5,0) {};
	\node [draw, rectangle] (sym) at  (0,1) {$\; \Sym^3 \;$};
	\node [draw, rectangle] (wedge) at (-.25,3) {$\; \wedge^2 \;$};
	\node (y) at (-.5,4) {};
	\node (output) at (0,4) {};
	\node (s) at (-.5,2) {};
	\node (t) at (0,1.75) {};
	\node (u) at (.5,2) {};
	\draw (a) to (sym.216);
	\draw (b) to (sym.270);
	\draw (c) to (sym.324);
	\draw (sym.144) to (s.center);
	\draw (sym.90) to (t.center);
	\draw (sym.36) to (u.center);
	\draw (s.center) to (wedge.233);
	\draw (u.center) to (wedge.307);
	\draw (wedge.125) to (y);
	\draw (wedge.53) to (output);
	\draw (s.center) to (t.center);
	\draw (t.center) to (u.center);
\end{tikzpicture}}

\newcommand{\TracelessJordanProofb}{%
\begin{tikzpicture}[baseline=2cm]
	\node (a) at (-.5,0) {};
	\node (b) at (0,0) {};
	\node (c) at (.5,0) {};
	\node [draw, rectangle] (sym) at  (0,1) {$\; \Sym^3 \;$};
	\node [draw, rectangle] (wedge) at (-.25,3) {$\; \wedge^2 \;$};
	\node (y) at (-.5,4) {};
	\node (output) at (0,4) {};
	\node (s) at (-.5,2) {};
	\draw (a) to (sym.216);
	\draw (b) to (sym.270);
	\draw (c) to (sym.324);
	\draw (wedge.233) to (sym.144);
	\draw [bend left=90, looseness=2] (sym.90) to (sym.36);
	\draw (s.center) to (wedge.307);
	\draw (wedge.125) to (y);
	\draw (wedge.53) to (output);
\end{tikzpicture}}

\newcommand{\TracelessJordanProofc}{%
\begin{tikzpicture}[baseline=2cm]
	\node (a) at (-.5,0) {};
	\node (b) at (0,0) {};
	\node (c) at (.5,0) {};
	\node [draw, rectangle] (sym) at  (0,1) {$\; \Sym^3 \;$};
	\node [draw, rectangle] (wedge) at (-.25,3) {$\; \wedge^2 \;$};
	\node (y) at (-.5,4) {};
	\node (output) at (0,4) {};
	\node (s) at (-.5,2) {};
	\draw (a) to (sym.216);
	\draw (b) to (sym.270);
	\draw (c) to (sym.324);
	\draw (wedge.233) to (sym.144);
	\draw (sym.90) to (s.center);
	\draw (sym.36) to (wedge.307);
	\draw (wedge.125) to (y);
	\draw (wedge.53) to (output);
\end{tikzpicture}}

\newcommand{\capadjointb}{%
\begin{tikzpicture}[baseline=.5cm]
	\node (a) at (-1.5,0) {};
	\node (b) at (-.5,0) {};
	\node (c) at (.5,1) {};
	\node (d) at (1.5,1) {};
	\draw (c.center) to [out=-90, in = -90, looseness =2] (d.center);
	\draw (a.center) to [out = 90, in = 90, looseness = 2] (c.center);
	\draw (b.center) to [out = 90, in = 90, looseness = 2] (d.center);
	\pgfresetboundingbox
	\draw(-1.5,0) (1.5,2);
\end{tikzpicture}}

\newcommand{\trivalentadjointa}{%
\begin{tikzpicture}[baseline=.5cm]
	\node (a) at (-1,0) {};
	\node (b) at (0,0) {};
	\node (c) at (1,0) {};
	\node (t) at (-.5,.5) {};
	\node (c') at (1,.5) {};
	\draw (a.center) to (t.center);
	\draw (b.center) to (t.center);
	\draw (c.center) to (c'.center);
	\draw (t.center) to [bend left=90, looseness=2] (c'.center);
	\pgfresetboundingbox
	\draw(-1,0) (1.25,1.5);
\end{tikzpicture}}

\newcommand{\trivalentadjointb}{%
\begin{tikzpicture}[baseline=.5cm]
	\node (a) at (-1,0) {};
	\node (b) at (0,0) {};
	\node (c) at (1,0) {};
	\node (t) at (1,.5) {};
	\draw (c.center) to (t.center);
	\draw (a.center) to [out=90,in=135, looseness =1] (t.center);
	\draw (b.center) to [out=90, in=45, looseness = 3] (t.center);
	\pgfresetboundingbox
	\draw(-1,0) (1.25,1.25);
\end{tikzpicture}}

\newcommand{\trivalentadjointc}{%
\begin{tikzpicture}[baseline=.5cm]
	\node (a) at (-1,0) {};
	\node (b) at (0,0) {};
	\node (c) at (1,0) {};
	\node (t) at (1,.5) {};
        \draw (c.center) to (t.center);
        \draw (a.center) to [out=90,in=45,looseness=2] (t.center);
        \draw (b.center) to [out=90,in=135,looseness=1] (t.center);
	\pgfresetboundingbox
	\draw(-1,0) (1.25,1.25);      
\end{tikzpicture}}

\newcommand{\JordanLowerTermsa}{%
\begin{tikzpicture}[baseline=2cm]
	\node (a) at (-.5,0) {};
	\node (b) at (.5,0) {};
	\node [draw, rectangle, minimum width = 1.4cm, minimum height = .7cm] (sym) at  (0,1) {$\; \Sym^2 \;$};
	\node [draw, rectangle, minimum width = 1.4cm, minimum height = .7cm] (wedge) at (0,3) {$\; \wedge^2 \;$};
	\node (y) at (-.5,4) {};
	\node (output) at (.5,4) {};
	\node (s) at (0,1.75) {};
	\node (t) at (0,2.25) {};
	\draw (a) to (sym.216);
	\draw (b) to (sym.324);
	\draw (sym.144) to (s.center);
	\draw (sym.36) to (s.center);
	\draw (t.center) to (wedge.216);
	\draw (t.center) to (wedge.324);
	\draw (wedge.144) to (y);
	\draw (wedge.36) to (output);
	\draw (s.center) to (t.center);
\end{tikzpicture}}

\newcommand{\JordanLowerTermsb}{%
\begin{tikzpicture}[baseline=2cm]
	\node (a) at (-.5,0) {};
	\node (b) at (.5,0) {};
	\node [draw, rectangle, minimum width = 1.4cm, minimum height = .7cm] (sym) at  (0,1) {$\; \Sym^2 \;$};
	\node [draw, rectangle, minimum width = 1.4cm, minimum height = .7cm] (wedge) at (0,3) {$\; \wedge^2 \;$};
	\node (y) at (-.5,4) {};
	\node (output) at (.5,4) {};
	\node (s) at (-.5,2) {};
	\node (t) at (+.5,2) {};
	\draw (a) to (sym.216);
	\draw (b) to (sym.324);
	\draw (sym.144) to (s.center);
	\draw (sym.36) to (t.center);
	\draw (s.center) to (wedge.216);
	\draw (t.center) to (wedge.324);
	\draw (wedge.144) to (y);
	\draw (wedge.36) to (output);
	\draw (s.center) to (t.center);
\end{tikzpicture}}

\newcommand{\JordanLowerTermsc}{%
\begin{tikzpicture}[baseline=2cm]
	\node (a) at (-.5,0) {};
	\node (b) at (.5,0) {};
	\node [draw, rectangle, minimum width = 1.4cm, minimum height = .7cm] (sym) at  (0,1) {$\; \Sym^2 \;$};
	\node [draw, rectangle, minimum width = 1.4cm, minimum height = .7cm] (wedge) at (0,3) {$\; \wedge^2 \;$};
	\node (y) at (-.5,4) {};
	\node (output) at (.5,4) {};
	\draw (a) to (sym.216);
	\draw (b) to (sym.324);
	\draw (sym.144) to (wedge.216);
	\draw (sym.36) to (wedge.324);
	\draw (wedge.144) to (y);
	\draw (wedge.36) to (output);
\end{tikzpicture}}

\begin{document}
\title{Towards the quantum exceptional series}

\author[Morrison]{Kim~Morrison}
\address{Lean Focused Research Organization}
\email{kim@tqft.net}
\thanks{This paper is based on work of the three authors from
2013--2020; the final version was prepared by NS and DPT}

\author[Snyder]{Noah~Snyder}
\address{Bloomington, Indiana, USA}
\email{nsnyder1@iu.edu}

\author[Thurston]{Dylan~P.~Thurston}
\address{Bloomington, Indiana, USA}
\email{dpthurst@iu.edu}

\begin{abstract}
  We find a single two-parameter skein relation on trivalent graphs,
  the \emph{quantum exceptional relation}, that specializes to a skein
  relation associated to each exceptional Lie algebra (in the adjoint
  representation). If a slight
  strengthening of Deligne's conjecture on the existence of a
  (classical) exceptional series is true, then this relation
  holds for a new two-variable quantum exceptional polynomial, at
  least as a power series near $q=1$. The
  single quantum exceptional relation can be viewed as a deformation of the
  Jacobi relation, and implies
 a deformation of the Vogel relation that motivated the conjecture on
 the classical exceptional series.

  We find a conjectural basis for the space of diagrams with $n$ loose
  ends modulo the quantum exceptional relation for $n \le 6$, with
  dimensions agreeing with the classical computations, and compute
  the matrix of inner products, and the quantum dimensions of idempotents.
  We use the
  skein relation to compute the conjectural quantum exceptional
  polynomial for many knots. In particular we
  determine (unconditionally) the values of the quantum polynomials
  for the exceptional Lie algebras on
  these knots. We can perform these computations for all links of Conway width less 
  than $6$, which includes
  all prime knots with 12 or fewer crossings. Finally, we prove several 
  specialization results relating our conjectural family to certain quantum
  group categories, and conjecture a number of exceptional analogues of 
  level-rank duality.
\end{abstract}


\date{\today}

\dedicatory{To Vaughan Jones, for all he
  taught us about knot polynomials and life}

\maketitle
{ \hypersetup{linkcolor=black}
  \tableofcontents }


\section{Introduction}
\label{sec:introduction}

There are two main 2-variable knot polynomials: the HOMFLY-PT polynomial \cite{MR776477,MR0945888} and
the Kauffman polynomial \cite{MR958895}.  The HOMFLY-PT polynomial interpolates between the
Reshetikhin\hyp Turaev invariants \cite{MR1036112} of the standard
representations of the quantum groups
$U_q(\mathfrak{gl}_n)$, while the Kauffman polynomial does the same for the
standard representations of $U_q(\mathfrak{o}_n)$ (or, in a slightly
different form, the standard representations of $U_q(\mathfrak{sp}_n)$).  The
two variable nature of these knot polynomials come from the two parameters $n$
and $q$, but we need to replace the discrete parameter $n$ with a continuous
one.   One can think of this process in two steps: first construct a family of
Lie algebra objects in symmetric ribbon categories which interpolate between
the $\mathfrak{gl}_n$ yielding something called $\mathfrak{gl}_t$ \cite{MR654325}, and second
quantize these Lie algebra objects.   This quantization can be done
systematically using the Kontsevich integral \cite{MR1318886}, provided it converges.

Work of Cvitanovi\'c, Vogel, Deligne, Cohen--de-Man, and others 
\cite{MR2418111,MR1378507, MR1411045, MR1952563, MR1381778, MR2769234, Vogel}, can be
summarized as a conjecture that there is a \emph{third} 1-parameter family of Lie
algebra objects in symmetric ribbon categories, called the (classical)
exceptional series. This should interpolate between the adjoint
representations of
$\mathfrak{a}_1$, $\mathfrak{a}_2$, $\mathfrak{d}_4$,
$\mathfrak{g}_2$, $\mathfrak{f}_4$, $\mathfrak{e}_6$, $\mathfrak{e}_7$, and
$\mathfrak{e}_8$ in an appropriate sense.  If this classical exceptional
family exists, we should again expect a quantization yielding a 2-variable
exceptional knot invariant~$\xi$.  In this paper we give a conjectural skein
theoretic description of this exceptional knot invariant. In fact, we prove
that these particular skein relations must hold if the exceptional knot
invariant exists.

We can use these skein relations to do many computations.  In particular, we
are able to calculate the value of this knot polynomial (if it exists!)\ on all
knots of Conway width 6, which includes all knots of 12 or fewer crossings.
Unconditionally these give the first calculations of the Reshetikhin-Turaev
knot invariants for the exceptional Lie algebras for most of these knots.  
(For 2- and 3-strand torus links and for algebraic/arborescent links such
calculations can already be found by specializing the calculations done by
Mironov-Mkrtchyan-Morozov \cite{MR3491191,MR3475991}.)
For example, we compute the following value for the left-handed
trefoil for the exceptional
series, and the Reshetikhin-Turaev invariant for $\mathfrak{e}_8$ is given by
setting $v=q^5$ and $w=q^{-1}$:
\begin{align}
  \label{eq:excep-trefoil}
  \xi(3_1) &= d - (v^6-v^{-6})d \times {}\\
  \nonumber &\qquad\quad \times \bigl(
v^{34}/w^4 - v^{44}/w^4
+ v^{36}/w^2 - v^{40}/w^2 + v^{44}/w^2 - v^{46}/w^2 + v^{50}/w^2 - v^{54}/w^2\\
\nonumber&\qquad\qquad\quad -v^6 - v^{18} + v^{32} + v^{38} - 2v^{42} + v^{46} - v^{50} + v^{52} + v^{54
}\\
\nonumber&\qquad\qquad\quad + v^{34}w^2 - v^{38}w^2 + v^{42}w^2 - v^{44}w^2 + v^{48}w^2
           - v^{52}w^2 + v^{30}w^4 - v^{40}w^4\bigr)\\
  \intertext{where $d$ is the quantum dimension}
  \nonumber d &= -\frac{(v^2+v^{-2})(wv^5-w^{-1}v^{-5})(wv^{-6} - w^{-1}v^6)}{(w-w^{-1})(wv^{-1}-w^{-1}v)}.
\end{align}

We are also able to do several other calculations, like giving quantum
dimension formulas for the summands of $\mathfrak{g}^{\otimes 3}$ quantizing
the dimension formulas of Cohen--de-Man \cite{MR1381778} in the classical case.
(For the summands of $\mathfrak{g}^{\otimes 2}$ such quantum dimension
formulas were already given by Tuba-Wenzl \cite{MR1815266}.)

In order to state these conjectures and theorems precisely and place them in
their proper context, we have written a  long introduction split into several
subsections.

\subsection{The classical exceptional series}
\label{sec:intro-classical}
We start by giving a skein theoretic
description of the classical exceptional series following Cvitanović,
Vogel, and Deligne. 

We begin with the symmetric ribbon category of Jacobi diagrams,
$\mathsf{Jac}_{R,D,B}$.

\begin{definition} \label{def:Jacobi}
Let $R$ be an algebra over $\mathbb{Q}$, let $D$ be an element of $R$
and let
$B$ be an invertible element of~$R$. (We use capital letters for these
classical quantities; the corresponding quantum quantities will get
lowercase letters.)
The objects of $\mathsf{Jac}_{R,D,B}$ are
collections of $n$ points on the line and the morphism space from $n$ to $m$
consists of $R$-linear combinations of planar graphs in a rectangle with $n$
boundary points on the bottom, $m$ boundary points at the top, and consisting
of trivalent vertices and $4$-valent symmetric crossings%
\footnote{A symmetric crossing means a crossing as in the symmetric
  group, i.e., not braided.}
modulo planar isotopy, the usual Reidemeister relations
for the symmetric crossing and trivalent vertex, and the following
circle, lollipop, bigon,
triangle, anti-symmetry of the vertex, symmetry of the Killing form,
and Jacobi relations.
\begin{align} 
\label{eq:classical-loop-lollipop}  \unknot \; &= D &  \loopvertex \; &=0 \\
\label{eq:classical-bigon-trigon}  \twogon\; &= B \;\onestrandid  & \threegon\; &= \frac{B}{2} \;\threevertex \allowdisplaybreaks \\
\label{eq:classical-symmetry-antisymmetry}   \symtwistvertex\; &= - \threevertex & \symtwist\; &= \drawcup
\end{align}
\begin{equation}
\drawI\; - \;\drawH\; + \;\drawsymcrossX\; = 0.
\label{eq:IHX}
\end{equation}
\end{definition}

There's an equivalence of ribbon categories $\mathsf{Jac}_{R,D,B} \cong \mathsf{Jac}_{R,D,xB}$
for any invertible $x \in R$ given by isotoping each diagram so every trivalent vertex has two edges coming in from
below and one edge going out above and then rescaling by $x$ for each Morse
minimum and $x^{-1}$ for each Morse maximum. In particular, if $R
= \mathbb{C}$ then $\mathsf{Jac}_{R,D,B}$ does not depend on the
choice of~$B$.  A convenient normalization is $B=12$.

If $\mathfrak{g}$ is a simple Lie algebra object in a symmetric ribbon
category with dimension $D$, by the diagram calculus  the symmetric ribbon
category of finite dimensional representations of $\mathfrak{g}$ is the target
of a functor from $\mathsf{Jac}_{R,D,B}$, where we interpret each strand as a
copy of $\mathfrak{g}$, each cup or cap using the Killing form or its inverse, each trivalent
vertex as the Lie bracket, and each crossing as the usual interchange
of tensor factors.
The category $\mathsf{Jac}_{R,D,B}$ 
is somewhat poorly behaved
since the space of $0$-boundary point diagrams is infinite dimensional.

Cvitanović \cite{MR2418111} and Vogel \cite{MR2769234}
independently observed that the adjoint
representations of the exceptional Lie algebras satisfy another relation, the
\emph{classical exceptional relation}.

\begin{definition}
If $2+D$ is invertible in $R$, let $\mathsf{Exc}_{R,D,B}$ be the
symmetric ribbon category which is the quotient of
$\mathsf{Jac}_{R,D,B}$ modulo the exceptional relation
\begin{equation}
\fourgon\; = \frac{B}{6} \left(\;\drawI\; + \;\drawH\; \right)
 + \frac{5}{6}\frac{B^2}{2+D} \left( \;\cupcap\; + \;\twostrandid\; + \;\symcross\; \right).
\label{eq:classical-except}
\end{equation}
\end{definition}

To give nicer formulas, Deligne
introduced a parameter $\lambda$ such that $D = -2
\frac{(\lambda+5)(6-\lambda)}{\lambda(1-\lambda)}$ \cite{MR1378507}.  This makes sense so long
as $\lambda$ and $1-\lambda$ are invertible in $R$, and such a $\lambda$ can
be found when $2+D$ is invertible and $(2+D)(242+D)$ has a square root in $R$.
Note that $\lambda$ and $1-\lambda$ correspond to the same $D$ and thus to the
same category.  If we normalize $B=12$ then the exceptional relation for
$\mathsf{Exc}_{R,-2   \frac{(\lambda+5)(6-\lambda)}{\lambda(1-\lambda)},12}$
(also written $\mathsf{Exc}_{R,\lambda}$) has the form
\begin{equation}
\fourgon\; = 2\left(\;\drawI\; + \;\drawH\;\right)
 - 2\lambda(1-\lambda) \left( \;\cupcap\; + \;\twostrandid\; + \;\symcross\; \right).
\label{eq:classical-except-2}
\end{equation}

Let us fix some notation. If $\mathcal{C}$ is a diagram category we denote the
Hom space from $k$ points to $\ell$ points by $\mathcal{C}(k,l)$ and the invariant
space $\mathcal{C}(k,0)$ by $\mathcal{C}(k)$. Since all our categories are ribbon,
we can use duality to canonically identify $\mathcal{C}(k,l) = \mathcal{C}(n) = \mathcal{C}(k',l')$ 
for any $k+l = n = k'+l'$ by using caps on the right to turn the outputs into inputs. 
We call this space the \emph{$n$-box space}. Of course to compose $n$-boxes
we need to be careful to turn the appropriate strands back upwards using cups.
In some cases, if the notation for $\cC$ already involves a parenthesis, we will instead use
$\cC_n$ to denote the $n$-box space, but hopefully this should be clear from context.

\begin{conjecture}[Classical sufficiency]
  \label{conj:class-suffic}
  So long as a certain finite set of non-zero polynomials
  in $\lambda$
are invertible,
the space of $n$-boundary point diagrams in
$\mathsf{Exc}_{R,\lambda}$
is
finitely generated as an $R$-module for every $n$.  Furthermore,
$\mathsf{Exc}_{R,\lambda}(0)$ is spanned by the empty diagram,
$\mathsf{Exc}_{R,\lambda}(1)$ is the $0$-module,
$\mathsf{Exc}_{R,\lambda}(2)$ is spanned by the strand,
$\mathsf{Exc}_{R,\lambda}(3)$ is spanned by the trivalent vertex, and
$\mathsf{Exc}_{R,\lambda}(4)$ is spanned by
\begin{equation}\label{eq:4-box-gen}
  \symcross\;, \qquad \drawI \;,
  \qquad \drawH \;, \qquad \twostrandid \;, \qquad \cupcap \; .
  \end{equation}
\end{conjecture}

We also conjecture specific spanning sets for $\mathsf{Exc}_{R,\lambda}(n)$
for $n \leq 6$; see \S\ref{sec:bases}.

Conjecture~\ref{conj:class-suffic} would hold trivially if
$\mathsf{Exc}_{R,\lambda}$ collapsed to the zero category; to address
that we have a complementary conjecture.

\begin{conjecture}[Classical consistency]
  \label{conj:class-consist}
So long as a certain finite set of non-zero polynomials in $\lambda$ are
invertible in~$R$, then $R$ injects into
$\mathsf{Exc}_{R,\lambda}(0)$ as multiples of the
empty diagram.  In particular, any two ways of simplifying the same closed
diagram to give an element of $R$ will give the same value.
\end{conjecture}

\begin{remark}
  The quantum Conjectures~\ref{conj:quant-suffic}
  and~\ref{conj:quant-consist} below suggest that for
  Conjectures~\ref{conj:class-suffic} and~\ref{conj:class-consist} the
  only polynomials that need to be inverted are $\lambda$ and $1-\lambda$.
\end{remark}

Note that if $D = -2\frac{(\lambda+5)(6-\lambda)}{\lambda(1-\lambda)}$ is 
invertible in $R$, it follows from classical consistency
that $R$ injects into $\mathsf{Exc}_{R,\lambda}(2)$ as multiples of the
strand.  If $B$ is also invertible in $R$, it follows from
classical consistency that $R$ injects into $\mathsf{Exc}_{R,\lambda}(3)$ as
multiples of the trivalent vertex.  Finally if in addition
$\lambda\notin\{3,-4\}$
then classical consistency implies that $R^5$ injects into $\mathsf{Exc}_{R,\lambda}(4)$ as
linear combinations of the five diagrams in~\eqref{eq:4-box-gen}.
(Lemma~\ref{lem:classical-eigenvalues}.)

If both consistency and sufficiency hold, it is also natural to conjecture
that $\mathsf{Exc}_{R,\lambda}(n)$ is always a free $R$-module of finite
rank.

\begin{proposition}[Classical specialization] \label{prop:class-spec}
  If $G$ is one of the (super)groups on the right of
  Table~\ref{tab:lambda-group}, and $\lambda$ is the value shown on the
  left (also equal to $-6/h^\vee$, where $h^\vee$ is the dual Coxeter
  number), then there is a full
and dominant functor from $\mathsf{Exc}_{\mathbb{C},\lambda}$ to the 
category of representations of the corresponding group on the
right of Table~\ref{tab:lambda-group}. 
\end{proposition}

In particular, the category of representations of the corresponding group on the
right of Table~\ref{tab:lambda-group} is the Cauchy completion (i.e., additive and
idempotent completion) of a quotient of $\mathsf{Exc}_{\mathbb{C},\lambda}$.
This result is in the style of the first fundamental theorem of invariant theory. We
would also like an explicit characterization of the kernel of this quotient in the style
of the second fundamental theorem of invariant theory, which we
do not have.  If Conjecture \ref{conj:class-suffic} holds, 
then the kernel of this quotient
can be characterized abstractly as the tensor ideal of negligible morphisms. 

\begin{table}
  \begin{tabular}{rl}
    \toprule
    $\lambda$ & Supergroup  \\ \midrule
    $-5$ & Trivial group \\
    $-4$ & $\SOSp(1|2)$ \\
    $-3$ & $\PSL(2)$ \\
    $-2$ & $\PSL(3) \rtimes \ZZ/2$ \\
    $-3/2$ & $G_2$  \\
    $-1$ & $\PSO(8) \rtimes S_3$ \\
    $-2/3$ & $F_4$ \\
    $-1/2$ & $E_6^{\mathrm{adj}} \rtimes \ZZ_2$\\
    $-1/3$ & $E_7^{\mathrm{adj}}$ \\
    $-1/5$ & $E_8$ \\
    \bottomrule
  \end{tabular}
  \caption{Values of $\lambda$ and the supergroup whose
    category of
    representations is related to $\mathsf{Exc}_{\CC,\lambda}$. In
    each case, take the Lie group
    with trivial center, semidirect product with its group of Dynkin diagram
    automorphisms.    \label{tab:lambda-group}}
\end{table}

Classical specialization is proved in \S\ref{sec:special-values} following
an idea appearing in \cite{MR1403861, MR3951762}.  Several cases were already
proved in \cite{2204.13642}.

\begin{conjecture}[Classical semisimplicity]\label{conj:class-semisimple}
There is a dense subset $U \subset \mathbb{C}$ so that, for
$\lambda \in U$, the Cauchy
completion
of $\mathsf{Exc}_{\mathbb{C},\lambda}$ is
semisimple.
\end{conjecture}
(Here as elsewhere, we use the standard, not Zariski, topology on $\CC$.)
We expect that in order for the Cauchy completion to be semisimple one
needs to invert infinitely many polynomials.
For example, for the Cauchy
completion of $\mathrm{GL}_t$ to be semisimple you need to invert $t-n$ for
every positive integer $n$.  From the dimension formulas of Cohen--de-Man and
others, we might speculate that for classical semisimplicity it is
enough to invert $k\lambda-\ell$ for $k, \ell \in \mathbb{Z}$. It
would follow from Conjecture~\ref{conj:class-semisimple} that the
Cauchy completion of the generic category
$\mathsf{Exc}_{\mathbb{C}(\lambda),\lambda}$ is semisimple.

Deligne's conjecture \cite{MR1378507} is in terms of the existence of
a suitable category. This was
improved by Cohen--de-Man \cite{MR1714606} who incorporated the
Cvitanović-Vogel exceptional
relation. We now give a version of their conjecture.

\begin{conjecture}[Deligne's conjecture]\label{conj:Deligne}
There exists
  \begin{itemize}
  \item a ring $R$ which is a localization of $\mathbb{Q}[\lambda]$ at finitely many polynomials;
  \item a rigid symmetric monoidal category
    $\mathsf{Exc}^{Del}_{R,\lambda}$ where the objects are
    non-negative integers and morphism spaces are finitely-generated
    $R$-modules; and
  \item a symmetric monoidal functor $\eval$ from
    $\mathsf{Exc}_{R,\lambda}$ to $\mathsf{Exc}^{Del}_{R,\lambda}$, that
    is the identity on objects and surjective on morphisms.  
  \end{itemize}
For any map $\varphi\colon R \rightarrow S$ we can define a base
extension $\mathsf{Exc}^{Del}_{S,\varphi(\lambda)} =
\mathsf{Exc}^{Del}_{R,\lambda} \otimes_\varphi S$.  Furthermore
$\mathsf{Exc}^{Del}_{\mathbb{C},\lambda}$ has the following properties for an open dense subset
$U \subset \mathbb{C}$.
\begin{itemize}
\item The Cauchy completion of
  $\mathsf{Exc}^{Del}_{\mathbb{C},\lambda}$ is semisimple for $\lambda \in U$.
\item $\mathsf{Exc}^{Del}_{\mathbb{C},\lambda}(n)$ has dimension $1,\allowbreak0,\allowbreak1,\allowbreak1,\allowbreak5,\allowbreak16,\allowbreak80$
for $\lambda \in U$ and $n=0,1,2,3,4,5,6$, respectively.
\item If $\lambda$ is one of the special values in
  Table~\ref{tab:lambda-group}, then the Cauchy completion of the
  quotient of
  $\mathsf{Exc}^{Del}_{\mathbb{C},\lambda}$ by negligible elements is
  equivalent to the category of representations of the corresponding
  supergroup.
\end{itemize}
\end{conjecture}

Conjecture~\ref{conj:Deligne} follows from
Conjectures~\ref{conj:class-suffic}, \ref{conj:class-consist},
and~\ref{conj:class-semisimple} (using
Proposition~\ref{prop:quant-spec}).
Conjecture~\ref{conj:Deligne} should be
viewed as closer to Conjecture~\ref{conj:class-consist}, since it
doesn't say anything about the kernel of $\eval$;
Conjecture~\ref{conj:class-suffic} says that the classical exceptional
relation is the only skein relation needed to define
$\mathsf{Exc}^{Del}_{R,\lambda}$.

\begin{remark}
  We have learned that Vogel has computational evidence against
  Deligne's conjecture (private communication). Note that it
  is possible for the quantum
  Conjectures~\ref{conj:quant-suffic} and~\ref{conj:quant-consist} to
  hold even if the classical conjectures do not. See
  Remark~\ref{rem:F4-deformation}.
\end{remark}

\subsection{Laurent polynomial conventions}

The bulk of this paper deals with quantum versions of the sufficiency and
consistency conjectures, now involving two variables $v$ and~$w$.%
\footnote{We avoid the letter $q$, because the relation of our
  parameters to the quantum group parameters is not obvious; see
  \S\ref{sec:Specializations}.}
Before stating these conjectures we fix some
notation for certain Laurent polynomials appearing in our formulas.  On the
one hand, we would like to use formulas that are irreducible as Laurent
polynomials in $v$ and $w$, so that we can easily read off which
primes need to be inverted.
On the other hand the Laurent polynomials that
occur in our formulas are often even or odd under $v \leftrightarrow
v^{-1}$ and/or $w \leftrightarrow w^{-1}$, so we would also like our formulas to have this symmetry.
So we often combine two irreducible factors into a single expression.  For
example, $v-v^{-1} = v^{-1}(v-1)(v+1) =
(v^{\frac{1}{2}}-v^{-\frac{1}{2}})(v^{\frac{1}{2}}+v^{-\frac{1}{2}})$ but
there is no way to factor it into odd or even Laurent polynomials so we will
leave $v-v^{-1}$ alone in our formulas.  This will never cause any confusion
since any time the irreducible factor $(v-1)$ appears it will have a
corresponding $(v+1)$ factor.

\begin{definition}
We define
\begin{align*}
[k\lambda + n] &= w^kv^n - w^{-k}v^{-n}\\
\{k\lambda + n\} = \frac{[2k\lambda + 2n]}{[k\lambda + n]} &= w^k v^n + w^{-k} v^{-n}.
\end{align*}
\end{definition}

\begin{warning}
Unlike with the usual quantum numbers, there is no denominator. To get
suitable limiting behavior, we need equal number of terms in the
numerator and denominator:
\[
  \lim_{\substack{v,w\to 1\\w=v^\lambda}} \frac{[a\lambda+b]}{[c\lambda + d]}
   = \frac{a\lambda+b}{c\lambda+d}.
\]
(In this equality, $\lambda$ is a formal variable in expressions of
the form $[k\lambda+n]$ and a rational number otherwise.)
\end{warning}

Note that $\{k\lambda + n\}$ is always irreducible,
while $[k\lambda + n]$ is reducible if $k$ and~$n$ have a common factor. (In
particular, for $n,k>1$, $[1]$ divides $[n]$ and $[k\lambda]$.)
When
$k=0$ we will generally avoid $[n]$ and instead factor into the following
symmetrized cyclotomic polynomials.  (Strictly speaking we should use
a similar convention for factoring $[k \lambda]$, but this is not needed
in practice.)  Let
\begin{align*}
  \Psi_n(v) &= \prod_{\zeta\colon \zeta^{n} = \pm 1 \text{ and } \zeta^m \neq \pm 1 \text{ for $m < n$}}
  (v^{\frac{1}{2}}-\zeta v^{-\frac{1}{2}}),\\
  \intertext{or, in terms of the usual cyclotomic polynomials $\Phi_n$ and Euler's totient function $\varphi$,}
  \Psi_n(v) &= v^{-\varphi(n)} \Phi_{n}(v^2).
\end{align*}
In particular, when $n$ is odd $\Psi_n$ vanishes exactly at primitive
$n$th and primitive $2n$th roots of unity, while when $n$ is even,
$\Psi_n$ vanishes exactly at primitive $2n$th roots of unity.
The notation $\{k\lambda + n\}$ can be viewed as another case of this:
$\{k\lambda+n\} = \Psi_{2}(w^kv^n)$, and for $n$ a power of~$2$,
$\{n\} = \Psi_{2n}(v)$. For convenience, here are the values of $\Psi_n$ that
appear in this paper.
\begin{align*}
  \Psi_1(v) & = [1] = v - v^{-1} &
  \Psi_2(v) & = \{1\} = v + v^{-1}\\
  \Psi_3(v) & = \frac{[3]}{[1]} = v^2 + 1 + v^{-2}&
  \Psi_4(v) & = \{2\} = v^2 + v^{-2}\\
  \Psi_5(v) & = \frac{[5]}{[1]} = v^4 + v^2 + 1 + v^{-2} + v^{-4}&
  \Psi_6(v) & = \frac{[6][1]}{[3][2]} = v^2 - 1 + v^{-2}\\
  \Psi_7(v) & = \frac{[7]}{[1]} = v^6 + v^4 + v^2 + 1 + v^{-2} + v^{-4} + v^{-6} &
  \Psi_8(v) & = \{4\} = v^4 + v^{-4} \\
  \Psi_{10}(v) & = \frac{[10][1]}{[5][2]} = v^4 - v^2 + 1 - v^{-2} + v^{-4} &
  \Psi_{12}(v) & = \frac{[12][2]}{[6][4]} = v^4 - 1 + v^{-4}
\end{align*}
\subsection{Quantum exceptional conjectures}
\label{sec:intro-quantum}
We now turn to our main conjectures.  Our first goal is to give a definition
of a ribbon category $\mathsf{QExc}$ which looks like a quantum deformation of
$\mathsf{Exc}$. In particular we need skein relations which deform the Jacobi
relation, the classical exceptional relation, and the symmetric relation
between the overcrossing and undercrossing.  Our specific choices of relations
will be motivated below by our main theorems.

\begin{definition}\label{def:qexc}
Let $R$ be an algebra over~$\mathbb{Q}$ and $v$ and $w$ invertible
elements of~$R$ such that
$[\lambda]$, $[\lambda-1]$, $\Psi_1$, and $\Psi_6$
are invertible.  Let
$\mathsf{QExc}_{R,v,w}$ be the ribbon category whose objects are collections
of $n$ points on a line, and whose morphisms spaces from $n$ to $m$ consist of
$R$-linear combinations of planar graphs in a rectangle with $n$ boundary
points on the bottom, $m$ boundary points at the top, and consisting of
trivalent vertices and $4$-valent braided crossings modulo planar isotopy, the
usual Reidemeister 2 and 3 relations for a crossing, the ``Reidemeister 2.5'' relation
saying you can pull a trivalent vertex through a crossing,
and the
following circle, lollipop, bigon, triangle, quantum anti-symmetry, and
quantum Killing symmetry relations.

 \begin{equation}
    \label{eq:simple-rels-spec}
  \begin{aligned}
    \unknot\; &= d&\quad
    \loopvertex\;&=0\\[5pt]
      \twogon&= b\;\onestrandid&
        \threegon &= t\threevertex\\[5pt]
    \twistvertex\; &= -v^{6}\;\threevertex&
      \twist\; &= v^{12}\;\drawcup
  \end{aligned}
  \end{equation}
  where
\begin{align*}
  d &= -\frac{\Psi_4 [\lambda+5][\lambda-6]}{[\lambda][\lambda-1]} \allowdisplaybreaks \\  
  b &= \{\lambda+2\}\{\lambda-3\}\Psi_3 \allowdisplaybreaks\\
  t &= \Psi_2 \bigl(\Psi_2 \{2\lambda -1\} + (v^4 - v^2 - 1 - v^{-2} + v^{-4})\bigr). 
\end{align*}

To these skein relations, we add three more relations:
\begin{itemize}
\item the \emph{quantum exceptional Jacobi relation}
  \begin{equation}
    \tag{QEJac-$w$}
v^{-3} \;
\drawcrossX
\;+ v \;
\drawI
\; -v^{-1} \;
 \drawH
\;
 - \frac{[\lambda][\lambda-1]}{\Psi_1}
\left[\; \braidcross \;
 + v^{4}\;
\cupcap
\; + v^{-4} \;
 \twostrandid \;
 \right] = 0.\label{eq:QEJac-w}
\end{equation}
\item the \emph{quantum exceptional square relation}
\begin{equation}
  \tag{QESq-$w$}
  \label{eq:QESq-w}
\fourgon + \frac{-\Psi_2\Psi_6^2[\lambda][\lambda-1]}{\Psi_1^2} \braidcross + \frac{x_1}{\Psi_1} \drawI + \frac{x_2}{\Psi_1} \drawH + \frac{x_3}{\Psi_1^2} \cupcap + \frac{x_4}{\Psi_1^2} \twostrandid =0
\end{equation}
where the $x_i$ are Laurent polynomials in $v$ and $w$ from Appendix \ref{app:coefficients}.
\item the \emph{quantum exceptional crossing relation}
  \begin{equation}
    \tag{QECross-$w$}
\label{eq:QECross-w}
\braidcross - \invbraidcross + \frac{\Psi_1}{\Psi_6} \left[\; \drawI \; - \; \drawH \; \right] + \frac{\Psi_1 [\lambda][\lambda-1]}{\Psi_6} \left[\; \twostrandid \; - \; \cupcap \; \right] = 0.
\end{equation}
\end{itemize}
\end{definition}

Note that if you set $w = v^\lambda$ and take the limit $v \rightarrow 1$
these relations recover the classical Jacobi relation, classical exceptional
relation, and the relation in a symmetric tensor category saying that the overcrossing
equals the undercrossing.  In order to make this
specialization process purely algebraic, we would need to introduce a larger
ring with formal symbols for expressions like $\frac{w-w^{-1}}{v-v^{-1}}$.

\begin{warning}
Although~\eqref{eq:QEJac-w} looks like an
analogue of the ordinary Jacobi relation with lower order terms, it is
important to emphasize that it is only the correct analogue of the Jacobi
relation in the exceptional setting.  In particular, we do not expect
any similar relation for
quantizations of non-exceptional Lie algebras.
\end{warning}

We now state our main conjectures. Let $\bfP$ be the set of
Laurent polynomials in $v$, $w$
\[
  \bfP = \{v,w,\Psi_1, \Psi_2, \Psi_3,
  \Psi_6, [\lambda], [1-\lambda]\}.
\]
(This list includes
the polynomials that we assume
to be invertible to define $\QExc_{R,v,w}$.)

\begin{conjecture}[Quantum sufficiency]
  \label{conj:quant-suffic}
  If the polynomials in $\bfP$ are invertible in~$R$, the
  space of $n$-boundary point
diagrams in $\mathsf{QExc}_{R,v,w}$ is finitely generated as an $R$-module for
every~$n$.  
Furthermore, $\mathsf{QExc}_{R,v,w}(0)$ is generated by the
empty diagram, $\mathsf{QExc}_{R,v,w}(1)$ is the zero module,
$\mathsf{QExc}_{R,v,w}(2)$ is generated by the strand,
$\mathsf{QExc}_{R,v,w}(3)$ is generated by the trivalent vertex, and
$\mathsf{QExc}_{R,v,w}(4)$ is generated by
\begin{equation}\label{eq:4-box-quant-gen}
  \braidcross\;, \qquad \drawI \;, \qquad \drawH \;, \qquad
  \twostrandid \;, \qquad \cupcap \; .
\end{equation}
\end{conjecture}

As before we have conjectured spanning sets for $\mathsf{QExc}_{R,v,w}(n)$
for $n=0,1,2,3,4,5,6$; see Conjecture~\ref{conj:braided-span}.

\begin{conjecture}[Quantum consistency]
  \label{conj:quant-consist}
So long as the polynomials in $\bfP$
are invertible, $R$ injects into
$\mathsf{QExc}_{R,v,w}$ as multiples of the empty diagram.  In particular, any
two ways of applying the above relations to a closed diagram to return a
rational function in $v$ and $w$ will yield the same rational function.
\end{conjecture}

As in the classical case, from this conjecture we get that if $d$
(resp.\ $d$ and $b$) are invertible,
then $R$ also injects into $\mathsf{QExc}_{R,v,w}(2)$ (resp.\
$\mathsf{QExc}_{R,v,w}(3)$) as multiples of the strand (resp.\ trivalent
vertex), and under some other invertibility assumptions $R^5$ injects
into $\mathsf{QExc}_{R,v,w}(4)$ as linear combinations of the five
diagrams in~\eqref{eq:4-box-quant-gen} (see Lemma~\ref{lem:lin-ind-w}).

\begin{proposition}[Quantum specialization]\label{prop:quant-spec}
If $\lambda$ is one of
the values on the left of Table~\ref{tab:lambda-group}, then there is a full
and dominant functor from $\mathsf{QExc}_{\mathbb{C},v,v^\lambda}$ 
to the category of representations of the quantum version of
the corresponding group, for $v$ in some dense subset of\/~$\mathbb{C}$.
\end{proposition}

As before, the category of representations of the corresponding group on the
right of Table~\ref{tab:lambda-group} is the Cauchy completion
of a quotient of $\mathsf{QExc}_{\mathbb{C},v,v^\lambda}$.
If Conjecture \ref{conj:class-suffic} holds, then the kernel of this quotient
can be characterized abstractly as the tensor ideal of negligible morphisms. 

We expect that quantum specialization holds for all $v$ that are not roots of unity.  This conjecture could 
also be stated more algebraically by working over the generic point $\mathsf{QExc}_{\mathbb{C}(v),v,v^\lambda}$.
Moreover, if you replace the quantum group category with the appropriate tilting modules at roots of unity, 
then it should hold except at some small roots of unity.

Quantum specialization is stated precisely and proved as
Proposition~\ref{prop:quant-spec-2} in 
\S\ref{sec:Specializations}, again following an idea appearing in \cite{MR3951762}. 
(The precise statement has the correspondence
between the parameters $q$ and~$v$, which is somewhat subtle.)

\begin{conjecture}[Quantum semisimplicity]
There is an open dense subset $U \subset \mathbb{C}^2$ for which the Cauchy
completion of $\mathsf{QExc}_{\mathbb{C},v,w}$, for $(v,w) \in U$, is
semisimple.
\end{conjecture}

We expect that it's enough to assume that $[k \lambda + \ell]$ is nonzero for all integers $k$ and $\ell$ (in particular, $v$ is not a root of unity).

To justify the definition of the $\mathsf{QExc}$ we have the
following theorems, which we can summarize as saying that under
some mild assumptions, any trivalent ribbon category where the $n$-box space
has dimension at most $1,0,1,1,5$ for $n=0,1,2,3,4$ satisfies a version of
\eqref{eq:QEJac-w}.  Furthermore, again under some mild assumptions, 
\eqref{eq:QEJac-w} implies \eqref{eq:QESq-w} and
\eqref{eq:QECross-w}.
To allow for more general statements, we use another set
of variables, $(v,\alpha)$ in place of $(v,w)$ above, at the cost of
messier formulas.

\begin{definition}
  \label{def:TrivalentRibbon}
A \emph{trivalent ribbon} category $(\cC, X, \tau)$ is a ribbon
category~$\mathcal{C}$ over~$R$ with  an object~$X$ where
\begin{itemize}
\item the 1-box space is zero, and
\item the 0-, 2-, and 3-box spaces are rank 1 free $R$-modules
generated by the empty diagram, the strand, and a `trivalent vertex' $\tau \in \cC(3)$ 
respectively, 
\end{itemize}
such that
\begin{itemize}
\item the category is generated as a ribbon category by $\tau$,
\item the value of the circle is a non-zero multiple~$d$ of the empty
  diagram, and
\item the value of the bigon is a
  non-zero multiple~$b$ of the strand.
\end{itemize}
\end{definition}

Being generated as a ribbon category by $\tau$ is the same as saying that
$\cC$ is generated as a pivotal category by $\tau$ and the crossing.  In
particular, a trivalent ribbon category need not be trivalent in the sense of
\cite{MR3624901} where the trivalent vertex alone is assumed to generate $\cC$ 
as a pivotal category.  Also we have weakened the non-degeneracy condition from
\cite{MR3624901}, replacing it with the assumption that the circle and bigon
are non-zero.

\begin{theorem} \label{thm:Jacobi}
Suppose $(\cC, X, \tau)$ is a trivalent ribbon category over
$\mathbb{C}$ and $\dim{\cC(4)} \leq 5$. Then there exists $v, \alpha,
b, d, t \in \CC$ such that the relations in~\eqref{eq:simple-rels-spec}
hold with
\begin{align*}
  [5] b &= - \alpha (d+\{8\}) \\
   \{2\} t &= b-[5]\alpha,
\end{align*}
as well as
\begin{equation}
  \tag{QEJac-$\alpha$}
v^{-3} \;
\drawcrossX
\;+ v \;
\drawI
\; -v^{-1} \;
 \drawH
\;
 + \alpha
\left[\; \braidcross \;
 + v^{4}\;
\cupcap
\; + v^{-4} \;
 \twostrandid \;
 \right] = 0.\label{eq:QEJac-alpha}
\end{equation}
\end{theorem}

If $\dim{\cC(4)} = 5$ then relation~\eqref{eq:QEJac-alpha} is
the only relation of this form up to simultaneously sending $v \mapsto -v$ and $\alpha \mapsto -\alpha$,
but if $\dim{\cC(4)} < 5$ then there will be at least two linearly independent such relations.

\begin{theorem}\label{thm:square-crossing}
Suppose $(\cC, X, \tau)$ is a trivalent ribbon category over $\mathbb{C}$ with
$\dim{\cC(4)} \leq 5$ and there exist $v$ and $\alpha$ satisfying
\eqref{eq:QEJac-alpha}, with $v^{10} \neq 1$ and $v^{12}
\neq 1$.  Then $\alpha$ and $b+[3]\alpha$ are
also non-zero, $v^8 \ne 1$, and we have the following relations:
\begin{itemize}
\item
$
  d = -\frac{[5] b}{\alpha} - \{8\}.
$
\item 
$
  t = \frac{b-[5] \alpha}{\{2\}}.
$
\item
The quantum exceptional square relation:
\begin{equation}
  \tag{QESq-$\alpha$}\label{eq:QESq-alpha}
  \fourgon + \frac{\alpha \Psi_6 (b+[3]\alpha)}{\Psi_1 \Psi_3 \Psi_4}  \braidcross + \frac{y_1}{\Psi_3 \Psi_4}  \drawI + \frac{y_2}{\Psi_3 \Psi_4}  \drawH + \frac{\alpha y_3}{\Psi_1 \Psi_3 \Psi_4} \cupcap + \frac{\alpha y_4}{\Psi_1 \Psi_3 \Psi_4} \twostrandid =0
\end{equation}
where the $y_i$ are Laurent polynomials in $b$, $\alpha$, and $v$ given in Appendix \ref{app:coefficients}.
\item
The quantum exceptional crossing relation:
\begin{equation}
  \tag{QECross-$\alpha$}\label{eq:QECross-alpha}
\braidcross - \invbraidcross + \frac{\Psi_1 \Psi_2 \Psi_3 \Psi_4}{b+[3]\alpha} \left[\; \drawI \; - \; \drawH \; \right] - \frac{\Psi_1^2 \Psi_2 \Psi_3 \Psi_4 \alpha}{b+[3]\alpha} \left[\; \twostrandid \; - \; \cupcap \; \right] = 0.
\end{equation}
\end{itemize}

\end{theorem}

To get back to the $(v,w)$ variables and nicer formulas in
Definition~\ref{def:qexc},
we make a change of variables
motivated by the Kontsevich integral approach and Deligne's
change of variable from $d$ to~$\lambda$.
(See
Lemma~\ref{lem:eigenvalues-twist} for an
alternate motivation.) We first remind the reader of some very
elementary mathematics which we nonetheless found quite confusing.

\begin{remark}
  Consider the rational change of variables from $(x,y)$ to $(x,w)$
  with $w= \frac{y}{x+y}$, which makes sense when $y \neq -x$. The
  inverse change of variables is given by the rational function
  $y = \frac{wx}{1-w}$, which makes sense when $w \neq 1$. However, in
  order for these transformations to be inverse we need to check that
  $$ y \overset{?}{=} \frac{x\frac{y}{x+y}}{1-\frac{y}{x+y}} = \frac{xy}{x}$$
  which requires in addition that $x \neq 0$. In this case, $y$ as a
  function of $w$ and vice versa are fractional linear
  transformations, which are invertible as long as the determinant of
  the corresponding $2 \times 2$ matrix
  $\left(\begin{smallmatrix} 1 & 0 \\ 1 & x \end{smallmatrix}\right)$
  is nonzero.
  More generally, by the inverse function
  theorem
  the transformation $(x,y) \mapsto (x,w)$ fails to have an
  inverse exactly when
  $\partial_y w = \frac{(x+y)-y}{(x+y)^2} = 0$.
\end{remark}

When $b+[3]\alpha$ is non-zero and the relevant square roots exist
in~$R$ we can make the change of variables from $\alpha$ to $w$ given by
\begin{equation}\label{eq:change-variables}
  \frac{w^2}{v} + \frac{v}{w^2} = \frac{\{1\}b-[3]\{5\}\alpha}{b+[3]\alpha}.
\end{equation}
Note that there are generically four values of $w$ for each value
of~$\alpha$.
We see that $\partial_\alpha (w^2/v+v/w^2) = 0$ exactly if $0 = \{1\}\cdot[3]
+ 1\cdot[3]\{5\} = \Psi_1 \Psi_2 \Psi_3 \Psi_4 \Psi_6$, so when $v$ is not an
$8$th or $12$th root of unity
the (multi-valued) change of variables is invertible, with
(single-valued) inverse given
by
$$\alpha =  -\frac{b[\lambda][\lambda-1]}{\{\lambda+2\}\{\lambda-3\}\Psi_1\Psi_3}.$$

Finally, we observe that if we rescale the choice of self-duality so that
$b = \{\lambda+2\}\{\lambda-3\}\Psi_3$,
formulas \eqref{eq:QEJac-alpha}, \eqref{eq:QESq-alpha},
and~\eqref{eq:QECross-alpha}
are generically identical to
the formulas in terms of $v,w$ given earlier as
\eqref{eq:QEJac-w}, \eqref{eq:QESq-w} and \eqref{eq:QECross-w}.

\begin{corollary}\label{cor:maincor}
Suppose $(\cC, X, \tau)$ is a trivalent ribbon category over
$\mathbb{C}$ with $\dim\cC(4) \le 5$ so that
\eqref{eq:QEJac-alpha} holds with $v^{10} \neq 1$ and $v^{12}
\neq 1$.  Then, there
is a $w \in \mathbb{C}$ and an evaluation functor
$\mathsf{QExc}_{\mathbb{C},v,w} \rightarrow \cC$ sending 
the trivalent vertex to $\tau$.
\end{corollary}

This corollary justifies our definition of $\mathsf{QExc}$, since if the
quantum exceptional family exists in any form it would have to satisfy our
quantum exceptional relations.

\begin{remark}\label{rem:QEJac-alpha}
  Throughout the paper, when we say that ``\eqref{eq:QEJac-alpha}
  holds'' in a trivalent ribbon category, we implicitly include the
  choice of~$v$ and the elementary relations
  \eqref{eq:simple-rels-spec} in the statement.
\end{remark}

\subsection{From classical to quantum}
The classical conjectures from \S\ref{sec:intro-classical} and
the quantum conjectures from \S\ref{sec:intro-quantum} are
related, but there are no formal implications between them. From the point of view
of the quantum conjectures, the classical specialization $v=1$ is excluded
by the assumption that $\Psi_1$ is invertible. There
is a partial implication the other direction: Conjecture~\ref{conj:Deligne}
implies a corresponding quantum conjecture.

\begin{conjecture}[Quantum version of Deligne's conjecture]\label{conj:quant-Deligne}
There exists
  \begin{itemize}
  \item a ring $R$ satisfying the hypotheses of Definition~\ref{def:qexc};
  \item a trivalent ribbon category
    $\mathsf{QExc}^{Del}_{R,v,w}$ where the objects are
    non-negative integers and morphism spaces are finitely-generated
    $R$-modules; and
  \item a ribbon functor $\mathrm{Qeval}$ from
    $\mathsf{QExc}_{R,v,w}$ to $\mathsf{QExc}^{Del}_{R,v,w}$, that
    is the identity on objects, surjective on morphisms, and sends
    the trivalent vertex to the trivalent vertex.
  \end{itemize}
\end{conjecture}

\begin{theorem}\label{thm:classical-quantum}
  If Conjecture~\ref{conj:Deligne} is true
  for $\mathsf{Exc}^{Del}_{R,\lambda}$, then
  Conjecture~\ref{conj:quant-Deligne} is true for
  $\mathsf{Exc}^{Del}_{R((h)),v,w}$ for $v=e^{h/2}$ and $w=e^{\lambda
    h/2}$ as formal
  Laurent series.
\end{theorem}
In interpreting Theorem~\ref{thm:classical-quantum}, recall that we
normalized $\mathsf{Exc}_{R,\lambda}$ so $B = 12$. In the language introduced in
\S\ref{sec:LAConventions}, this normalization of $B$ corresponds to what we call
 the ``Convenient'' choice of invariant inner product.

Theorem \ref{thm:classical-quantum} is proved in \S \ref{sec:classical-quantum}.

\subsection{Calculations}
In \S 5 we calculate some properties of
$\mathsf{QExc}_{\mathbb{C},v,w}$ and the corresponding $2$-variable knot
polynomial, under the assumptions of our main conjectures.

First we give conjectural bases for the box spaces up to $n=6$, and check that
these diagrams are linearly independent (assuming
Conjecture~\ref{conj:class-consist}).  We unconditionally give a
$16$-dimensional representation of the $5$-strand affine braid group, and an
$80$-dimensional representation of the $6$-strand affine braid group, obtained
from the braid group actions on these $5$- and $6$-boundary point diagrams.

Using this, we give a divide-and-conquer algorithm for computing the
conjectural quantum exceptional invariant of all knotted trivalent
graphs of Conway width at most~6.
This includes all prime knots of at most 12 crossings, and nearly all prime links of at most 12 crossings,
including many non-algebraic examples.

Note that without assuming any of our conjectures, by
Proposition~\ref{prop:quant-spec} these 2-variable knot polynomials 
must specialize at particular values of $w$ to give the quantum invariants of the adjoint
representations of the actual Lie algebras.
To our knowledge these are the
first calculations of the quantum knot polynomials of non-torus knots for the
adjoint representations of any of the exceptional Lie algebras, and they are
much more efficient than doing the calculations directly in $U_q(\mathfrak{g})$.
(For torus knots these calculations can be done using the Rosso-Jones formula
\cite{MR1209320}, see \cite{MR3491191}.)

Using our skein relations, we also compute the quantum versions of Cohen--de-Man's  \cite{MR1381778} formulas for the dimensions of the objects appearing in the third tensor
power of adjoint representation in the exceptional series.  (For the summands of $\mathfrak{g}^{\otimes 2}$ such quantum dimension
formulas were already given by Tuba-Wenzl \cite{MR2132671}.)

The programs internally work with formulas
$\QExc_{\bbC,v,\alpha}$, since the degrees are smaller.

\subsection{Symmetries and the parameter plane}
\label{sec:symmetries}

For reference, we list some global symmetries of the category
$\mathsf{QExc}_{R,v,w}$. Consider the following involutions on our parameters:

\begin{align}
  v &\leftrightarrow v &
    w & \leftrightarrow -w &
      b & \leftrightarrow b &
        t & \leftrightarrow t &
         \alpha  & \leftrightarrow \alpha \label{eq:sym-negw} \\
  v &\leftrightarrow v &
    w & \leftrightarrow v/w&
      b & \leftrightarrow b &
        t & \leftrightarrow t &
         \alpha  & \leftrightarrow \alpha \label{eq:sym-lambda-auto} \\
  v &\leftrightarrow -v &
    w & \leftrightarrow w &
      b & \leftrightarrow -b &
        t & \leftrightarrow -t &
         \alpha  & \leftrightarrow -\alpha \label{eq:sym-neg-vertex}\\
  v &\leftrightarrow v^{-1} &
    w & \leftrightarrow w^{-1} &
      b & \leftrightarrow b &
        t & \leftrightarrow t &
         \alpha  & \leftrightarrow \alpha \label{eq:sym-or-reverse}
\end{align}

Of these, symmetries \eqref{eq:sym-negw}
and~\eqref{eq:sym-lambda-auto} do not
change
any of the skein relations. The first is an artifact of
the parameterization, and specifically our insistence on formulas
that are balanced under $w \leftrightarrow w^{-1}$. The
second interchanges the roles of $\lambda$ and $1-\lambda$ in the
quantum number formulas,
and is thus a quantization of the symmetry of $\mathsf{Exc}_{R,\lambda}$
interchanging
$\lambda$ and $1-\lambda$. For instance it switches two
eigenvalues of the twist acting  on the
$4$-box space (Lemma~\ref{lem:eigenvalues-twist}).

Symmetry~\eqref{eq:sym-neg-vertex} rescales $b$ by $-1$, as explained
in \S\ref{sec:intro-classical}. Specifically it acts on diagrams in
$D$ with $n$ incoming vertices, $m$ outgoing vertices, and $k$
trivalent vertices by
\[
  D \mapsto (-1)^{(m-n+k)/2}\cdot D.
\]

Symmetry~\eqref{eq:sym-or-reverse} acts on diagrams by reflecting in the plane
of the knot diagram, switching the sign of all crossings. (The powers of $v$
appearing in \eqref{eq:QEJac-alpha} were chosen so that $\alpha$ does
not change under this reflection.)

These symmetries generate a group of order $16$. (Symmetries
\eqref{eq:sym-lambda-auto}
and~\eqref{eq:sym-neg-vertex} do not commute).
If we restrict the
action to the torus $v = e^{i \theta}$ and $w = e^{i \varphi}$ for real
$\theta, \varphi$, one fundamental domain for the group is
\begin{align*}
  0 & \le \theta \le \pi/2 \\
  (\theta - \pi)/2 & \le \varphi \le \theta/2,
\end{align*}
as shown in Figure~\ref{fig:fund-domain}, along with the lines
$w = v^\lambda$ corresponding to the exceptional series and certain
other lines explained in \S\ref{sec:spec-f4-g2}.
\begin{figure}
  \[
    \includegraphics{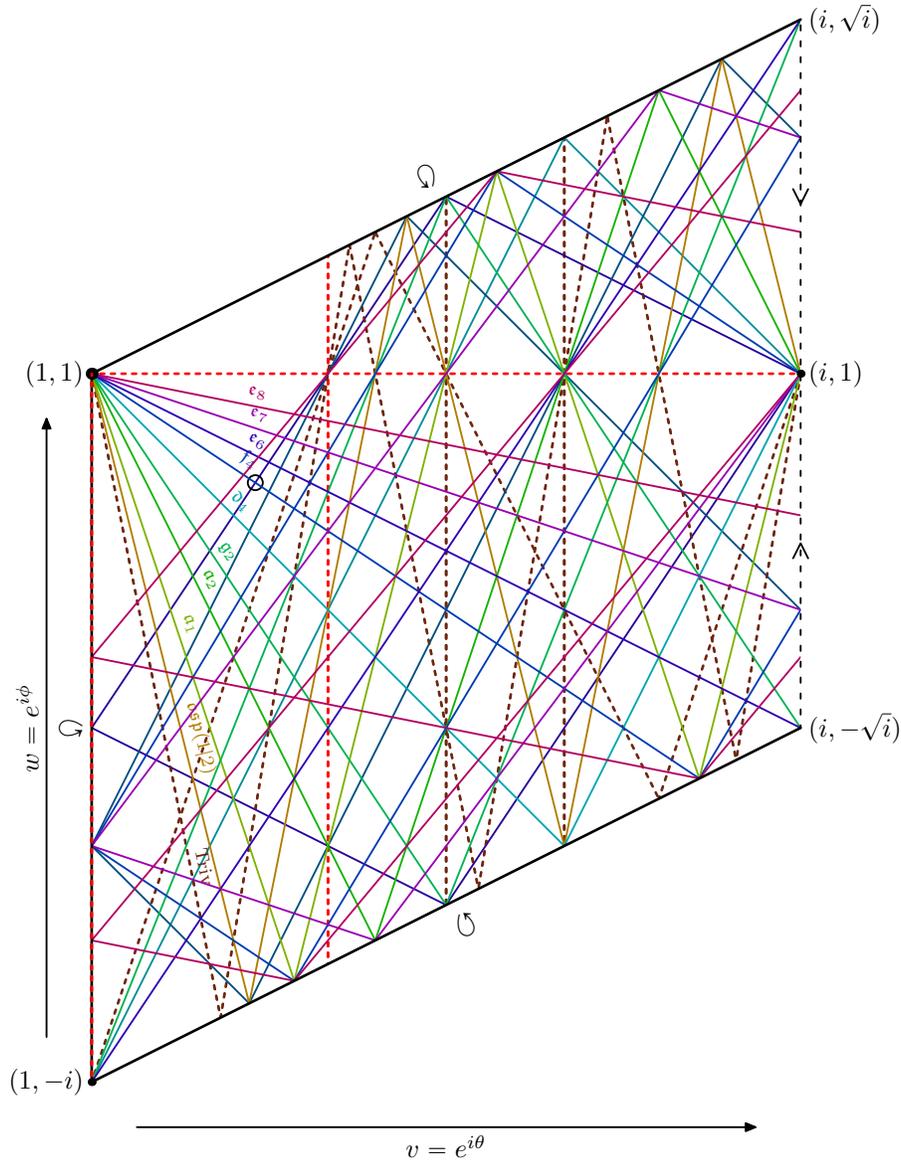}
  \]
  \caption{A fundamental domain for the action of symmetries on
    $(v,w)$, restricted to the torus $\abs{v} = \abs{w} = 1$. The left
  side is glued to itself by $(v,w) \leftrightarrow (1/v,w/v)$, the
  top by $(v,w) \leftrightarrow (v,v/w)$, the bottom by $(v,w)
  \leftrightarrow (v,-v/w)$, and the top of the right side is glued to
  the bottom of the right side by
  $(v,w)\leftrightarrow (-1/v,1/w)$.
  The colored solid lines are
  the loci $w = v^\lambda$ and their symmetric images, for $\lambda$
  one of the values in
  Table~\ref{tab:lambda-group}.
  The dashed red lines
  are loci where denominators in the definition of
  $\mathsf{QExc}$ vanish. The dashed brown lines are the loci where
  either $d=0$
  or $b=0$. This figure also includes other known specializations
  (\S\ref{sec:spec-f4-g2}), and there
  are other loci that are exceptions to the main theorems
  (Figure~\ref{fig:exceptions}).
}
  \label{fig:fund-domain}
\end{figure}

\subsection{Special values}
In \S\ref{sec:special-values}, we look at what happens in
certain special cases. For many of these special cases 
the change of variables is not well-defined, so we will often need
to return to the $\alpha$-version of the main relations.

We begin by outlining how to prove specialization theorems in the style of
Proposition~\ref{prop:class-spec}, using ideas from Kuperberg's paper
\cite[Thm. 5.1]{MR1403861}, as expanded on by
\cite[Thms. 2.1 and 2.3]{MR3951762} and \cite[Thm. 5.3]{GSZ23:DiagrammaticsF4}.

Second in \S\ref{sec:classical-G2F4}, we collect and expand upon some results about Deligne's symmetric group category,
 the classical $G2$ and $F4$ families, and the relationship between these three categories and
 Frobenius algebras, Hurwitz algebras, and cubic Cayley-Hamilton Jordan algebras.

Third, we write down explicitly which instances of the quantum exceptional relation hold
in various special cases where $\dim \cC(4) < 5$ and so multiple relations hold.

Fourth, when $v$ is a $10$th or $12$th root of unity,
Corollary~\ref{cor:maincor} does not
apply.  We show that when $v = \pm 1$ then the classical exceptional relations
must hold (see \cite{2204.13642} for some related results).  We show that when $v = \pm i$
and an additional non-degeneracy assumption holds, the relations of the classical $F4$ series 
must hold (see \cite{GSZ23:DiagrammaticsF4} for some related results).  Finally we show that when $v$
is a primitive $3$rd or $6$th root of unity, then the relations of the classical $G2$ series
must hold (see \cite{Boos:thesis, 1011.6197, Street19:VectProd, twf169} for related results).  
We delay a full treatment of all root of unity cases to later work.

We prove the classical and quantum specialization results for
the exceptional family in
\S\ref{sec:Specializations}. Then in \S\ref{sec:spec-f4-g2} we
prove analogous results for
the $G2$ and~$F4$ families; see \S\ref{sec:classical-G2F4} for
background on the classical $G2$ and~$F4$ families, each of which has
several specializations to Lie (super)groups.

Finally we
conjecture a number of coincidences of quantum group categories.

\begin{remark}\label{rem:F4-deformation}
  The category $\mathsf{QExc}_{R,v,w}$, viewed in this paper as a
  quantized version of the conjectural exceptional family, can also be
  viewed as a quantization of the $F4$ family. Indeed, the main
  difference between the classical exceptional and $F4$ families is
  whether the generating trivalent vertex is anti-symmetric or
  symmetric, and thus the relation \eqref{eq:QEJac-w} can be viewed
  as a quantum deformation of the defining $F4$ relation from
  Definition~\ref{def:classical-F4}. The fact that the classical
  consistency fails for the $F4$ family \cite{ThurstonF4Family} 
  (see \S \ref{sec:classical-F4})
   has no implications about
  quantum consistency for $\mathsf{QExc}_{R,v,w}$. Such an
  interpretation doesn't really work for the classical $G2$ family,
  since the $4$-box space for the classical $G2$ family only has
  dimension~$4$.
\end{remark}

\subsection{Other work}\label{sec:other-work}

There has been a great deal of work on the classical (1-parameter)
exceptional series and the larger (2-parameter) Vogel plane, but
relatively little on the quantized versions. On one level, some
previous work has calculated universal knot invariants, which can be
thought of as the knot invariants coming from the putative quantized
Vogel plane, hence giving specializations to the quantum knot
invariants of \emph{any} adjoint representation, exceptional or not.
Mironov-Mkrtchyan-Morozov \cite{MR3491191} explain the calculation
of the universal knot invariants for all 2- and 3-strand torus knots.
Westbury \cite{1510.08307} explains the context of the quantized Vogel
plane, and gives the invariants for $(2,n)$ torus knots (in \S 6.1 of
that paper).
Mironov-Morozov \cite{MR3475991} explain how to calculate the
universal invariants for all algebraic (a.k.a.\ arborescent) links.
Note that the class we calculate includes many non-algebraic links.
The first is $8_{18}$; see Appendix~\ref{app:E8}.

Finally, Westbury uses similar techniques \cite[\S4]{2211.04270}  in a
related context with two different types of strands.

\subsection{Acknowledgements}
KM was partially supported by Discovery Project grants  ‘Subfactors
and symmetries’ DP140100732 and ‘Low dimensional categories’
DP160103479, and a Future Fellowship ‘Quantum symmetries’ FT170100019
from the Australian Research Council.
NS was supported
by NSF grants DMS-1454767 and DMS-2000093, and would like to
thank MSRI where some of this work was done; research at MSRI is
supported in part by NSF grant DMS-1928930.
DPT was supported by NSF
grants DMS-1008049,
DMS-1507244,
and DMS-2110143.
Thanks to Joshua Chen,
whose algorithms we used in drawing some of the figures in this paper.
Thanks to Emily Peters, Alistair Savage, Pierre Vogel, and Bruce Westbury for helpful
conversations, and to the referee for helpful comments.



\section{Background} \label{sec:Background}

\subsection{Trivalent tangles and braided trivalent categories}\label{sec:Tri}

Let $R$ be a ring and let $\mathsf{Tri}_R$ denote the ribbon category
of $R$-linear combinations of framed unoriented knotted trivalent
graphs.  That is, its objects are points on the line, and its
morphisms are formal linear combinations of framed unoriented knotted
trivalent graphs with composition given by vertical gluing, and tensor
product given by horizontal disjoint union.  Note in particular that
the trivalent vertex is rotationally invariant.

Suppose that $X$ is an object in a ribbon category $\cC$ over $R$,
$s\colon X \rightarrow X^*$ is a self-duality
(i.e., one where $(p_x)^{-1} \circ s^* = s$
where $p_X \colon X \to X^{**}$
is the pivotal structure),
and $\tau\colon X \otimes X \rightarrow X$ is a map which is rotationally invariant in the sense that
the map $$X \otimes X \rightarrow X \otimes X \otimes X \otimes X^* \rightarrow X^* \otimes X \otimes X \rightarrow X$$ 
given by
\[(\mathrm{ev} \otimes \mathrm{id}) \circ (s \otimes \tau \otimes s^{-1})
\circ (\mathrm{id} \otimes \mathrm{id} \otimes \mathrm{coev})\]
is equal to $\tau$.  Then the ribbon diagram calculus \cite{MR2767048} defines a functor 
$\mathsf{Tri}_R \rightarrow \cC$ sending the strand to $X$ and the trivalent vertex to $\tau$.

In particular, if $\cC$ is a trivalent ribbon category in the sense of
Definition \ref{def:TrivalentRibbon}, then by \cite[Lemma
2.2]{MR3624901}, $X$ is symmetrically self-dual, and by
\cite[Lemma 8.2]{MR3624901}, the trivalent vertex is rotationally
invariant, so we get a ribbon functor as above.
This functor depends on a choice of isomorphism $X \rightarrow X^*$ and a choice of map $X \otimes X \rightarrow X$, each of which is well-defined up to rescaling.

Similarly, we can consider the symmetric ribbon category $\mathsf{SymTri}_R$ of framed unoriented symmetric graphs where we do not
distinguish between overcrossings and undercrossings.  Again if we
have a symmetric ribbon category $\cC$ over $R$
with a choice of an object $X$, a symmetric self-duality $s$, and a rotationally invariant map $X \otimes X \rightarrow X$, then we get a functor $\mathsf{SymTri}_R \rightarrow \cC$.

Several of our categories have more than one tensor generator which
support a product making it a trivalent ribbon category, giving
distinct functors from $\mathsf{Tri}_R$. In ambiguous cases we frequently
add the tensor generator to disambiguate. For instance, $G_2$ arises
with both its defining $7$-dimensional representation, denoted
$V_{(1,0)}$ in our weight ordering, and its adjoint $14$-dimensional
representation $V_{(0,1)}$. Thus $(G_2,V_{(1,0)})$ and
$(G_2,V_{(0,1)})$ denote the same category $\mathrm{Rep}(G_2)$ but
indicate different functors from $\mathsf{SymTri}$, and likewise the
quantized versions $(G_2,V_{(1,0)})_q$ and $(G_2,V_{(0,1)})_q$ are the
same ribbon category but with different functors from $\mathsf{Tri}$.
(In principle we should follow Definition~\ref{def:TrivalentRibbon}
and also include the trivalent vertex $\tau$, but in almost all cases
the only ambiguity
there is the scale, which is mostly fixed by specifying~$b$.)
We also use $G$ as shorthand for the category $\mathsf{Rep}(G)$, and
similarly for the quantized versions.

\subsection{Diagrams, Lie algebras, and inner product conventions}
\label{sec:LAConventions}

Consider the symmetric ribbon category of Jacobi diagrams, $\mathsf{Jac}_{R,D,B}$
from Definition~\ref{def:Jacobi}.

\begin{definition}
A \emph{metric Lie algebra object} in a symmetric ribbon category $\cC$ is an
object $\mathfrak{g}$ with a bracket $[\cdot,\cdot]\colon \mathfrak{g}
\otimes \mathfrak{g} \rightarrow \mathfrak{g}$
satisfying the Jacobi identity and antisymmetry, and a symmetric pairing
$(\cdot,\cdot)\colon \mathfrak{g} \otimes \mathfrak{g} \rightarrow 1$ which is invariant under
the adjoint action and such that there exists (a necessarily unique)
copairing $C\colon 1 \rightarrow \mathfrak{g} \otimes \mathfrak{g}$ exhibiting
$\mathfrak{g}$ as its own dual, i.e., satisfying the zig-zag axioms.
We will denote the copairing on $\mathfrak{g}^*$ using angle brackets,
as
$\langle\cdot,\cdot\rangle \colon \mathfrak{g}^* \otimes
\mathfrak{g}^* \to 1$.
\end{definition}

The strand in $\mathsf{Jac}_{R,D,B}$
with the trivalent vertex as the bracket and the cap and cup as the pairing and copairing
is a metric Lie algebra object. Moreover it has the following well-known universal 
property, essentially described by Bar-Natan \cite{MR1318886}.

\begin{lemma}
  Suppose $\cC$ is a symmetric ribbon category over~$R$ with $\End(1) = R$
  and $\mathfrak{g}$ a metric Lie algebra
object with $\End(\mathfrak{g}) = R$.  Then there exist $D,B \in R$
and a functor
$\mathsf{Jac}_{R,D,B} \rightarrow \cC$ sending the strand to $\mathfrak{g}$ and
preserving the metric Lie algebra structure.
\end{lemma}

Suppose $\mathfrak{g}$ is a simple Lie algebra over $\mathbb{C}$. Then there exists a 
$\mathfrak{g}$-invariant bilinear form which is unique up to rescaling. Each choice of such
bilinear form makes $\mathfrak{g}$ into a metric Lie algebra and thus gives us a symmetric functor 
from $\mathsf{Jac}_{R,D,B}$.  It's easy to see that $D = \dim \mathfrak{g}$. The value of $B$ will
depend on the choice of bilinear form.  In particular, rescaling the choice
of bilinear form will change the interpretation of a graph by a scalar. If the graph is isotoped so that
all trivalent vertices have two strands coming in from below and one from above, then the value of 
the diagram is homogeneous of degree the number of local
maxima in the diagram minus the number of local minima.

More generally given a metric Lie algebra object
$\mathfrak{g}$ and a representation object $V$,
we can interpret diagrams with unoriented edges labelled by $\mathfrak{g}$, oriented edges 
labelled by $V$, and trivalent vertices corresponding to both the Lie bracket (satisfying the Jacobi
and antisymmetry relations) and to the action of $\mathfrak{g}$ on $V$,
satisfying the axioms of an action.
(See \cite[\S2.4]{MR1318886}, though see \cite{MR2304469} for sign
conventions.)

The goal of this section is to compare several conventions for choice of bilinear form
on a simple Lie algebra $\mathfrak{g}$, and to do a few simple diagram calculations.

We will be interested in four different invariant bilinear forms on $\mathfrak{g}$:
\begin{itemize}
\item the ``Short'' form $(x,y)_S$ where short roots have squared length $2$,
\item the ``Long'' form $(x,y)_L$ where long roots have squared length $2$,
\item the Killing form $(x,y)_K$ defined by $\mathrm{Tr}(\mathrm{ad}_x \circ \mathrm{ad}_y)$,
\item the ``Convenient'' form $(x,y)_C$ which is $1/12$ of the Killing form.
\end{itemize}
Note that the roots in the definition of the Short and Long form live
in $\mathfrak{h}^*$ (or its complexification), and use the copairing
$\langle \cdot, \cdot \rangle$
on that space (which is a quotient of $\mathfrak{g}^*$).

The main reason for using all four of these conventions at various times is to use results
in the literature which already made a preferred choice.  A secondary reason is that when
considering the Kontsevich integral it is more natural to use the Killing form or the Convenient form, 
while for quantum groups it's more natural to use the Short form or
the Long form.
We will put subscripts on the various quantities that depend on
the bilinear forms, notably the bigon parameter~$B$; in these contexts,
$*$ is a generic
subscript that can represent any of the four options.

\begin{remark} \label{rem:SuperKilling}
If $\mathfrak{g}$ is a simple Lie superalgebra (i.e., a Lie algebra object in the category
of super vector spaces), then we can similarly define a Killing form using the supertrace.
However, unlike in the case of simple Lie algebras, for a superalgebra it can happen that 
the Killing form is identically zero. If the Killing form is non-zero then it is automatically 
non-degenerate, and we can again make the above definitions. In this paper we will
only ever consider Lie superalgebras that have non-zero Killing form. See \cite[\S 23]{MR1773773}
for a list of which simple Lie superalgebras have zero or non-zero Killing form.
\end{remark}

Recall that the Dual Coxeter number $h^\vee$ can be defined to be half
of the ratio of the Killing form
to the Long form
\cite[\S 4.8]{MR0268192} \cite{25680}, while
$\Delta$ (usually denoted~$D$, but we use $D$ for certain
dimensions)
is the \emph{lacing number}, which is the ratio of the Long form to
the Short form.  These two facts let us
easily determine the ratio between any two of these forms. For example the Convenient form is
$h^\vee\Delta/6=-\Delta/\lambda$ times the Short form. (Here $\lambda
= -6/h^\vee$ is the parameter from the introduction.) See also
Table~\ref{tab:bilinear-normalizations}.

\begin{lemma}
For any choice of invariant bilinear form, the following diagram maps to the Killing form in $\mathrm{Rep}(\mathfrak{g})$:
$$\mathcenter{\drawbubblecap}.$$
\end{lemma}
\begin{proof}
Let $e_i$ and $e^i$ denote dual bases of $\mathfrak{g}$ with respect
to the given form $(\cdot,\cdot)$.
The following diagram is isotopic to the one in the statement of the
lemma:
$$\mathcenter{\begin{tikzpicture}[scale=0.25]
	\begin{pgfonlayer}{nodelayer}
		\node (inleft) at (-5.25, -2.25) {};
		\node (inright) at (-3.5, -2.25) {};
		\node (1) at (-2.25, 0.25) {};
		\node (2) at (-2.25, 1.75) {};
                \node (3) at (-1.125, 2.875) {};
		\node (4) at (0, 1.75) {};
		\node (5) at (0, 0.25) {};
                \node (6) at (-1.125, -0.875) {};
	\end{pgfonlayer}
	\begin{pgfonlayer}{edgelayer}
		\draw [bend left=23] (inleft.center) to (2.center);
		\draw [bend left=15] (inright.center) to (1.center);
                \draw (1.center)--(2.center)..(3.center)..(4.center)--(5.center)..(6.center)..(1.center);
	\end{pgfonlayer}
\end{tikzpicture}}.$$

By the diagram calculus, this diagram maps to the composition 
$$x \otimes y \mapsto \sum_i x \otimes y \otimes e_i \otimes e^i \mapsto x \otimes [y,e_i] \otimes e^i \mapsto (e^i,[x,[y,e_i]])$$
which is exactly the Killing form.
\end{proof}

In particular, we see that with respect to the Killing form the bigon parameter is $B_K = 1$, and
if we rescale the pairing $(\cdot,\cdot)$ by $c$ the copairing and the bigon
parameter both scale by $1/c$.

\begin{definition}
Fix an invariant bilinear form on $\mathfrak{g}$.
Suppose that $V$ is an irreducible representation of $\mathfrak{g}$. Define
$C_V$ to be the eigenvalue of the quadratic Casimir on~$V$, i.e., the unique scalar such that
$$
\casimirbubble = C_V \; \orientedonestrandid \;,
$$
where the unoriented strand is labelled $\mathfrak{g}$, the oriented
strand is labelled $V$, and the trivalent vertices are the action map
$\mathfrak{g} \otimes V \rightarrow V$.  In other words,
$$\sum_i e^i(e_i v) = C_V v.$$
\end{definition}

The quadratic Casimir $C_V$ involves a cup and no caps, so
scales like the bigon.

Let $\rho \in \mathfrak{h}^*$
 be the Weyl vector. (In the super case,
this is half the sum of
positive even roots minus half the sum of the positive odd roots.) The
following fact is well-known.
(See, e.g., \cite[Prop.\ 5.28]{MR1920389} for a proof for
$(\cdot,\cdot)_K$, and note that the two sides scale the same way as the pairing varies.)
\begin{proposition}\label{prop:quadcasimir}
Let $\mu$ be a dominant weight for a simple Lie
(super)algebra~$\mathfrak{g}$ and $V_\mu$ the corresponding highest
weight
representation. With respect to any form, the eigenvalue of the
quadratic Casimir for $V_\mu$ is
$C_\mu  \coloneq C_{V_\mu} = \langle\mu, \mu+2\rho\rangle$.
\end{proposition}

Let $\theta$ denote the highest root, i.e., the highest weight for the adjoint representation.  
By definition $C_\theta  = B$.  In particular, with respect to the Long form we have 
$2h^\vee = B_L = \langle\theta, \theta + 2 \rho\rangle_L = 2 + 2 \langle\theta, \rho\rangle_L$.  This gives a second characterization
of $h^\vee$, namely $h^\vee = 1+\langle\theta, \rho\rangle_L$.

\begin{lemma}
With respect to the Long form, we have $C_{2\theta,L} = 4(h^\vee+1)$.
\end{lemma}
\begin{proof}
We can directly calculate:
\[
  C_{2 \theta,L} = \langle2\theta, 2\theta + 2 \rho\rangle_L = 8 + 4 \langle\theta,
  \rho\rangle_L = 8 + 4(h^\vee -1) = 4(h^\vee+1).\qedhere
\]
\end{proof} 

Finally we will want a few results about the eigenvalues of the ladder
operator
\[\drawH\colon \mathfrak{g} \otimes \mathfrak{g} \rightarrow \mathfrak{g} \otimes \mathfrak{g},\]
or in other words $\smallfig{\drawH}(x,y) = \sum_i [x,e_i][e^i,x]$.

\begin{lemma}
If $V_\mu$ is a simple summand of $\mathfrak{g} \otimes \mathfrak{g}$
and $L_\mu$ is the eigenvalue of the ladder operator as it acts on $\End(\mathfrak{g} \otimes \mathfrak{g})$
with eigenvector the projection map onto $V_\mu$,
then
$$C_\mu = 2B - 2L_\mu.$$
(This equation is homogeneous in rescaling the bilinear form so holds for all forms.)
\end{lemma}
\begin{proof}
  Since $\mathfrak{g}$ acts on the tensor product $\mathfrak{g} \otimes \mathfrak{g}$
  by the Leibniz rule,
  i.e., by adding its action on each factor, we see that
$$\casimirbubble = \ladderaa\; + \ladderab\; + \ladderba\;+ \ladderbb\; = (2B - 2L_\mu)\; \orientedonestrandid\;,$$
where the minus sign comes from antisymmetry of the bracket.
\end{proof}

\begin{corollary} \label{cor:ladderontwicelongestroot}
With respect to the Long form, $L_{2\theta,L} = -2$. With respect to the Convenient form, $L_{2\theta,C} = -12/h^\vee$.
\end{corollary}

\begin{lemma} \label{lem:ladderzero}
Suppose that $W$ is any subrepresentation of $\bigwedge^2 \mathfrak{g}$ which is
perpendicular to $\mathfrak{g}$. Then $L_W = 0$ (with respect to any form).
\end{lemma}
\begin{proof}
The projection onto the perpendicular complement of $\mathfrak{g}$ in
$\bigwedge^2 \mathfrak{g}$ is given by
$$\frac{1}{2}\twostrandid -\frac{1}{2}\symcross - \frac{1}{B} \drawI.$$
Acting on this by ladder operator, and using that triangles give a
factor of $B/2$, we get
$$\frac{1}{2}\drawH -\frac{1}{2}\drawsyminvcrossX - \frac{1}{2} \drawI.$$
But this expression is zero by the Jacobi identity.
\end{proof}

We summarize the results of this section by
Table~\ref{tab:bilinear-normalizations}, showing the values of the relevant quantities
with respect to all four bilinear forms.

\begin{table}
\begin{center}
  \begin{tabular}{rcccc}
    \toprule
    &\multicolumn{4}{c}{Bilinear form} \\ \cmidrule(l){2-5}
    Quantity & $(v,w)_L$ & $(v,w)_S$ & $(v,w)_C$ & $(v,w)_K$ \\ \midrule
    $\langle\mu,\mu\rangle$ for $\mu$ long & $2$ & $2\Delta$
          & $12/h^\vee = -2\lambda$ & $1/h^\vee = -\lambda/6$ \\[2pt]
    $\langle\mu,\mu\rangle$ for $\mu$ short & $2/\Delta$ & $2$
         & $12/(\Delta h^\vee) = -2\lambda/\Delta$ & $1/(\Delta h^\vee) = -\lambda/(6\Delta)$ \\[2pt]
$B = C_\theta = \langle\theta,\theta+2\rho\rangle$ & $2h^\vee$ & $ 2h^\vee\Delta$ & $12$ & $1$ \\[2pt]
$C_{2\theta}$ &  $4(h^\vee+1)$ & $4\Delta(h^\vee+1)$ & $24(1+1/h^\vee)$ & $2(1+1/h^\vee)$ \\[2pt]
$L_{2\theta}$ & $-2$ & $-2\Delta$ & $-12/h^\vee = 2 \lambda$ & $-1/h^\vee = \lambda/6$ \\[2pt]
$L_W$ & $0$ & 0 & $0$ & $0$ \\
\bottomrule
\end{tabular}
\end{center}  
  \caption{Comparing the different normalizations of bilinear forms on
    a simple Lie algebra. In this table, $\lambda = -6/h^\vee$ is the parameter of the Classical Exceptional Series, $\theta$ is the highest root,
$L_\mu$ is the eigenvalue of the ladder
operator, and $W$ is any subrepresentation of $\bigwedge^2\mathfrak{g}$
perpendicular to $\mathfrak{g}$.}
  \label{tab:bilinear-normalizations}
\end{table}

\subsection{Quantum group conventions}
\label{sec:QGConventions}
  
In this subsection we fix conventions for quantum groups. The details
here will only come up in choosing some conventions for
certain diagram categories, and in the proof of quantum specialization in \S\ref{sec:special-values}.
This section is aimed at experts wishing to compare our conventions to those in other papers.
  
If $\mathfrak{g}$ is a complex semisimple Lie algebra, let $U_s(\mathfrak{g})$
denote the Drinfel'd-Jimbo quantum group, and let $\Rep{U_s(\mathfrak{g})}$
denote its ribbon category of representations.  
We follow the conventions from \cite{MR2286123, MR2783128}.  See \cite[p. 2]{MR2286123}
for a comprehensive summary of how his conventions line up with those in other
sources. (Thus our $q$ is the same as Sawin's~$q$ and Lusztig's~$v$.) In
particular we have variables $s$ and $q$ and the relation $s^L = q$
where $L$ is the smallest integer such that
$L(\mu,\nu)_S \in \mathbb{Z}$ for any two weights $\mu$ and $\nu$.
We always consider Type I finite dimensional representations.
The quantum group itself and its representation theory
only depend on $q$, while the braiding and the ribbon category depend on the
additional choice of $s$.   When talking about quantum groups
we will typically use the bilinear form $( v,w)_S$ where
short roots have squared
length two. More generally, for $s$ not a root of unity,
the braiding between irreps $V_\mu$ and $V_\nu$ depends only on 
$s^{L( \mu, \nu)_S}$ and so for $\mu$ in the
root lattice
(as usually happens for us) the braiding depends only on $q$.

Our reason for fixing these conventions
is that it's compatible with
all the other papers by KM and NS, as well as KM's
quantum
groups Mathematica package \cite{QuantumGroups}. Note that using the
short form is somewhat inconvenient
for our purposes in this paper, and will result in some slightly complicated changes of 
variables when we compare $\mathsf{QExc}_{R,v,w}$ to actual quantum group
categories.

Since our tangles are unoriented, there's no way to distinguish positive and
negative crossings in a link diagram, because an isotopy that rotates the crossing by 
$90$ degrees will switch the two crossings. Nonetheless as a tangle with four boundary
points, there are two crossings which are not isotopic rel boundary and we need to decide which 
of these crossings is declared positive (i.e., which acts on
representations using the $R$-matrix, rather than $R^{-1}$).
Since we compose tangles vertically, our convention is that a crossing 
 is positive if it would be positive when oriented vertically, namely
$$\invbraidcross \; \text{is the positive crossing}.$$
However, the following diagram is the \emph{positive} twist written
horizontally and thus is a positive power of $q$ times the cup, even
though the middle of the diagram is a negative crossing when
considered as a 4-boundary point tangle:
\begin{equation}\label{eq:neg-twist-positive}
  \twist\; \text{is a positive power of $q$ times}\; \drawcup.
\end{equation}
For this reason, when we analyze the eigenvalues of a crossing acting
on $\Hom(X^{\otimes 2}, X^{\otimes 2})$ we will typically use the
\emph{negative} crossing, an operator we denote $\HTw$.

If $s$ is not a root of unity, and $\Lambda'$ is a sublattice of the
weight lattice containing the root lattice,
then the representations whose weights lie in $\Lambda'$ form a ribbon
subcategory.  In particular, we can talk about the \emph{adjoint subcategory}
corresponding to the root lattice.  This subcategory is exactly the subcategory
generated under tensor product, direct sums, and direct summands by
the adjoint representation.  By the above discussion,
the adjoint subcategory as a ribbon category depends only on $q$ and not~$s$.
We will abuse notation often by referring to subcategories supported on $\Lambda'$
as quantum versions of the group for which representations supported on~$\Lambda'$ lift
to representations of the group.   For example, $SO(3)_q$ denotes the adjoint subcategory
of the category of Type I finite\hyp dimensional representations of $U_q(\mathfrak{sl}_2)$.

Finally, if we have a group of automorphisms of the Dynkin diagram,
then this gives a group of braided autoequivalences of the corresponding
ribbon category, and hence we can take the equivariantization \cite[\S2.7]{MR3242743} by this
action to get a new ribbon category.  We will again abuse notation
by referring to this as the quantum analogue of the corresponding semidirect
product group.  For example $(PSL(3) \rtimes \mathbb{Z}/2\mathbb{Z})_q$ denotes the
equivariantization of the subcategory  of
representations of
$U_q(\mathfrak{sl}_3)$
which are supported on the root lattice by the group of Dynkin diagram automorphisms of $A_2$.
This category deforms the category of representations of  the group 
$PSL(3) \rtimes \mathbb{Z}/2\mathbb{Z}$.

We use similar conventions for quantum enveloping algebras attached to simple basic Lie superalgebras with
non-zero Killing form (see Remark \ref{rem:SuperKilling}).  We follow the definitions of \cite{MR1266383,MR1892766}, except that we always normalize the inner product so that the short root has length two and use $q = q_S$.  Also note that in \cite{MR1266383} their coproduct is opposite from ours.  In general the quantum enveloping algebras depend not just on the Lie superalgebra but on the choice of simple roots.
For a collection of conventions for simple basic Lie superalgebras see \cite{MR1773773}.  Again, as in the classical case, by the 
quantum analogue of a Lie supergroup we will mean the appropriate equivariantization of the appropriate subcategory
of representations of the quantized enveloping algebra.

In fact, the only quantum supergroup which we will need to understand in detail is $SOSp(1|2)_q$.
In \S\ref{sec:some-braid-triv}, we will explain how this relates to $SO(3)_q$, and so understanding the general conventions for
quantum supergroups will not be important for this paper.  (Warning:
in \cite{MR3624901} this category was denoted $OSp(1|2)_q$,
but this was a poor choice of notation which is not compatible with how we denote categories attached to groups.)

\subsection{Explicitly comparing conventions} \label{sec:comparing}

It can sometimes be difficult to work out what conventions another paper is using, and thus to find the right change of variables between our conventions and their conventions.  In this subsection we'll explain why other conventions occur and how to find the correct change of variables.

The main source of different $q$ conventions in the literature is that
there are
variables $q_i$ for each simple root $\mu_i$ where the formula for $q_i$
depends on the length of~$\mu_i$.  Our choice of $q$ is
$q = q_S$ corresponding to the short roots, so that $q_L = q_S^\Delta
= q^\Delta$.
It is also common in other papers to pick $q_L$ as $q$.   The former
convention is natural
if you use $\langle\cdot,\cdot\rangle_S$ as your inner product on the root space, while the latter
convention is natural if you use $\langle\cdot,\cdot\rangle_L$.
(More generally, there are $q$ variables associated
to any pairing;
see Remark~\ref{rem:q-inner-prod}.)
Similarly, if you're studying a family like $\mathfrak{so}(2n+1)$, if you uniformly choose $q$ to be $q_L$,
then $q_n = \sqrt{q}$ and if you then specialize to $\mathfrak{so}(3)$
you end up with a choice of $q$ which is neither
$q_S$ nor $q_L$ (which are equal because there's only one root).

In addition to the conventions for $q$, 
there are many slight variants one can give for the generators and
relations. These typically
do not change the category of representations at generic $q$ at all because they just correspond to rescaling the
choices of $E_i$ and $F_i$ either by scalars or by powers of $K$ and so generate the same algebra $h$-adically.
Because of this rescaling freedom it can sometimes be difficult to work out what $q$ conventions are being used:
the relation $[E_i,F_i] = (K_i-K_i^{-1})/(q_i-q_i^{-1})$ changes under rescaling $E_i$ which allows
you to change the denominator on the right-hand side.  However, if you look at the Serre relation normalized
so that the leading coefficient is one, then the coefficient of the next term is $-[1-a_{ij}]_{q_i}$, which does not change
under rescaling of $E_i$ and so will tell you what conventions are being used.  For example, in \cite{MR1188811}
the generators and relations for quantum $\mathfrak{so}_{2n+1}$ given use a non-standard normalization of the
final $E_n$, but by looking at the Serre relations we see that $[2]_{q_i} = q+q^{-1}$  for 
$i < n$ but $[3]_{q_n} = q+1+q^{-1}$ so they've chosen $q=q_i=q_L$ and
$q_n = q_S = q^{\frac{1}{2}} = q^{\frac{1}{\Delta}}$, so Zhang's $q$
is our $q^2$.

Comparing conventions is much easier when you allow yourself to
consider the $R$-matrix and the ribbon structure on the category of
representations of the quantum group. The positive twist acts on
$V_\mu$ by the scalar $q^{\langle\mu, \mu + 2\rho\rangle}$ where
$\rho$ is the Weyl vector and
$\langle\cdot,\cdot\rangle$ is
the (dual) inner product in the choice of~$q$. For $V_\xi$ a summand
of $V_\mu \otimes V_\nu$,
the
square of the negative
braiding acts on $\Hom(V_\mu \otimes V_\nu, V_{\xi})$
 by acting by the positive twist on $V_\mu$ and $V_\nu$ and the
negative twist on $V_{\xi}$, and thus acts by the scalar
$$q^{\langle\mu, \mu + 2\rho\rangle + \langle\nu, \nu + 2\rho\rangle - \langle\xi, \xi + 2\rho\rangle}.$$
In particular, on the summand where 
$\xi = \mu + \nu$, the square of the negative
braiding acts by $q^{-2 \langle\mu, \nu\rangle}$.

If $\mu = \nu$ and $\xi$
appears with multiplicity one, we can also ask which square root of
the double braiding eigenvalue is the single braiding eigenvalue, and this can
usually be worked out by looking at what happens as $q$ goes to $1$
and determining whether $V_\xi$ lives in the symmetric square or the
alternating square of $V_\mu$. One special case is that when $\xi
= \mu + \nu = 2\mu$: in this case by looking at the formula for the
$R$-matrix acting on the tensor product of highest weight vectors, we
can see that the single negative braiding acts by
$q^{-\langle\mu,\mu\rangle}$ in the ordinary case, and $\pm q^{-\langle\mu,\mu\rangle}$ in
the super case, where the sign is the
parity of~$\mu$.

\subsection{Some trivalent ribbon categories}
\label{sec:some-braid-triv}
We now recall the definitions and fix conventions for some ribbon categories.

\subsubsection{\texorpdfstring{$SO(3)_q$}{SO3\_q}}
\label{sec:SO3q}
\begin{definition}\label{def:SO3q}
When $q$ is not a primitive $8$th root of unity, let $SO(3)_q =
(SO(3),V_2)_q$ be the trivalent ribbon category given by the following
relations.
\begin{gather*}
  \drawH \; - \; \drawI \; + \frac{b}{q^2+q^{-2}}\; \twostrandid \; -  \frac{b}{q^2+q^{-2}}\; \cupcap\; = 0\\[5pt]
\braidcross\;  =  q^{-2}\; \twostrandid\; + (q^2 - 1)\; \cupcap\; - \frac{q^2+q^{-2}}{b}\; \drawI\\[5pt]
\begin{aligned}
    \unknot\; &= q^2+1+q^{-2}&\qquad
      \twist\; &= q^4 \;\drawcup&\qquad
        \twistvertex\; &= -q^2\;\threevertex\\[5pt]
    \loopvertex\;&=0&
      \twogon\;&= b\;\onestrandid&
        \threegon\; &= \frac{q^2-1+q^{-2}}{q^2+q^{-2}} b\;\threevertex
\end{aligned}
\end{gather*}
\end{definition}

(The same quantum group with its $5$-dimensional representation, the
trivalent ribbon category $(SO(3),V_4)_q$, also arises, as does the
$2$-dimensional representation
$(SU(2),V_1)_q$; see
\S\ref{sec:spec-f4-g2}. But the above category is much more frequent
as a special case, so we give it the shorter name.)

As usual we can choose any non-zero value of $b$ without changing the
category.  Also sending $q \mapsto -q$ doesn't change the category.
Sending $q \mapsto q^{-1}$ doesn't change the underlying pivotal
category but does change the braiding.
The standard normalization of $b$ which arises from thinking of $SO(3)_q$ as the 
full subcategory of Temperley-Lieb generated by the second Jones-Wenzl projection 
is $b = \frac{q^2+q^{-2}}{q+q^{-1}}$.
In \cite{MR3624901} we normalize
so that $b=1$.

Note that here we have not semisimplified, so when $q$ is a root of
unity this category will be degenerate. For all but a few small roots
of unity it will agree with the category of tilting modules for
quantum $SO(3)$ defined algebraically. (This is well-known and appears
many places in the literature; a particularly nice argument is given
by Webster \cite[Lemma A.7]{MR3624396}.)

\subsubsection{\texorpdfstring{$SOSp(1|2)$}{SOSp(1|2)}}
\label{sec:sosp12}
When $q=\pm i$ the diagrammatic category defined in the previous
subsection comes from the supergroup $SOSp(1|2)$ instead of $SO(3)$,
see \cite{MR3624901}.  We want to explain this correspondence in more
detail to clarify conventions and to explain the relationship between
the quantized versions $SOSp(1|2)_t$ and $SO(3)_q$.  This is based on
the descriptions in \cite{MR1010993, MR1057559}, but modified to match
our conventions.  In particular,
$\mathfrak{osp}(1|2)$ has two positive
roots, an odd root $\alpha$ and an even root $2 \alpha$, and our
general convention is to normalize the short root so that
$\langle\alpha,\alpha\rangle = 2$, and thus $\langle 2 \alpha,
2\alpha\rangle = 8$.  This is an
unusual convention when studying $\mathfrak{osp}(1|2)$ specifically, and
leads to our formulas having slightly different $q$ conventions from
what is in the literature.

We first recall the definition of $\mathfrak{osp}(1|2)$ together with some facts about its representation theory, so that we can explain why $SOSp(1|2)$ is the relevant supergroup.

We order the three vectors of a $(1|2)$ super vector space in the
order: odd, even, odd.  With this ordering, the super transpose on
$\mathfrak{gl}(1|2)$ is given by:
$$\begin{pmatrix}
a_{11} & a_{12} & a_{13} \\
a_{21} & a_{22} & a_{23} \\
a_{31} & a_{32} & a_{33}
\end{pmatrix}^{st} = 
\begin{pmatrix}
a_{11} & a_{21} & a_{31} \\
-a_{12} & a_{22} & -a_{32} \\
a_{31} & a_{23} & a_{33} 
\end{pmatrix}
$$

The Lie algebra $\mathfrak{osp}(1|2)$ consists of supermatrices $X \in \mathfrak{gl}(1|2)$ which satisfy $X^{st} \Omega + \Omega X =0$ under the superbracket $[x,y] = xy - (-1)^{|x||y|} yx$, where 

$$\Omega = \begin{pmatrix} 0 & 0 & 1 \\ 0 & 1 & 0 \\ -1 & 0 & 0 \end{pmatrix}.$$
It is spanned by matrices
\[
  H = \begin{pmatrix} 1 & \phantom{0}& \phantom{0} \\   & 0 &  \\  &  & -1 \end{pmatrix},\,
E_+ = \begin{pmatrix} \phantom{0} & 1 & \phantom{0} \\   &  & 1 \\  &  &  \end{pmatrix}, \,
E_- = \begin{pmatrix} \phantom{0} & \phantom{0} & \phantom{0} \\ 1  &  &  \\  & -1 &  \end{pmatrix}, \,
J_+ = \begin{pmatrix} \phantom{0} & \phantom{0} & 1 \\   &  &  \\  & &  \end{pmatrix}, \,
J_- = \begin{pmatrix} \phantom{0} & \phantom{0} & \phantom{0} \\   &  &  \\ 1 & &  \end{pmatrix}.
\]
Here $H, J_+, J_-$ are even and $E_+, E_-$ are odd. The elements
$E_\pm$ are the simple roots, while $J_\pm$ are the root vectors for
twice the simple roots.

The defining relations that these generators satisfy under the supercommutator are: 
\begin{multline*}[H, E_\pm] = \pm E_\pm, \hspace{.2in} [H, J_\pm] = \pm 2 J_\pm, \hspace{.2in} 
[J_\pm, E_\pm] = 0, \hspace{.2in} [J_+,J_-] = H, \hspace{.2in} [E_+,E_-] = H, \\ [E_\pm,E_\pm] = \pm 2 J_\pm, \hspace{.2in} [J_\pm, E_\mp] = - E_\pm.\end{multline*}

The standard Cartan subalgebra is generated by $H$. Let $\alpha_\nu\in
\mathfrak{h}^*$ denote
the weight attached to a root vector $\nu$, i.e., $[x, \nu] = \alpha_\nu(x) \nu$ for $x \in \mathfrak{h}$.

This Lie superalgebra is generated just by $H, E_+, E_-$
without any need for $J_+$ and $J_-$. 
Note that the span of $H, J_+, J_-$ is a Lie subalgebra isomorphic to $\mathfrak{sl}_2$,
and hence any weight vector~$w$ in any representation has weight that
is an integer in the sense that $H w = \mu w$
 for $\mu \in \mathbb{Z}$;
that is, weights lie in $\mathbb{Z} \alpha_{E_+}$.
Here $E_{\pm}$ raise and lower weight by $1$ while $J_{\pm}$ raise
and lower weight by $2$.
The Weyl vector $\rho$ is $\frac{1}{2} \left(\alpha_{J_+} -\alpha_{E_+}\right) = \frac{1}{2} \alpha_{E_+}$.

\begin{remark}
We will prefer the Short form $(\cdot,\cdot)_S$ on
$SOSp(1|2)$. Another natural choice of invariant bilinear form is $(x, y)_{\mathrm{STr}} : = \mathrm{STr}(x \circ y)$.
In order to compare this to the short and long inner products (which are defined in terms of 
roots in the dual space $\mathfrak{h}^*$)  we can compute what the corresponding
bilinear form is on $\mathfrak{h}^*$. First note that $(H, H)_{\mathrm{STr}} = 2$.
Let $\varphi_{\mathrm{STr}}$ be the corresponding identification $\mathfrak{h} \rightarrow \mathfrak{h}^*$ 
defined by $\varphi_{\mathrm{STr}}(x) = y \mapsto \mathrm{STr}(x \circ y)$.
In particular, $(\varphi(H))(H) = 2$.
Since $[H, J_+] = 2 J_+$, we have that $\alpha_{J+}(H) = 2$ and thus $\varphi(H) = \alpha_{J_+}$.
In particular, the corresponding inner product on $\mathfrak{h}^*$ satisfies $\langle\alpha_{J_+}, \alpha_{J_+}\rangle_{\mathrm{STr}} = (H, H)_{\mathrm{STr}} = 2$, and thus $\mathrm{STr}(x \circ y) = (x,y)_L$.
\end{remark}

In fact, every finite\hyp dimensional
representation of $\mathfrak{osp}(1|2)$ is semisimple and the
irreducible representations are $V_n^\pm$ which have a highest
weight vector of weight $n \alpha_{E_+}$ and parity $\pm$. Note that
$V_{n}^+$ has graded dimension $(n+1 | n)$ and $V_{n}^-$ has
graded dimension $(n | n+1)$. The tensor product
rules are, for $n \le m$,
$$V_{n}^\sigma \otimes V_{m}^\tau =  \bigoplus_{i=0}^{n} V_{(n+m-i)}^{{(-1)}^i \sigma \tau}.$$
In particular, the defining representation $X = V_1^-$ has graded
dimension $(1|2)$ and thus super dimension $-1$, and satisfies
$$X \otimes X \cong 1 \oplus X \oplus V_2^+.$$Note, however, that this does not tensor-generate representations of
$\mathfrak{osp}(1|2)$, which is generated instead by $V_1^+$ satisfying
\[
  V_1^+ \otimes V_1^+ \cong 1 \oplus V_1^- \oplus V_2^+.
\]
Thus $\mathrm{Rep}(\mathfrak{osp}(1|2))$ is not trivalent.

To be more precise about the
representation category, we pass to the Lie supergroup. The supergroup
$OSp(1|2)$ has underlying group
$O(1) \times Sp(2,\bbR)$. The non-trivial element of the center of
$OSp(1|2)$ is (in the matrix representation above)
$\mathrm{diag}(-1,-1,-1)$. The order-$2$ elements $\varepsilon_A =
\mathrm{diag}(-1,1,-1)$ and $\varepsilon_B = \mathrm{diag}(1,-1,1)$
are not central
in the supergroup, but instead \emph{parity elements}: involutions
so that conjugation by them gives the parity map (on, say, the Lie
superalgebra). There is also a subgroup $SOSp(1|2)$ (or equivalently
the quotient
$POSp(1|2)$); that group has a unique parity
element~$\varepsilon_A$.

When looking at the representations of a Lie supergroup, one typically
studies a pair of a supergroup~$G$ and a fixed parity element
$\varepsilon\in G$, and denotes $\mathrm{Rep}(G,\varepsilon)$ the
category of representations where $\varepsilon$ acts by the parity
map \cite{MR1944506}. For $\mathrm{Rep}(SOSp(1|2),\varepsilon_A)$, this imposes a
restriction on the parity of the representations: the allowed
representations are $V_{2\mu}^{(-1)^\mu}$. On the other hand,
$\mathrm{Rep}(OSp(1|2),\varepsilon_B)$ is isomorphic to
$\mathrm{Rep}(\mathfrak{osp}(1|2))$, with $\varepsilon_B$
acting by parity. $\mathrm{Rep}(OSp(1|2),\varepsilon_A)$ is likewise
not a trivalent category.

For this paper we are interested in $\mathrm{Rep}(SOSp(1|2),\varepsilon_A)$, or
equivalently the subcategory of $\mathrm{Rep}(\mathfrak{osp}(1|2))$
generated by~$X=v_2^-$. This is a symmetric trivalent category;
to see
the sign of the trivalent vertex, note that
$\Sym^2 X$ contains the trivial representation (the
symmetric bilinear form), while $\bigwedge^2 X$ contains $V_2^+$
(the adjoint representation). Then a dimension count shows that
$\Sym^2 X \cong V_0^+ \oplus X$, so this vertex is symmetric and
$(SOSp(1|2), X)$ is in the
$F4$ family (\S\ref{sec:classical-F4}).
This is a trivalent category with
$\dim \Hom(X^{\otimes 2}, X^{\otimes 2}) = 3$, and so must have a
skein theory which agrees with some specialization of the $SO(3)_q$
skein theory defined in the previous section for some $q$ 
(as can be extracted from \cite{MR3624901}, but which we prove directly
in Corollary \ref{cor:small-dims-summary}).
and it is
easy to see from the dimension that $q = \pm i$. (The two values $\pm i$ give
the same ribbon category.)

We could likewise look at
$\mathrm{Rep}(SOSp(1|2),\varepsilon_A,V_2^+)$, taking the adjoint
representation $Y = V_2^+$ as
tensor generator. Since
\[
  Y \otimes Y \cong 1 \oplus V_1^- \oplus Y \oplus V_3^- \oplus V_4^+,
\]
this gives a trivalent category with
$\dim \Hom(Y^{\otimes 2}, Y^{\otimes 2}) = 5$, with an antisymmetric
vertex. Thus $(SOSp(1|2), Y)$ is in
the classical exceptional series. As per the remark above it is also
$(SO(3),V_{(4)})_{\pm i}$.

Fixing conventions for the quantum version of $\mathfrak{osp}(1|2)$ is
somewhat miserable. In particular, since we use $(x,y)_S$ for our quantum
parameter conventions instead of $(x,y)_L$ which agrees with the usual
pairing defined by supertrace, there will be some strange factors of four.
Instead of dealing with these conventions head-on,
 we explain the correct parameterization
for $SOSp(1|2)_t$ indirectly.
Any deformation $SOSp(1|2)_t$ will again
be a trivalent category with
$\dim \Hom(X^{\otimes 2}, X^{\otimes 2}) = 3$, and hence $SO(3)_q$ for
some change of variables between $q$ and $t$. As discussed in \S
\ref{sec:comparing}, the easiest way to find this change of variables
is to look at the action of $\HTw$ on the trivalent vertex. In
$SO(3)_q$ this trivalent twist factor is given by $-q^2$.
(This
follows from the formulas in the $SO(3)_q$ section, but can also be
seen by noting that the square of the braiding acts on
$\Hom(V_{(2)} \otimes V_{(2)}, V_{(2)})$ by $q^4$ (\S\ref{sec:comparing}) and the square root must be
$-q^2$ because as $q\rightarrow 1$ we get the antisymmetric Lie
bracket on the adjoint representation.) In $SOSp(1|2)_t$, 
we see that the square of
the braiding acts on $\Hom(V_1^- \otimes V_1^-, V_1^-)$ by 
$t^{\langle 1 \alpha_{E_+}, 1 \alpha_{E_+} + 2\rho \rangle_S} = t^{\langle 1 \alpha_{E_+}, 2 \alpha_{E_+} \rangle_S} = t^4$, and
hence the braiding acts by $t^2$.
(More generally, the eigenvalues for the negative crossing on
$\Hom(X^{\otimes 2}, X^{\otimes 2})$ for $SO(3)_q$ are
$(q^4, -q^2, q^{-2})$ while for $SOSp(1|2)_t$ they are
$(t^4, t^2, -t^{-2})$.) Thus we have the following.
\begin{lemma}\label{lem:SOSp12-SO3}
  The trivalent ribbon categories $(SOSp(1|2),V_1^-)_t$ and
  $(SO(3),V_{(2)})_{i t}$ are isomorphic, taking $V_1^-$ to $V_{(2)}$.
\end{lemma}
Since $SO(3)_q \cong \SO(3)_{-q}$,
as before, it doesn't matter which square
root of $-1$ you take.

\subsubsection{Golden categories}

\begin{definition}
When $q$ is a primitive $5$th or $10$th root of unity then $SO(3)_q$
has an additional quotient we call the \emph{golden category} with the
extra relation
$$\drawI = b \twostrandid - b \frac{1}{q^2+1+q^{-2}} \cupcap.$$

Here $\unknot = d = q^2+1+q^{-2}$ is $\frac{1\pm \sqrt{5}}{2}$.
\end{definition}

The Cauchy completion of this category is exactly the usual golden
category (also called the Fibonacci category) with simple objects $1$
and $X$ with $X \otimes X \cong X
\oplus 1$. Indeed, there's clearly a full and dominant functor from
the category defined by generators and relations to the algebraic
category, and the functor is faithful (and thus an equivalence after
Cauchy completion) because the dimensions of $\mathrm{Hom}$ spaces on both sides
are given by the Fibonacci numbers. See \cite[Remark 4.13]{MR3624901}.

\subsubsection{\texorpdfstring{$(G_2)_q$}{(G\_2)\_q}}

\begin{definition}
For $q$ not a primitive $3$rd, $4$th, $6$th, $7$th, $14$th, $16$th, or $24$th
root of unity, let $(G_2)_q = (G_2,V_{(1,0)})_q$ be the trivalent
ribbon category defined by the following relations.

\begin{align*}
\unknot\; &= \Psi_7(q) \Psi_{12}(q) = q^{10} + q^8 + q^2 + 1 + q^{-2} + q^{-8} + q^{-10}
\end{align*}

\begin{align*}
      \twist\; &= q^{12}  \;\drawcup&\qquad
        \twistvertex\; &= -q^6 \;\threevertex\\[5pt]
    \loopvertex\;&=0&
      \twogon\;&= b\;\onestrandid&
        \threegon\; &= - \frac{\Psi_{6}(q)}{\Psi_{8}(q)} b \;\threevertex
\end{align*}

\begin{align*}
\ngon[45]{4} &=  b \frac{\Psi_4(q)}{\Psi_3(q)  \Psi_{8}(q)} \left(\;\drawI \; +\; \drawH \; \right) +  b^2 \frac{1}{\Psi_3(q)  \Psi_{8}(q)^2} \left(\; \cupcap \; + \; \twostrandid \; \right) \displaybreak[1] \\[5pt]
\ngon[90]{5} &= - b \frac{1}{\Psi_3(q) \Psi_{8}(q)} \left(\mathfig{0.1}{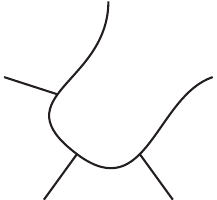} + \mathfig{0.1}{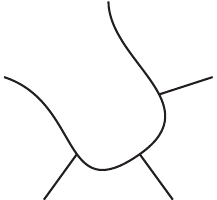} + \mathfig{0.1}{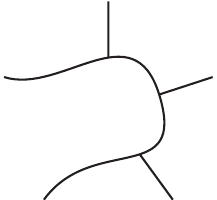} + \mathfig{0.1}{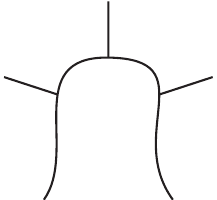} + \mathfig{0.1}{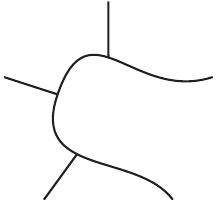} \right) \\
& \qquad - b^2 \frac{1}{\Psi_3(q)^2  \Psi_{8}(q)^2}  \left(\mathfig{0.1}{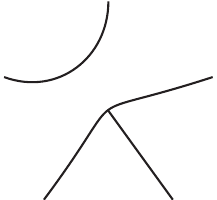} + \mathfig{0.1}{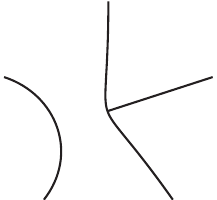} + \mathfig{0.1}{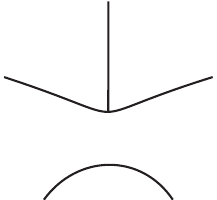} + \mathfig{0.1}{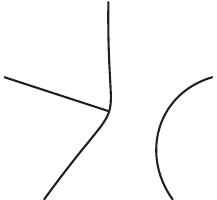} + \mathfig{0.1}{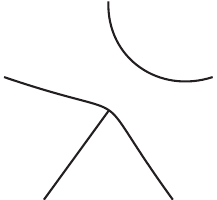}\right) \displaybreak[1] \\[5pt]
\braidcross  &= \frac{1}{\Psi_2(q)} \left(q^{-3} \twostrandid + q^{3} \cupcap\right) 
	- \frac{1}{b} \frac{\Psi_3(q) \Psi_8(q)}{\Psi_2(q)} 
	\left( q^{-1} \drawH + q \drawI \right).
\end{align*}
\end{definition}

The complicated relation for simplifying pentagons actually follows from the others by \cite[\S 8]{MR3624901}.

Sending $q \mapsto -q$ doesn't change the category, while sending $q \mapsto q^{-1}$ doesn't change the underlying pivotal category but does change the braiding. 

The usual normalization is $b = -\Psi_3(q) \Psi_{8}(q)$, which makes
the coefficients of the pentagon relation $+1$ and $-1$ (not quite as
in \cite{MR1403861}, which contains an error; see \cite[\S
1.1.3]{MR3624901}) and simplifies the crossing relation.
In \cite{MR3624901} we normalized
so that $b=1$.

When $q = \pm i$, all the relations except for the crossing relation
still make sense, so we still have a trivalent category but it is not
braided. In this case, we have that $(G_2)_{\pm i}$ has a quotient to
$SO(3)_{\pm i}$, which is the value corresponding to $SOSp(1|2)$.
Recall that
in \cite{MR3624901} we took
the non-degenerate quotient,
but we do not do so here.

\subsubsection{\texorpdfstring{$\mathrm{Rep}(S_3)$}{Rep(S\_3)}}\label{sec:Rep-S3}

\begin{definition}
Let $\mathrm{Rep}(S_3)$ be the trivalent symmetric ribbon category defined by
\begin{gather*}
 \drawH \; - \; \drawI \; + b\; \twostrandid \; -  b\; \cupcap = 0\\[5pt]
\symcross = \frac{1}{b}\;\drawI + \frac{1}{b}\;\drawH\\[5pt]
\begin{aligned}
    \unknot\; &= 2 &\qquad
      \symtwist\; &=  \;\drawcup&\qquad
        \symtwistvertex\; &=  \;\threevertex\\[5pt]
    \loopvertex\;&=0&
      \twogon\;&= b\;\onestrandid&
        \threegon\; &= 0
\end{aligned}
\end{gather*}
\end{definition}

Note that as a pivotal category $\mathrm{Rep}(S_3)$ takes a functor
from $SO(3)_{\zeta_{12}}$ (where $\zeta_{12} = e^{2\pi i/12}$), but
this functor is not a ribbon functor
because they have different braidings (the former is symmetric while
the latter is not) \cite[Ex.\ 8.6]{MR3624901}. This functor is clearly
essentially surjective, and is full because of the formula simplifying
the crossing. However, this functor is not faithful, because
$\mathrm{Rep}(S_3)$ is semisimple, and so every negligible morphism in
$SO(3)_{\zeta_{12}}$ must be killed and we see that
$\mathrm{Rep}(S_3)$ is the semisimplification of $SO(3)_{\zeta_{12}}$
as a pivotal category but not as a ribbon category. The smallest negligible
morphism in $SO(3)_{\zeta_{12}}$ lives in the $5$-box space.

As the notation suggests, the Cauchy completion of $\Rep(S_3)$ is equivalent
to the usual category of representations of the symmetric group $S_3$ under 
a functor sending the strand to the $2$-dimensional standard representation.

\subsection{Trivalent ribbon categories with small \texorpdfstring{$4$}{4}-box space}
\label{sec:3dim4box}

The main focus of this paper is on trivalent ribbon categories with
$\dim \cC(4) = 5$. Here we collect some results when
$\dim \cC(4) \leq 3$; see \S\ref{sec:4dim4box} for more, including the case
$\dim \cC(4) \leq 4$. These are minor modifications of results from
\cite{MR3624901}.

\begin{lemma}\label{lem:I-H-indep}
In a trivalent ribbon category $\cC$, the diagrams
$$\cupcap \qquad\text{and}\qquad \twostrandid$$
are non-zero and are linearly independent.
\end{lemma}
\begin{proof}
  If there is any non-zero linear relation, then multiplying by a
  trivalent vertex on either the top or the right you find that the
  trivalent vertex is zero, which is a contradiction.
\end{proof}

\begin{lemma} \label{lem:Dim3impliesIequalsH}
If $\cC$ is a trivalent ribbon category where the diagrams   
  \[
  \cupcap\;,\qquad\twostrandid\;,
    \qquad\drawI\;,\qquad\drawH\;\;
   \]
   are linearly dependent, then $d \ne 1$ and
   \begin{equation}\label{eq:IequalsH}
     \drawI\; - \; \drawH - \frac{b}{d-1} \left(\; \twostrandid - \cupcap \; \right) = 0.
   \end{equation}
\end{lemma}
\begin{proof}
By the previous lemma and considering eigenvalues under $90^\circ$
rotation, we must have a relation of the form
$$\drawI\;\pm \; \drawH = z \left(\; \twostrandid \pm \cupcap \;\right).$$

Capping off we see that $b= z (d \pm 1)$, so $z = \frac{b}{d \pm 1}$
(with non-zero denominator since $b \ne 0$).  So all we need to do is
handle the $+$-sign case.

In that case, attaching a trivalent vertex we see that $t + b = \frac{b}{d + 1}$, so $t =  -\frac{b d}{d + 1}$.  Adding $\drawH$ to the top gives
$$t \drawI + \fourgon = \frac{b}{d+1} \left(\; \drawH + b \cupcap \; \right).$$
If we take this relation minus its $90$-degree rotation the square term will vanish yielding
$$ b \frac{d-1}{d+1} \left(\; \drawI -\drawH \; \right) = \frac{b^2}{d+1} \left(\; \twostrandid - \cupcap \; \right).$$
If $d=1$, we get a contradiction to Lemma~\ref{lem:I-H-indep}. Otherwise,
dividing by $b \frac{d-1}{d+1}$ yields the desired equation with a
minus sign.
\end{proof}



\section{Main results}

\begin{convention}
  Throughout this section $(\cC, X, \tau)$ is a trivalent ribbon
  category over~$\mathbb{C}$, in the sense of
  Definition~\ref{def:TrivalentRibbon}.
\end{convention}
As explained in \S\ref{sec:Tri}, we can interpret any
trivalent ribbon tangle as a morphism in $\cC$.

\begin{lemma} \label{lem:constants} For any $\cC$ as above, there are
  numbers $b, d, t, u \in \mathbb{C}$ with $u \ne 0$ satisfying
  the following relations.
  \begin{equation}
    \label{eq:simple-rels}
  \begin{aligned}
    \unknot\; &= d&\qquad
      \twist\; &= u^2\;\drawcup&\qquad
        \twistvertex\; &= -u\;\threevertex\\[5pt]
    \loopvertex\;&=0&
      \twogon\;&= b\;\onestrandid&
        \threegon\; &= t\;\threevertex
  \end{aligned}
  \end{equation}
\end{lemma}

\begin{proof}
All of these follow directly from $\dim \cC(n)$ starting $1,0,1,1$ for
$n=0,1,2,3$ except that
we need to see that the constant for the twisted strand is the square of the
constant for the twisted vertex.  This follows, as in
\cite[Lemma~8.2]{MR3624901}, from the following isotopy:
\begin{equation*}
\overviolin = \rotatedtrivalent. \qedhere
\end{equation*}
\end{proof}

\subsection{Quantum Jacobi relation}
\label{sec:QEJac}
\begin{proposition}\label{prop:Jacobi}
If $\dim \cC(4) \leq 5$, then for some $\beta, \alpha \in \bbC$ and some choice
of~$v \in \bbC$ with $v^6 = u$, we have
\begin{equation}
\beta \left[\; v^{-3} \;
\drawcrossX
\;+ v \;
\drawI
\; -v^{-1} \;
 \drawH
\;\right]
 + \alpha
\left[\; \braidcross \;
 + v^{4}\;
\cupcap
\; + v^{-4} \;
 \twostrandid \;
 \right] = 0.\label{eq:quant-except-init}
 \end{equation}
\end{proposition}

The conditions in
Proposition~\ref{prop:Jacobi} are somewhat stronger than needed. In particular,
we don't need to know the exact dimensions of the $n$-box spaces, just that
certain diagrams are linearly dependent. We do not spell out the details here.

\begin{proof}
  By assumption, the space spanned by the six diagrams
  \[
  \cupcap\;,\qquad\twostrandid\;,\qquad\braidcross\;,
    \qquad\drawI\;,\qquad\drawH\;,\qquad\drawcrossX\;
  \]
  is at most $5$-dimensional.  These six diagrams should be thought
  of as having the symmetries of a tetrahedron (up to framing
  change). To be more precise, consider the operation~$R$ that cyclically
  rotates three of the input strands:
  \[
  R\left(\,\,\idtangle\,\,\right) = \Rcycle.
  \]
  The operation~$R$ permutes the six diagrams above up to powers of $\pm
  u^k$:
  \begin{align*}
    \begin{tikzpicture}
      \node[inner sep=5pt] (A) at (150:1.4cm) {\cupcap};
      \node[inner sep=5pt] (B) at (30:1.4cm) {\twostrandid};
      \node[inner sep=5pt] (C) at (-90:1.4cm) {\braidcross};
      \draw[|->,bend left=15] (A) to node[above,cdlabel] {u^{-2}} (B);
      \draw[|->,bend left=15] (B) to (C);
      \draw[|->,bend left=15] (C) to (A);
    \end{tikzpicture}&&
    \begin{tikzpicture}
      \node[inner sep=5pt] (D) at (150:1.4cm) {\drawI};
      \node[inner sep=5pt] (E) at (30:1.4cm) {\drawH};
      \node[inner sep=5pt] (F) at (-90:1.4cm) {\drawcrossX};
      \draw[|->,bend left=15] (D) to node[above,cdlabel] {-u^{-1}} (E);
      \draw[|->,bend left=15] (E) to node[below right,cdlabel] {\!\!-u^{-1}} (F);
      \draw[|->,bend left=15] (F) to (D);
    \end{tikzpicture}
  \end{align*}
  (The notation means, for instance, that
  \(
  R\left(\,\smallfig{\cupcap}\,\right) = u^{-2}\,\,\smallfig{\twostrandid}\,\,
  \).)
Notice that $R^3$ acts by multiplication by $u^{-2}$. (Indeed,   $R^3$ does
not permute the strands, and is equivalent by   Reidemeister moves to changing
the   framing on the upper-right strand by~$-1$.)  Since $R^{3}$ acts by a
nonzero scalar, we have that $R$ acts diagonalizably on the space
spanned by these six
diagrams, and in particular on the nonzero space of relations among these six
diagrams in $\cC$.  Thus we must have a relation which is an eigenvector for
the action of $R$ with eigenvalue $\lambda$ with $\lambda^{3} = u^{-2}$.   We
can then pick a square root $\mu$ of $\lambda^{-1}$ such that $\mu^3 = u$, and
then pick $v=\pm\sqrt{\mu}$.  These are the only
two values of~$v$
satisfying $\lambda = v^{-4}$ and $v^6 = u$.
  
Thus we have a relation of
the form of~\eqref{eq:quant-except-init}.
\end{proof}

\begin{remark}
Typically, if $\dim \cC(4) = 5$ there will only be a
one-dimensional space of relations among the six diagrams appearing in
~\eqref{eq:quant-except-init}, and so there will be a unique relation
of the form~\eqref{eq:quant-except-init} up to rescaling. However, if
$\dim \cC(4) < 5$ or if the six diagrams happen to span a strict
subspace of $\cC(4)$ (as will happen in many of the cases considered
in \S \ref{sec:special-values}), then the vector space of relations
will be more than one-dimensional. In this space of relations,
relations of the form of ~\eqref{eq:quant-except-init} are those which
are eigenvectors for the action of the operator $R$. So, for example,
if the space of relations among the six diagrams is $2$-dimensional,
and the eigenvalues for the action of $R$ are distinct, then $\cC$
would have exactly two relations of the form of
~\eqref{eq:quant-except-init} up to rescaling.
\end{remark}

We now address the case $\beta = 0$.

\begin{lemma} \label{lem:betazero}
  Suppose
  \begin{equation}\label{eq:golden-rel}
    \braidcross \; + v^{4}\;\cupcap\; + v^{-4} \; \twostrandid = 0.
  \end{equation}
  Then $\cC$ is the golden category with $d = \frac{1 + \sqrt{5}}{2}$ or its Galois conjugate.  
\end{lemma}
\begin{proof}
Multiplying by a trivalent vertex we see that $-v^6+v^{-4}=0$, so $v$ is a
$10$th root of unity.  This yields two distinct cases: $v= \pm 1$ or $v = \pm
\zeta_5$ for $\zeta_5$ a fifth root of unity. WLOG we take the $+$ sign.

Placing $\smallfig{\drawH}$ on the top of Eq.~\eqref{eq:golden-rel} gives a second relation $\psi$
between the six diagrams.  The rotations $R(\psi)$ and
$R^2(\psi)$ gives two additional relations.

When $v=1$ the relations we get are:
\begin{align*}
0 &= \braidcross + \twostrandid + \cupcap \\
0 &= \drawcrossX +\drawH + b\cupcap \\
0 &= \drawI - \drawcrossX + b\twostrandid \\
0 &= -\drawH -\drawI + b \braidcross
\end{align*}
Combining the first equation and the last one we see that:
\[ 0 = \drawI+\drawH + b \twostrandid + b\cupcap.\]
Multiplying by a trivalent vertex we see that $t=-2b$.  Pairing
with the square we see that $0 = 2t^2 bd + 2b^3 d$; dividing by $2bd$
(which is non-zero) we get $t^2+b^2=4b^2+b^2 =0$, a contradiction.

When $v$ is a fifth root of unity we get the relations:
\begin{align*}
0 &= \braidcross + v\twostrandid + v^{-1}\cupcap \\
0 &= \drawcrossX +v \drawH + b v^{-1}\cupcap \\
0 &= \drawI - \drawcrossX + b v^{-3} \twostrandid \\
0 &= -v^{-1} \drawH -\drawI + b v^{-3} \braidcross
\end{align*}
Capping off the first equation we see that $v^2 + v +v^{-1}d = 0$, so $d =
-v^2-v^3$ is the Golden ratio or its Galois conjugate.  These four equations
are linearly independent, and by taking the proper linear combinations we see
that
\[
  \drawI =  b\twostrandid  -\frac{b}{d}\; \cupcap
\]
which is the other defining relation of the Golden category (besides
the hypothesis).
\end{proof}

When $\beta \neq 0$, WLOG we assume that $\beta = 1$, and thus arrive
at \eqref{eq:QEJac-alpha}.

We will now start collecting exceptions to the
main theorems; those that can be represented in the $(v,w)$ plane are
summarized in
Figure~\ref{fig:exceptions}.

\begin{figure}
  \centering
  \includegraphics{mpdraws/param-space-5}
  \caption{Values excluded from the main theorems, in the $(v,w)$
    plane, as in Figure~\ref{fig:fund-domain}. Note that some of the excluded values are only seen in the
    $(v,\alpha)$ coordinates, since the change of variables fails.
    Red: denominator in \eqref{eq:QEJac-w} or in $d$ vanishes, no version of
    $\mathsf{QExc}_{R,v,w}$ is defined. Brown: $d$ vanishes. Orange:
    $b$ vanishes, not a trivalent category.
    Green: $v^{10} = 1$, exception to second part of
    Lemma~\ref{lem:dandtvalues}.
    Purple: $v$ a primitive $4$th or $12$th root of unity, derivation of
    \eqref{eq:QECross-w} fails. Cyan: action of $\HTw$ is not diagonalizable
    by Lemma~\ref{lem:eigenvalues-twist}. Black: Expected basis for
    $4$-box space is not linearly independent, as in
    Lemma~\ref{lem:lin-ind-w}. These black lines correspond to known
    existing quantum
    group lines, but satisfy more than one version of
    \eqref{eq:QEJac-w}; see \S\ref{sec:special-cases}. Gold crosses
    are the points with specializations that give the golden category,
    which is
    particularly exceptional; see Lemma~\ref{lem:golden}. The change
    of variables is not invertible when $v$ is a root of unity of
    order $1,2,3,4,6,8,12$, all called out for other reasons.}
  \label{fig:exceptions}
\end{figure}

\begin{lemma} \label{lem:dandtvalues}
  Suppose $\cC$ satisfies \eqref{eq:QEJac-alpha} for some $v$ and~$\alpha$. Then 
\begin{align}
  [5] b &= - \alpha (d+\{8\})\label{eq:b-rel} \\
   \{2\} t &= b-[5] \alpha\label{eq:b-t-rel}
\end{align}
and $v$ is not a primitive $8$th
root of unity.  If $v$ is also not a 10th root of unity, then $\alpha$
is non-zero and we have the following formulas.
\begin{align*}
  d &= -\frac{[5] b}{\alpha} - \{8\}   \\
  t  &= \frac{b-[5] \alpha}{\{2\}}.
\end{align*}
\end{lemma}
\begin{proof}
Equation~\eqref{eq:b-rel} follows from closing off
\eqref{eq:QEJac-alpha} with a cap, and
then solving for~$b$. Equation~\eqref{eq:b-t-rel} follows from   closing off
\eqref{eq:QEJac-alpha} by attaching two ends   to a trivalent
vertex, and then solving
for~$t$.
  
If $v$ is a primitive $8$th root of unity, then $0 = \{2\}t = b-[5] \alpha$,
so $b = [5] \alpha$.  Since $[5]$ is nonzero when
$v$ is a primitive $8$th root of unity, we see that $\alpha = \frac{b}{[5]}$.
Now the first equation becomes
$bd = -b([5]^2 + \{8\}) = b(2 - 2)$,
which
contradicts $b,d \neq 0$.

Finally, if $v$ is not a $10$th root of unity, then $0 \neq
[5] b = - \alpha (d+\{8\})$, so $\alpha$
is non-zero and we can solve for~$d$. We can also solve for $t$ since
$\{2\} \ne 0$.
\end{proof}

\begin{proof}[Proof of Theorem \ref{thm:Jacobi}]
By Proposition \ref{prop:Jacobi} and Lemma \ref{lem:betazero} we have that
\eqref{eq:QEJac-alpha} is satisfied for some $v$ and~$\alpha$, or
$\cC$ is the golden category. It is easy to see that the golden
category also satisfies instances of \eqref{eq:QEJac-alpha} (Lemma
\ref{lem:golden}).
By Lemma~\ref{lem:constants}
and Lemma~\ref{lem:dandtvalues} we have the other relations.
\end{proof}

\subsection{Squares and crosses}
\label{sec:squares-crosses}

\begin{proposition} \label{prop:square-crossing}
  Suppose $\cC$ satisfies \eqref{eq:QEJac-alpha} for some $v$ and~$\alpha$.
  Then the following equations hold.
\begin{itemize}
\item A multiplied version of \eqref{eq:QESq-alpha}:
    \begin{equation}
      \tag{QESq-$\alpha'$}\label{eq:QESq-alpha'}
      \Psi_3\Psi_1 \fourgon + \frac{\alpha \Psi_6 (b+[3]\alpha)}{\Psi_4} \braidcross + \frac{y_1 \Psi_1}{\Psi_4} \drawI + \frac{y_2 \Psi_1}{\Psi_4} \drawH + \frac{\alpha y_3}{\Psi_4} \cupcap + \frac{\alpha y_4}{\Psi_4} \twostrandid =0
\end{equation}
where the $y_i$ are Laurent polynomials in $v$, $b$, and $\alpha$ given in Appendix \ref{app:coefficients}.
  \item
A multiplied version of \eqref{eq:QECross-alpha}:
\begin{multline}
  \tag{QECross-$\alpha'$}\label{eq:QECross-alpha'}
\frac{\Psi_6 \alpha (b+[3]\alpha)}{\Psi_4} \left(\braidcross -
  \invbraidcross \right) + \Psi_1 \Psi_2 \Psi_3 \Psi_6 \alpha \left[\;\drawI \; - \; \drawH \; \right] \\
- \Psi_1^2 \Psi_2 \Psi_3 \Psi_6 \alpha^2 \left[\; \twostrandid \; - \; \cupcap \; \right] = 0.
\end{multline}
\end{itemize}
 \end{proposition}
 
\begin{proof}
  We attach $\smallfig{\drawH}$ to the \eqref{eq:QEJac-alpha} relation
  and then look at the orbit
of this relation under the action of the operator $R$.  Here is the action of
$R$ on the square.
$$
\begin{tikzpicture}
      \node[inner sep=5pt] (A) at (150:1.4cm) {\fourgon};
      \node[inner sep=5pt] (B) at (30:1.4cm) {\twistedsquarehor};
      \node[inner sep=5pt] (C) at (-90:1.4cm) {\twistedsquarever};
      \draw[|->,bend left=15] (A) to (B);
      \draw[|->,bend left=15] (B) to node[right,cdlabel] {v^{-12}} (C);
      \draw[|->,bend left=15] (C) to (A);
\end{tikzpicture}
$$
Attach $\smallfig{\drawH}$ to the bottom of \eqref{eq:QEJac-alpha}
(using that $\smallfig{\braidcross}$ and $\smallfig{\drawH}$
commute) and then apply $R$ and $R^2$.
\begin{gather*}
v^{-3} \twistedsquarever + v t \drawI  -v^{-1} \fourgon + \alpha \drawcrossX + \alpha v^4 b \cupcap + \alpha v^{-4} \drawH = 0\\
v^{-3} \fourgon -v^{-5} t \drawH  -v^{-1} \twistedsquarehor + \alpha \drawI + \alpha v^{-8} b \twostrandid - \alpha v^{-10} \drawcrossX = 0\\
v^{-3} \twistedsquarehor +v^{-11} t \drawcrossX  -v^{-13} \twistedsquarever -v^{-6} \alpha \drawH + \alpha v^{-8} b \braidcross - \alpha v^{-10} \drawI = 0
\end{gather*}
Multiply the equations by $v^{-2}$, $v^6$, and $v^8$, respectively,and
add them. All the twisted squares cancel,
yielding
\begin{multline*}[3] \fourgon +  (t+[1]\alpha)v^{-3} \drawcrossX -(vt
  +v^{-2}[4]\alpha) \drawH +(v^{-1} t+v^2 [4]\alpha)\drawI \\+ \alpha
  b \braidcross + \alpha b v^{-2} \twostrandid + \alpha b v^2 \cupcap
  = 0.
\end{multline*}

We now use \eqref{eq:QEJac-alpha} to remove the $\smallfig{\drawcrossX}$ term, yielding
\eqref{eq:QESq-alpha'}.  Then \eqref{eq:QESq-alpha'} minus its
90-degree rotation
gives \eqref{eq:QECross-alpha'}.
\end{proof}

When $v$ is not a $6$th root of unity we can divide \eqref{eq:QESq-alpha'} by $[3]$,
and if $v$ is not a $10$th root of unity we can divide
\eqref{eq:QECross-alpha'} by $\alpha$.  All that remains in order to
prove Theorem
\ref{thm:square-crossing} is to show that when $v$ is not a $10$th or $12$th
root of unity then $b+[3]\alpha$ is non-zero.  Note that if $v$ is not a
$12$th root of unity and $b+[3]\alpha$ is zero, then
\eqref{eq:QECross-alpha'}
becomes close to an $SO(3)_q$ relation as in \S\ref{sec:SO3q} between
the four planar
diagrams.  We take a
brief detour to analyze this situation in light of
Lemma~\ref{lem:Dim3impliesIequalsH}.

\begin{lemma} \label{lem:IequalsH}
Suppose $\cC$ satisfies \eqref{eq:QEJac-alpha} for some $v$ and $\alpha$, and suppose that  $d \ne 1$ and
$$\drawI\;- \; \drawH - \frac{b}{d-1} \left(\; \twostrandid - \cupcap \;\right) = 0.$$
Then one of the following occurs:
\begin{itemize}
\item $v$ is a primitive $4$th root of unity;
\item $v$ is a primitive $12$th root of unity; or
\item
$\cC$ is a quotient of $SO(3)_q$ for some $q$ with the usual braiding (which
includes the possibility $SOSp(1|2)$ when $q= \pm i$).
\end{itemize}
\end{lemma}

Note that
in the case of $SO(3)_q$, the assumption $d \ne 1$ implies that $q=v^3$
is not an $8$th root of unity.

\begin{proof}
We apply the tetrahedral symmetry operator $R$ to the relation to get:
\begin{align*}
-v^{-6} \drawH + v^{-6} \drawcrossX -\frac{b}{d-1} \left(\; \braidcross - v^{-12} \twostrandid \; \right) &= 0\\
v^{-12} \drawcrossX + v^{-6} \drawI -\frac{b}{d-1} \left(\; \cupcap -v^{-12} \braidcross \; \right) &= 0.
\end{align*}
Multiply the first equation by $-v^{-6}$ and then add to the second relation to make the $\smallfig{\drawcrossX}$ terms vanish. We obtain
$$v^{-6} \drawI + v^{-12} \drawH - \frac{b}{d-1} \left(\;\cupcap -v^{-12} \braidcross -v^{-6}\braidcross + v^{-18}\twostrandid  \; \right) = 0.$$
Unless $v^6$ is $-1$, this says that $\smallfig{\braidcross}$ simplifies.  So we split into two cases.

If $\smallfig{\braidcross}$ simplifies, then
$\cC$ agrees with $SO(3)_q$ from Definition~\ref{def:SO3q} as a
pivotal category for $q^2+q^{-2}=d-1$. (Without the assumption that
the crossing
simplifies, $\cC$ might not be generated as a pivotal category by the
vertex.) So we need only solve for all braidings for
$SO(3)_q$, which yields only the usual braiding (with $q=v^3$) or when $d=2$ the
additional braiding giving $\Rep(S_3)$ (see \cite[Example 8.6]{MR3624901}).
For $\Rep(S_3)$, all instances of \eqref{eq:QEJac-alpha} have $v$ a primitive
$4$th or $12$th root of unity, which can be seen by diagonalizing the action
of the tetrahedral symmetry operator on the $3$-dimensional space of relations.
(See Lemma \ref{lem:Rep-S3} for the exact form of the three relations.) 

If $v^{6} = -1$, then $v$ is a primitive $4$th or $12$th root of unity.
\end{proof}

\begin{lemma} \label{lem:bplus3alpha}
Suppose $\cC$ satisfies \eqref{eq:QEJac-alpha} for some $v$ and $\alpha$ with $v$ not a
primitive $4$th or $12$th root of unity. Then $b+[3]\alpha$ is
non-zero.
\end{lemma}
\begin{proof}
If $v$ is a sixth root of unity or if $\alpha$ is zero, then
$b+[3]\alpha = b \neq 0$.  So we may assume that $v$ is not a $12$th
root of unity and $\alpha$ is non-zero.
Suppose for the sake of contradiction that $b+[3]\alpha = 0$.  Then by Proposition
\ref{prop:square-crossing} we have that $\cC$ satisfies the assumptions of
Lemma \ref{lem:IequalsH}, so $\cC$ is $SO(3)_q$.  We can easily compute all 
instances of the Exceptional relation which hold; see
\S\ref{sec:special-cases} for details.  
For $SO(3)_q$, we have $q=v^3$ and
$\alpha = -b \frac{\Psi_1(v)}{\Psi_{12}(v)}$.
The $b+[3]\alpha$ simplifies to $b \frac{\Psi_6}{\Psi_{12}}$ which is non-zero since $v$ is not a
$12$th root of unity, giving a contradiction.
\end{proof}

\begin{proof}[Proof of Theorem \ref{thm:square-crossing}]
Since $v$ is not a $10$th root of unity, by Lemma \ref{lem:dandtvalues} the
equations for $d$ and $t$ hold, $v^8 \ne 1$, and $\alpha \ne 0$.  Since $v$ is not a
$12$th root of unity, $[3]$
is non-zero.
By Lemma \ref{lem:bplus3alpha}, $b+[3]\alpha$ is
nonzero.  Thus we can divide by~$[3]$ in \eqref{eq:QESq-alpha'} from
Proposition~\ref{prop:square-crossing} and divide by $[3](b+[3]\alpha)$ in
\eqref{eq:QECross-alpha'} yielding the corresponding relations in
Theorem \ref{thm:square-crossing}.
\end{proof}

\subsection{Change of variables}
\label{sec:change-variables}

Finally let us motivate the change of variables to $w$ given in the
introduction---we will see that one of the eigenvalues for the action of
the crossing on the $4$-box space is given by exactly $w^2$.
For this, we need to exclude certain special values of~$v$.
Let $P_{G_2}(q) = q^{10}+q^{8}+q^{2}+1+q^{-2}+q^{-8}+q^{-10}$ be the quantum
dimension of the $7$-dimensional representation of $G_2$, and let
$\zeta_3 = e^{2\pi i/3}$.

\begin{lemma} \label{lem:lin-ind-w}
Suppose $\cC$ satisfies \eqref{eq:QEJac-alpha} for some $v$ and $\alpha$ such that $v$ is
not a $10$th or $12$th root of unity and in addition $d \neq v^6+1+v^{-6}$, $d
\neq P_{G_2}(\zeta_3 v)$, and $d \neq P_{G_2}(\zeta_3^2 v)$.
Then the
following diagrams are linearly independent:
$$\twostrandid\;,\qquad \cupcap\;,\qquad \drawI\;,\qquad\braidcross\;,\qquad\invbraidcross.$$
\end{lemma}
After the change of variables, the excluded values of $d$ correspond to
$w = \pm v^k$ for $k \in \{-3,4\}$ and $w = \pm \zeta_3^{\pm1}v^k$ for
$k \in \{-1,2\}$; these lines are black in Figure~\ref{fig:exceptions}.

\begin{proof}
If you pair any two of these diagrams together, the resulting ribbon graph can
be evaluated using the \eqref{eq:QEJac-alpha}, \eqref{eq:QESq-alpha}, and \eqref{eq:QECross-alpha} relations.  The
determinant of this matrix is
\begin{multline*}
\frac
{
d^5 b^2 [2][3]^2[6] (d-(v^{6}+1+v^{-6}))^2 (d + \{8\}) (d-P_{G_2}(\zeta_3 v)) (d-P_{G_2}(\zeta_3^2 v))
}
{
[1]^{2}(d + \{2\})^{2}
} =\\=
  -\frac{b^2[4]^5[5]^3[6]^2\bigl([\lambda-4][\lambda+3]\bigr)^2\bigl([\lambda+5][\lambda-6]\bigr)^5\{\lambda-3\}\{\lambda+2\}}{[1][2]^5\bigl([\lambda][\lambda-1]\bigr)^8}
  \frac{[3\lambda+3][3\lambda-6]}{[\lambda+1][\lambda-2]}.
\end{multline*}
(See \S\ref{sec:independence} for more details.)

By our assumptions none of the factors in the first form can vanish
except possibly
$d+\{8\}$.  If $d+\{8\} = 0$, then by Lemma~\ref{lem:dandtvalues}
$[5]b = -\alpha(d+\{8\})= 0$, which implies
that $v$ is a $10$th root of unity.
\end{proof}

\begin{remark}
We will see that $SO(3)_{v^3}$ satisfies a Jacobi relation for $v$ with 
$d = q^2 + 1 + q^{-2} = v^6+1+v^{-6}$ (Lemma~\ref{lem:SO3q}), and
similarly $(G_2)_{\zeta_3^{\mp 1} v}$ in the $7$-dimensional representation
satisfies a Jacobi relation with $d=P_{G_2}(q) = P_{G_2}(\zeta_3^{\pm
  1} v)$ (Lemma~\ref{lem:G2q}).
For these categories, the diagrams in Lemma \ref{lem:lin-ind-w} will not be linearly independent.
It's possible that there are additional degenerate trivalent ribbon categories with those values of $d$ 
which have these examples as quotients and where the five diagrams are linearly independent.
\end{remark}

As is \S\ref{sec:QGConventions}, let $\HTw$ be the \emph{negative} half twist
$\smallfig{\braidcross}$ acting on the $4$-box space.

\begin{lemma}\label{lem:eigenvalues-twist}
  If $v$ is not a $10$th or $12$th root of unity, and (with $w$ given by
  Equation~\eqref{eq:change-variables})
  $w \neq \pm v^k$ for $k \in \{-3,4\}$,
  $w \neq \pm \zeta_3^{\pm 1}v^k$ for $k \in \{-1,2\}$,
  $w \neq \pm i v^k$ for $k \in \{0,1\}$,
  and $v \neq \pm w^2$,
  then $\HTw$ acts
  diagonalizably on the five-dimensional subspace of the $4$-box space
  spanned by the diagrams from Lemma~\ref{lem:lin-ind-w}
  with eigenvalues $v^{12}$, $-v^6$, $-1$, $w^2$, and~$\frac{v^2}{w^2}$.
\end{lemma}
\begin{proof}
  The first few excluded values guarantee that
  Lemma~\ref{lem:lin-ind-w} applies.
  Note that $\frac{1}{d} \smallfig{\cupcap}$ and $\frac{1}{b} \smallfig{\drawI}$ are projections
and are eigenvectors for $\HTw$ with eigenvalues $v^{12}$ and
$-v^6$, respectively. We expect three more eigenvalues.
To find them, we look at the action of $\HTw$
on the image of the projection
$$\pi = \;\twostrandid\; -\frac{1}{d}\;\cupcap\; -\frac{1}{b}\;\drawI\;.$$
Start with \eqref{eq:QECross-alpha} and multiply by $\HTw$, then use
\eqref{eq:QEJac-alpha} to remove
the $\smallfig{\drawcrossX}$ term, and \eqref{eq:QECross-alpha} to remove the $\smallfig{\drawH}$ term.  Multiply
again by $\HTw$ and then by $\pi$ to get a cubic equation in the
$\HTw$ on the image of~$\pi$.  By Lemma~\ref{lem:lin-ind-w}, $\HTw$
does not satisfy a quadratic equation on this subspace, so this cubic
polynomial is
the minimal polynomial and its roots are all eigenvalues.  A tedious
calculation shows that the roots of this cubic equation are exactly $-1$,
$w^2$, and $\frac{v^2}{w^2}$.

The action is diagonalizable
exactly when the
minimal polynomial does not have repeated roots, that is unless
$w = \pm i$, $v = \pm i w$, or $v = \pm w^2$.
\end{proof}

Note that  when $w = \pm i$, $v = \pm i w$, or $v = \pm w^2$ the
action of the crossing is not diagonalizable, and so
$\mathsf{QExc}_{R,v,w}$ cannot be semisimple; these lines are cyan in
Figure~\ref{fig:exceptions}.

\begin{remark}
Recall that the operators $\smallfig{\braidcross}$ and 
$\smallfig{\drawH}$ commute.
Thus the eigenvectors for the former (which all have multiplicity $1$
generically) must also be eigenvectors for the latter.  The
eigenvalues for the two operators are
$$\left(v^{12}, b\right), \left(-v^6,t\right), \left(-1, [\lambda][\lambda-1]\right), \left(w^2,
\frac{[\lambda]\{\lambda+2\}}{\Psi_1}\right), \left(\frac{v^2}{w^2},
-\frac{[\lambda-1]\{\lambda-3\}}{\Psi_1}\right).$$
There is a less-tedious argument that the above pairs of eigenvalues are the
only possible
ones.  Namely, suppose $r$ is an eigenvector in the image of $\pi$
with eigenvalues $(\beta, \gamma)$. Multiply $r$ by \eqref{eq:QEJac-w} and reduce
using the known eigenvalues. We find:
$$v^{-3} \beta \gamma -v^{-1} \gamma - \frac{[\lambda][\lambda-1]}{\Psi_1}( \beta + v^{-4}) = 0.$$
Similarly, multiplying by \eqref{eq:QECross-w} we find:
$$\beta-\beta^{-1} - \frac{\Psi_1}{\Psi_6} \gamma + \frac{\Psi_1[\lambda][\lambda-1]}{\Psi_6} = 0.$$
Solving these equations for $\beta$ and $\gamma$ give the list of possible eigenvalues.  We still need some version of Lemma \ref{lem:lin-ind-w} in order to see that all the eigenvalues are realized.
\end{remark}

We will also need the classical version of Lemma~\ref{lem:eigenvalues-twist}.  Again one could guess the
answer quickly heuristically using the previous remark.

\begin{lemma} \label{lem:classical-eigenvalues}
Let $\cC$ be a symmetric ribbon category that is a quotient of
$\mathsf{Exc}_{R,\lambda}$ with $\lambda \notin \{-3, 4\}$. Then the
crossing
operator and the ladder operator act
diagonalizably on a $5$-dimensional subspace of $\cC(4)$ with eigenvalues 
$(+1, 12)$, $(-1, 6)$, $(-1, 0)$, $(+1, 2\lambda)$, and $(+1, 2-2\lambda)$.
\end{lemma}
\begin{proof}
The determinant of the matrix of inner products of the diagrams
$$\twostrandid\;,\qquad \cupcap\;, \qquad \drawI\;, \qquad \drawH\;, \qquad \symcross$$
is $\frac{3}{4}B^4D^5(D-3)^2(2+D)$.  Thus those diagrams are linearly independent
unless $D=3$ (corresponding to $\lambda = -3$ or $\lambda = 4$) or $D=-2$
(which does not happen for any $\lambda$).
We can now directly
calculate the eigenvalues (specializing as usual to $B=12$).
\end{proof}

Of course if Conjecture~\ref{conj:class-suffic} holds then the $5$-dimensional subspace in the previous lemma is the whole $4$-box space.



\section{From the classical to the quantum conjecture}
\label{sec:classical-quantum}
This section is devoted to a proof of
Theorem~\ref{thm:classical-quantum}, saying that Conjecture~\ref{conj:Deligne} implies  Conjecture~\ref{conj:quant-Deligne}.

A key component of the construction is a functor
$\mathsf{Tri}_R \rightarrow \mathsf{Exc}_{R((h)), v, w}$ for suitable
rings~$R$ and
$v = e^{h/2}$ and $w = e^{\lambda h/2}$ as formal power series.
The construction here is given by the \emph{Kontsevich integral}~$Z$. The
Kontsevich integral was first defined for unframed knots \cite{MR1318886}, and
later extended to framed trivalent graphs
\cite{MR1473309,MR2304469,MR2661529}. We briefly recall its
properties.
We will
specialize for now to the case corresponding to the adjoint
representation of the Lie
algebra, removing the usual distinction between different kinds of
edges.

Let $\mathsf{Tri}_R$ be the category of $R$-linear combinations of framed, unoriented, trivalent
graph tangles from \S\ref{sec:Tri}. Then, we have a Kontsevich integral functor $Z$ from
$\mathsf{Tri}_R$ to $\mathsf{Jac}_{R[[h]],D,B}$. (This uses only that $R$ contains~$\QQ$ and $D,B \in R$.) 
The variable $h$ is used to keep track of
the grading. The functor $Z$ is usually defined as a functor from the category of
parenthesized (non-associative) tangle graphs; to define it as a
functor from our
category, where the objects are just integers, pick arbitrarily, for
each~$n \ge 3$,  a parenthesization of $n$ objects and apply the usual
functor.

We will use the following properties of~$Z$.
\begin{enumerate}
\item\label{item:Kontsevich-perturb} $Z(G) = G + O(h)$, where on the
  right hand side we think of the
  underlying trivalent graph of~$G$ as a Jacobi diagram.
\item\label{item:Kontsevich-crossing} On a single crossing, $Z$ takes
  the value
  \begin{equation}
    \label{eq:Kontsevich-crossing}
    Z\biggl(\;\braidcross\;\biggr) =
    \exp\left(\frac{h}{2}\;\drawH\;\right)
    \cdot\; \symcross\;.
  \end{equation}
\end{enumerate}

Equation~\eqref{eq:Kontsevich-crossing} appears in
\cite[\S6.4]{MR1881401}. Sign conventions differ in different sources;
in
the terminology of
  \cite[Figure 2]{MR2304469}, we use Convention~2. (Note the crossing
  is negative.)

\begin{remark}\label{rem:q-inner-prod}
  We can see already from \eqref{eq:Kontsevich-crossing} that the
  different choices of $q$ depending on the inner product in
  \S\ref{sec:comparing} correspond uniformly to setting $q=e^{h/2}$ as
  a power series. Specifically, with this substitution
  and Proposition~\ref{prop:quadcasimir} we see that
  \[
    Z\biggl(\;\twist\;\biggr) = e^{Bh/2} Z\biggl(\;\drawcup\;\biggr)
     = q^{\langle \theta, \theta+2\rho\rangle} Z\biggl(\;\drawcup\;\biggr),
 \]
 where $\theta$ is the highest weight of the adjoint representation.
 More generally (in the theory with multiple kinds of edges), the same
 argument says that the positive twist acts
 on $V_\mu$ by $q^{\langle \mu,\mu+2\rho\rangle}$ for any inner
 product and $q=e^{h/2}$.
\end{remark}
  
Now specialize and set $R$ to be a localization of $\mathbb{Q}[\lambda]$
at finitely many polynomials, and set
$D = -2 \frac{(\lambda+5)(6-\lambda)}{\lambda(1-\lambda)}$ and
$B=12$. Let $\eval$ be the functor from $\mathsf{Jac}_{R,D,B}$ to
$\mathsf{Exc}^{Del}_{R,\lambda}$ given by the hypothesis of
Theorem~\ref{thm:classical-quantum}.
Then the desired functor $\Qeval$ is the composition of $Z$ with the
quotient by the classical exceptional relations and the given
classical evaluation $\eval$, with scalars extended to $R[[h]]$:
\[
  \mathsf{Tri}_{R[[h]]} \overset{Z}{\longrightarrow}
  \mathsf{Jac}_{R[[h]],D,B} {\longrightarrow}
  \mathsf{Exc}_{R[[h]],\lambda} \overset{\eval}{\longrightarrow}
  \mathsf{Exc}^{Del}_{R[[h]],\lambda}.
\]
We then localize at $h$ to take values in Laurent series $R((h))$ (so
that, e.g., $[\lambda]$ is invertible).

We must check that this functor has the desired properties (in
particular the values of $v$ and~$w$).
By hypothesis, the $k$-box spaces of $\mathsf{Exc}^{Del}_{R((h)),\lambda}$ for $k \le 3$
have the required dimensions for a ribbon trivalent category. By
property~\eqref{item:Kontsevich-perturb} of~$Z$,
we have that $\mathsf{Exc}^{Del}_{R((h)),\lambda}$.
is generated as a ribbon category over $R((h))$ by a
trivalent vertex.
Furthermore, Equation~\eqref{eq:Kontsevich-crossing}, combined with $B=12$ and
$T=6$, show that
\begin{align*}
  \Qeval\biggl(\twist\;\biggr) &= e^{6h}\eval\biggl(\;\drawcup\;\biggr)\\
  \Qeval\biggl(\! \twistvertex\;\biggr) &= -e^{3h}\eval\biggl(\threevertex\biggr).
\end{align*}
Comparison with Equation~\eqref{eq:simple-rels-spec} shows that we
must choose $v^6 = e^{3h} \in R((h))$.
 Then by Theorem~\ref{thm:Jacobi}, the
 functor $\Qeval$
 satisfies \eqref{eq:QEJac-alpha} for
some value of~$\alpha$ and $v = \zeta e^{h/2}$ for some $\zeta^6 = 1$.
(If necessary, temporarily extend the scalars so that $R$ has these
roots of unity.)
For this value of~$v$, note that $v^n - 1$ is
invertible in $R((h))$ for all $n \ne 1$. Thus by Lemma~\ref{lem:bplus3alpha} and
Theorem~\ref{thm:square-crossing} we have quantum exceptional square
and crossing relations. 
Since $v$ is not a root of unity, we can
renormalize~$b$ and make the
change of variables
from Equation~\eqref{eq:change-variables} to get the category
$\mathsf{QExc}_{R((h^t)),v,w}$ for some rational~$t$ and some choice
of~$w$. (The possible exponent~$t$ comes from the square roots.) To
identify~$w$ and see that we can take $t=1$ and $\zeta=1$,
we find the eigenvalues
of the twist operator on the 4-box space.

By Lemma~\ref{lem:classical-eigenvalues}, the eigenvalues of the crossing
operator and the ladder operator in $\mathsf{Exc}_{R[[h]],\lambda}$
are $(+1, 12)$, $(-1, 6)$, $(-1, 0)$, $(+1, 2\lambda)$, and $(+1, 2-2\lambda)$.
Equation~\eqref{eq:Kontsevich-crossing} then shows
that the corresponding eigenvalues of the twist operator are
$e^{6h}$, $-e^{3h}$, $-1$,  $e^{\lambda h}$, and
$e^{(1-\lambda)h}$, respectively.  On the other hand, by Lemma~\ref{lem:eigenvalues-twist}, the
eigenvalues of the twist
operator in $\mathsf{QExc}$ are $v^{12}$, $-v^6$, $-1$,  $w^2$, and
$v^2/w^2$. We need to be a bit careful to see that they're given in
the same order. Notice 
 that the first two eigenvalues are the twist and trivalent twist (the corresponding eigenvectors
 are the diagrams $\smallfig{\cupcap}$ and $\smallfig{\drawI}$),
 and the last two are interchanged by the symmetry replacing $w$ with $v/w$.  Thus, we must have that the
 Quantum Exceptional relations hold with $v = e^{h/2}$ and
 $w = e^{\lambda h/2}$, and we did not in fact need to extend scalars
 in finding~$v$.

This completes the proof of Theorem~\ref{thm:classical-quantum}.

\begin{remark}
  Given Conjecture~\ref{conj:Deligne}, to deduce a version of
  Conjecture~\ref{conj:quant-consist} with a non-infinitesimal
  deformation (valued in, say, complex numbers rather than a
  power series ring), we would have to get control over the
  coefficients in the power series that arise in the functor $\Qeval$
  constructed above. This has been done for the Kontsevich
  integral~$Z$ \cite{MR1669276,MR2764999}. The evaluation algorithm
  implicit in
  $\eval$ would also have to have computable bounds that are not too
  large. Most reasonable evaluation algorithms do give such bounds,
  but we do not have an evaluation algorithm.
\end{remark}



\section{Calculations}
\label{sec:calculations}

In this section we do a number of calculations related to the quantum exceptional series
and the corresponding Reshetikhin-Turaev link invariant. These
calculations were done
by KM, with the algebra-spiders package at
\begin{center}\url{https://github.com/kim-em/toolkit}.\end{center}
We (NS and DPT)
have been able to reproduce some of the calculations, but for several of the larger calculations
we haven't been able to reproduce the full details (either because they require a larger computer
or because there were some additional tricks used in the calculation that weren't well documented).
In particular, we were unable to reproduce the calculation of $\det M_6$ in \S \ref{sec:independence},
the calculation of the operators $T$ in \S \ref{sec:operators}, and the calculations of the null-spaces in
\S \ref{sec:idempotents-dimensions}.
Although the level of rigor in this section is lower than in the rest of the paper, we do not have serious concerns
about the validity of these calculations, which pass a number of sanity tests.

When we were able to identify exactly what went into a particular calculation we have also included
with the \texttt{arXiv} source of this paper the relevant Mathematica notebook. For readers who do not have
Mathematica we have also included a PDF version of each notebook as well with the same filename but
with the extension \texttt{.pdf} instead of \texttt{.nb}.

\subsection{Candidate bases}
\label{sec:bases}
We now consider two candidates bases, $\cB^{\text{braided}}$ and $\cB^{\text
{planar}}$ for each of the spaces $\cP_n$, for $n = 2,3,4,5,6$.
At this point we can only give heuristics for choosing these sets of diagrams;
in the next section we will establish that they are linearly independent, but
it is beyond the scope of this paper to show that they span. (This is
probably at least as hard as showing that the quantum exceptional relation
suffices to evaluate all closed diagrams.) Nevertheless they have the expected
sizes:
from Cohen--de-Man's decompositions of $\mathfrak g^{\otimes
  n}$ for $n \le 4$ \cite{MR1381778}, we can
compute the dimension of the $n$-box space for $n \le 8$, and observe
that they agree with the invariant spaces for
$\mathfrak{f}_4$ (OEIS sequence \oeis{A179685} \cite{EIS}):
\[
\begin{tabular}{rr}
  \toprule
  $n$ & $\dim \Hom_{\mathfrak{f}_4} (1 \to \mathfrak{f}_4^{\otimes n})$ \\
  \midrule
  0 & 1 \\ 1 & 0 \\ 2 & 1 \\ 3 & 1 \\ 4 & 5 \\ 5 & 16 \\
  6 & 80 \\ 7 & 436 \\ 8 & 2891 \\ 9 & 22248 \\ 
  \bottomrule
\end{tabular}
\]
Note that the Lie algebra $\mathfrak{f}_4$ works better for computing
dimensions of invariant spaces than, say, $\mathfrak{e}_7$ or
$\mathfrak{e}_8$, which have representations that degenerate in
$\mathfrak{g}^{\otimes 3}$ and $\mathfrak{g}^{\otimes 4}$,
respectively, while $\mathfrak{e}_6$ has Dynkin diagram automorphisms
which results in additional complications. The number for $n=9$ is also presumably
correct for the exceptional series, since it can be computed as $\dim
\Hom_{\mathfrak{f}_4}(\mathfrak{f}_4^{\otimes 4} \to
\mathfrak{f}_4^{\otimes 5})$, i.e., only representations appearing in
$\mathfrak{f}_4^{\otimes 4}$ are relevant.

In order to write our proposed bases compactly, we introduce the following
notations.  Let $D_{*k}$ denote the set of all $k$ inequivalent rotations of the diagram $D$, 
including $D$ itself (so when $D$ has $n$ boundary points, $k$ divides $n$).
Let $D_{*k \clubsuit}$ denote the set of the first $k$ counterclockwise rotations of the
diagram $D$ (including $D$ itself), even though these do not exhaust all of the
inequivalent rotations.  In particular, $D_{*1 \clubsuit}$ denotes just the diagram itself.

We define `braided' bases, which use crossings but not squares, as:
\begin{align*}
\cB^{\text{braided}}_0 & = \cB^{\text{planar}}_0 = \left\{ \diagram
{0.1}{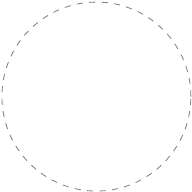}_{*1}\right\} &
\cB^{\text{braided}}_1 & = \cB^{\text{planar}}_1 = \Bigg\{ \quad \Bigg\} \displaybreak[0]\\
\cB^{\text{braided}}_2 & = \cB^{\text{planar}}_2 = \left\{ \diagram
{0.1}{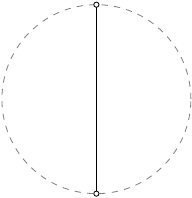}_{*1} \right\} &
\cB^{\text{braided}}_3 & = \cB^{\text{planar}}_3 = \left\{ \diagram
{0.1}{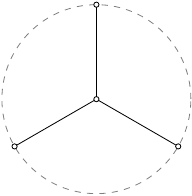}_{*1} \right\}
\end{align*} 
\begin{align*}
\cB^{\text{braided}}_4 & = \left\{ 
  \diagram{0.1}{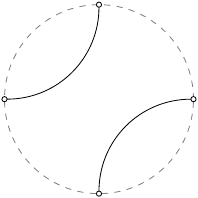}_{*2},
  \diagram{0.1}{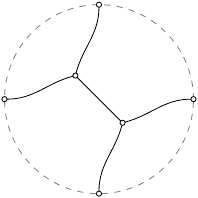}_{*2},
  \diagram{0.1}{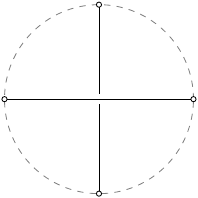}_{*1\clubsuit}
  \right\} \displaybreak[1]\\
\cB^{\text{braided}}_5 & = \left\{ 
  \diagram{0.1}{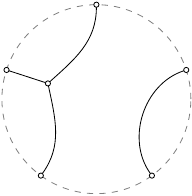}_{*5},
  \diagram{0.1}{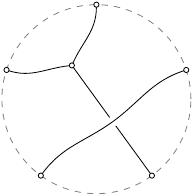}_{*5},
  \diagram{0.1}{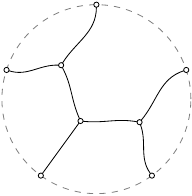}_{*5},
  \diagram{0.1}{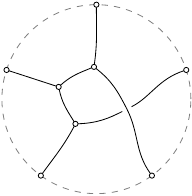}_{*1\clubsuit}
  \right\} \displaybreak[1]\\
\cB^{\text{braided}}_6 & = \Bigg\{ 
  \diagram{0.1}{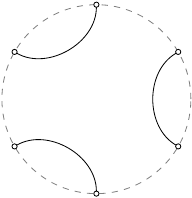}_{*2},
  \diagram{0.1}{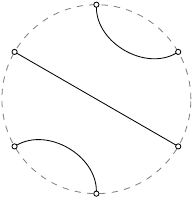}_{*3},
  \diagram{0.1}{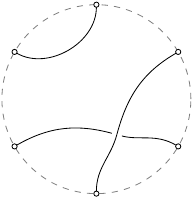}_{*6},
  \diagram{0.1}{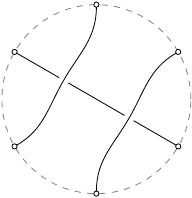}_{*3},
  \diagram{0.1}{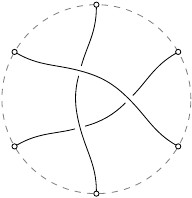}_{*1\clubsuit}, \displaybreak[1] \\
  & \qquad
  \diagram{0.1}{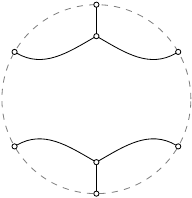}_{*3}, 
  \diagram{0.1}{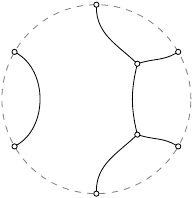}_{*6},
  \diagram{0.1}{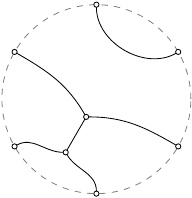}_{*6},
  \diagram{0.1}{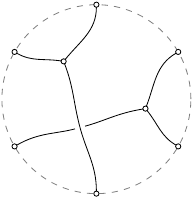}_{*6},
  \diagram{0.1}{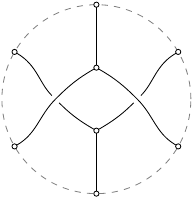}_{*1\clubsuit} \displaybreak[1] \\
  & \qquad
  \diagram{0.1}{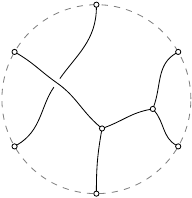}_{*6},
  \diagram{0.1}{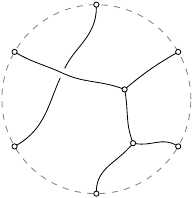}_{*6},
  \diagram{0.1}{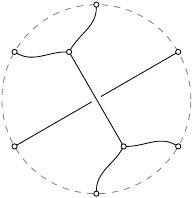}_{*3}, 
  \diagram{0.1}{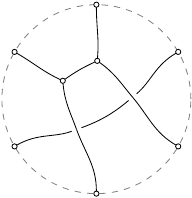}_{*3\clubsuit} \displaybreak[1] \\
  & \qquad
  \diagram{0.1}{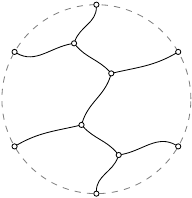}_{*3},
  \diagram{0.1}{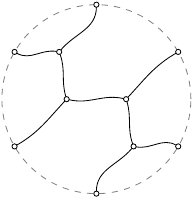}_{*3},
  \diagram{0.1}{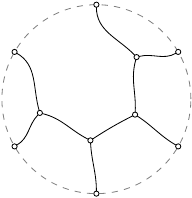}_{*6},
  \diagram{0.1}{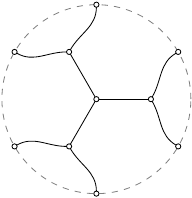}_{*2}, 
  \diagram{0.1}{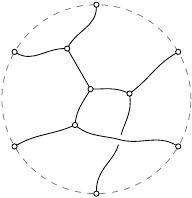}_{*6},
  \diagram{0.1}{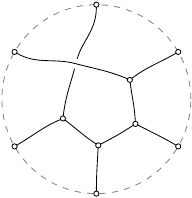}_{*4\clubsuit}, \displaybreak[1] \\
  & \qquad
  \diagram{0.1}{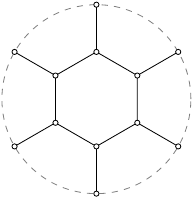}_{*1}
\Bigg\}
\end{align*}

We then have an extension of Conjecture~\ref{conj:quant-suffic}. (See
also Proposition~\ref{prop:basis}.)

\begin{conjecture}\label{conj:braided-span}
  So long as the polynomials in $\bfP$ are
  invertible in~$R$, $\mathsf{QExc}_{R,v,w}(k)$ is spanned by
  $\cB^{\text{braided}}_k$ for $k \le 6$.
\end{conjecture}

In most of the cases above where we only take some of the rotations of a given
diagram, it is because the other rotations are already equivalent (using
3-dimensional isotopies and/or crossing changes) to the chosen ones.
The exceptions are that we only take some of the rotations~of
\[\diagram{0.1}{6c16eee9baf0b996fa06ef435e4f0229abade522}
  \qquad\text{and}\qquad
  \diagram{0.1}{8a27a57e82333b072ffb4d5b78de022da5c0852b}.
\]

The following heuristic produces these proposed bases.  Consider the classical
setting where the overcrossing and undercrossing are equal.  Now focus on the case
of a diagram whose underlying graph is a tree with $k$ leaves.  If you fix the last vertex 
as a root vertex and label the remaining vertices clockwise as $x_1, x_2, \ldots x_{k-1}$, 
then such a diagram can be interpreted as a bracket formula like $[x_2,[x_1, x_3]]$ in
the free Lie algebra on $k-1$ letters.  But due to the Jacobi and antisymmetry relation we should expect these
diagrams to span a space that's smaller than naively given by counting
binary trees.
In fact, by 
the Witt Dimension Formula applied to words of degree $(1,1,\ldots)$ \cite[Satz 3]{MR1581553}, they span an $(k-2)!$ dimensional space which has as a basis
the diagrams obtained by taking the
left-associated bracket formula $[[[[x_1,x_2],x_3],\dots],x_{k-1}]$
and permuting the $k-2$ legs $\{x_2,\dots,x_{k-1}\}$ arbitrarily.

More generally, for forests pick a partition of the boundary points into parts with 
at least two elements. For each such partition, and for each part with $k$ elements, 
pick an abstract (unembedded) tree with those $k$ points as the leaves, and form 
the $(k-2)!$ trees (still unembedded) obtained by the symmetric group action on the leaves.
For each such unembedded tree, pick an embedding in the plane which involves as few crossings
as possible.  Finally in the quantum setting we choose over- and under-crossings arbitrarily.

The first line of the proposed basis $\cB^{\text{braided}}_6$ comes from the partition $6=2+2+2$,
the second line from $6=3+3$, the third line from $6=4+2$, and the fourth line from $6=6$.
(Note that in the first three rows, the `omitted' diagrams marked with $\clubsuit$ are omitted because the additional rotations differ from the included ones just by changing under- and over-crossings, while in the fourth row two diagrams are omitted because we've already reached the expected $(6-2)!$ diagrams.)  Finally, we observe that this procedure only gives 79 elements, and so we throw in the hexagon to get the desired
count of 80.  As one can tell from this final step, we do not have a
general procedure for producing conjectural
bases (or even the dimensions) beyond $6$ boundary points.

The calculations below further allow one to verify that if we add in any natural small diagram not on this list,
the rank of the matrix of pairwise inner products between these diagrams does not increase. (Because we 
have no argument that the inner product is nondegenerate, this
observation itself does not lead to new relations.)

We then define corresponding `planar' bases:

\begin{align*}
\cB^{\text{planar}}_4 & = \left\{ 
  \diagram{0.1}{34a876bb5dae61c3ae047f077a51c37b92ee0789}_{*2},
  \diagram{0.1}{a3a76e07ebde41b851fc63a935929666fca9dcd6}_{*2},
  \diagram{0.1}{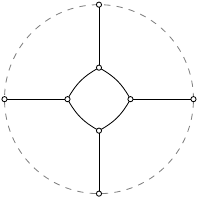}_{*1}
  \right\} \displaybreak[1] \\
\cB^{\text{planar}}_5 & = \left\{ 
  \diagram{0.1}{e55fafbd03afa4f32c7a584ff8fc14d93149f64e}_{*5},
  \diagram{0.1}{25d22603781a5a699a84b26b177c0c228abe0e7a}_{*5},
\diagram{0.1}{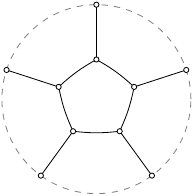}_{*1},
\diagram{0.1}{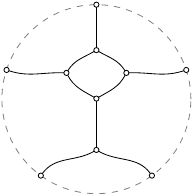}_{*5}
  \right\} \displaybreak[1] \\
\cB^{\text{planar}}_6 & = \Bigg\{ 
  \diagram{0.1}{a906b4797b1a0522970b8c97cdb5415a37a14d50}_{*2},
  \diagram{0.1}{1d5fda6573f6e14e019449f842ce86389a666ef3}_{*3},
  \diagram{0.1}{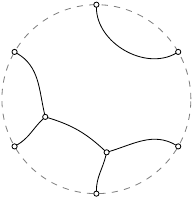}_{*6},
  \diagram{0.1}{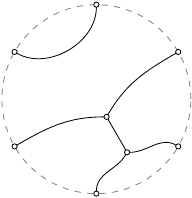}_{*6},
  \diagram{0.1}{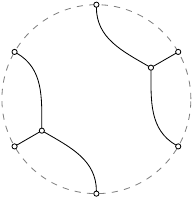}_{*3},
  \diagram{0.1}{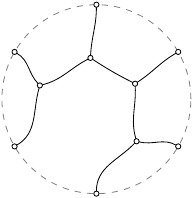}_{*3}, \displaybreak[1] \\
  & \qquad 
  \diagram{0.1}{1baad19a7b800a28edd9aa033244492c73330c5b}_{*3},
  \diagram{0.1}{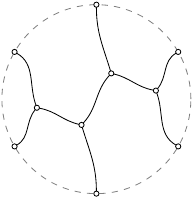}_{*3},
  \diagram{0.1}{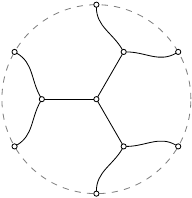}_{*2},
  \diagram{0.1}{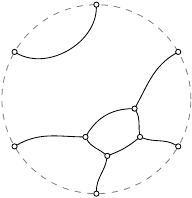}_{*6},
  \diagram{0.1}{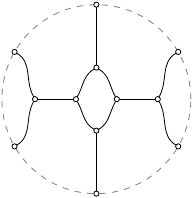}_{*3},
  \diagram{0.1}{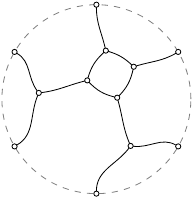}_{*6}, \displaybreak[1] \\
  & \qquad 
  \diagram{0.1}{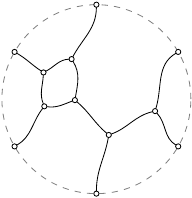}_{*6},
  \diagram{0.1}{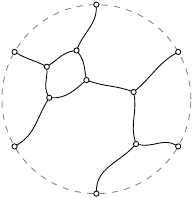}_{*6},
  \diagram{0.1}{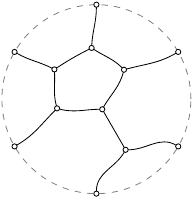}_{*6},
  \diagram{0.1}{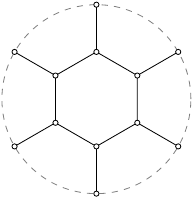}_{*1},
  \diagram{0.1}{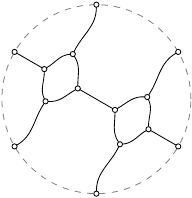}_{*3},
  \diagram{0.1}{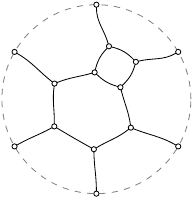}_{*4\clubsuit}, \displaybreak[1] \\
  & \qquad 
  \diagram{0.1}{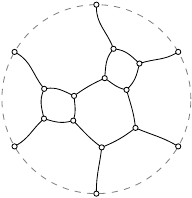}_{*3\clubsuit}, 
  \diagram{0.1}{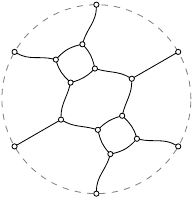}_{*1\clubsuit},
  \diagram{0.1}{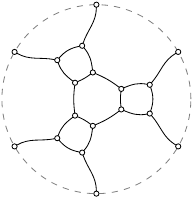}_{*1\clubsuit}
\Bigg\} 
\end{align*}

\begin{proposition}\label{prop:planar-span}
  If Conjecture~\ref{conj:braided-span} holds, then so long as all
  the polynomials in $\bfP$ are
  invertible, $\mathsf{QExc}_{R,v,w}(k)$ is spanned by
  $\cB^{\text{planar}}_k$ for $k \le 6$.
\end{proposition}
\begin{proof}
  The relation \eqref{eq:QESq-w} lets us turn a crossing into a
  square (with a factor) plus some simpler diagrams. Furthermore, a
  square next to a pentagon can be turned
  into a single
  pentagon plus simpler diagrams, by Lemma~\ref{lem:pentasquare}
  below. We thus get a triangular
  change-of-basis matrix, with coefficients rational functions of $v$
  and~$w$, from $\cB^{\text{braided}}_k$ to $\cB^{\text{planar}}_k$ or
  vice versa.
\end{proof}

Proposition~\ref{prop:planar-span} gives an implicit bijection
between these proposed bases. We've chosen not to
display $\cB^{\text{planar}}_k$ in the order given by this bijection, but rather
by the apparent planar complexity of the diagrams.

\begin{remark}
If you ignore our heuristic which starts with the non-planar basis, and instead simply guess an appropriate
number of relatively simple planar diagrams you can easily wind up with a set which is
linearly dependent for the $6$-boundary point diagrams.  In particular, if you replace the 
final five diagrams (involving a hexagon and more than one square)
with the five
diagrams $$\left\{\diagram{0.1}{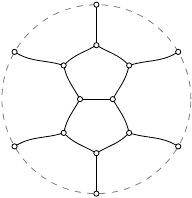}_{*3},
  \diagram{0.1}{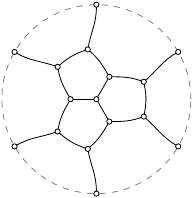}_{*2}\right\}$$
then those $80$ diagrams have a singular matrix of inner products so are not
a good candidate basis. (In particular, if the original list of 80 is
a basis then
the modified list will be linearly dependent.)
\end{remark}

\begin{lemma}\label{lem:pentasquare}
  In $\mathsf{QExc}_{R,v,w}$ in a ring where $\Psi_1$, $\Psi_6$,
  $[\lambda]$, and $[\lambda-1]$ are
  invertible, there is a $K \in R$ so that
  \[
    \pentasquare \;=\; K\;\pentagon + (\text{lower
      order terms}),
  \]
  where the lower order terms are an $R$-linear combination of the other 15
  diagrams in $\cB_5^{\mathrm{planar}}$ (or, by an easy equivalence,
  an $R$-linear combination of the first 15 diagrams in
  $\cB_5^{\mathrm{braided}}$).
\end{lemma}

\begin{proof}
  Write $\equiv$ to mean equality up to combinations of the other 15
  diagrams. Then by \eqref{eq:QESq-w}, we have
  \[
    \pentasquare\;\equiv\;
      \frac{\Psi_2\Psi_6^2[\lambda][\lambda-1]}{\Psi_1^2}\;\twisttreefourfive
      \;-\; \frac{x_2}{\Psi_1}\;\pentagon.
  \]
  To deal with the twisted tree, first note that \eqref{eq:QECross-w}
  lets us switch the sign of the crossing:
  \[
    \twisttreefourfive \;-\; \twisttreefivefour \;\equiv\;
    \frac{\Psi_1}{\Psi_6}\;\pentagon.
  \]
  Next, an application of \eqref{eq:QEJac-w} (at the upper-right edges
  in the diagrams below) and 3D isotopies lets us rotate the
  twisted tree:
  \[
    v^{-1}\;\twisttreefourfive \;\equiv\; v\; \twisttreefourthree.
  \]
  By repeatedly applying rotated and/or reflected versions of this
  relation, we find
  \[
    v^{-1}\;\twisttreefourfive
    \;\equiv \dots \equiv\; v\;\twisttreefivefour
    \;\equiv\; v\;\twisttreefourfive - v \frac{\Psi_1}{\Psi_6}\pentagon
  \]
  or
  \begin{equation}\label{eq:twisttree-pent}
    \twisttreefourfive \;\equiv\; \frac{v}{\Psi_6}\pentagon.
  \end{equation}
  Algebra yields the desired result.
\end{proof}

\begin{remark}
 Mousaaid \cite[Eq.~4.3]{2204.13642} explicitly identified a relation
 among $5$-boxes for classical E8
 which rewrites the pentagon in terms of the braided basis.
 Above we used Eq.~\eqref{eq:twisttree-pent} to write a twisted tree
 in terms of a pentagon plus lower-order terms, but we could also use
 that equation the other way to extract a $5$-box relation for the
 quantum exceptional
 family which will presumably specialize to Mousaaid's relation.
\end{remark}

\subsection{Linear independence}
\label{sec:independence}
We first check that these proposed bases are linearly independent, assuming 
Conjecture~\ref{conj:quant-consist}.
For brevity we will drop the subscripts on $\QExc_{R,v,w}$.
We have a pairing $\langle \cdot , \cdot \rangle\colon \QExc(n) \otimes \QExc(n)
\to \QExc(0)$ defined
by 
\[
  ( x, y ) =
  \begin{tikzpicture}[baseline, thick]
    \node[draw, shape=semicircle, shape border rotate=180, minimum size=1.5em] (x) {$x$};
    \node[draw, shape=semicircle, shape border rotate=180, minimum size=1.5em, right=0.5cm of x] (y) {$y$};
    \draw (x.90) to[in=90, out=90] (y.90);
    \draw (x.135) to[in=90, out=90] (y.45);
    \draw (x.45) to[in=90, out=90] (y.135);
  \end{tikzpicture}
\]
and for many (conjecturally all) closed diagrams in $\QExc(0)$ we can use
the relations to evaluate to a rational function in \(\QQ(v,w)\). Of course,
our basis is far from orthonormal with respect to this pairing!

We first look at the braided bases described above, and find that, up to $n=6$, we can evaluate all matrix entries of $M_n =
( e_i, e_j )_{i,j}$ for
$e_i\in\cB_n^{\mathrm{braided}}$. Explicit matrices are available with the {\tt arXiv} sources of this article at
$$\text{{\small\tt arxiv-code/matrices/$n$-box-innerproducts.m}},$$
for $n \leq 6$.  These matrices were computed
using variables $d$ and $v$ in a
normalization where $b=1$. (It is easy to recover the power of~$b$:
for a closed diagram, multiply by $b$ raised to half the number of trivalent
vertices. For the determinants of matrices of inner products below,
this amounts to multiplying by $b$ raised to the total number of
trivalent vertices in all diagrams in $\cB_n$.)

While the bases in this matrix agree with the braided basis described
above, the
order of the basis elements is for convenience of
computers rather than humans. The explicit order of basis elements
(here and in other matrices) can
be seen diagrammatically in
\[\text{{\small\tt 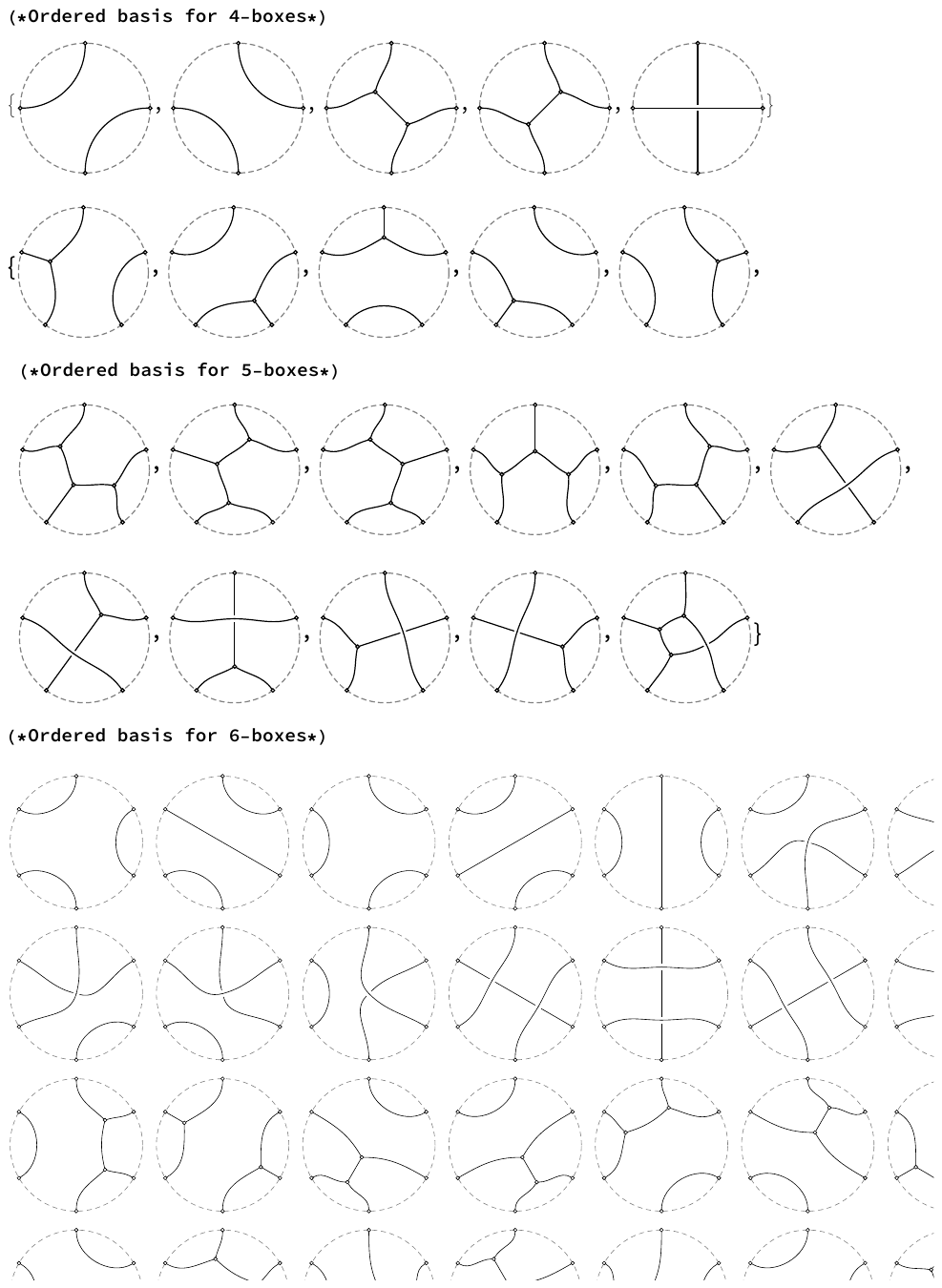}}.\]

By computer, we can
calculate the
determinants of these matrices. The determinants are computed using 
\[\text{{\small\tt arxiv-code/determinants-of-innerproducts.nb}},\]
and give the following results (after conversion to the $v,w$ variables).
\begin{align*}
\det M_3 & = bd = - \Psi_3 \Psi_4 \frac{
	           [\lambda - 6] [\lambda + 5] \{\lambda-3\}\{\lambda+2\}
             }{
                   [\lambda - 1] [\lambda]
             } \\
\det M_4 & = -\Psi_1^2 \Psi_2^2 \Psi_3^4 \Psi_5^3 \Psi_6^4
           \frac{\bigl([\lambda - 6] [\lambda + 5]\bigr)^5}{\bigl([\lambda - 1] [\lambda]\bigr)^8}
           \bigl(\{\lambda - 3\} \{\lambda + 2\}\bigr)^3\bigl([\lambda - 4][\lambda + 3]\bigr)^2
 \frac{ [3\lambda - 6] [3\lambda + 3]}{ [\lambda - 2][\lambda + 1]}\displaybreak[1]\\
  \det M_5 & = -v^{-2}\Psi_2^{12} \Psi_3^{16} \Psi_4^{16} \Psi_5^{13} \Psi_6^{20}
             \frac{\bigl([\lambda - 6] [\lambda + 5]\bigr)^{16}}{\bigl([\lambda - 1] [\lambda]\bigr)^{26}}\bigl(\{\lambda - 3\} \{\lambda + 2\}\bigr)^{16}\bigl([\lambda - 4][\lambda + 3]\bigr)^{10}\times\\
  &\qquad\times \left(\frac{[3\lambda - 6] [3\lambda + 3]}{[\lambda - 2] [\lambda + 1]}\right)^6
   \{2\lambda - 3\} \{2\lambda + 1\} \displaybreak[1]\\
  \det M_6 & = -v^{12}\frac{\Psi_2^{94}\Psi_3^{79}\Psi_4^{80}\Psi_5^{75}\Psi_6^{165}}{\Psi_1^{32}}
             \frac{\bigl([\lambda - 6] [\lambda + 5]\bigr)^{80}}{\bigl([\lambda - 1] [\lambda]\bigr)^{150}}
             \bigl(\{\lambda - 3\} \{\lambda + 2\}\bigr)^{75}\bigl([\lambda - 4] [\lambda + 3]\bigr)^{65} \times \\
  &\qquad\times
    \left(\frac{[3\lambda - 6] [3\lambda + 3]}{[\lambda - 2] [\lambda + 1]}\right)^{45}
    \bigl(\{2\lambda-3\}\{2\lambda+1\}\bigr)^{10}
    \bigl([\lambda - 5] [\lambda + 4]\bigr)^{15}
           \bigl([\lambda - 3] [\lambda + 2]\bigr)^{5}\times\\
         &\qquad\times
           \bigl(\{\lambda - 2\} \{\lambda + 1\}\bigr)^5
           [5\lambda - 6] [5\lambda + 1].
\end{align*}

A version of the determinant $\det M_4$, with a different basis, was
already used in Lemma~\ref{lem:lin-ind-w}. That determinant can be
computed from this since
\eqref{eq:QECross-w} gives us row and column operations to convert the
bases, from which we see that the determinants
differ by $(\Psi_1/\Psi_6)^2$.

These determinants have denominators that
are powers of $[\lambda][\lambda-1]$ and sometimes $\Psi_1$.
(Recall here and below that the bracket polynomials are
not generally irreducible.)
While additional zeroes of the numerator appear as we increase~$n$, these
zeroes in the numerator are all explained by degenerate cases or categories 
we already know about.

In particular, we see that $\det M_3$ vanishes (and thus the category
is not trivalent) when $v$ is a primitive
$k$th root of unity for $k=6,8$, or when $w = \pm v^k$ for $k=-5,6$,
or when $w = \pm iv^k$ for $k = -2,3$. It has poles at $w = \pm v^k$
for $k=0,1$.

Then $\det M_4$ vanishes in addition for
\begin{itemize}
\item $v$ a primitive $k$th root of unity for $k=1,2,4,5,10,12$;
\item $w = \pm v^k$ for $k=-3,4$, where $(SO(3),V_{(2)})_q$ appears as a
  quotient; and
\item $w = \pm \zeta_3^{\pm 1}v^k$ for $k=-1,2$ with $\zeta_3 =
  e^{2\pi i/3}$, where $(G_2,V_{(1,0)})_q$
  appears as a quotient, as part of the $G2$ family
  (Definition~\ref{def:G2-family}).
\end{itemize}

In addition to these, $\det M_5$ no longer vanishes for $v=\pm1$ and vanishes for
\begin{itemize}
\item $w^2 = \pm i v^k$ for $k=-1,3$, where $(F_4,V_{(0,0,0,1)})_q$ appears as a
  quotient, as part of the $F4$ family (Definition~\ref{def:F4-family}).
\end{itemize}

Finally, $\det M_6$ has poles at $v=\pm1$ and vanishes in addition for
\begin{itemize}
\item $w = \pm v^k$ for $k=-4,5$, where $(SO(3),V_{(4)})_q$ appears as a
  quotient as part of the $F4$ family (or equivalently $(SOSp(1|2),V_2^+)_q$ as
  part of the exceptional family);
\item $w = \pm v^k$ for $k=-2,3$, where $(A_2,V_{(1,1)})_q$ appears as
  a quotient, in either the exceptional or $F4$ families;
\item $w^5 = \pm v^k$ for $k=-1,6$, where $(E_8)_q$ appears as
  a quotient; and
\item $w = \pm iv^k$ for $k=-1,2$, where $(C_3,V_{(0,1,0)})_q$ appears as a quotient
  as part of the $F4$ family.
\end{itemize}

\begin{proposition}\label{prop:basis}
  If Conjecture~\ref{conj:quant-consist}
  is true and a certain finite set of non-zero
  polynomials in $v$ and~$w$ are
  invertible,
 $\cB^{\text{braided}}_k$ and $\cB^{\text{planar}}_k$
  freely generate their (common) span in $\QExc_{R,v,w}$ for $k\leq 6$.
\end{proposition}
 (The finite set here
  is larger than $\bfP$, as seen above.)

\begin{proof}
  This follows from the fact that the determinants computed above do
  not vanish identically.
\end{proof}

\subsection{Matrices for operations}
\label{sec:operators}
We next obtain matrices for the braiding, rotation, and
related operators on
 $\QExc(n)$, `relative to' $\cB^{\text{braided}}$ described above.
We have to be careful here, as we don't know that $\cB^{\text{braided}}$
actually spans $\QExc(n)$, or even that these operators map
the span of $\cB^{\text{braided}}$ to itself.

Showing the braiding and rotation  preserve the span of $\cB^{\text{braided}}$ 
is much more tractable than proving our main conjectures; we would only need to
apply braiding or rotation to each element, and show that it
is possible to use the relations to rewrite the result as a linear combination
of elements of $\cB^{\text{braided}}$. (We did one case of this in
Lemma~\ref{lem:pentasquare}.) For practical computational reasons we
take a slightly different approach,
computing via pairings the matrix which \emph{would} be associated to each of the
operations \emph{if} $\cB^{\text{braided}}$ really were a basis, and then
verifying that these matrices satisfy expected properties.

Given an operator $T\colon \QExc(n) \to \QExc(m)$, if the $\cB^{\text{braided}}_n$
are bases, the associated matrix for $T$ would be $AM_m^{-1}$, where
$A = ( T e_i, e_j )_{i,j}$.
Rather unfortunately inverting the matrix $M_6$ is computationally
infeasible. (Recall that arithmetic over $\QQ(v,w)$ may be quite inefficient!)

We find, nevertheless, that for $T$ any of
\begin{align*}
  (\rho_n \colon \QExc(n) \to \QExc(n)) & = 
    \begin{tikzpicture}[baseline, thick, scale=0.8]
      \node[draw, shape=semicircle, shape border rotate=180, minimum size=1.5em] (x) {};
      \draw (x.90) -- ++(0,0.5);
      \draw (x.45) -- ++(0,0.5);
      \draw (x.135) to[in=90, out=90] ++(-0.75,0) 
                     to[in=180,out=-90] ($(x.-90)+(0,-0.25)$) 
                     to[out=0,in=-90] ($(x.45)+(0.75,0)$)
                     -- ++(0,0.5);
    \end{tikzpicture} &
  (\beta_n \colon \QExc(n) \to \QExc(n)) & = 
    \begin{tikzpicture}[baseline, thick, scale=0.8]
      \node[draw, shape=semicircle, shape border rotate=180, minimum size=1.5em] (x) {};
      \draw (x.45) -- ++(0,0.5);
      \draw (x.135) to[out=90,in=-90] ($(x.90)+(0,0.5)$);
      \draw[knot] (x.90) to[out=90,in=-90] ($(x.135)+(0,0.5)$);
    \end{tikzpicture} 
  \displaybreak[1] \\
  (\cap_n \colon \QExc(n) \to \QExc({n-2})) & = 
    \begin{tikzpicture}[baseline, thick, scale=0.8]
      \node[draw, shape=semicircle, shape border rotate=180, minimum size=1.5em] (x) {};
      \draw (x.45) -- ++(0,0.5);
      \draw (x.90) arc (0:180:0.2);
    \end{tikzpicture}
  &
  (\cup_{n-2} \colon \QExc({n-2}) \to \QExc({n})) & = 
    \begin{tikzpicture}[baseline, thick, scale=0.8]
      \node[draw, shape=semicircle, shape border rotate=180, minimum size=1.5em] (x) {};
      \draw (x.90) -- ++(0,0.5);
      \draw (x.45) -- ++(0,0.5);
      \draw (x.135) -- ++(0,0.5);
      \draw ($(x.135)+(0,0.5)+(-0.3,0)$) arc (0:-180:0.2);
    \end{tikzpicture}
  \displaybreak[1] \\
  (\fork_{n-1} \colon \QExc({n-1}) \to \QExc({n})) & = 
    \begin{tikzpicture}[baseline, thick, scale=0.8]
      \node[draw, shape=semicircle, shape border rotate=180, minimum size=1.5em] (x) {};
      \draw (x.90) -- ++(0,0.5);
      \draw (x.45) -- ++(0,0.5);
      \draw (x.135) -- ++(0,0.25) -- ++(0.125,0.25);
      \draw    ($(x.135)+(0,0.25)$) -- ++(-0.125,0.25);
    \end{tikzpicture}
  &
  (\fuse_n \colon \QExc(n) \to \QExc({n-1})) & = 
    \begin{tikzpicture}[baseline, thick, scale=0.8]
      \node[draw, shape=semicircle, shape border rotate=180, minimum size=1.5em] (x) {};
      \draw (x.45) -- ++(0,0.5);
      \draw (x.90) -- ($(x.114)+(0,0.25)$);
      \draw (x.135) -- ($(x.114)+(0,0.25)$) -- ++(0,0.25);
    \end{tikzpicture}
  \\
  (H_n \colon \QExc(n) \to \QExc(n)) & = 
    \begin{tikzpicture}[baseline, thick, scale=0.8]
      \node[draw, shape=semicircle, shape border rotate=180, minimum size=1.5em] (x) {};
      \draw (x.45) -- ++(0,0.5);
      \draw (x.90) -- ++(0,0.5);
      \draw (x.135) -- ++(0,0.5);
      \draw ($(x.90)+(0,0.25)$) -- ($(x.135)+(0,0.25)$);
    \end{tikzpicture}
\end{align*}
with $n \leq 6$,
we can evaluate all of the matrix entries of $A$ as rational functions.
(Even though $H_n = \rho_{n}^{-1} \fuse_n \rho_{n+1} \fork_n$, we can not
use this to compute $H_6$ as we would need to know $\rho_7$ and $\fuse_7$.)

Even though directly inverting $M$ is infeasible, we find that we can
calculate the product $A M^{-1}$ in each case. (This remains a
difficult calculation---we use some tricks, calculating the product at
various specializations of the variables and then using interpolation, finally
verifying our answer by multiplying by $M$.)

Explicit matrices are available with the {\tt arXiv} sources of this article,
as the files
$$\text{\texttt{\small arxiv-code/matrices/$n$-box-$\mathit{OP}$-matrix.m}},$$
where $n \leq
6$, and $\mathit{OP}$ is one of \texttt{rotation}, \texttt{braiding},
\texttt{inverse-braiding}, \texttt{cup}, \texttt{cap}, \texttt{fork},
\texttt{fuse}, \texttt{H}.  The
matrices are again written in terms of the variables $d, v$ normalized
with $b=1$.

\begin{remark}
We check that these matrices satisfy a number of sanity checks in the ``Verifying Matrices'' part
of the Mathematica notebook \texttt{\small{QES-Knot-Polynomials.nb}}.
In particular,
rotating $n$ times is the identity, the cupcap squared is $d$ times
the cupcap, applying
a cup and then fusing is zero (by the lollipop relation),
the braiding composed with the inverse braiding is the identity, and applying a fork
or cup and then an appropriate crossing does have the correct twist values.
\end{remark}

\subsection{Representations of braid groups}
The affine $n$-strand braid group has generators $\rho$ and $\beta$, with
relations $\rho^n = 1$, $\beta \rho^\pm \beta \rho^\mp \beta = \rho^\pm \beta
\rho^\mp \beta \rho^\pm \beta \rho^\mp$, and $\beta \rho^i \beta \rho^{-i} =
\rho^i \beta \rho^ {-i} \beta$ for all $i = 2, \ldots, n-2$.

At this point, we may verify that the matrices $\rho = \rho_n$ and $\beta =
\beta_n$ provide a representation of the affine $n$-strand braid group merely
by multiplying out the matrices.

\begin{lemma}
The matrices described in the previous section give a representation of the affine $n$-strand braid group for $3 \leq n \leq 6$.
\end{lemma}

\subsection{A 2-variable link polynomial}

\subsubsection{Limited planar operations}
Thinking of the vector spaces $\QExc(n)$ as forming a planar algebra, 
at this point we have limited access to the full set of operations indexed by
planar tangles. We have computed above conjectural versions of
$\cap_n$,
$\cup_{n-2}$, and
$\rho_n$ for $n \leq 6$. We can also find
\begin{align*}
(m_{4,4,2} \colon \QExc(4) \otimes \QExc(4) \to \QExc(4)) & = 
  \begin{tikzpicture}[baseline, thick]
    \node[draw, shape=semicircle, shape border rotate=180, minimum size=1.5em] (x) at (0,0) {};
    \node[draw, shape=semicircle, shape border rotate=180, minimum size=1.5em] (y) at (1.5,0) {};
    \draw (x.130) -- +(0,0.5);
    \draw (x.105) -- +(0,0.5);
    \draw (x.75) to[out=90,in=90] (y.105);
    \draw (x.50) to[out=90,in=90] (y.130);
    \draw (y.75) -- +(0,0.5);
    \draw (y.50) -- +(0,0.5);
  \end{tikzpicture}
 \displaybreak[1] \\
(m_{4,6,2} \colon \QExc(4) \otimes \QExc(6) \to \QExc(6)) & = 
  \begin{tikzpicture}[baseline, thick]
    \node[draw, shape=semicircle, shape border rotate=180, minimum size=1.5em] (x) at (0,0) {};
    \node[draw, shape=semicircle, shape border rotate=180, minimum size=1.5em] (y) at (1.5,0) {};
    \draw (x.130) -- +(0,0.5);
    \draw (x.105) -- +(0,0.5);
    \draw (x.75) to[out=90,in=90] (y.120);
    \draw (x.50) to[out=90,in=90] (y.140);
    \draw (y.100) -- +(0,0.5);
    \draw (y.75) -- +(0,0.5);
    \draw (y.55) -- +(0,0.5);
    \draw (y.40) -- +(0,0.5);
  \end{tikzpicture}
 \displaybreak[1] \\
\intertext{and}
(m_{4,4,1} \colon \QExc(4) \otimes \QExc(4) \to \QExc(6)) & = 
  \begin{tikzpicture}[baseline, thick]
    \node[draw, shape=semicircle, shape border rotate=180, minimum size=1.5em] (x) at (0,0) {};
    \node[draw, shape=semicircle, shape border rotate=180, minimum size=1.5em] (y) at (1.5,0) {};
    \draw (x.130) -- +(0,0.5);
    \draw (x.105) -- +(0,0.5);
    \draw (x.75) -- +(0,0.5);
    \draw (x.50) to[out=90,in=90] (y.130);
    \draw (y.100) -- +(0,0.5);
    \draw (y.75) -- +(0,0.5);
    \draw (y.50) -- +(0,0.5);
  \end{tikzpicture}
\end{align*}
 as follows.

Recall that our chosen basis for $\QExc(4)$ is
\[
  \left(
    \tikz[baseline=-7, thick]{\draw(0,0) arc (-180:0:0.2); \draw(0.6,0) arc(-180:0:0.2);},
    \tikz[baseline=-7, thick]{\draw(0,0) arc (-180:0:0.4); \draw(0.2,0) arc(-180:0:0.2);},
    \tikz[baseline=-7, thick]{\draw(0,0) arc (-180:0:0.2); \draw(0.6,0) arc(-180:0:0.2);
                       \draw(0.2,-0.2) arc (-180:0:0.3);},
    \tikz[baseline=-7, thick]{\draw(0,0) arc (-180:0:0.4); \draw(0.2,0) arc(-180:0:0.2);
                       \draw(0.4,-0.4) -- (0.4,-0.2);},
    \tikz[baseline=-7, thick]{\draw(0,0) arc (-180:0:0.3); \draw[knot] (0.3,0) arc (-180:0:0.3);}
  \right)
\]
Expressed in our chosen bases, we then have 
\begin{align*}
  m_{4,n,2}((x_1, x_2, x_3, x_4, x_5) \otimes y)
    & = x_1 \cup_{n-2}\cap_n y \\
      & \quad + x_2 y \\
      & \quad + x_3 \fork_{n-1}\fuse_n y \\
      & \quad + x_4 H_n y \\
      & \quad + x_5 \beta y.
\end{align*}
Finally, we can write $m_{4,4,1}$ in terms of the other operations,
as, e.g.,
$$
 m_{4,4,1}(x \otimes y) = m_{4,6,2}(x \otimes \rho_6^{-1} (\cup_4 (\rho_4(y)))
 = \begin{tikzpicture}[baseline, thick]
     \draw (0,0) arc (-180:0:0.5) -- cycle;
     \draw (2,0) arc (-180:0:0.5) -- cycle;
     \draw (2.2,0) arc (0:180:0.2) to[out=-90,in=180] (2.5,-0.65) to[out=0,in=-90] (3.2,0) arc (180:0:0.1) to[out=-90,in=0] (2.5,-0.8) to[out=180,in=-90] (1.2,0) arc (0:180:0.2);
     \draw (0.6,0) arc (180:0:0.4) arc (-180:0:0.1) -- +(0,0.4);
     \draw (0.2,0) -- +(0,0.4);
     \draw (0.4,0) -- +(0,0.4);
     \draw (2.4,0) -- +(0,0.4);
     \draw (2.6,0) -- +(0,0.4);
     \draw (2.8,0) -- +(0,0.4);
   \end{tikzpicture}
$$ 

\subsubsection{Conway width}
Typically, we may choose to think about a diagrammatically defined tangle
invariant as a morphism of planar algebras  $\mathsf{Tangle} \to \mathcal{Q}$
for some planar algebra $\mathcal{Q}$, sending the crossing to a suitable
element of $\mathcal{Q}_4$. With our limited set of operations in the
target planar
algebra, we can only compute the (conjectural) invariants of links
which we can generate
from the crossing using available operations.

\begin{definition}
The \emph{width} of a generic embedding of a planar graph in the plane is the 
maximum number of intersections with a horizontal line. The width of
an abstract planar
graph is the minimum width over all generic embeddings.

The \emph{basic polyhedron} of a 4-valent planar graph is the graph
obtained by repeatedly collapsing all digons to a point.
The \emph{Conway width} of a 4-valent planar graph is the width of its
basic polyhedron. The Conway width of a link
diagram is the
Conway width of the graph where all crossings are replaced by
$4$-valent vertices, and the Conway width of a link is the minimum
over all diagrams. In particular, algebraic (a.k.a.\ arborescent) links
are those with
Conway width~$2$, and there are no links with Conway width~$4$.
\end{definition}

Motivation for these definitions comes from Conway's
classification of
small knots \cite{MR258014}.
It is easy to see that with the available planar operations, we can evaluate 
all links of Conway width at most 6.

\begin{lemma}
All knots and links with at most 12 crossings have Conway width at most 6, with exceptions of the two alternating links
$$
\diagram{0.2}{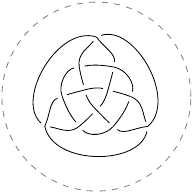} \qquad
\diagram{0.2}{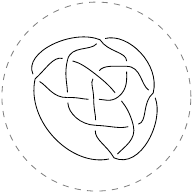}
$$
and their non-alternating cousins.
\end{lemma}
\begin{proof}
One can enumerate the basic polyhedra
with $n$ vertices using Brinkmann and McKay's {\tt plantri} program
\cite{MR2357364,MR2186681}, with the command {\tt plantri -qda -c2
  $\langle n+2\rangle$}. 
All basic polyhedra with less than 12 vertices have width at most 6, and of 
the 12 vertex basic polyhedra all but 12C and 12G (using 
the naming scheme from \cite{MR679310}) have width at most 6. These graphs,
made into links, are the exceptions described above.
\end{proof}
If we were able to work with Conway width 8, we could nearly exhaust the 
available tables of knots and links; every Conway basic polyhedron with at 
most 19 vertices has width at most 8.

There is an obvious greedy algorithm to compute link invariants using
our limited
operations (implicitly showing a diagram has Conway width~$6$): repeatedly
collapse digons; when that is not
possible apply $m_{4,4,1}$ to merge two 4-boxes into a 6-box, and afterwards
hope that $m_{6,4,2}$  suffices to combine all remaining 4-boxes into that
6-box.
This works on all prime knots up to 12 crossings. (For a handful of knots,
if you apply $m_{4,4,1}$ `in the wrong place' it is possible to get stuck---%
relabeling the strands and trying again quickly succeeds.)
Our tabulations are based on this algorithm.

\subsubsection{Examples}
Using the above observations,
we have computed the (conjectural!) 
invariant~$\xi$ of all prime links up to 11 crossings, and all prime knots up
to 12 crossings.
(We take the homological $0$-framing on each link component as usual.) 
We observe that these values of $\xi$ are always of the form
$$\xi(L) - d^{k} = 
\frac{[4][6][\lambda-6][\lambda+5]}{[2][\lambda]^k [\lambda-1]^k} \mathcal E(L),$$
where 
\[d = \xi(\textrm{unknot}) =
-\frac{[4][\lambda-6][\lambda+5]}{[2][\lambda][\lambda-1]},\]
$k$ is the number of components of $L$, 
and $\mathcal E(L)$ is a integer Laurent polynomial
in $v$ and $w$. We offer no explanation at this point for this
factorization behavior.
(We are abusing notation somewhat, since we
need Conjecture~\ref{conj:class-consist} to know that $\xi(L)$ and
$\mathcal{E}(L)$ are
actually invariants of~$L$.)

These polynomials $\mathcal E (L)$ are tabulated
with the \texttt{arXiv} version of this paper at
\begin{center}
\texttt{\small arxiv-code/QES-Knot-Polynomials.txt.xz}  
\end{center}
They are computed by installing the KnotTheory and QuantumGroups Mathematica packages \cite{katlas, QuantumGroups} and running the files 
$$\text{{\small\tt arxiv-code/QES-Knot-Polynomials.nb}}$$
 (which contains the functions which compute the polynomials) and 
 $$\text{{\small\tt arxiv-code/StoringQESCalculations.nb}}$$ (which computes them on these knots and links and saves the answer).

It would be desirable to verify directly that these conjectural
computed invariants specialize
correctly to the quantum link invariants of the adjoint representations of the
exceptional Lie algebras. (We will prove unconditionally in \S \ref{sec:Specializations}
that these do in fact specialize correctly---but it would be nice to be able to check computations.)
Unfortunately, these invariants have previously been
rather hard to compute, so there is little available to check against. The Lie
algebra $\mathfrak{sl}_2$ lies on the exceptional curve, and so the 2nd
colored Jones polynomial  (that is, colored by the 3-dimensional adjoint
representation) should be the specialization to $v=q^{1/3}$ and either
$w=\pm q^{4/3}$ or $w=\pm q^{-1}$ (both work). This is
indeed the case for all prime links up to 9 crossings, and all prime knots up to 10 crossings.

Even for the adjoint representation of $G_2$, an earlier program
by KM
for computing quantum knot invariants gets only as far as the trefoil.
Sure enough, for the left-handed trefoil we have the value quoted in
Equation~\eqref{eq:excep-trefoil}, which at $v= q^2, w= q^5$ gives
\begin{align*}
& q^{144}-q^{126}+q^{122}-q^{116}+2 q^{112}+q^{110}-q^{108}+2 q^{104}+q^{102}-q^{98}-q^{96}+q^{94}-2 q^{90}-2 q^{88} \\ 
& \qquad -q^{86}-q^{84}-2 q^{82}-3 q^{80}-2 q^{78}-2 q^{76}-
  2 q^{74}-2 q^{72}-2 q^{70}-q^{68}-q^{64}-q^{62}+q^{60}+q^{58} \\
& \qquad +q^{56}+2 q^{54}+q^{52}+2 q^{50}+3 q^{48}+2 q^{46}+
   2 q^{44}+3 q^{42}+2 q^{40}+2 q^{38}+3 q^{36}+2q^{34}+q^{32} \\
& \qquad +2 q^{30}+q^{28}+q^{26}+q^{24}+q^{18},
\end{align*}
agreeing with the direct calculation using $R$-matrices \cite{QuantumGroups}.
One could also attempt to use the adjoint representation skein theory for $G_2$ from
\cite{MR1403861}, although we haven't pursued this.

As an example,
the value of $\xi(8_{18})$, the first non-algebraic knot, is too large to display here. (It has 426 monomial terms.)
Its specialization to the quantum $E_8$ curve shows that the $E_8$ adjoint quantum knot invariant is the polynomial given in Appendix \ref{app:E8}.
This calculation was previously well beyond what was possible.

 We can evaluate a variety of knotted trivalent graphs. The actual
 values are mostly too complicated
to warrant repeating here, but we have
\begin{align*}
\xi\left(\mathfig{0.05}{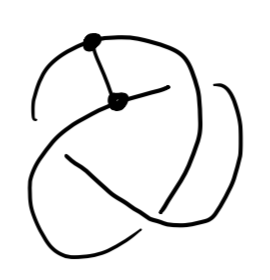}\right) & = 
\frac{\left(v^4+1\right) \left(v^6-w\right) \left(v^6+w\right) \left(v^5 w-1\right)\left(v^5 w+1\right) }{v^{36} (w-1) w^4 (w+1) (v-w) (v+w)} \times \\
& \quad \times
\big(v^{30} w^4-v^{30} w^2+v^{30}-v^{28} w^6+v^{28} w^4+v^{26} w^8-v^{26} w^4+v^{26} w^2+v^{24} w^6+v^{24} w^4 
   \\ 
& \qquad -2 v^{24} w^2+v^{24}-2 v^{22} w^6+3 v^{22} w^4-v^{22} w^2+v^{20} w^8-v^{20} w^6-2 v^{20} w^4+2 v^{20} w^2 \\
& \qquad -v^{20} +2 v^{18} w^6-2 v^{18} w^4-v^{16} w^8+3 v^{16} w^4-2 v^{16} w^2-2 v^{14} w^6+2 v^{14} w^2-v^{14} \\
& \qquad +2 v^{12} w^6-5 v^{12} w^4 +2 v^{12} w^2-v^{10} w^8+2 v^{10} w^6-v^{10}
   w^2-v^8 w^6+2 v^8 w^4-v^6 w^4 \\
& \qquad +v^6 w^2+v^4 w^6-v^4 w^4+v^2 w^4-v^2 w^2-w^6+w^4\big)
\end{align*}
Each of the random selections from the list of 7-crossing knotted thetas
\cite{MR2507922} we tried is computable.
We can also compute the value of the
dodecahedron---even though we haven't explicitly written down
any relation that reduces the dodecahedron and it has width~$7$, one
can easily check it
can be written in terms of the planar operations discussed above.

\subsection{Idempotents in the 6-box space} \label{sec:idempotents-dimensions}

In order to analyze the 6-box space, we compute in the span of
$\cB^{\text{braided}}_6$ the simultaneous eigenvectors for the three
commuting operators
\begin{align*}
\Tw_3 = (\beta_6 \rho_6 \beta_6 \rho_6^{-1})^3 & = \begin{tikzpicture}[baseline, thick]
      \braid[number of strands=6, xscale=0.35, yscale=-0.35] (braid) a_1^{-1} a_2^{-1} a_1^{-1} a_2^{-1} a_1^{-1} a_2^{-1};
      \node[draw, shape=semicircle, shape border rotate=180, minimum size=2.5em] at ($(braid-1-s)+(0.88,-0.45)$) (x) {};
    \end{tikzpicture}
&
\beta_6 & = \begin{tikzpicture}[baseline, thick]
      \braid[number of strands=6, xscale=0.35, yscale=-0.35] (braid) a_1^{-1};
      \node[draw, shape=semicircle, shape border rotate=180, minimum size=2.5em] at ($(braid-1-s)+(0.88,-0.45)$) (x) {};
    \end{tikzpicture}
&
\rho_6^4 \beta_6 \rho_6^{-4} & = \begin{tikzpicture}[baseline, thick]
      \braid[number of strands=6, xscale=0.35, yscale=-0.35] (braid) a_5^{-1};
      \node[draw, shape=semicircle, shape border rotate=180, minimum size=2.5em] at ($(braid-1-s)+(0.88,-0.45)$) (x) {};
    \end{tikzpicture}.
\end{align*}
(As before, all of these are negative twists.)

The characteristic polynomial of $\Tw_3$ is 
\begin{align*}
w^{-72}
\left(v^{36}-\lambda \right)
\left(v^{24}-\lambda \right)^{25}
\left(v^{12}-\lambda \right)^{16}  
& \left(v^{12} w^4-\lambda \right)^9
\left(v^{16}-\lambda  w^4\right)^9 \\
\left(1 - \lambda\right)
\left(w^{12}-\lambda \right) 
&\left(v^{12}-\lambda  w^{12}\right)
\left(v^4-\lambda \right)^9 
\left(w^6-\lambda \right)^4  
\left(v^6-\lambda  w^6\right)^4.
\end{align*}

By a direct calculation of null-spaces we find this triple of operators has distinct eigenvalues, as shown in Table \ref{tab:braid-gen}, so we obtain 80 eigenvectors, each of which is uniquely determined up to scale. 
\begin{table}
  \centering
  \begin{tabular}{ccccccc}
    \toprule
    &&\multicolumn{5}{c}{$\beta_6$ and $\rho_6^4 \beta_6 \rho_6^{-4}$ eigenvalues} \\ \cmidrule(l){3-7}
   $\Tw_3$ & Mult & $v^{12}$ & $-v^{-6}$ & $-1$ & $w^2$ & $v^2w^{-2}$ \\
    \midrule
   $v^{36}$        & $1$  & - & 1 & - & - & - \\[2pt]
   $v^{24}$        & $25$ & 1 & 1 & 1 & 1 & 1 \\[2pt]
   $v^{12}$        & $16$ & - & 1 & 1 & 1 & 1 \\[2pt]
   $v^{12}w^4$     & $9$  & - & 1 & 1 & 1 & - \\[2pt]
   $v^{16}w^{-4}$  & $9$  & - & 1 & 1 & - & 1 \\[2pt]
   $1$             & $1$  & - & - & 1 & - & - \\[2pt]
   $w^{12}$        & $1$  & - & - & - & 1 & - \\[2pt]
   $v^{12}w^{-12}$ & $1$  & - & - & - & - & 1 \\[2pt]
   $v^4$           & $9$  & - & - & 1 & 1 & 1 \\[2pt]
   $w^6$           & $4$  & - & - & 1 & 1 & - \\[2pt]
   $v^6w^{-6}$     & $4$  & - & - & 1 & - & 1 \\
    \bottomrule
  \end{tabular}
  \caption{Decomposition of the 6-box space into eigenspaces of the
    full twist and braid generators. For a fixed eigenvalue of
    $\Tw_3$, the eigenvalues for $\beta_6$ and
    $\rho_6^4\beta_6\rho_6^{-4}$ may be chosen independently.}\label{tab:braid-gen}
\end{table}

In fact, writing a joint eigenvector with eigenvalues $x,y,z$ as $v_{x,y,z}$, we see that
they must provide matrix units (up to some scalars) with respect to a
multiplication operation. Let $\bullet$ or $m_{6,6,3}$ be the operation
\[
(m_{6,6,3} \colon \QExc(6) \otimes \QExc(6) \to \QExc(6) = 
  \begin{tikzpicture}[baseline, thick]
    \node[draw, shape=semicircle, shape border rotate=180, minimum size=1.5em] (x) at (0,0) {};
    \node[draw, shape=semicircle, shape border rotate=180, minimum size=1.5em] (y) at (1.5,0) {};
    \draw (x.135) -- +(0,0.5);
    \draw (x.120) -- +(0,0.5);
    \draw (x.100) -- +(0,0.5);
    \draw (x.80) to[out=90,in=90] (y.100);
    \draw (x.60) to[out=90,in=90] (y.120);
    \draw (x.45) to[out=90,in=90] (y.135);
    \draw (y.80) -- +(0,0.5);
    \draw (y.60) -- +(0,0.5);
    \draw (y.45) -- +(0,0.5);
  \end{tikzpicture}.
\]
We can't compute the matrix entries of $\bullet$ directly, but we can
reason about it.
Certainly $v_{x,y,z} \bullet v_{x',y',z'} = 0$ unless $x = x'$, since the full twist $\Tw_3$ is central.
Moreover as
\begin{align*}
\rho_6^4 \beta_6 \rho_6^{-4}(v_{x,y,z}) \bullet v_{x,y',z'} 
  & = v_{x,y,z} \bullet \beta_6 (v_{x,y',z'})
\end{align*}
and all the joint eigenspaces are one dimensional, we must have that
$v_{x,y,z} \bullet v_{x,y',z'}$ is a
multiple of $\delta_{z,y'} v_{x,y,z'}$.

In particular, the $v_{x,y,y}$ (when non-zero) are minimal idempotents (up to scale), and picking one
such eigenvalue $y_x$ for each eigenvalue $x$ of $\Tw_3$ so
$v_{x,y_x,y_x} \neq 0$, we get a complete
set of representatives of the minimal idempotents.

We can compute the quantum traces of these idempotents, as
$$\frac{( v_{x,y,y}, \mathrm{id} )^2}{( v_{x,y,y}, v_{x,y,y} )}.$$
(This quantity is homogeneous, so we may assume that we have picked $v_{x,y,y}$ to actually be an idempotent, and then $( v_{x,y,y}, v_{x,y,y} ) = ( v_{x,y,y}, \mathrm{id} )$ so after
canceling we have the quantum dimension.)

We obtain the dimensions displayed in Table \ref{tab:quantum-dimensions},
and one can check that after taking the limit $v\to 1$ (achieved by replacing $[k\lambda+l]$ with $k\lambda+l$) these agree
with the tables in \cite[p.\ 431]{MR1381778}.
The quantum dimensions are also depicted graphically in
Figure~\ref{fig:dimensions}.
Note the similarities with the kinds of factors that come out of the
Weyl character formula.

\begin{table}
  \centering
  \begin{tabular}{ccc}
    \toprule
   eigenvalue of $\Tw_3$ & representation & quantum dimension \\
    \midrule
   $v^{36}$        & $1$     & $1$ \\[6pt]
   $v^{24}$        & $X_1$   & $-\frac{[4][\lambda+5][\lambda-6]}{[2][\lambda][\lambda-1]}$ \\[6pt]
   $v^{12}$        & $X_2$   & $\frac{[5][\lambda+5][\lambda+3][\lambda-4][\lambda-6][2\lambda+4][2\lambda-6]}{[1][\lambda+2][\lambda][\lambda-1][\lambda-3][2\lambda][2\lambda-2]}$ \\[6pt]
   $v^{12}w^4$     & $Y_2$   & $-\frac{[4][5][6][\lambda+5][\lambda-4][3\lambda-6]}{[2][\lambda][\lambda-1][\lambda-2][2\lambda][2\lambda-1]}$ \\[6pt]
   $v^{16}w^{-4}$  & $Y'_2$ & $-\frac{[4][5][6][\lambda-6][\lambda+3][3\lambda+3]}{[2][\lambda-1][\lambda][\lambda+1][2\lambda-2][2\lambda-1]}$ \\[6pt]
   $1$             & $X_3$   & $- \frac{[4][5][\lambda+5][\lambda+4][\lambda-5][\lambda-6][2\lambda+4][2\lambda-6][3\lambda+3][3\lambda-6]}{[1][2][\lambda+1][\lambda][\lambda-1][\lambda-2][2\lambda][2\lambda-2][3\lambda][3\lambda-3]}$ \\[6pt]
   $w^{12}$        & $Y_3$   & $- \frac{[4][5][6][\lambda+5][\lambda-4][\lambda-5][\lambda-6][2\lambda-4][5\lambda-6]}{[2][\lambda][\lambda-1][\lambda-2][2\lambda][2\lambda-1][2\lambda-2][3\lambda][3\lambda-1]}$ \\[6pt]
   $v^{12}w^{-12}$ & $Y'_3$ & $- \frac{[4][5][6][\lambda-6][\lambda+3][\lambda+4][\lambda+5][2\lambda+2][5\lambda+1]}{[2][\lambda-1][\lambda][\lambda+1][2\lambda-2][2\lambda-1][2\lambda][3\lambda-3][3\lambda-2]}$ \\[6pt]
   $v^4$           & $A$     & $- \frac{[6][\lambda+5][\lambda+4][\lambda+3][\lambda-4][\lambda-5][\lambda-6][3\lambda-6][3\lambda+3]}{[2][\lambda+1][\lambda]^2[\lambda-1]^2[\lambda-2][3\lambda-1][3\lambda-2]}$ \\[6pt]
   $w^6$           & $C$     & $\frac{[4][5][6][\lambda+5][\lambda+3][\lambda-5][2\lambda+4][2\lambda-4][2\lambda-6][4\lambda-6]}{[1][\lambda+2][\lambda]^2[\lambda-1][\lambda-2][2\lambda-1][2\lambda-3][3\lambda][3\lambda-2]}$ \\[6pt]
   $v^6w^{-6}$     & $C'$   & $\frac{[4][5][6][\lambda-6][\lambda-4][\lambda+4][2\lambda-6][2\lambda+2][2\lambda+4][4\lambda+2]}{[1][\lambda-3][\lambda-1]^2[\lambda][\lambda+1][2\lambda-1][2\lambda+1][3\lambda-3][3\lambda-1]}$ \\ 
    \bottomrule
  \end{tabular}
  \caption{Quantum dimensions of minimal idempotents in the 6-box space}\label{tab:quantum-dimensions}
\end{table}
\begin{figure}
  \begingroup
  \newcommand{\pnode}[2]{
    \node at (#1,#2) {$\bullet$};
    \node at (-#1,-#2) {$\bullet$}
  }
  \newcommand{\mnode}[2]{
    \node at (#1,#2) {$\times$};
    \node at (-#1,-#2) {$\times$}
  }
  \newcommand{\mnodes}[3]{
    \node[circle,draw,inner sep=0pt] at (#1,#2) {$\times$};
    \node[circle,draw,inner sep=0pt] at (-#1,-#2) {$\times$}
  }
  \newcommand{\axes}[2]{
    \draw (-#1-0.5,0) to (#1+0.5,0);
    \draw (0,-#2-0.5) to (0,#2+0.5);
    \foreach \y in {-#2,...,0}
      \foreach \x in {\numexpr-#1-\y,...,#1}
        \fill (\x,\y) circle [radius=0.5pt];
    \foreach \y in {0,...,#2}
      \foreach \x in {-#1,...,\numexpr#1-\y}
        \fill (\x,\y) circle [radius=0.5pt];
  }
  \tikzset{slant grid/.style={x=0.5cm,y={(0.25cm,0.3536cm)},baseline=0}}
  \centering
\begin{tabular}{cc}
    $\underset{\textstyle\dim X_1}{
\begin{tikzpicture}[slant grid]
  \axes{6}{2}
  \pnode{4}{0}; \pnode{5}{1}; \pnode{-6}{1};
  \mnode{2}{0};
  \mnode{0}{1};
  \mnode{-1}{1};
\end{tikzpicture}}$
     &
    $\underset{\textstyle\dim X_2}{
\begin{tikzpicture}[slant grid]
  \axes{6}{2}
  \pnode{5}{0};
  \pnode{5}{1}; \pnode{3}{1}; \pnode{-4}{1}; \pnode{-6}{1};
  \pnode{4}{2}; \pnode{-6}{2};
  \mnode{1}{0};
  \mnode{2}{1}; \mnode{0}{1}; \mnode{-1}{1}; \mnode{-3}{1};
  \mnode{0}{2}; \mnode{-2}{2};
\end{tikzpicture}}$\\[5pt]
    $\underset{\textstyle\dim Y_2}{
\begin{tikzpicture}[slant grid]
  \axes{6}{3}
  \pnode{4}{0}; \pnode{5}{0}; \pnode{6}{0};
  \pnode{5}{1}; \pnode{-4}{1};
  \pnode{-6}{3};
  \mnode{2}{0};
  \mnode{0}{1}; \mnode{-1}{1}; \mnode{-2}{1};
  \mnode{0}{2}; \mnode{-1}{2};
\end{tikzpicture}}$&
    $\underset{\textstyle\dim Y'_2}{
\begin{tikzpicture}[slant grid]
  \axes{6}{3}
  \pnode{4}{0}; \pnode{5}{0}; \pnode{6}{0};
  \pnode{-6}{1}; \pnode{3}{1};
  \pnode{3}{3};
  \mnode{2}{0};
  \mnode{-1}{1}; \mnode{0}{1}; \mnode{1}{1};
  \mnode{-2}{2}; \mnode{-1}{2};
\end{tikzpicture}}$\\
    $\underset{\textstyle\dim X_3}{
\begin{tikzpicture}[slant grid]
  \axes{6}{3}
  \pnode{4}{0}; \pnode{5}{0};
  \pnode{5}{1}; \pnode{4}{1}; \pnode{-5}{1}; \pnode{-6}{1};
  \pnode{4}{2}; \pnode{-6}{2};
  \pnode{3}{3}; \pnode{-6}{3};
  \mnode{1}{0}; \mnode{2}{0};
  \mnode{1}{1}; \mnode{0}{1}; \mnode{-1}{1}; \mnode{-2}{1};
  \mnode{0}{2}; \mnode{-2}{2};
  \mnode{0}{3}; \mnode{-3}{3};
\end{tikzpicture}}$ &
    $\underset{\textstyle\dim Y_3}{
\begin{tikzpicture}[slant grid]
  \axes{6}{5}
  \pnode{4}{0}; \pnode{5}{0}; \pnode{6}{0};
  \pnode{5}{1}; \pnode{-4}{1}; \pnode{-5}{1}; \pnode{-6}{1};
  \pnode{-4}{2}; \pnode{-6}{5};
  \mnode{2}{0};
  \mnode{0}{1}; \mnode{-1}{1}; \mnode{-2}{1};
  \mnode{0}{2}; \mnode{-1}{2}; \mnode{-2}{2};
  \mnode{0}{3}; \mnode{-1}{3};
\end{tikzpicture}}$ \\
    $\underset{\textstyle\dim A}{
\begin{tikzpicture}[slant grid]
  \axes{6}{3}
  \pnode{6}{0}; 
  \pnode{5}{1}; \pnode{4}{1}; \pnode{3}{1}; \pnode{-4}{1}; \pnode{-5}{1}; \pnode{-6}{1}; 
  \pnode{-6}{3}; \pnode{3}{3};
  \mnode{2}{0};
  \mnode{1}{1}; \mnodes{0}{1}{2}; \mnodes{-1}{1}{2}; \mnode{-2}{1};
  \mnode{-1}{3}; \mnode{-2}{3};
\end{tikzpicture}}$ &
    $\underset{\textstyle\dim C}{
\begin{tikzpicture}[slant grid]
 \axes{6}{4}
 \pnode{4}{0}; \pnode{5}{0}; \pnode{6}{0}; 
 \pnode{5}{1}; \pnode{3}{1}; \pnode{-5}{1};
 \pnode{4}{2}; \pnode{-4}{2}; \pnode{-6}{2};
 \pnode{-6}{4};
 \mnode{1}{0};
 \mnode{2}{1}; \mnodes{0}{1}{2}; \mnode{-1}{1}; \mnode{-2}{1};
 \mnode{-1}{2}; \mnode{-3}{2};
 \mnode{0}{3}; \mnode{-2}{3};
\end{tikzpicture}}$ 
\end{tabular}
\endgroup
\caption{A graphical depiction of the quantum dimension formulae. A
  $\bullet$, $\times$, or $\otimes$ at
$(x,y)$ means that $[x+y\lambda]$ appears, respectively, in the
numerator, in the denominator, or twice in the denominator. (If there is a mark
at $(x,y)$, there will be the same mark at $(-x,-y)$.) The $y$ axis is
drawn slanted to properly reflect the symmetry of the theory
$w \leftrightarrow v/w$ (see \S\ref{sec:symmetries}): the
representations (like~$X_n$) that are invariant
under this symmetry are vertically symmetric, while those (like~$Y_n$)
that appear in unprimed and primed versions are mirror images of each
other. To illustrate this, both $Y_2$ and $Y'_2$ are shown, while $Y_3$ and $C$ are only
shown in the unprimed version.}
\label{fig:dimensions}
\end{figure}
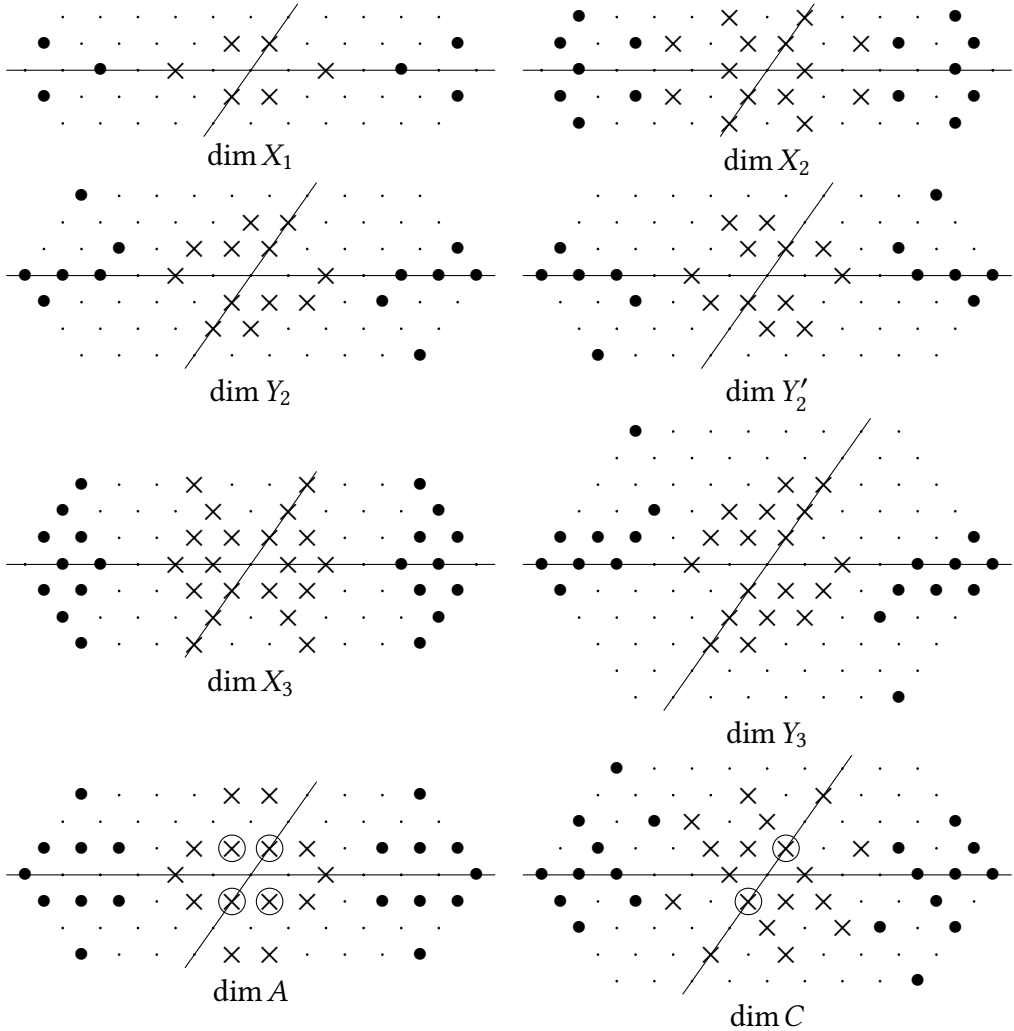

\begin{remark}
  This approach to finding quantum dimensions, by simultaneously
  diagonalizing commuting twist operators, is reminiscent of the
  Jucys-Murphy approach to the representation of the symmetric group,
  as explained by Vershik and Okounkov \cite{MR2050688}. (It would
  look more similar if, instead of $\Tw_3$, we used the equivalent operator
  $\Tw_3(\beta_6)^{-2}$ that wraps the third strand around the first
  two.) Unfortunately without some
  refinement this approach is likely to fail at the next stage, since
  (in the language of Vershik-Okounkov) the classical Bratteli
  diagram is not simple at the next stage: for instance,
  \[
    X_1 \otimes A \simeq 2A \oplus \cdots.
  \]
  It is correspondingly unlikely that the joint eigenspaces of the
  analogous operators acting on the 8-box space are all 1-dimensional.
\end{remark}

\subsection{Heuristic calculations using the Kontsevich integral}\label{sec:heuristics}

In this section we briefly sketch how one can more quickly guess
several of the results from \S\ref{sec:idempotents-dimensions}
assuming Deligne's conjecture, using the
Kontsevich integral construction and Cohen--de-Man's computations.  In
addition to what we used in
\S\ref{sec:classical-quantum}, we will use the version
of the Kontsevich integral that incorporates another line type labelled
by a representation.  In particular, we need that if the quadratic Casimir
of a Lie algebra object $\mathfrak{g}$ acts on an irrep $V$ by a scalar $C$
then the
framing change operator $\mathrm{Fr}$  acts on $V$ by the power
series $e^{hC/2}$, as in Remark~\ref{rem:q-inner-prod}.  The
calculations in this section generalize
those in the last section of \cite{MR1815266}.

For the classical exceptional series, Deligne \cite{MR1378507} gives
decomposition formulas for up to the third tensor power of the adjoint
representation. 
\begin{align*}
  X_1^{\otimes 0} &\simeq 1\\
  X_1^{\otimes 1} &\simeq X_1\\
  X_1^{\otimes 2} &\simeq 1 \oplus X_1 \oplus X_2 \oplus Y_2 \oplus Y'_2\\
  X_1^{\otimes 3} &\simeq 1 \oplus 5X_1 \oplus 4X_2 \oplus 3Y_2 \oplus 3Y'_2 \oplus X_3
                    \oplus Y_3 \oplus Y'_3 \oplus 3A \oplus 2C \oplus 2C'.
\end{align*}
(Recall that $X_1 = \fg$.)
He also gives information about how the symmetric group $S_n$ acts on
$X_1^{\otimes n}$.

In a quantum group, we expect $X_1^{\otimes n}$ to have
commuting representations of $U_q(\fg)$ and of $B_n$. Of course
representations of~$B_n$ are much more complicated than
representations of~$S_n$, so the simple plethysm-type formulas aren't
going to make a lot of sense. The four different copies of (for instance)
$X_2$ inside of $X_1^{\otimes 3}$, which appear in three different
representations of $S_3$ in the plethysm formulas, presumably come
from a single irreducible representation of $B_3$ which happens to
become reducible at $q=1$.

In our diagrammatic calculus, we have no direct access to the
representation $X_1^{\otimes 3}$, since we can only talk about
invariant maps between representations and not the representations
themselves. Instead, we can compute the 6-box
space, which is $\End(X_1^{\otimes 3})$ and comes with a left and
right action by $B_3$. All the multiplicities end up getting squared
when we do that. Note that
\[
80 = 1 + 5^2 + 4^2 + 2\cdot 3^2 + 1 + 2\cdot 1 + 3^2 + 2\cdot 2^2.
\]

To sort out which representation is which, we look at the center
of~$B_n$, for which we take as generator the negative
full-twist operator $\Tw_n$ (see \eqref{eq:neg-twist-positive}).
In general, braid
group representations break up as a direct sum of representations with
different eigenvalues of this full twist. It is also related to the
quadratic Casimir. To get that relation more precise, it's a little
better to look at a related operator, the framing change
operator~$\Fr$, which does a positive Reidemeister 1 move to the
entire bundle of
strands.
This is the same as doing a positive full twist and adding a positive kink
(Reidemeister 1) to each strand individually. As such, we have
\[
\Fr = v^{12n} (\Tw_n)^{-1},
\]
since the framing change on each strand contributes a factor of~$v^{12}$.

For $n=2$, we also have the half-twist operator $\HTw_2$, which does
only a (negative) half-twist. (This makes sense in general, but is only central
for $n=2$.)

The data for all representations up to the third tensor power is shown
in Table~\ref{tab:e-vals}.
The ``Casimir'' column comes from
Deligne \cite[\S(\textbf{D})]{MR1378507}, substituting $a = -\lambda/6$ and $a^* = (-1+\lambda)/6$ and
multiplying by $12$ since we use the Convenient normalization rather
than the Killing normalization that Deligne uses. For the
representations appearing in $\fg^{\otimes 2}$ we already computed 
the eigenvalues of the half-twist in Lemma~\ref{lem:eigenvalues-twist}.
The framing change eigenvalue can be read off from the Casimir as
$e^{-hC}$ as above.

\begin{table}
  \centering
\begin{tabular}{cccccccccc}
  \toprule
      &         &   \multicolumn{4}{c}{Eigenvalues}    & \multicolumn{4}{c}{Multiplicities} \\
  \cmidrule(lr){3-6} \cmidrule(l){7-10}
  Rep & Casimir & $\Fr$ & $\HTw_2$ & $\Tw_2$ & $\Tw_3$ & $\fg^{\otimes0}$ & $\fg^{\otimes1}$ & $\fg^{\otimes2}$ & $\fg^{\otimes3}$\\ \midrule
  1   & 0       & 1     & $v^{12}$ & $v^{24}$ & $v^{36}$
      & 1 & 0 & 1 & 1\\[3pt]
  $\fg=X_1$ & 12 & $v^{12}$ & $-v^6$ & $v^{12}$ & $v^{24}$
      &   & 1 & 1 & 5\\[3pt]
  $X_2$ & 24     & $v^{24}$ & $-1$ & $1$      & $v^{12}$
      &   &   & 1 & 4\\[3pt]
  $Y_2$ & $12 - 4\lambda$ & $v^{24}w^{-4}$ & $w^2$ & $w^4$ & $v^{12}w^4$
      &   &   & 1 & 3\\[3pt]
  $Y'_2$ & $20 + 4\lambda$ & $v^{20}w^{4}$ & $v^2 w^{-2}$ & $v^4 w^{-4}$ & $v^{16}w^{-4}$
      &   &   & 1 & 3\\[3pt]
  $X_3$ & $36$ & $v^{36}$ & & & $1$ &&&& 1\\[3pt]
  $Y_3$ & $36 - 12\lambda$ & $v^{36}w^{-12}$ &&& $w^{12}$ &&&& 1\\[3pt]
  $Y'_3$ & $24 + 12\lambda$ & $v^{24}w^{12}$ &&& $v^{12}w^{-12}$ &&&& 1\\[3pt]
  $A$ & $32$ & $v^{32}$ &&& $v^4$ &&&& 3\\[3pt]
  $C$ & $36 - 6\lambda$ & $v^{36}w^{-6}$ &&& $w^6$ &&&& 2\\[3pt]
  $C'$ & $30 + 6\lambda$ & $v^{30}w^{6}$ &&& $v^6w^{-6}$ &&&& 2\\[3pt]
  \bottomrule
\end{tabular}
\caption{Eigenvalues and multiplicities for representations appearing in $\fg^{\otimes n}$ for $n \le 3$.}
\label{tab:e-vals}
\end{table}

For each of the representations appearing in the third tensor power,
we can deduce the eigenspace decomposition of the braid
generator. Suppose we have a $d$-dimensional action of $B_3$, and
consider the determinant of
$\Tw_3$. If a (negative)
braid generator has eigenvalues $\lambda_1,\dots,\lambda_d$ and the
full twist has eigenvalue $\lambda_{\text{tot}}$, this is
\[
(\lambda_1\lambda_2\cdots\lambda_d)^6 = (\lambda_{\text{tot}})^d.
\]
On the other hand, each $\lambda_i$ must be in the set
$\{v^{12},-v^{-6},-1,w^2,v^2w^{-2}\}$, since we can factor through a
projection on to one of the irreducible representations appearing in
$\fg^{\otimes2}$. In each case, this uniquely determines
the~$\lambda_i$ given~$\lambda_{\mathrm{tot}}$. The results recover
exactly Table~\ref{tab:braid-gen}
which we computed directly. We
could also read this off from Deligne's formulas for decomposition of
tensor powers. For instance, reading down the ``$-1$'' eigenvalue
column of Table~\ref{tab:braid-gen} tells us that
\[
X_2 \otimes \fg \simeq X_1 \oplus X_2 \oplus Y_2 \oplus Y'_2 \oplus X_3
\oplus A \oplus C \oplus C',
\]
since $X_2$ is the representation in $\fg^{\otimes2}$ where the braid
generator has eigenvalue $-1$.

We can also heuristically make the same calculations for
the fourth tensor power.
Cohen--de-Man
\cite{MR1381778} say that the fourth tensor power should decompose as
\begin{multline*}
\fg^{\otimes4} = 5\cdot 1 \oplus 16X_1 \oplus 23X_2 \oplus 18Y_2^{(\prime)}
  \oplus 21A \oplus 16C^{(\prime)} \oplus 12X_3 \oplus 6Y_3^{(\prime)} \\
  \oplus 6D^{(\prime)} \oplus 8E \oplus 6F^{(\prime)} \oplus 3G^{(\prime)} \oplus 2H^{(\prime)}
    \oplus 3I^{(\prime)} \oplus 3J \oplus X_4 \oplus Y_4^{(\prime)},
\end{multline*}
where the notation $Y_2^{(\prime)}$, for instance, means $Y_2 \oplus Y'_2$.

For the representations newly appearing in the fourth tensor power, we
have Casimirs and eigenvalues in Table~\ref{tab:fourth-tensor}. The
decomposition formulas can also
tell us the eigenvalues of the full twist on three strands within each
of these representations and the eigenvalues of a braid
generator.

\begin{table}
  \centering
  \begin{tabular}{ccccccccccc}
    \toprule
    Rep & Mult & Casimir & $\Fr$ & $\Tw_4$ \\ \midrule
    $J$ & 3 & $40$ & $v^{40}$ & $v^8$ \\[3pt]
    $E$ & 8 & $42$ & $v^{42}$ & $v^6$ \\[3pt]
    $F$ & 6 & $44 - 4\lambda$ & $v^{44}w^{-4}$ & $v^4w^4$ \\[3pt]
    $D$ & 6 & $44 - 8\lambda$ & $v^{44}w^{-8}$ & $v^4w^8$ \\[3pt]
    $X_4$ & 1 & $48$ & $v^{48}$ & 1\\[3pt]
    $I$ & 3 & $48 - 8\lambda$ & $v^{48}w^{-8}$ & $w^8$\\[3pt]
    $H$ & 2 & $48 - 12\lambda$ & $v^{48}w^{-12}$ & $w^{12}$\\[3pt]
    $G$ & 3 & $48 - 16\lambda$ & $v^{48}w^{-16}$ & $w^{16}$ \\[3pt]
    $Y_4$ & 1 & $48 - 24\lambda$ & $v^{48}w^{-24}$ & $w^{24}$ \\[3pt]
    \bottomrule
  \end{tabular}
  \caption{Conjectural representations newly appearing in the fourth tensor power. The
    representations have been sorted by the value of the Casimir. Only
  one out of each primed pair is listed.}
  \label{tab:fourth-tensor}
\end{table}



\section{Special values}
\label{sec:special-values}

The goal of this section is to analyze several special cases.
First, we introduce the classical $F4$ and $G2$ families and Deligne's symmetric group family.
Second, we record which instances of the quantum
exceptional relation hold in particular examples when more than one
relation holds.
Third, when $v$ is a $10$th or $12$th root of unity, 
the proof of our main theorem breaks down.  We delay a full treatment of what happens when 
$v$ is a $10$th or $12$th root of unity to later work; here we will restrict ourselves to some results 
characterizing the classical $G2$, $F4$, and exceptional families, and a classification under
the additional assumption that the crossing is symmetric.
Fourth, if $\fg$ is a Lie 
algebra lying in the exceptional, $F4$, or $G2$ families we show that an appropriate quantum
group category takes a full and dominant functor from the $\mathsf{QExc}_{\mathbb{C},v,w}$.
Assuming a weak version of quantum sufficiency, this implies that these quantum group categories
are the Cauchy completion (i.e., additive and Karoubi completion) of the quotient of
$\mathsf{QExc}_{\mathbb{C},v,w}$ for that $v$ and $w$ by negligible morphisms.
Finally we conjecture a number of equivalences between quantum group categories
which correspond to intersections between lines corresponding to different quantum groups.
These may be thought of as an exceptional analogue of Level-Rank duality in the spirit of
Beliakova and Blanchet \cite{MR1710999, MR1854694}.

This section is somewhat longer than expected. In part this is simply because
we need to discuss what is happening in each of four families (the Exceptional family,
the $G2$ family, the $F4$ family, and the symmetric group family, corresponding to $v$
being plus or minus a primitive first, third, fourth, or twelfth root of unity). 
 The other reason is that the $F4$ family section is particularly long.
A lot of what we've written follows what was done by 
Gandhi, Savage, and Zainoulline \cite{GSZ23:DiagrammaticsF4} for the group
$F_4$ itself, but not for the rest of the $F4$ family, so we need to combine their
approach with results of Benkart, Elduque, and Cvitanovi\'c.

\subsection{Specialization theorems} \label{sec:specialization-outline}
We begin by outlining how to prove specialization theorems in the style of
Proposition~\ref{prop:class-spec}; this is used in Propositions
\ref{prop:G2specialization},
\ref{prop:F4specialization}, and~\ref{prop:ClassicalSpecializationProof}.
These classical specialization theorems can then be used to prove
quantum specialization theorems in the style of
Proposition~\ref{prop:quant-spec}, as in
Propositions~\ref{prop:quant-spec-2}, \ref{prop:quant-f4-spec},
and~\ref{prop:quant-g2-spec}.

The essential ideas here all go back to Kuperberg's original paper
\cite[Thm. 5.1]{MR1403861}, but see also
\cite[Thms. 2.1 and 2.3]{MR3951762} and \cite[Thm. 5.3]{GSZ23:DiagrammaticsF4}.

\begin{definition}
Suppose that $\cC$ is a category and that $\cD$ is a semisimple category. We say that $\cF\colon \cC \rightarrow \cD$ is \emph{dominant}
if the functor from the Cauchy completion of $\cC$ to $\cD$ guaranteed by the 
universal property of Cauchy completion is essentially surjective.
\end{definition}

Note that if $\cF\colon \cC \rightarrow \cD$ is full then it is dominant if and only if every simple object of $\cD$ is isomorphic to a direct 
summand of an object in the image of $\cF$.

The following result is well-known and follows easily from character theory 
(see this MathOverflow answer for details \cite{192103}).

\begin{lemma} \label{lem:faithful}
Suppose that $G$ is a compact group and $V$ is a finite dimensional
representation of~$G$. Then every irreducible representation of~$G$ 
appears as a direct summand of a tensor power of $V$ if and only if $V$ is faithful.
\end{lemma}

Suppose that $A$ is a commutative or anti-commutative real algebra 
(not necessarily unital or associative) with a definite symmetric bilinear
 form $\langle \cdot, \cdot \rangle \colon A \otimes A \rightarrow \mathbb{R}$
 which is associative in the sense that 
$\langle ab, c \rangle = \langle a, bc \rangle$.%
 \footnote{We allow the case of a negative definite form
because we will need to consider the case of the compact form of a semisimple Lie algebra
where the usual Killing form is negative definite.}
We define two signs, one measuring the commutativity and the other
the definiteness, namely $\sigma,\delta \in \{\pm 1\}$ where $xy = \sigma yx$ and $\delta \langle x , x\rangle \geq 0$.

The data of such an algebra yields a functor $\cF_A\colon \mathsf{Tri}_\mathbb{R} \rightarrow \Vect_\mathbb{R}$ with the property 
that the self-duality maps to a definite pairing, where the image of
the trivalent vertex is the multiplication map.
Complexifying we also get a functor $\cF_A\colon \mathsf{Tri}_\mathbb{C} \rightarrow \Vect_\mathbb{C}$. We can endow $\Vect_\mathbb{C}$
with the structure of a $C^*$ category by defining the adjoint to be the adjoint with respect to the positive definite Hermitian form coming
from complexifying $\delta \langle -, - \rangle$ (or, equivalently, endow $\Vect_\mathbb{R}$ with the structure of a real $C^*$ category
using the positive definite form $\delta \langle -, - \rangle$). 

We want to describe the adjoint more explicitly, but there is one subtle point which comes up. Namely if $V$ and $W$ are vector spaces
with an inner product, then the inner product on $V \otimes W$ requires swapping two coordinates:
\[V \otimes W \otimes V \otimes W \rightarrow V \otimes V \otimes W \otimes W \rightarrow \mathbb{R}.\]
This occurs even when factors are repeated, as when computing inner
products on $A^{\otimes n}$.
Thus, if $f$ is a real morphism 
$A^{\otimes n} \rightarrow A^{\otimes m}$,
its adjoint is the unique real morphism $f^*\colon A^{\otimes m}
\rightarrow A^{\otimes n}$ which satisfies
\begin{equation} \label{eq:innerproductswaps}
\prod_{i=1}^m \delta \langle (f(x)_i),  y_i \rangle = \prod_{i=1}^n \delta \langle x_{i},  f^*(y)_i\rangle.
\end{equation}
Since we pair the $i$th coordinate with the $i$th coordinate, when drawn diagrammatically (as in the examples below), this
equation will involve crossings.

\begin{lemma} \label{lem:star}
In the situation above, we have
\[\cF_A \left(\;  \drawcap \; \right)^* = \cF_A \left(\drawcup\right) \]
and
\[\cF_A \left(\upsidedownthreevertex \right)^* =  \sigma \delta \cF_A \left(\threevertex \right) \]
\end{lemma}
\begin{proof}
Rewriting the formula for the adjoint in Equation \ref{eq:innerproductswaps} using the diagram calculus, we directly 
check the defining property of the adjoints:
\[  \cF_A \left(\; \delta^0 \drawcap\; \right)  = \cF_A \left(\delta^2 \; \mathcenter{\capadjointb} \; \right)\]
and
\[\cF_A \left( \; \delta \trivalentadjointa \; \right)
  \; = \;  \cF_A \left( \delta^2 \;  \trivalentadjointb  \;\right).
  \; = \;  \sigma\cF_A \left(\; \trivalentadjointc  \;\right).
\]
Note that in the second equation the
inputs to the multiplication are swapped, which is why the sign $\sigma$ appears.
\end{proof}

\begin{remark}
Once you see that the adjoint of the trivalent vertex is plus or minus its reflection it is easy to work out which sign arises, 
because the composition of the adjoint of the trivalent vertex with
the vertex must be a positive multiple of the identity: the
trivalent vertex is adjoint to its reflection if the bigon value is positive and
adjoint to the negative of its reflection otherwise.
In particular, there's a sign error in \cite[Thm. 5.1]{MR1403861}: since at $q=1$ the value of the bigon is negative, the 
adjoint of the trivalent vertex should be negative its reflection. (Also compare to \cite[\S 5]{MR3464395}.)
(For $G2$, we have that $\sigma = -1$, so Kuperberg is choosing
the pairing to be positive definite with $\delta = 1$, so $\sigma \delta = -1$.) This sign issue is completely 
inessential to the proof because all that's actually used is that the subcategory generated by the trivalent vertex is closed under
taking adjoints, which will be true for either sign.
\end{remark}

For $\epsilon\in\{\pm1\}$, we define a $\ast$-category%
\footnote{Recall that a $\ast$-category has an anti-linear involution,
  which in a $C^*$ category is furthermore required to give a definite inner
  product.}
$\mathsf{Tri}^\epsilon_{\CC}$ from
$\mathsf{Tri}_\mathbb{C}$ by defining the adjoint of a
morphism to complex conjugate coefficients, reflect diagrams
vertically, and multiply by $\epsilon$ raised to the number of
trivalent vertices. Then Lemma~\ref{lem:star} tells us that
$\cF_A\colon \mathsf{Tri}^{\sigma\delta}_\mathbb{C} \rightarrow
\Vect_\mathbb{C}$ is a $\ast$-functor.

Let $G_A$ denote the group of algebra automorphisms 
of $A$ which preserve the inner product. Note that $G_A$ is a compact group by definiteness of the inner product. Let 
$\Rep_\mathbb{R}(G_A)$ and $\Rep_\mathbb{C}(G_A)$ denote,
respectively, the symmetric ribbon category of real representations of~$G_A$ and 
the $C^*$ symmetric ribbon category of complex representations
of~$G_A$. Again we have a functor
$\cG_A\colon \mathsf{Tri}_\mathbb{R} \rightarrow \Rep_\mathbb{R}(G_A)$ and a $\ast$-functor
$\cG_A\colon \mathsf{Tri}^{\sigma\delta}_\mathbb{C} \rightarrow\Rep_\mathbb{C}(G_A)$.
We have the following result which is a special case of \cite[Thm. 2.3]{MR3951762}.

\begin{proposition}\label{prop:definite-full-dominant}
  Suppose that $A$ is a commutative or anti-commutative real algebra with an associative definite
symmetric bilinear form. Then the functor
$\cG_A\colon \mathsf{Tri}_\mathbb{C} \rightarrow \Rep_\mathbb{C}(G_A)$ is full and dominant.
\end{proposition}
\begin{proof}
The Cauchy completion of the image of $\cG_A$ is a $C^*$-subcategory $\cC$ of $\mathrm{Rep}_\mathbb{C}(G_A)$ 
 and in particular is semisimple and hence abelian.  By post-composing with the fiber functor 
 of $\mathrm{Rep}(G_A)$, we see that $\cC$ has a fiber functor to vector spaces.  By Tannaka-Krein 
 duality, we see that $\cC$ is the category of representations of a group $H$ with 
 $G_A \subseteq H \subseteq \mathrm{O}(A)$, where $O(A)$ is the orthogonal group of the vector space $A$ with
 its symmetric billinear form.  But the multiplication is a map of $H$-representations, 
 hence $H \subseteq \mathrm{Aut}(A)= G_A$.  Thus $H = G_A$ and the functor
 is full.

Now the functor is dominant by Lemma~\ref{lem:faithful} because $A$ is
a faithful representation of $G_A$.
\end{proof}

Of course this result gives us no control of the kernel of $\cG_A$. In our examples we will show that certain morphisms are in
the kernel of $\cG_A$, and so the functor factors through an appropriate quotient, but we usually cannot verify that these 
are the only morphisms in the kernel.

\begin{remark} \label{rem:full-dominant-quotient}
Suppose the functor $\cG_A\colon \mathsf{Tri}_\mathbb{C} \rightarrow \Rep_\mathbb{C}(G_A)$ factors through a quotient
category $\cC$. Then clearly the quotient functor $\cC \rightarrow \Rep_\mathbb{C}(G_A)$ is also full and dominant.
\end{remark}

We now turn our attention to quantum specialization results. Suppose that $G$ is a compact group and $G_q$
is corresponding quantum group category deforming $\Rep(G)$. We fix a functor 
$\cF\colon \mathsf{Tri}_\mathbb{C} \rightarrow \Rep(G)$
and let $X$ denote the image of the strand under $\cF$.
We will say that a deformation $\cF_q\colon \mathsf{Tri}_\mathbb{C}
\rightarrow G_q$ defined for $q$ in
some open dense subset of $\mathbb{C}$
is \emph{given by rational functions} if the image of any trivalent
tangle written as a matrix with respect to the
standard bases of the representations (i.e., the PBW basis of lowering operators applied to the highest weight vector) 
has matrix entries given by rational functions of~$q$.%
\footnote{In general one should use the variable $s$ instead of $q$ throughout,
but as discussed in \S\ref{sec:QGConventions} since we generally look
at representations whose highest weight is a root, everything will
be a rational function in $q = s^L$ except when we look at some
entries of the $F4$ and $G2$ families.}

\begin{remark}
A more principled approach here is to use an algebro-geometric viewpoint instead of a complex-analytic one,
and introduce a rational form of the quantum group over $\mathbb{Q}[q,q^{-1}]$ and to
pass between different rings in the following arguments. See \cite{MR3951762, GSZ23:DiagrammaticsF4} where this is approach is
done carefully. We take a complex-analytic viewpoint so that we don't need to worry about some roots which could appear
in the formulas for $v$ and $w$ in terms of $q$. In our examples, $v$ and $w$ will turn out to be rational functions of $q$, 
but this depends on explicit calculations.
\end{remark}

\begin{proposition} \label{prop:rationaldeformation}
  Let $G$ be a compact group with a self-dual simple object $X \in
  \mathsf{Rep}(G)$
  that appears with multiplicity one inside $X \otimes X$.
  Then the induced
  functor $\mathcal{F} \colon
  \mathsf{Tri}_{\bbC} \to \mathsf{Rep}(G)$ has a deformation $\cF_q\colon \mathsf{Tri}_\mathbb{C} \rightarrow G_q$ given by rational functions.
Moreover, we may assume WLOG that after rescaling the self-duality the value of the bigon can be taken
to be any rational function of $q$ that we like.
\end{proposition}
\begin{proof}
We fix a family of objects~$X_q$ in each $G_q$ deforming~$X$. As in \cite[Lemma 2.4]{MR3951762}, there exists a
family of morphisms deforming the self-duality and the trivalent vertex given by rational functions (indeed by Laurent
polynomials). Since $X$ is irreducible and since $X$ appears with multiplicity one inside $X \otimes X$, these morphisms
are unique up to rescaling.

In $\mathsf{Rep}(G)$, the object~$X$ is symmetrically self-dual (since
$G$ is compact) and the trivalent vertex has rotational eigenvalue~$1$
(since $G$ is a group). This remains true under deformation, since rotational eigenvalues form 
a discrete set (and thus cannot change under deformation).

Finally the braiding is given by rational functions by 
the $R$-matrix formula.  Since the trivalent vertex, the self-duality, and the braiding generate all morphisms 
in $\mathsf{Tri}_\mathbb{C}$, we have a family of functors 
$\cF_q\colon \mathsf{Tri}_\mathbb{C} \rightarrow G_q$ given by rational functions.

 Since the functor constructed in the first paragraph is given by rational functions, 
 it sends the bigon to some rational function in $q$ 
times the identity morphism. If we rescale the choice of self-duality by any rational function then
the functor will still be given by rational functions, and so we can choose the self-duality in order to
assure that the value of the bigon is our favorite rational function.
\end{proof}

Note that this rescaling to fix the value of the bigon will introduce some additional denominators,
so the functor may no longer be defined by Laurent polynomials.

\begin{corollary} \label{cor:given-by-rational}
Suppose that $(\Rep(G), X, \tau)$ is a trivalent ribbon category with $\dim \Rep(G)_4 \leq 5$.
Then there are analytic functions $v(q)$ and $w(q)$ defined for $q$ on
an open dense subset in the complex topology on\/~$\bbC$
and a family of evaluation functors
$\mathsf{Tri}_{\mathbb{C}} \rightarrow \mathrm{QExc}_{\mathbb{C},v(q),w(q)} \rightarrow G_q$ which are 
given by rational functions.
\end{corollary}
\begin{proof}
By Proposition \ref{prop:rationaldeformation} we have a family of functors $\cF_q\colon \mathsf{Tri}_\mathbb{C} \rightarrow G_q$
given by rational functions. By Corollary \ref{cor:maincor}, after rescaling the self-duality to match our conventions for the bigon, 
outside of a finite set of points this functor factors through $\mathrm{QExc}_{\mathbb{C},v(q),w(q)}$ for some $v(q)$ and $w(q)$.
Moreover, by looking at the proof of Corollary \ref{cor:maincor},
it's easy to verify that $v(q)$ and $w(q)$ are given by algebraic
functions. Thus by choosing appropriate branch cuts in~$\bbC$
each is given by an analytic function of~$q$.
\end{proof}

\begin{proposition} \label{prop:rational-implies-full-dominant}
Suppose that $\cF_q\colon \mathsf{Tri}_\mathbb{C} \rightarrow G_q$ is given by rational functions,
and that $\cF_1$ is full and dominant. Then $\cF_q$ is full and dominant for $q$ in some dense 
subset of $\mathbb{C}$.
\end{proposition}
\begin{proof}
For fixed~$n$ the map is surjective on the $n$-box space at $q=1$ by
assumption, and is thus
surjective for $q$ in a Zariski open dense set, and thus a dense set in the ordinary complex topology.  
Therefore the functor is full on a countable intersection of dense sets, 
and thus by the Baire Category Theorem the functor is full on a dense set.

To see that it is dominant, we shrink the dense set to avoid all roots of unity, so that we can assume
that $G_q$ is semisimple. Then the fusion rules for $G_q$ agree with those of $\Rep(G)$, so the
image of the strand tensor generates and the functor is dominant.
\end{proof}

\subsection{Other \texorpdfstring{$1$}{1}-parameter classical families}
\label{sec:classical-G2F4}

In this section we introduce three families of symmetric ribbon
categories, namely $\mathit{Sym}(d+1)$, $G2(d)$,
and $F4(d)$. It will turn out these are related to
$\mathsf{QExc}_{R,v,w}$ with $v$ being a primitive $12$th root, a
primitive $3$rd root, 
or a primitive $4$th root of unity, respectively, much as the
classical exceptional family is related to $v=1$.

In all three cases a parallel phenomenon occurs: if $X$ is the strand,
then in the additive
completion $\mathbf{1} \oplus X$ admits the structure of an algebra
object satisfying certain axioms,
and that moreover these three families are universal for symmetric ribbon categories with such an algebra.
We will see that for the $\mathit{Sym}$ family $\mathbf{1} \oplus X$ is a commutative Frobenius algebra object, for the $G2$ family $\mathbf{1} \oplus X$
is a Hurwitz algebra object, and for the $F4$ family $\mathbf{1} \oplus
X$ is a unital Jordan algebra object with a trace
that moreover satisfies the cubic Cayley-Hamilton identity.

\subsubsection{\texorpdfstring{$\mathit{Sym}(t)$}{Sym(t)}}
\label{sec:Sym-t}
In this section we recall a description of Deligne's symmetric group family $S_t^{\mathrm{spec}}$ \cite{MR2348906}
and a closely related trivalent symmetric ribbon category $\mathit{Sym}(t)$, and the relationship between
these categories and commutative Frobenius algebras. This is based on
NS's diagrammatics for
$S_t$ which are described in \cite{2007.11640} and inspired by the diagrammatic description
given by Harriger \cite{JordynThesis}. Since this topic is somewhat tangential to the main point of the paper,
 and are not difficult given the results already in \cite{JordynThesis}, we will quickly state the correct theorems 
 without proof. More details should appear eventually in an expanded version of \cite{JordynThesis}joint with NS\relax.
 We will distinguish two slightly different variants, the special version $S_t^{\mathrm{spec}}$ and the
 normalized version $S_t^{\mathrm{norm}}$.

\begin{definition}
$S_t^{\mathrm{spec}}$ is the symmetric ribbon category generated by a trivalent vertex and a one-valent vertex modulo the following relations:

\begin{align*}
    \unknot\; &= t &\qquad
      \symtwist\; &=  \;\drawcup&\qquad
        \symtwistvertex\; &=  \;\threevertex\\[5pt]
    \loopvertex\;&=0&
      \twogon\;&= \;\onestrandid&\\
	\leftunitvertex \; &= \; \vertonestrandid \; = \; \rightunitvertex &\dogbone \;&= t & \drawH \; &= \; \drawI \; \\
\end{align*}
\end{definition}

\begin{remark}
When $t$ is not a positive integer, the Cauchy completion of $S_t^{\mathrm{spec}}$ is a semisimple abelian category.
When $t$ is a positive integer, then the quotient of the Cauchy
completion of $S_t^{\mathrm{spec}}$ by negligibles agrees with the representation
category $\Rep(S_t)$. (See \S\ref{sec:Rep-S3} for the case $t=3$.)
When $t$ is a positive integer there is also an ``abelian envelope''
of $S_t^{\mathrm{spec}}$
which is a non-semisimple abelian category \cite{MR2348906, MR2737787,
  1601.03426}.
The construction of this abelian category is much more subtle than
the diagrammatic version, and we will have no need to consider it.
\end{remark}

From the last defining
relation, we see that any two
trivalent trees on $n$ vertices are equal
when thought of as morphisms in $S_t^{\mathrm{spec}}$. Similarly, any two forests
with the same connected components give the same morphism. The
following sufficiency and
consistency result follows.

\begin{proposition}
Partitions of~$n$ give a basis of $S_t^{\mathrm{spec}}(n)$. In particular $\dim (S_t^{\mathrm{spec}}(n))$ is given
by the Bell Numbers, OEIS \oeis{A000110} \cite{EIS}.
\end{proposition}

\begin{definition}
  A \emph{commutative Frobenius algebra over $k$} is a unital
  commutative associative algebra $A$ over $k$ with a trace
  $\tr \colon A \rightarrow k$ such that $\tr(xy)$ is a non-degenerate
  symmetric bilinear form. A \emph{commutative Frobenius algebra
    object} in a symmetric ribbon category is a unital commutative
  associative algebra object $A$ with a trace
  $\tr \colon A \rightarrow \mathbf{1}$ such that $\tr \circ \mu$ is the
  evaluation of a symmetric self-duality $\mathrm{ev}$ on $A$.
\end{definition}

\begin{remark}
  We have started to distinguish between different units: plain~$1$ is
  the multiplicative unit in the base field~$k$,
  bold~$\bfOne$ is the unit object in a monoidal category (usually $k$
  as a vector space), and blackboard bold~$\bbOne$ is the unit in
  an algebra over~$k$.
\end{remark}

If $A$ is a Frobenius algebra object then
$(\mathrm{id} \otimes \mu)\circ(\mathrm{coev} \otimes \mathrm{id})$
(where $\mathrm{coev}$ is the inverse of $\mathrm{ev}$)
defines a coproduct $\Delta\colon A \rightarrow A \otimes A$ and Frobenius
algebras are often axiomatized in terms of this coproduct instead of
the trace.

\begin{definition}
A Frobenius algebra object is \emph{normalized} if $\tr(\bbOne) = 1$ and
\emph{special} if $\mu \circ \Delta = \mathrm{id}$.
\end{definition}

The following proposition is immediate from the above definition, and since this universal
property characterizes $S_t^{\mathrm{spec}}$, \cite[Prop.
8.3]{MR2348906} shows that it agrees with any other
definition of~$S_t^{\mathrm{spec}}$.

\begin{proposition}
If $\cC$ is a symmetric ribbon category, the groupoid of functors $S_t^{\mathrm{spec}} \rightarrow \cC$ and natural isomorphisms is equivalent to the groupoid of special commutative Frobenius algebra objects in $\cC$.
\end{proposition}

We also have the following slight variation on $S_t^{\mathrm{spec}}$, which is equivalent as a symmetric
ribbon category when $t \neq 0$ and which satisfies a similar
universal property for \emph{normalized} commutative
Frobenius algebras.

\begin{definition}
Deligne's $S_t^{\mathrm{norm}}$ is the symmetric ribbon category generated by a trivalent vertex and a one-valent vertex modulo the following relations:

\begin{align*}
    \unknot\; &= t &\qquad
      \symtwist\; &=  \;\drawcup&\qquad
        \symtwistvertex\; &=  \;\threevertex\\[5pt]
    \loopvertex\;&=0&
      \twogon\;&= t \;\onestrandid&\\
	\leftunitvertex \; &= \; \vertonestrandid \; = \; \rightunitvertex &\dogbone \;&= 1 & \drawH \; &= \; \drawI \; \\
\end{align*}
\end{definition}

\begin{proposition}\label{prop:St-norm}
If $t \neq 0$, then $S_t^{\mathrm{spec}}$ and $S_t^{\mathrm{norm}}$ are equivalent as symmetric ribbon categories.
For all $t$, partitions form a basis for  $S_t^{\mathrm{norm}}(n)$ and so the dimensions 
of Hom-spaces are again given by Bell numbers.

If $\cC$ is a symmetric ribbon category, the groupoid of functors $S_t^{\mathrm{norm}} \rightarrow \cC$ and natural isomorphisms is equivalent to the groupoid of normalized commutative Frobenius algebra objects in $\cC$.
\end{proposition}

If $A$ is a Frobenius algebra object and $\tr(\bbOne) \neq 0$ then
$\bigl(a \mapsto \tr(a)\cdot \bbOne/\tr(\bbOne)\bigr) \co A \to A$
is
a projection, and so $A \cong \mathbf{1} \oplus X$ for some sub-object
$X$. The restriction of multiplication composed with projection
onto~$X$ gives a map $M\colon X \otimes X \rightarrow X$. It is easy
to see that the Karoubi envelope of $S_t^{\mathrm{norm}}$ with the
object $X$ and the map $M$ is a trivalent category, and we can easily
read off its
defining generators and relations by expanding $A \cong \mathbf{1}
\oplus X$ in the defining relations of $S_t^{\mathrm{norm}}$.

\begin{remark} \label{rem:element-convention} If $A$ is an algebra
  object in a symmetric monoidal category $\cC$ with a multiplication
  $\mu$ and some additional structures such as a unit
  $u\co \mathbf{1}\rightarrow A$, or trace $\tr\co A \rightarrow \mathbf{1}$, we will
  still write expressions ``in elements'' using dummy variables (as
  above). This notation will cause no confusion provided that each
  variable appears exactly once in each term. We can sum together such
  terms or set such sums equal to each other. For example, we can
  write $x \mapsto \tr(\bbOne x)$ to mean the composition
  $\tr \circ \mu \circ (u \otimes \mathrm{id})\colon A = k \otimes A
  \rightarrow k$. Or we can write the Jacobi identity for a Lie
  algebra object (using bracket as the algebra operation) as
  $[x,[y,z]] = [[x,y],z] + [y,[x,z]]$. Similarly, if $A$ is endowed
  with a symmetric self-duality, we will write a duality pairing
  $A \otimes A \rightarrow k$ using the notation
  $\langle x , y \rangle$.
\end{remark}

\begin{definition}
Let  $\mathit{Sym}(d+1)$ be the symmetric trivalent ribbon category given by the following relations.
\begin{align*}
    \unknot\; &= d &\qquad
      \symtwist\; &=  \;\drawcup&\qquad
        \symtwistvertex\; &=  \;\threevertex\\[5pt]
    \loopvertex\;&=0&
      \twogon\;&= (d-1) \;\onestrandid&
        \threegon\; &= (d-2) \;\threevertex
\end{align*}
\vspace{5pt}
\[  \drawH \; - \; \drawI \; +  \;\twostrandid \; -  \; \cupcap = 0.\]
\end{definition}

\begin{proposition}
The Cauchy completions of $S_t^{\mathrm{norm}}$ and $\mathit{Sym}(t)$ are equivalent.
This equivalence sends the strand in $S_t^{\mathrm{norm}}$ to the direct sum of the trivial and
the strand in $\mathit{Sym}(t)$.
\end{proposition}

Thus we can think of $\mathit{Sym}(t)$ as a interpolated symmetric group category,
but the strand in $\mathit{Sym}(t)$ does not correspond to the permutation
representation, but instead to the ``standard representation'',
i.e., the irreducible representation
corresponding to the partition $(d,1)$, which is the permutation
representation minus the trivial representation.  The endomorphism
algebras in this category are called the \emph{quasi-partition algebras} in
\cite{MR3177889}.

\begin{definition}
  A \emph{$\mathit{Sym}$-algebra} is
a (non-unital) $k$-algebra $A$ with a symmetric bilinear form with a
symmetric monoidal functor $\mathit{Sym}(t) \rightarrow \Vect_k$
sending the object to~$A$ and the trivalent vertex to the
multiplication. Similarly an \emph{$\mathit{Sym}$-algebra object} in
a symmetric ribbon category $\cC$ is
an object $A$ in together with a functor $\mathit{Sym}(t) \rightarrow \cC$ sending
the strand to~$A$.
\end{definition}

The
consistency and sufficiency theorem for $S_t^{\mathrm{norm}}$,
Proposition~\ref{prop:St-norm}, yields a similar
consistency and sufficiency theorem for $\mathit{Sym}(t)$. For each $m>1$ fix a trivalent tree 
with $m$ leaves. For each partition $\pi$ of $n$ into parts with no singletons, we have 
an element $v_\pi$ of $\mathit{Sym}(d+1)_n$ given by connecting the vertices in each part by the fixed tree. 

\begin{lemma}
The elements $v_\pi$, where $\pi$ ranges over partitions of $n$ 
into parts with no singletons, form a basis of $\mathit{Sym}(t)_n$. In particular,
$\dim \mathit{Sym}(t)_n$ is the sequence OEIS \oeis{A000296} \cite{EIS}.
\end{lemma}

The universal property satisfied by $\mathit{Sym}(t)$ is:

\begin{proposition}
  Suppose $\cC$ is a Karoubi complete symmetric ribbon category with a
  normalized commutative Frobenius algebra object $A \cong \bfOne \oplus X$
 with $\dim X = d$.
  Then $\cC$ takes a symmetric monoidal functor from $\mathit{Sym}(d+1)$ sending
  the strand to $X$ and the trivalent vertex to the restricted
  multiplication~$M$. (That is, $X$ is an $\mathit{Sym}$-algebra object in~$\cC$.)

Conversely, in the additive completion of $\mathit{Sym}(d+1)$, the direct sum of the 
unit and the strand admits the structure of a normalized commutative Frobenius algebra.
\end{proposition}

\begin{remark} \label{rem:Sym0SOSp}
Since all the relations defining $\mathit{Sym}(0)$ hold in $\Rep(SOSp(1|2),V_1^-)$,
there's a quotient functor $\mathit{Sym}(0) \rightarrow  \Rep(SOSp(1|2),V_1^-)$.
Note that the $4$-box space for the former is $4$-dimension, and the latter is $3$-dimensional;
this functor kills
\[\symcross  +1 \twostrandid + 2 \cupcap - \frac{2}{b} \drawI.\]
In fact, it follows from \cite[\S 5]{MR4450143} that the semisimplification of
$\mathit{Sym}(0)$ is $\Rep(SOSp(1|2),V_1^-)$. 
\end{remark}

We conclude this section with a brief discussion about normalization of the bigon.
  
\begin{remark} \label{rem:SymBigons}
There are three reasonable normalizations for $b$ in the definition of $\mathit{Sym}(d+1)$:
the ``spherical'' choice $b=\frac{d-1}{\sqrt{d+1}}$, the
``special'' choice $b = \frac{d-1}{d+1}$, and the ``normalized''
choice $b = d-1$. 
When $d +1 \neq 0$
the resulting categories are equivalent by changing the self-duality,
but when $d+1 = 0$ more
care should be taken.

In particular, when $d=-1$, our $\mathit{Sym}(0)$ does not agree with
Deligne's $S_0$.
This is because Deligne's $S_0^{\mathrm{spec}}$ has a special Frobenius algebra object
with $\tr(\bbOne) = 0$, while our $\mathit{Sym}(0)$ is equivalent to $S_0^{\mathrm{norm}}$ 
and thus has a normalized Frobenius
algebra object with $\tr(\bbOne) = 1$ but $\mu \circ \Delta = 0$.
In
particular, in our category the permutation object breaks up as a
direct sum of the standard and the trivial, whereas for Deligne's the
permutation will be a nontrivial extension of the trivial and the
standard.
A similar story applies for $SO(3)_q$ at $q=\pm i$ but for
(not necessarily commutative) Frobenius algebra objects in a (not necessarily symmetric) tensor category.
We hope to
address this further in a later paper.
\end{remark}

\subsubsection{Classical \texorpdfstring{$G2$}{G2} family}
\label{sec:classical-G2}

\begin{definition}
  The \emph{classical $G2(d)$ family} is the symmetric ribbon category
  generated by a trivalent vertex modulo the following relations:
\begin{gather*}
  \begin{align*}
    \unknot\; &= d &\qquad
      \symtwist\; &=  \;\drawcup&\qquad
        \symtwistvertex\; &=  - \;\threevertex\\[5pt]
    \loopvertex\;&=0&
      \twogon\;&= (1-d) \;\onestrandid&
        \threegon\; &=  (d-4) \;\threevertex
  \end{align*}\\[5pt]
  \drawI + \drawH = 2 \symcross - \twostrandid - \cupcap
\end{gather*}
\end{definition}

Following the PhD thesis of Boos \cite{Boos:thesis} in German, and its several English language expositions \cite{1011.6197, Street19:VectProd, twf169}, the classical $G2$ family is closely related to Hurwitz algebras. We summarize this as follows.

\begin{definition}\label{def:G2-algebra}
A \emph{vector product algebra} is
a (non-unital) $k$-algebra $A$ with a symmetric bilinear form
with a symmetric monoidal functor $G2(d) \rightarrow \Vect_k$
sending the object to~$A$ and the trivalent vertex to the
multiplication. We will also call it a \emph{$G2$-algebra}. Similarly
we will refer to an object $A$ in a symmetric ribbon category $\cC$
together with a functor $G2(d) \rightarrow \cC$ as a \emph{$G2$-algebra object}
in $\cC$.
\end{definition}

\begin{example}
$\mathbb{R}^3$ with the standard bilinear form and multiplication given 
by the cross product is a vector product algebra (or $G2$-algebra) over 
$\mathbb{R}$. For this example $d=3$. If we endow $\mathbb{R}^3$ with 
the defining action of $SO(3)$ then the standard bilinear form and the cross 
product are $SO(3)$-equivariant. Thus the defining representation of $SO(3)$ 
is a $G2$-algebra object in $\Rep(SO(3))$.
\end{example}

\begin{definition}
A \emph{Hurwitz algebra} over $k$ is a unital (not necessarily
associative) algebra $A$ over $k$ with a non-degenerate
symmetric bilinear form
$\langle\cdot ,\cdot\rangle\colon A \otimes A \rightarrow k$ such that the
norm map $N(v)\coloneq \langle v, v\rangle$ satisfies $N(\bbOne) = 1$ and $N(uv) = N(u)N(v)$.
\end{definition}

Hurwitz algebras are also called unital composition algebras. Often you will see composition 
algebras and Hurwitz algebras defined directly using a quadratic norm map $N$ in the 
definition instead of a bilinear form, but this is not necessary outside of characteristic $2$.

In any Hurwitz algebra we can define an involution $\overline{\vphantom{n}\cdot}$
via $\overline{v} = 2\langle v, \bbOne\rangle -v$, and it follows that
$N(v) \bbOne = \overline{v} v$. 

\begin{example}\label{examp:real-Hurwitz}
  The real numbers, complex numbers, quaternions, and octonions, with
  their standard norms, are Hurwitz algebras over $\mathbb{R}$.
  Moreover, in these examples the bilinear form is positive definite,
  from which it follows that they're all division algebras with the
  inverse of $x$ given by $\overline{x}/N(x)$. Their
  complexifications ($\mathbb{C}$, $\mathbb{C} \oplus \mathbb{C}$,
  $M_2(\mathbb{C})$, and the bi-octonions) are Hurwitz algebras
  over $\mathbb{C}$. (Since the quadratic form is no longer
  definite, they are not division algebras.)
\end{example}

We would like to give a definition of Hurwitz algebra objects in
symmetric ribbon categories, but there's a difficulty in generalizing
directly, since the construction of $N$ implicitly uses a diagonal
map. In order to remedy this situation we take the ``polarization'' of
the identity $N(uv) = N(u)N(v)$ as follows. Substitute
$u = a+b$ and $v=c+d$ and then take the terms that are homogeneous of
degree $(1,1,1,1)$, yielding (after dividing by $2$):
\begin{equation} \label{eqn:PolarizedComposition}
\langle ac, bd\rangle + \langle ad,bc\rangle = 2 \langle a, b\rangle \langle c, d \rangle
\end{equation}
Conversely Equation \eqref{eqn:PolarizedComposition} implies
the usual composition identity $N(uv) = N(u)N(v)$ by setting $a=b=u$
and $c=d=v$ and dividing by $2$.

\begin{definition}
  If $\cC$ is a symmetric ribbon category, then a \emph{Hurwitz algebra
  object} in $\cC$ is an object $A$ which is symmetrically self-dual,
  endowed with a multiplication $\mu\colon A \otimes A \rightarrow A$ and a
  unit $u\colon \bfOne \rightarrow A$, satisfying the unit axioms, and such that
   $\langle \bbOne, \bbOne\rangle  = 1$ and $\langle ac, bd\rangle +
   \langle ad,bc\rangle = 2 \langle a, b\rangle \langle c, d \rangle$.
\end{definition}

In this definition we have used the convention from Remark \ref{rem:element-convention}
to interpret these algebraic expressions as morphisms in $\cC$. We could also express these
identities using diagram notation as follows:
\[\mathcenter{\unitcap}  = 1 \quad \text{and} \quad \HurwitzTerma \; + \; \HurwitzTermb = 2\; \twocap\]

\begin{warning}
The multiplication map on a Hurwitz algebra is typically not
rotationally invariant (i.e., the inner product is typically not
associative).
\end{warning}

\begin{remark}
  We could treat vector product algebras in a similar way, and get
the axioms in Definition~\ref{def:G2-algebra} from polarizing the
natural identities for the cross product and dot product, namely 
$\langle x, x y\rangle = 0 = \langle x y, y\rangle$ and
$\langle xy, xy\rangle = \langle x, x\rangle \langle y, y\rangle - \langle x, y \rangle^2$ 
\cite{wiki:7dimcross}.
\end{remark}

If $A$ is a Hurwitz algebra object, then we have a trace map
which is dual to the unit map. Moreover, the first axiom of a Hurwitz
algebra object tells you that $\tr(\cdot) \bbOne\colon A \rightarrow A$ is
an idempotent, and so in the idempotent completion of $\cC$, we have
$A \cong \mathbf{1} \oplus X$ for some sub-object $X$.
The restriction of multiplication composed with
projection onto~$X$ gives a map $M\colon X \otimes X \rightarrow X$.

\begin{proposition}[{\cite[Thm. 25]{Street19:VectProd}}]
  Suppose $\cC$ is an idempotent complete symmetric ribbon category with a
  Hurwitz algebra object $A \cong \mathbf{1} \oplus X$ and $\dim X = d$.
  Then $\cC$ takes a symmetric monoidal functor from $G2(d)$ sending
  the strand to $X$ and the trivalent vertex to the restricted
  multiplication~$M$ (i.e., $X$ is a $G2$-algebra object in $\cC$).

Conversely, in the additive completion of $G2(d)$, the direct sum of the 
unit and the strand admits a unique structure of a Hurwitz algebra such that its
traceless part agrees with the $G2(d)$-algebra structure on $X$.
\end{proposition}

\begin{remark}
Explicitly, the multiplication on the Hurwitz algebra $\bfOne \oplus X$ is given by
$(a,x)(b,y) = (ab - \langle x, y\rangle, ay+bx+xy)$.
In particular, the component of the multiplication map in
$X \otimes X \rightarrow \bfOne$ is not the pairing $X \otimes X \rightarrow \bfOne$ induced
by the symmetric self-duality of $A$, but instead those two maps are negatives of
each other. One good way to think about this is that $\langle x , y \rangle = \tr(x\overline{y})$
and so they're negatives of each other when restricted to $X$, since $\overline{x} = -x$ in $X$.
In general, there are some parameters involved in turning a non-unital
algebra~$X$ into a unital algebra $\bfOne \oplus X$; see the proof of
Proposition~\ref{prop:tracial-traceless-Jordan}.
\end{remark}

The following lemma proves sufficiency for the classical $G2$ family.

\begin{lemma}
The\ $\Hom$ spaces for $G2(d)$ are spanned by
planar non-elliptic diagrams 
and thus are bounded by the dimensions of the\/ $\Hom$ spaces for the
category of 
representations of $G_2$.
\end{lemma}
\begin{proof}
  First, use the last (defining) identity to remove any crossings to
  reduce to planar diagrams.
  By \cite[Eq.~15]{1011.6197}, these relations imply a relation
  simplifying the square into a sum of planar terms with no faces, and
  by the same argument as \cite[\S 3.1]{1011.6197} the pentagon also
  simplifies into a sum of planar terms with no faces. Thus, by Euler
  characteristic (see \cite[\S 4]{MR1403861}) every $\Hom$ space is
  spanned by the non-elliptic diagrams.
\end{proof}

The following propositions show that consistency at generic values of
the parameter is false, and show
specialization at the points that are consistent; together they give a
generalization of the classification of Hurwitz algebras due to
Hurwitz.

\begin{proposition}[\cite{1011.6197}]
  $G2(d)$ is the zero category unless $d=0$, $1$, $3$, or $7$. 
\end{proposition}

If $A$ is one of the four real Hurwitz algebras from Example~\ref{examp:real-Hurwitz}, 
then let $G_A$ denote the group of unital algebra
automorphisms of the real Hurwitz algebra $A$ which preserve the
bilinear form, as in
\S\ref{sec:specialization-outline}. (It's not difficult to see that
these groups can also be described as the groups of automorphisms of the corresponding
real vector product algebra~$X$ which preserve the bilinear form.)
Since the inner product on the real Hurwitz algebras are positive definite, it follows that
$G_A$ is a compact group, and that its category of complex representations has a unitary
structure. The groups $G_A$ are trivial,
$S_2$, $SO(3)$, and the compact form of $G_2$, respectively, as shown
in Table~\ref{tab:G2} on page~\pageref{tab:G2}.

\begin{proposition} \label{prop:G2specialization}
If $d = 0$, $1$, $3$, or $7$, take $A$ to be 
  $\mathbb{R}$,
  $\mathbb{C}$,
  $\mathbb{H}$,
  or $\mathbb{O}$, respectively
  and define $G_A$ as above. Then at these four values of $d$, the category
   $G2(d)$ 
  has a full and dominant functor to
  the category $\Rep(G_A)$.
\end{proposition}

\begin{proof}
Since $A$ is a real Hurwitz
algebra of dimension $d+1$ we have a functor $\cF\co G2(d) \rightarrow \Rep(G_A)$.

This functor is full and dominant by
Proposition~\ref{prop:definite-full-dominant} and
Remark~\ref{rem:full-dominant-quotient}.
\end{proof}

Note that when $A=\mathbb{R}$ the image of the strand
in $G2(0)$ is the zero object in $\Rep(G_A)$ and when 
$A = \mathbb{C}$ the image of the
trivalent vertex in $G2(1)$ is the zero map. (The map $M$ defined
above is zero.) So in both of these cases
the target category is not trivalent.

\subsubsection{Classical \texorpdfstring{$F4$}{F4} family}
\label{sec:classical-F4}

\begin{definition}\label{def:classical-F4}
The classical $F4(d)$ family, studied earlier by DPT
\cite{ThurstonF4Family}, is the symmetric ribbon category
generated by a trivalent vertex modulo the following relations:
\begin{gather}
\begin{aligned}
    \unknot\; &= d &\qquad
      \symtwist\; &=  \;\drawcup&\qquad
        \symtwistvertex\; &=  \;\threevertex\\[5pt]
    \loopvertex\;&=0&
      \twogon\;&= \frac{d+2}{4} \;\onestrandid&
        \threegon\; &=  \frac{2-d}{8} \;\threevertex
\end{aligned}\nonumber\\[5pt]
\drawI + \drawH + \drawsymcrossX = \frac{1}{2} \left(\; \cupcap +
    \twostrandid + \symcross \; \right)\label{eq:F4}
\end{gather}\end{definition}

We have used a different bigon normalization than the one used
in \cite{ThurstonF4Family}, for reasons that will become apparent below.

As in the case of the classical exceptional family, we can change
variables $\nu = (d+4)/6$ to get a somewhat better behaved parameter,
following \cite{MR1952563}.  For our purposes an even better
parameterization is by $\eta = (d-2)/12 = (\nu - 1)/2$.
(This is not
the same as the $\eta$ parameter in \cite{ThurstonF4Family}.)

We have the following sufficiency conjecture, with some partial
evidence given in \cite{ThurstonF4Family}.

\begin{conjecture}
$F4(d)$ has finite dimensional\/ $\Hom$ spaces.
\end{conjecture}

As with the $G2$ family, consistency fails.

\begin{proposition}[\cite{ThurstonF4Family}]\label{prop:F4consistency}
$F4(d)$ is the zero category unless $d \in \{-2, -1, 0, 2, 5, 8, 14, 26\}$.
\end{proposition}

Our remaining goal is to prove a version of specialization at most of
these values.

\begin{definition}
An \emph{$F4$-algebra} is a (non-unital) $k$-algebra $X$ with a
symmetric bilinear form
with a symmetric monoidal functor $F4(d) \rightarrow \Vect_k$
sending the object to~$X$ and the trivalent vertex to the
multiplication. Similarly
we will refer to an object $X$ in a symmetric ribbon category $\cC$
together with a functor $F4(d) \rightarrow \cC$ as an \emph{$F4$-algebra object}
in~$\cC$.
\end{definition}

In the last section we saw that $G2$-algebras appear exactly as the
kernel of the trace $x \mapsto \langle x, \bbOne\rangle$ on a Hurwitz
algebra $A$. We will give a similar characterization here of $F4$-algebras
as the traceless part of cubic tracial Jordan algebras, following \cite{MR2418111, LieAsInvariance, GSZ23:DiagrammaticsF4}.

\begin{definition}\label{def:Jordan-alg}
A \emph{Jordan algebra} over $k$ is a commutative (non-associative)
$k$-algebra~$A$ satisfying the Jordan identity
\[(xy)(xx) = x(y(xx)).\]

A \emph{tracial Jordan algebra} is a Jordan algebra together with a
unit $k \rightarrow A$ satisfying the usual unit axioms, and a
non-degenerate symmetric bilinear form $\langle \cdot, \cdot \rangle\colon A \otimes A \rightarrow k$
which is associative, $\langle x, yz \rangle = \langle xy, z\rangle$, and satisfies $\langle \bbOne, \bbOne \rangle = 1$.
In a tracial Jordan algebra we let $\tr(x) \coloneq \langle x, \bbOne\rangle$.
Associativity guarantees that
$\langle x, y\rangle = \tr(xy)$ and $\tr((xy)z) = \tr(x(yz))$.
\end{definition}

Since $\langle x, y\rangle = \tr(xy)$ we will often specify a tracial Jordan algebra
by specifying the trace instead of the bilinear form; of course one
then needs to check
that the resulting bilinear form is non-degenerate.
The assumption that $\langle \bbOne, \bbOne\rangle = 1 = \tr(\bbOne)$
is not essential, we could instead
assume only that $\langle \bbOne, \bbOne\rangle \neq 0$. But this normalization is convenient
and one can always rescale the bilinear form to ensure it. However there are a few
places where our normalization assumption appears unexpectedly, see
Remark~\ref{rem:TracelessNormalization}.

Diagrammatically we will denote $A$ by an unoriented strand (since the bilinear form
gives it the structure of a symmetrically self-dual vector space), the
unit by a bullet,
and the multiplication by trivalent vertex (which is
allowed since it is rotationally invariant by associativity of the trace).

\begin{example}
If $A$ is an associative algebra, then $x \circ y \coloneq \frac{xy+yx}{2}$
makes $A$ into a unital Jordan algebra. If $A = M_n(k)$ then
together with the normalized trace this turns $A$ into a tracial Jordan algebra.
\end{example}

As before, Definition~\ref{def:Jordan-alg} is not amenable to
generalization to algebra
objects in a ribbon category due to the repeated arguments in the
Jordan identity. Again, we polarize by substituting $x = a+b+c$ and looking
at the terms that are homogeneous of degree $(1,1,1)$. As before, in
characteristic $0$ this polarized identity is equivalent to the usual Jordan identity.
Each side of the Jordan identity contributes a sum over permutations of three letters.

\begin{definition}
A \emph{tracial Jordan algebra object} in a symmetric ribbon category
is a unital algebra
object $A$ together with a choice of symmetric self-duality 
and a rotationally invariant map $A \otimes A \rightarrow A$
which satisfies the Jordan algebra axiom
\begin{equation}\label{eq:TracialJordan}
0 = \; \TracialJordanAxiomLonga -\;  \TracialJordanAxiomLongb \; =
\TracialJordanAxiom
\end{equation}
(where $\Sym$ and $\wedge$ denote the symmetrizer and anti-symmetrizer,
and we've labelled strands by $a$,$b$,$c$,$y$ to allow easy comparison
with the Jordan axiom; the unlabelled strand is the output value)
and trace axiom
$$\mathcenter{\unitcap} = 1.$$
\end{definition}

\begin{remark} \label{rem:unitobvious}
If you compose with the unit (or trace) on any of the five outer
strings of Eq.~(\ref{eq:TracialJordan}), then it's easy to
check that the resulting identity automatically holds.
\end{remark}

Since $\tr(\cdot) \bbOne\colon A \rightarrow A$ is an idempotent, we have that
$A = \mathbf{1} \oplus X$ for some object $X$ in the idempotent completion of $\cC$. 
Moreover $X$ inherits a multiplication structure $M(x,y) = xy -
\tr(xy)\bbOne \colon X \otimes X \rightarrow X$
and a symmetric self-duality.

\begin{definition}
A \emph{traceless Jordan algebra object} in a symmetric ribbon
category is a symmetrically self-dual object $X$ with a rotationally
invariant
multiplication $M \co X \otimes X \rightarrow X$ satisfying the
identity
\begin{equation}\label{eq:TracelessJordan}
  \TracelessJordanAxioma \; + \; \TracelessJordanAxiomb \; = 0.
\end{equation}
\end{definition}

\begin{remark} \label{rem:TracelessNormalization}
The coefficient in the traceless Jordan algebra axiom \eqref{eq:TracelessJordan} implicitly
depends on our normalization $\tr(\bbOne)=1$, because the formula
for the projection onto $X$ depends on that normalization.
A non-polarized version of Eq.~\eqref{eq:TracelessJordan} is
\[
  (xy)(xx) + \langle x,y\rangle xx = x(y(xx)) + x\langle y,xx\rangle.
\]
\end{remark}

Note that $X$ retains an associative normalized duality pairing $X \otimes X \rightarrow \mathbf{1}$,
but does not have a trace since there's no unit in $X$.

\begin{proposition}\label{prop:tracial-traceless-Jordan}
The traceless part of a tracial Jordan algebra object~$A$ is a
traceless Jordan algebra object~$X$.
Conversely, if $X$ is a traceless Jordan algebra then there exists a
unique tracial Jordan algebra
structure on $\mathbf{1} \oplus X$ whose traceless part is~$X$.
\end{proposition}
\begin{proof}
For the first part, rewrite the identity on~$A$ as a sum of
projections onto $\mathbf{1}$ and~$X$,
and then take the component of~\eqref{eq:TracialJordan} where you
choose the projection onto~$X$
at every boundary point. This results in a sum over all ways of deleting or
not deleting each interior edge. Only two terms survive,
giving \eqref{eq:TracelessJordan}.

In the converse direction, first check uniqueness: we show that
$A = \mathbf{1} \oplus X$ must have
trace $\tr(a,x) = a$ and multiplication
$$(a,x) \cdot_A (b,y) = (ab + \langle x, y\rangle, ay + bx + xy).$$

The formula for the trace follows from the normalization of the trace and that
$X$ is supposed to be the traceless part of $X$.
In order to get the formula for multiplication we look component-wise. Unitality
and the fact that the multiplication on $X$ is inherited from the
multiplication on $A$ determines multiplication on
every component except for the $X \otimes X \rightarrow \mathbf{1}$ component.
Looking at the inner products, we see
$$\langle x, y \rangle_X = \langle (0,x), (0,y) \rangle_A = \tr((0,x)\cdot_A (0,y))$$
and hence the component $X \otimes X \rightarrow \mathbf{1}$ must be given by $\langle x, y\rangle$.

Now we need to check that $A$ with the above multiplication and trace is a
tracial Jordan algebra. Non-degeneracy of the bilinear form and normalization of
the trace are easy. We check the Jordan axiom in components, i.e., for
each of the four
legs we have to pick either $\mathbf{1}$ or $X$ and then check that the relation holds on that component.
This is automatic, essentially by Remark \ref{rem:unitobvious}, but is also easy to check directly. For example,
if one of the inputs at the bottom is $\mathbf{1}$ and all the others are $X$ then we need to check that the following 
expression in~$X$ vanishes:
\[\JordanLowerTermsa \; + \; 2 \; \JordanLowerTermsb + \; 0 \; + \; 2 \; \JordanLowerTermsc \]
This follows immediately from symmetry considerations.
\end{proof}

The relevance of Jordan algebras to the $F4$ family is explained by the following lemma.

\begin{lemma}[{\cite[Eqn.\ 15.15]{LieAsInvariance}}]\label{lem:cubic-Jordan}
An $F4(d)$ algebra object is a traceless Jordan algebra object.
\end{lemma}
\begin{proof}
  Apply the $F4(d)$
  relation~\eqref{eq:F4} to the lower internal edge in
  Eq.~\eqref{eq:TracelessJordan}
and combine terms which differ by applying a crossing to the symmetrizer to yield
\[-2 \; \TracelessJordanProofa \; -\frac{1}{2} \; \TracelessJordanProofb \; - \; \TracelessJordanProofc + \; \TracelessJordanAxiomb.\]
The first term in this sum vanishes because it equal to its own negative by applying a full twist at the top and
the bottom of the diagram. By easier symmetries, the second term is
zero and the
third and fourth terms cancel, so we get $0$ as desired.
\end{proof}

The Jordan algebras related to the $F4$ family satisfy an additional identity, which is the
Cayley-Hamilton identity satisfied by $3$-by-$3$ matrices. Note that Jordan algebras are 
power associative, so $x^3$ is unambiguous.

\begin{definition}
A tracial Jordan algebra is \emph{cubic Cayley-Hamilton} if, for
all $x$,
\begin{equation}\label{eq:cubicCH}
x^3  -3\tr(x) x^2 + \left(\frac{9}{2} \tr(x)^2 - \frac{3}{2} \tr(x^2) \right) x - \left(\frac{9}{2}\tr(x)^3 -\frac{9}{2} \tr(x)\tr(x^2) + \tr(x^3) \right)\bbOne = 0.
\end{equation}

A traceless Jordan algebra is \emph{cubic Cayley-Hamilton} if it satisfies the identity
\[a^3 = \frac{1}{2} \langle a, a\rangle a.\]
\end{definition}

\begin{remark} \label{rem:tracecCH}
The coefficients in the cubic Cayley-Hamilton relation depend on our choice of normalization.
That is, if $M$ is a $3$-by-$3$ matrix with the normalized trace, then the sum of the eigenvalues
of $M$ is $3\tr(M)$. The coefficients would be nicer if we had normalized so that $\tr(\bbOne) = 3$ 
so that it agreed with the matrix trace for $3$-by-$3$ matrices.

Note that if you don't make assumptions about the normalization of the
trace then Eq.~\eqref{eq:cubicCH} on its own does restrict the allowed normalizations. 
Eq.~\eqref{eq:cubicCH} forces $\tr(\bbOne)$ to be one of $1$, $2/3$, or $1/3$,
but although there are examples where each of the three traces appears, this identity 
is only related to the Cayley-Hamilton theorem if we additionally assume $\tr(\bbOne) = 1$.
\end{remark}

\begin{example}\label{examp:Jordan-matrix}
If $R$ is a Hurwitz algebra, $M_3(R)$ with its standard tracial Jordan algebra 
structure is cubic Cayley-Hamilton. 
If $R$ is a Hurwitz algebra, let $H_3(R)$ denote the Hermitian
$3$-by-$3$ matrices over $R$. (Here Hermitian means equal to
its conjugate transpose, where conjugate is the conjugate
in the sense of Hurwitz algebras.) This is a Jordan subalgebra of $M_3(R)$
and thus also cubic Cayley-Hamilton. The significance of $H_3(R)$ is
that, if the inner product on $R$ is positive-definite (as for
$R=\bbR,\bbC,\bbH,\bbO$), then the inner product on $H_3(R)$ is the
Frobenius inner product and thus also positive-definite.
\end{example}

\begin{example}\label{examp:Jordan-diag}
Suppose that $(c_1, c_2, \ldots c_n)$ are non-zero numbers summing to $1$. Then
$\mathbb{R}^n$ with component-wise multiplication and 
$$\tr(a_1, \ldots, a_n) = \sum_{i=1}^n c_i a_i$$
is a tracial Jordan algebra.

For $n = 3$ and trace given by $(1/3,1/3,1/3)$ we have that $\mathbb{R}^3$ satisfies cubic Cayley-Hamilton,
because it is a diagonal subalgebra of $M_3(\mathbb{R})$. Note that
the proper summands do not satisfy $\tr(\bbOne) = 1$
but do satisfy the Jordan algebra \eqref{eq:TracialJordan} and cubic
Cayley-Hamilton \eqref{eq:cubicCH} axioms. These give examples of the 
phenomenon mentioned in Remark \ref{rem:tracecCH}. For $n = 2$ and trace given by $(2/3, 1/3)$ 
we have that $\mathbb{R}^2$ satisfies Eq.~\eqref{eq:cubicCH} because it
is the diagonal subalgebra of $M_3(\mathbb{R})$ 
where we repeat the first entry. Finally, for $n=1$ and trace given by $(1)$ we again have a cubic
Cayley-Hamilton Jordan algebra.
\end{example}

As usual to move from algebras to algebra objects we need to
polarize the cubic Cayley-Hamilton identities.

\begin{definition}
A tracial Jordan algebra object is \emph{cubic Cayley-Hamilton} if it
satisfies
\begin{equation}\label{eq:cubicCHobject}
  \TracialThreefoldCHa \;  -3 \; \TracialThreefoldCHb \;  + \frac{9}{2}  \; \TracialThreefoldCHc \; - \frac{3}{2} \; \TracialThreefoldCHd \; - \frac{9}{2} \;\TracialThreefoldCHg\; +\frac{9}{2}  \;\TracialThreefoldCHf \; - \; \TracialThreefoldCHe\;  = 0.
\end{equation}

A traceless Jordan algebra object is \emph{cubic Cayley-Hamilton} if
it satisfies
\[\drawI + \drawH + \drawsymcrossX = \frac{1}{2} \left(\; \symcross + \twostrandid + \cupcap \; \right).\]
\end{definition}

As in Remark \ref{rem:unitobvious}, if you compose with the unit (or trace) on any of the four outer strings
of \eqref{eq:cubicCHobject}, then it's easy to check that the
resulting identity automatically holds.
Thus the traceless part of the identity is the only interesting part.

\begin{lemma}
If $A$ is a cubic Cayley-Hamilton tracial Jordan algebra object then its traceless part $X$ is cubic Cayley-Hamilton.
Conversely, if $X$ is a cubic Cayley-Hamilton traceless Jordan algebra object, then the unique Jordan algebra 
structure on $\mathbf{1} \oplus X$ is cubic Cayley-Hamilton.
\end{lemma}
\begin{proof}
The forward direction follows quickly from looking at the $X \otimes X \otimes X \rightarrow X$ component of the tracial
cubic Cayley-Hamilton identity. Only the first and fourth terms of the identity survive, but the first term contributes a sum of two terms:
$$0 = \;\TracialThreefoldCHa\; + \; \TracialThreefoldCHd \; - \frac{3}{2} \; \TracialThreefoldCHd \; = \;\TracialThreefoldCHa  \; - \frac{1}{2} \; \TracialThreefoldCHd \;.$$
Now expand the symmetrizer.

In the backwards direction the same calculation run backwards shows that the tracial cubic Cayley-Hamilton
axiom holds on the $X \otimes X \otimes X \rightarrow X$ component. We
need to also check the other components. 
This is a tedious but straightforward check: for example, if one of the inputs is $\mathbf{1}$ and the others are $X$
then only the first two terms of the Cayley-Hamilton identity contribute. For the first term, each permutation
contributes a trivalent vertex, while for the second term all but two permutations vanish and the remaining two
contribute a trivalent vertex. Thus we get $6-3\cdot 2 = 0$ copies of the trivalent vertex. The other calculations
are similar.
\end{proof}

Note that the polarized cubic Cayley-Hamilton identity in the traceless case
is exactly the key relation in defining the $F4$ family! In particular the following
Lemma follows.

\begin{lemma}
$X$ is an $F4$-algebra object in $\cC$ if and only if $X$ is a cubic Cayley-Hamilton traceless Jordan algebra object such that the lollipop vanishes.
\end{lemma}
\begin{proof}
  Given an $F4$-algebra object~$X$, the traceless cubic Cayley-Hamilton
  and vanishing lollipop are immediate, and $X$ is a traceless Jordan algebra
object by
Lemma~\ref{lem:cubic-Jordan}. For the other direction,
we have an object~$X$ satisfying
\begin{gather*}
  \drawI + \drawH + \drawsymcrossX = \frac{1}{2} \left(\; \symcross
      + \twostrandid + \cupcap \; \right)\\[5pt]
  \unknot\; = d \hspace{.5in} \symtwist\; = \;\drawcup
    \hspace{.5in} \symtwistvertex\; = \;\threevertex \hspace{.5in}
    \loopvertex\;=0
\end{gather*}
and we need to compute the values of the bigon and the triangle.
Attaching a cap to the bottom of the main relation and dividing by $2$
yields the bigon value, while attaching a trivalent vertex to the
bottom, using the value of the bigon, and solving for the triangle,
yields the triangle value.
\end{proof}

We next note that there is a one-parameter family of cubic
Cayley-Hamilton traceless Jordan algebra objects where the
lollipop is non-zero. We anticipate that this is the only such example, but since the proof is not short we will not attempt it here.

\begin{example}
Let $\cS\cM_d$ be the symmetric ribbon category generated by a rotationally invariant 3-valent vertex and a 1-valent vertex modulo the following relations:
\begin{gather*}
\threevertex  = \firstthreebox - \frac{1}{2} \biggl(\; \secondthreebox \; + \; \thirdthreebox \; + \; \fourththreebox \;\biggr)\\
\dogbone = 2 \quad \text{ and } \quad  \unknot \;= d \;
\end{gather*}
The strand in $\cS\cM_d$ with the trivalent vertex as multiplication is a cubic Cayley-Hamilton traceless Jordan algebra. Moreover, if $d \neq 2$, we have
$$\loopvertex \; = \; \frac{2-d}{2} \onevalent \; \neq 0.$$

Since the trivalent vertex can be simplified, this category is
generated by just the $1$-valent vertex. Moreover, as usual, we have a
projection $\frac{1}{2} \tr(\cdot) \bbOne\colon X \rightarrow X$ so in the
Karoubi envelope $X \cong \mathbf{1} \oplus Y$.  Conversely, if $Y$ is any
self-dual object in a symmetric ribbon category then in the additive
envelope $\bfOne \oplus Y$ we can define a cubic Cayley-Hamilton
traceless Jordan algebra by the above formula. The corresponding
tracial Jordan algebra is a direct
sum of Jordan algebras $\mathbf{1} \oplus (\mathbf{1} \oplus Y)$ where the first factor
has the trivial Jordan algebra structure but with $\tr(\bbOne) =
\frac{1}{3}$ and the second factor is a ``spin factor'' Jordan algebra
where $(a,v)(b,w) = (ab + \frac{3}{2}\langle v, w \rangle, aw + bv)$
and $\tr(a,v) = \frac{2}{3} a$. Note that the individual factors are
not cubic Cayley-Hamilton: they satisfy
the defining identity, but have $\tr(\bbOne) = 1/3$ and $\tr(\bbOne) = 2/3$ respectively 
(instead of $\tr(\bbOne)=1$ as we have required).
\end{example}

In the specialization result, in addition to the examples of Jordan algebras given above, we will also use two Jordan
superalgebras.

\begin{example}\label{examp:Jordan-superalg}
In characteristic not $2$, $3$, or $5$ there are exactly two simple cubic Cayley-Hamilton
tracial Jordan superalgebras (i.e., tracial Jordan algebra objects in the category
of super vector spaces) which are not purely even \cite[Prop. 5.1]{MR1989868} (see \cite[Prop.\ 3.2]{MR2848587} for a correction in characteristic $5$ due to McCrimmon).
The first is the $(1|2)$ dimensional algebra $J^{0|2}$ with basis $\bbOne, x, y$ with
$(ax+by)(cx+dy) =
\det\left(\begin{smallmatrix}a&b\\c&d\end{smallmatrix}\right)$.
The second is the
$(2|2)$-dimensional algebra $D_2$ spanned by
even elements $e$ and $f$ satisfying $e^2 = e$, $f^2 = f$, $ef = 0$  and $\bbOne=e+f$, and two odd
elements $x$ and $y$ satisfying $ex = \frac{1}{2} x = fx$,  $ey =
\frac{1}{2} y = fy$, $xy = e+2f = \bbOne+f$ and with $\tr(e) = \frac{2}{3}$, $\tr(f) = \frac{1}{3}$, and $\tr(x) = \tr(y) = 0$.
\end{example}

If $A$ is real cubic Cayley-Hamilton tracial Jordan algebra with positive definite bilinear form, 
then let $G_A$ denote the group of unital algebra
automorphisms of $A$ which preserve the bilinear form. (It's not difficult to see that
these groups can also be described as the groups of automorphisms of the corresponding
traceless Jordan algebra.) Since it preserves a positive definite form, $G_A$ is compact 
and the category $\Rep(G_A)$ has a $C^*$-structure.
One can similarly define a supergroup $G_A$ of automorphisms
of a tracial Jordan superalgebra,
though we do not specify the notion of
compactness there.

If $d = -2, -1, 0, 2, 5, 8, 14, 26$, take $A_d$ to be the respective following $d+1$ dimensional 
real cubic Cayley-Hamilton tracial Jordan (super)-algebras from
Examples \ref{examp:Jordan-superalg}, \ref{examp:Jordan-diag}, and~\ref{examp:Jordan-matrix}:
$J^{0|2}$, $D_2$, $\mathbb{R}$, 
$\mathbb{R}^3$, $H_3(\mathbb{R})$, $H_3(\mathbb{C})$, $H_3(\mathbb{H})$, $H_3(\mathbb{O})$.
\begin{proposition}
  For $d = -2,-1,0,2,5,8,14,26$, the (super)-group $G_{A_d}$ is, respectively,
  $(SL(2),\varepsilon=-1)$, $(SOSp(1|2),\varepsilon_A)$, trivial,
  $S_3$, $SO(3)$, $\mathit{PU}(3) \rtimes \bbZ/2\bbZ$, $\mathit{PUSp}(3)$, and $F_4$.
\end{proposition}
These groups are also listed in Table~\ref{tab:F4Table} on
page~\pageref{tab:F4Table}. In the cases when $A_d$ is a
super-algebra, we are looking at a category of super-representations
as explained in \S\ref{sec:sosp12}.
Here $\mathit{USp}(3)$ means the compact
form of $C_3$ (i.e., the unitary quaternion group), and likewise for $F_4$
we take the compact form.

\begin{proof}
  For $d=0,2$, the result follows immediately by inspection of the
  algebra.
  For $d=-2$, by
  inspection we see that the automorphisms of $J^{0|2}$ are $SL(2)$;
  this representation $V_{(1)}$ should be viewed as odd, and the
  tensor category this generates is $\mathrm{Rep}(SL(2),\varepsilon=-1)$.
  For $d=-1$, by \cite[Lemma 2.4]{MR1989868} we see that the Lie
  superalgebra of derivations of $D_2$ is $\mathfrak{osp}(1|2)$, with
  traceless part the representation $V_1^-$. From the analysis in
  \S\ref{sec:sosp12}, it follows that the Lie supergroup is $SOSp(1|2)$.

  The remaining cases are automorphisms of $H_3(R)$ for $R$ a Hurwitz
  algebra, which are described by
  Orlitzky~\cite{Orlitzky23:JordanAut}.
\end{proof}

\begin{proposition} \label{prop:F4specialization}
 If $d = 0, 2, 5, 8, 14, 26$, then $F4(d)$ 
  has a full and dominant functor to  the category $\Rep(G_{A_d})$.
\end{proposition}

\begin{proof}
  For $d \neq 2$, the lollipop is zero since the traceless part of $A$
  is a non-trivial simple
  $G_A$-module.
  A direct calculation for $d=2$ shows the lollipop is zero in that
  case too.
Since $A$ is a cubic Cayley-Hamilton tracial Jordan algebra of dimension $d+1$ with 
lollipop zero, we have a functor $\cF\co F4(d) \rightarrow \Rep(G_{A_d})$.
This functor is full and dominant by
Proposition~\ref{prop:definite-full-dominant} and
Remark~\ref{rem:full-dominant-quotient}.
\end{proof}

We expect that Proposition~\ref{prop:F4specialization} also holds at $d= -2, -1$, but there are some 
technical issues in the supergroup case, since we need both the appropriate version of 
Tannaka-Krein duality and the right analogue of compactness; these are
both used in Proposition~\ref{prop:definite-full-dominant}. Note also
that for $d=-2$ the quantization of this functor will be given
by rational functions in $s = q^{1/2}$ rather than rational functions in~$q$.

\begin{remark}
  The series of Lie algebras on the last four lines of
  Table~\ref{tab:F4Table} is well-known, particularly the construction
  of $F_4$ as $H_3(\bbO)$. There are descriptions of the other
  automorphism groups going back to
  Kalisch \cite[Thm.\ 6]{MR20551} (for $H_3(\bbR)$ and $H_3(\bbH)$)
  and Jacobson \cite{MR106936} (in very large generality), but
  the fact that $\mathrm{Aut}(H_3(\bbC))$ is not connected seems to be
  relatively little-noticed in the literature.
  The series appears prominently in
  the work of Freudenthal \cite{MR170974} and Tits \cite{MR0219578} as
  part of their work on the \textit{magic square}, and has seen many
  other developments and explanations since then. However, almost all
  references only consider the square as Lie algebras and not the
  specific Lie groups necessary to get the correct representation
  theory. This is perhaps because the natural version of the magic
  square as Lie groups is \emph{not} symmetric, as seen in the work of
  Deligne and Gross \cite{MR1952563}, who give an extended magic
  triangle, agreeing with the groups in Table~\ref{tab:exceptional} on
  the $E8$ row, Table~\ref{tab:F4Table} on the $F4$ row, and
  Table~\ref{tab:G2} on the $G2$ row. (The $E6$ and $E7$ rows in
  Deligne-Gross do not give trivalent categories in our sense.)

  Deligne-Gross make assertions about uniformity of the representation
  theory on each of their rows, but do not prove it for the $F4$ row.
Most of our work on this paper came before the recent developments
\cite{GSZ23:DiagrammaticsF4,2204.13642, 2204.11976},
but the results and ideas here in \S \ref{sec:classical-F4}
 are strongly influenced by reading \cite{GSZ23:DiagrammaticsF4}.
\end{remark}

\begin{remark}
In \cite{GSZ23:DiagrammaticsF4} they find an additional pentagon relation in the kernel of the map
$F4(26) \rightarrow \mathrm{Rep}(F_4)$, using the fact that $\dim
\mathcal{B}_5 = 15$ for $(F_4,V_{(0,0,0,1)})$. For the
$\mathfrak{c}_3$, $\mathfrak{a}_2$, and $\mathfrak{a}_1$ lines in the
$F4$ family, the computations in \S\ref{sec:independence} show
that the $5$-box space is at least $16$-dimensional, and so there is
no reason to believe this pentagon relation holds; as a result, we do
not use it in our definition of $F4(d)$.
\end{remark}

\subsection{Quantum exceptional relation for special cases}\label{sec:special-cases}

In this section, for the known examples that satisfy more than one version of
\eqref{eq:QEJac-alpha}
\begin{equation*}
\left[\; v^{-3} \;
\drawcrossX
\;+ v \;
\drawI
\; -v^{-1} \;
 \drawH
\;\right]
 + \alpha
\left[\; \braidcross \;
 + v^{4}\;
\cupcap
\; + v^{-4} \;
 \twostrandid \;
 \right] = 0,
\end{equation*}
we collect the values of $v$ and $\alpha$ for which they satisfy the relation.
Recall that there's a symmetry $(v,\alpha) \mapsto (-v,-\alpha)$ that
give the same relation, and we only record one version of each
relation.
We will also say that \eqref{eq:QEJac-alpha} holds with $\alpha=\infty$ if
\begin{equation*}
\braidcross \;
 + v^{4}\;
\cupcap
\; + v^{-4} \;
 \twostrandid \;= 0.
 \end{equation*}

 Each of the following lemmas is proved the same way. We compute the
 value of the trivalent twist; this gives the value of $-v^6$. This
 leaves three options (up to sign) for $v$. Then we take a relation
 $X$ between the six diagrams and compute $X+v^{4} R(X)+ v^{-4}R^2(X)$
 to get relations of the correct form. We also give the formula for
 $w$ when $v$ is not an $8$th or $12$th root of unity so the change of
 variables
 is non-singular. (The resulting values of $v$ and $w$ are shown in
 Figure~\ref{fig:exceptions}.)
 Finally, we record the eigenvalues for the
 negative crossing on the $4$-box space in each case, using
 representation theory. The details of the calculation can be found in
 the \texttt{arXiv} sources of this article in the Mathematica
 notebook
 $$\text{{\small\tt
     arxiv-code/CalculatingExceptionalRelationsForSpecialCases.nb}}$$

 \begin{lemma}\label{lem:SO3q}
$(SO(3),V_{(2)})_q$ satisfies three instances of \eqref{eq:QEJac-alpha} with $\alpha = -b \frac{\Psi_1(v)}{\Psi_{12}(v)} $ for $v$ a cube root of $q$.  Under the change of variables we see that this value of $\alpha$ corresponds to $w = \pm v^4$ or $w = \pm v^{-3}$.  The eigenvalues of the negative crossing on the $4$-box space are $(q^4,-q^2,q^{-2})$.
\end{lemma}

For $q \neq \pm i$ the previous lemma can be proved by starting with
either of the two defining relations of $SO(3)_q$.
However, when $q = \pm i$ (i.e., $v$
is a primitive $4$th or $12$th root of unity)
the first defining relation only yields two non-zero relations, but using the second
relation still yields the same three relations.

\begin{lemma} \label{lem:SOSp}
$(SOSp(1|2), V_1^-)$ satisfies three instances of \eqref{eq:QEJac-alpha}: 
\begin{itemize}
\item $v= -i$ with $\alpha = 2 i b$;
\item $v = e^{2 \pi i/12}$ with $\alpha = \frac{i b}{2}$; and
\item $v=e^{5\cdot2 \pi i/12}$ with $\alpha = \frac{i b}{2}$.
\end{itemize} 
The eigenvalues of the negative crossing on the $4$-box space are $(1,1,-1)$.
\end{lemma}

\begin{lemma}\label{lem:golden}
The golden categories are quotients of $SO(3)_q$ at $q$ a fifth or tenth root of unity.  The three relations for $SO(3)_q$ given above correspond to $v = q^2$, to $v = q^2 \omega$, or to $v = q^2 \omega^2$ for a primitive cube root $\omega$.  In addition, for the golden categories, there's a $2$-dimensional space of relations with $v = q^2$, i.e., there's such a relation for every $\alpha$ including $0$ and $\infty$.  The eigenvalues of the negative crossing on the $4$-box space are $(q^4, -q^2)$.
 \end{lemma}
 \begin{proof}
   Plug the golden category formula for $\smallfig{\drawI}$ into the
   $SO(3)_q$ braiding relation to get a formula for the crossing in
   terms of the Temperley-Lieb diagrams. This gives a relation with
   $\alpha=\infty$ and $v=q^2$. Linear combinations of this relation
   with the $v = q^2$ relation from the previous lemma gives relations
   with $v = q^2$ and $\alpha$ arbitrary.

   The eigenvectors are just the cupcap diagram and $\smallfig{\drawI}$
   (i.e., the projections onto the trivial and standard objects), so
   they are just the twist factors $(-q^2, q^4)$.
 \end{proof}

\begin{lemma}\label{lem:G2q}
$(G_2, V_{(1,0)})_q$ satisfies two instances of
\eqref{eq:QEJac-alpha}, namely $\alpha = - \frac{b}{\Psi_3(q)
  \Psi_8(q)} (\omega^2 q-\omega q^{-1})$ and $v = \omega q$ for
$\omega$ a primitive cube root of unity.  This $\alpha$ corresponds to
$w = \pm q^{-1} = \pm \omega v^{-1}$ or $w = \pm \omega q^2 = \pm \omega^2 v^2$.  The eigenvalues of the negative crossing on the $4$-box space are $(-q^6, q^{12}, q^{-2}, -1)$.
\end{lemma}

\begin{lemma}
  $\mathit{Sym}(t)$ satisfies two instances of \eqref{eq:QEJac-alpha}, namely
\begin{itemize}
\item $v = e^{2 \pi i/12}$ with $\alpha = \frac{i b}{2}$;
\item $v=e^{5\cdot2 \pi i/12}$ with $\alpha = \frac{i b}{2}$.
\end{itemize}
(Recall this is singular for the change of variables.)
The
eigenvalues of the negative crossing on the $4$-box space are
$(1,1,-1,1)$.
\end{lemma}

When $t = 0$ the relations for $\mathit{Sym}(0)$ 
includes two of the three relations that 
$(SOSp(1|2),\allowbreak V_1^-)$ satisfies in Lemma \ref{lem:SOSp}. See 
Remark~\ref{rem:Sym0SOSp}.

\begin{lemma}\label{lem:Rep-S3}
$\Rep(S_3)$ satisfies three instances of \eqref{eq:QEJac-alpha}, namely
\begin{itemize}
\item $v = e^{2 \pi i/12}$ with $\alpha = -i b$;
\item $v = e^{5 \cdot 2 \pi i/12}$ with $\alpha = - i b$; and
\item $v = i$ with $\alpha = -\frac{i}{2} b$.
\end{itemize}
The eigenvalues of the negative crossing on the $4$-box space are $(1,1,-1)$.
\end{lemma}

\subsection{4-dimensional 4-box space} \label{sec:4dim4box}

In this section we give a partial classification of trivalent ribbon
categories with $\dim \cC(4) \leq 4$, extending \S\ref{sec:3dim4box}.
We will need this in our results about roots of unity.

\begin{lemma} \label{lem:linindorfib}
If $\cC$ is a trivalent ribbon category then either $\cC$ is one of the
Golden Categories or the following three diagrams are linearly independent:
\[
  \cupcap\;,\qquad\twostrandid\;,
    \qquad\drawI\;.
  \]
\end{lemma}
(Compare to Lemma~\ref{lem:betazero}, which had a crossing in place of
the last diagram.)
\begin{proof}
If there is a linear relation between these diagrams with zero
coefficient on the third diagram, then
composing on the top or right with a trivalent vertex shows the trivalent vertex
is zero, which is a contradiction. So WLOG we have a relation of the form
\[
\drawI\; = x \; \twostrandid\; + y \; \cupcap\;.
\]

Attaching a cap on the top and on the right yields $0 = x + yd$ and $b = x d + y$.
Attaching a trivalent vertex on the top says $b = x$. Solving, we see $y = -b/d$ and thus
$b = b d -b/d$. Thus $d$ satisfies $d^2 = d+1$ and we have
\[
\drawI\; = b \; \twostrandid\; -\frac{b}{d} \; \cupcap\;.
\]
Thus the pivotal subcategory generated by the trivalent vertex is the Golden Category.
Now, composing with the crossing on top says that the crossing simplifies, and so the
whole category is the Golden Category. In particular
\[\; \braidcross \; = z \; \twostrandid \; + w \; \cupcap\; \]
and a standard argument (rotate 90-degrees and compose)
shows the coefficients must be given by the usual Kauffman bracket
formula ($z = w^{-1}$ and $d = -z^2 -z^{-2}$) 
and thus is one of the standard braidings on one of the Golden Categories.
\end{proof}
 
We next have the following lemma, which, when combined with results
from \S\ref{sec:special-cases}, is a more precise version of
Lemma~\ref{lem:IequalsH}.

\begin{lemma} \label{lem:IHimpliesv12}
Suppose that $\cC$ is a trivalent ribbon category with  $d \ne 1$ and
$$\drawI\;- \; \drawH - \frac{b}{d-1} \left(\; \twostrandid - \cupcap \;\right) = 0.$$
Then one of the following occurs:
\begin{itemize}
\item $\cC$ satisfies \eqref{eq:QEJac-alpha} for $v$ any primitive $12$th root of unity 
(by symmetry, this yields two distinct relations, not four), $\cC$ satisfies no other instances 
of \eqref{eq:QEJac-alpha}, and the crossing does not simplify as a sum
of the above diagrams;
\item $\cC$ is $\Rep(S_3)$; or
\item
$\cC$ is a quotient of $SO(3)_q$ for some $q$ with the usual braiding (including $SOSp(1|2)$).
\end{itemize}
\end{lemma}

\begin{proof}
We will follow the notation and techniques of Proposition~\ref{prop:Jacobi}.
We apply the tetrahedral symmetry operator~$R$ to the given relation to get
\begin{align*}
-u^{-1} \drawH + u^{-1} \drawcrossX -\frac{b}{d-1} \left(\; \braidcross - u^{-2} \twostrandid \; \right) &= 0\\
u^{-2} \drawcrossX + u^{-1} \drawI -\frac{b}{d-1} \left(\; \cupcap -u^{-2} \braidcross \; \right) &= 0.
\end{align*}
(Recall that the trivalent twist rescales by $-u$, which is $-v^6$ in
instances of \eqref{eq:QEJac-alpha}.)
Multiply the first equation by $-u^{-1}$ and then add to the second relation to make the $\smallfig{\drawcrossX}$ terms vanish. We obtain
$$u^{-1} \drawI + u^{-2} \drawH - \frac{b}{d-1} \left(\;\cupcap -u^{-2} \braidcross -u^{-1}\braidcross + u^{-3}\twostrandid  \; \right) = 0.$$
Unless $u$ is $-1$, this says that $\smallfig{\braidcross}$ simplifies.  So we split into two cases.

If $\smallfig{\braidcross}$ simplifies, then
$\cC$ agrees with $SO(3)_q$ from Definition~\ref{def:SO3q} as a
pivotal category for $q^2+q^{-2}=d-1$. (Without the assumption that
the crossing simplifies, we would only have that the subcategory of
graphs with no crossings agrees with $SO(3)_q$.) So we need only solve for all braidings for
$SO(3)_q$, which yields only the usual braiding (with $q=v^3$) or when $d=2$ the
additional braiding giving $\Rep(S_3)$ (see \cite[Example 8.6]{MR3624901}).

If $\smallfig{\braidcross}$ does not simplify,
then $u = -1$. To derive an instance of \eqref{eq:QEJac-alpha}, we
pick a value of $v^{-4} = \lambda \in \{1, \omega, \omega^2\}$ and
sum the above three relations. When $\lambda=1$, the sum is equal to zero.
When $\lambda\in \{\omega, \omega^2\}$ then $v$ is a primitive $12$th root of
unity as desired.

It remains to show in this last case that $\mathcal{C}$ does not
satisfy another instance of \eqref{eq:QEJac-alpha}, so suppose it
does. That would give us
a three-dimensional space of relations between the six diagrams
  \[
  \cupcap\;,\qquad\twostrandid\;,\qquad\braidcross\;,
    \qquad\drawI\;,\qquad\drawH\;,\qquad\drawcrossX\;.
  \]
(The two prior relations had two different eigenvalues under the action of $R$, so
a third relation cannot be a linear combination of the
two.)
Since the crossing does not simplify,
the three diagrams
\[
\cupcap\;,\qquad\twostrandid\;,\qquad\drawI
\]
are linearly dependent. But this contradicts Lemma~\ref{lem:linindorfib}, since in the Golden Category the crossing simplifies.
\end{proof}

This argument interacts in an interesting way with Lemma~\ref{lem:IequalsH}, which was
proved in a similar way. It's possible for $\cC$ to satisfy \eqref{eq:QEJac-alpha} for
$v = \pm i$
and also satisfy the given relation, but then it follows that $\cC$ also satisfies two additional
\eqref{eq:QEJac-alpha} relations for $v$ primitive $12$th roots of unity. Indeed this happens for
$\Rep(S_3)$ and $SOSp(1|2)$. (The reason that the additional relation
has $v = \pm i$ is that 
$-v^6$ is the trivalent twist.)

Along with Lemma \ref{lem:Dim3impliesIequalsH},
 we have the following corollary.

\begin{corollary} \label{cor:BestIH}
If $\cC$ is a trivalent ribbon category, and the four diagrams
  \[
  \cupcap\;,\qquad\twostrandid\;,
    \qquad\drawI\;,\qquad\drawH
   \]
   are linearly dependent, then one of the following occurs:
\begin{itemize}
\item $\cC$ satisfies two instances of \eqref{eq:QEJac-alpha} for $v$ a primitive $12$th root of unity and no other instances
of \eqref{eq:QEJac-alpha};
\item $\cC$ is $\Rep(S_3)$; or 
\item $\cC$ is a quotient of $SO(3)_q$ with the usual braiding (including $SOSp(1|2)$).
\end{itemize}
\end{corollary}

\begin{remark} \label{rem:SymtBadCase}
Note that if $\cC$ satisfies even one instance of
\eqref{eq:QEJac-alpha} for $v$ a primitive $12$th root of unity,
then \eqref{eq:QECross-alpha'}
guarantees a relation between
\[
  \cupcap\;,\qquad\twostrandid\;,
    \qquad\drawI\;,\qquad\drawH
   \]
and hence the corollary applies and either $\cC$ satisfies two
instances of \eqref{eq:QEJac-alpha} for $v$ a primitive $12$th root of unity and no
other instances of \eqref{eq:QEJac-alpha}, or $\cC$ is $\Rep(S_3)$, or $SOSp(1|2)$.
The main example we know of in the first case is $\mathit{Sym}(t)$.
But note this has $\dim \mathit{Sym}(t)_4 = 4 < 5.$ 
We believe that with the stronger assumption $\dim \cC(4) \leq 4$ then 
$\mathit{Sym}(t)$ is the only example at $v$ a primitive $12$th root of unity 
but since the proof is long we will not give it here.
On the other hand, in the first case of the Corollary if
$\dim \cC(4) = 5$
we do not know how to get any control over the case where the five diagrams
\[ \twostrandid\;,\qquad \cupcap \; ,\qquad \drawI \; ,\qquad \braidcross \; ,\qquad \invbraidcross \]
are linearly independent and so there is no simple relation relating
the two crossings.
\end{remark}

\begin{proposition} \label{prop:dim4}
If $\cC$ is a trivalent ribbon category and the six diagrams  
  \[
  \cupcap\;,\qquad\twostrandid\;,
    \qquad\drawI\;,\qquad\drawH\;,\qquad\braidcross\;,\qquad\fourgon
   \]
span a space of dimension $4$ or smaller, then one of the following occurs:
\begin{itemize}
\item $\cC$ satisfies two instances of \eqref{eq:QEJac-alpha} for $v$ a primitive $12$th root of unity and no other instances
of \eqref{eq:QEJac-alpha};
\item $\cC$ is $\Rep(S_3)$;
\item $\cC$ is a quotient of $SO(3)_q$ with the usual braiding (including $SOSp(1|2)$); or
\item $\cC$ is a quotient of $(G_2)_q$ with the usual braiding.
\end{itemize}
\end{proposition}
\begin{proof}
If the first four diagrams are linearly dependent then we're in one of the
first three cases by the previous lemma.  Otherwise, we have
relations simplifying both the
crossing and the square in terms of the first four diagrams.  This means that
$\cC$ is generated by the trivalent vertex and $\dim \cC(4) \leq 4$, and thus
by a modification of \cite[\S 5]{MR3624901} we see that $\cC$ is a
quotient of $(G_2)_q$ as a pivotal category. By \cite[\S 2.3]{MR1265145} 
(see also \cite[\S 6]{MR3624901}), after possibly replacing $q$ with $q^{-1}$
(which doesn't change the pivotal category) we may assume the braiding on $(G_2)_q$
is the standard one.

Specifically,
in \cite{MR3624901} we only considered
non-degenerate
trivalent categories, but this can be avoided as follows.  The
argument that the pentagon reduces still applies, so all small faces
reduce, and non-elliptic diagrams span all $\cC(k)$.  If the $10$
non-elliptic diagrams in $\cC(5)$ are independent then Kuperberg
\cite{MR1265145} says that $\cC$ is a quotient of $(G_2)_q$.  If
there's a relation between these ten diagrams, the relation must be
negligible.  By \cite{MR3624901} either there's a single relation and
the category is a quotient of $(G_2)_{\zeta_{10}}$, or the
non-degenerate quotient of $\cC$ is an $ABA$ category which
contradicts $\cC$ being braided.
\end{proof}

\begin{corollary} \label{cor:DiagramsSpan}
If $\cC$ is a trivalent ribbon category with $\dim \cC(4) \leq 5$
which satisfies \eqref{eq:QEJac-alpha} for some $v$ which is not a
primitive $12$th root of unity, then the six diagrams
  \[
  \cupcap\;,\qquad\twostrandid\;,
    \qquad\drawI\;,\qquad\drawH\;,\qquad\braidcross\;,\qquad\fourgon
   \]
span $\cC(4)$.
\end{corollary}
\begin{proof}
If the given diagrams do not span $\cC(4)$ then they span a space of dimension $4$ or fewer, and Prop. \ref{prop:dim4} applies. Since $\cC$ satisfies \eqref{eq:QEJac-alpha} for some $v$ which is not a primitive $12$th root of unity, the first case is impossible, and in all the other cases the given diagrams span.
\end{proof}

We can summarize much of this as follows.
\begin{corollary} \label{cor:small-dims-summary}
  Let $\cC$ be a trivalent ribbon category with small 4-box space.
  Then we have the following as ribbon categories (with their usual braidings).
  \begin{itemize}
  \item If $\dim \cC(4) = 2$, then $\cC$ is the golden category.
  \item If $\dim \cC(4) = 3$, then $\cC$ is $\mathsf{Rep}(S_3)$ or a
    quotient of $SO(3)_q$ (including $SOSp(1|2)$).
  \item If $\dim \cC(4) = 4$, then $\cC$ is a quotient of $\mathit{Sym}(t)$,
    a quotient of $(G_2)_q$, or another category with $v$ a primitive $12$th root of
    unity.
  \end{itemize}
\end{corollary}

\begin{proof}
  If $\dim \cC(4) = 2$, then the three diagrams in
  Lemma~\ref{lem:linindorfib} are dependent and $\cC$ is the golden
  category (and $v$ is a primitive $5$th or $10$th root of unity).

  If $\dim \cC(4) = 3$, then the four diagrams in
  Corollary~\ref{cor:BestIH} are linearly dependent, so we get one of
  the two desired possibilities or $v$ is a primitive $12$th root of
  unity. But in the latter case, those four diagrams already span
  $\cC(4)$ (since we are not in the golden category), so the crossing
  simplifies, contradicting Lemma~\ref{lem:IHimpliesv12}.

  Similarly if $\dim \cC(4) = 4$, we can apply
  Proposition~\ref{prop:dim4}.
\end{proof}

\subsection{Some results when \texorpdfstring{$v$}{v} is a root of
  unity}
Let $\zeta_n = e^{2\pi i/n}$ be a primitive $n$th root of unity.

\subsubsection{The case \texorpdfstring{$v = \pm 1$}{v = ± 1}}

Throughout this section we assume WLOG $v=1$.  The results in this section are closely related to those in \cite{2204.13642}, but we do not assume that the category is symmetric.

Recall $t = \frac{b-[5]\alpha}{\{2\}} = \frac{b}{2}$.  We also have that $0= [5]b = -\alpha (d+\{8\})$, so we split into two cases based on whether $\alpha = 0$.

\begin{lemma}
There are no trivalent ribbon categories $\cC$ satisfying
\eqref{eq:QEJac-alpha} with $v=1$ and $\alpha \neq 0$.
\end{lemma}
\begin{proof}
By
Proposition~\ref{prop:square-crossing}, \eqref{eq:QESq-alpha'} holds
and says that
$$\braidcross\; + \;\cupcap\; + \;\twostrandid\; = 0.$$
But $d+\{8\}=0$ so $d=-2$, and
this case is excluded by Lemma~\ref{lem:betazero}.
\end{proof}

\begin{lemma} \label{lem:SO3Case}
If $v=1$ and $\alpha = 0$ and the four diagrams
\[
   \cupcap\;,\qquad\twostrandid\;, \qquad\drawI\;,\qquad\drawH \]
are linearly dependent, then $\cC$ is $SO(3)$, and 
in particular satisfies the classical Jacobi, anti\hyp symmetry,
and exceptional relations.
\end{lemma}
\begin{proof}
By Corollary \ref{cor:BestIH}, $\cC$ must be a quotient
of some $SO(3)_q$ or is $\Rep(S_3)$. But $\Rep(S_3)$ is excluded by
Lemma~\ref{lem:Rep-S3} at $v=1$. Then
 we must have $q = v^3 = 1$, and so
 it is a quotient of classical $SO(3)$, which satisfies all the classical exceptional
 relations. Since $SO(3)$ is semisimple it has no non-trivial quotients other than itself.
\end{proof}

\begin{proposition}
If $\cC$ is trivalent ribbon category satisfying \eqref{eq:QEJac-alpha} with $v=1$, $\alpha =0$, and $\dim \cC(4) \leq 5$,
then $\cC$ is symmetric and satisfies the Jacobi and anti-symmetry relations.
\end{proposition}
\begin{proof}
  Anti-symmetry is immediate from $v=1$. (See Remark~\ref{rem:QEJac-alpha}.)
  
 \eqref{eq:QEJac-alpha} says that the Jacobi relation holds using an
 overcrossing.  Composing this relation with the undercrossing shows
 that the Jacobi relation also holds using an undercrossing.

Since $\dim \cC(4) \leq 5$, the following six diagrams are linearly dependent:
 \[
  \cupcap\;,\qquad\twostrandid\;,
    \qquad\drawI\;,\qquad\drawH\;,\qquad\braidcross\;,\qquad \invbraidcross
   \]
   More specifically, there is a relation between them that is an
   eigenvector of $90^\circ$ rotation with eigenvalue $\pm 1$.
   By Lemma~\ref{lem:SO3Case} we may assume that the first four terms
   are linearly independent, so the coefficient
   on the positive crossing can be normalized to~$1$:
\begin{equation}\label{eq:v1-rel}
  \; \invbraidcross \pm \braidcross + x\left(\drawI \pm \drawH\; \right) + y \left(\; \twostrandid\pm\cupcap\; \right)=0
\end{equation}

If the $+$ sign holds,
we can attach an $\smallfig{\drawH}$ to the top and use the Jacobi relations to get
\[ x \fourgon +\left(\frac{bx}{2} -2\right) \drawI + (2+y)\drawH + b y \cupcap = 0.\]
Subtracting this relation from its rotation we get
\[\left(\frac{bx}{2} -y -4\right) \left(\; \drawI-\drawH\; \right) -b y \left(\; \twostrandid - \cupcap\; \right) = 0.\]
By our linear independence assumption, this relation must be trivial, so we see that $y=0$ and $x = 8/b$.  But capping off our original relation we get $0 = 2+bx+(d+1)y = 2+8 = 10$, which is a contradiction.  So the relation cannot have rotational eigenvalue $+1$.

Now we turn to a $-$ sign in Eq.~(\ref{eq:v1-rel}) and apply the same argument.
We get
\[-x \fourgon + \left(\frac{xb}{2} -y\right) \drawI + y \drawH - b y
  \cupcap\]
which anti-symmetrizes to
\[\left(\frac{xb}{2}-y\right) \left(\; \drawI - \drawH\; \right) + b y \left(\; \twostrandid - \cupcap\; \right) = 0\]
which implies that $x = 0$ and $y=0$.

Thus we see that the crossing must be symmetric, and in particular
\eqref{eq:QEJac-alpha} becomes exactly the usual Jacobi relation.
\end{proof}

The next proposition is essentially due to Cvitanovi\'c \cite[Chapter
17]{MR2418111} and Vogel \cite[\S7.5]{MR2769234}. (Both authors
had much earlier statements.)

\begin{proposition} \label{prop:classical-identification-sym}
If $\cC$ is a trivalent symmetric ribbon category satisfying the Jacobi and anti-symmetry relations and $\dim \cC(4) \leq 5$, then $\cC$ satisfies the relations of the classical exceptional series.
\end{proposition}
\begin{proof}
Since $\dim \cC(4) \leq 5$ there must be a linear dependence between the six diagrams
 \[
  \cupcap\;,\qquad\twostrandid\;,
    \qquad\drawI\;,\qquad\drawH\;,\qquad\symcross\;,\qquad \fourgon\;.
   \]
We claim that WLOG the coefficient on the square is non-zero and hence can be
normalized to $1$. Indeed otherwise there's a relation between the other five diagrams,
which by Lemma~\ref{lem:SO3Case} can be assumed to have non-zero coefficient on
the crossing.
By Lemma~\ref{lem:betazero} this relation must also have a non-zero
coefficient on $\smallfig{\drawI}$ or $\smallfig{\drawH}$, and thus if
we use this formula to simplify the crossing in the Jacobi relation we get
a linear relation with a non-zero coefficient on the square.

Moreover, since the square is rotationally invariant, there must be such a relation
which is a rotational eigenvector with eigenvalue $1$:
$$0 = \fourgon + x \symcross + y \left(\; \drawI + \drawH \; \right) + z\left(\; \twostrandid + \cupcap \; \right).$$
Multiplying this by a crossing and using the Jacobi relation we get
\begin{align*}
0 &= \fourgon - \frac{b}{2} \drawI + x \twostrandid + y \left(-\drawI + \drawH -\drawI \right) + z\left(\; \symcross + \cupcap \; \right) 
\\ &= \fourgon + z \symcross + \left(-2y -\frac{b}{2} \right) \drawI +y \drawH +x \twostrandid + z\cupcap. 
\end{align*}

If these two relations are distinct, then Proposition \ref{prop:dim4} applies, and the only case with $v=1$ is classical $SO(3)$ (note that there is no trivalent category $(G_2)_{\zeta_3}$), which satisfies the classical exceptional relation.  Otherwise, we have $x=z$ and $y = -\frac{b}{6}$.  

Capping off this relation gives $0 = b^2 + x -\frac{b^2}{6}+x(1+d) =  (2+d)x +\frac{5 b^2}{6}$, so we have $x = -\frac{5 b^2}{6(2+d)}$.  This gives the classical exceptional square relation.
\end{proof}

Combining the previous results we have the following corollary.

\begin{corollary} \label{cor:classical-identification}
If $\cC$ is trivalent ribbon category satisfying
\eqref{eq:QEJac-alpha} with $v=1$, and in addition
$\dim \cC(4) \leq 5$, then $\cC$ satisfies the relations of the
classical exceptional series.
\end{corollary}

\begin{remark}
One can think of the classical exceptional series as a limit of the quantum exceptional series where you set $w = v^\lambda$ and then take the limit as $v$ goes to $1$.
\end{remark}

\subsubsection{The case \texorpdfstring{$v$}{v} is a primitive third or sixth root of unity}

\begin{proposition} \label{prop:G2-recognition}
If $\cC$ is trivalent ribbon category satisfying
\eqref{eq:QEJac-alpha} with $v$ a primitive third or sixth root of unity and $\dim \cC(4) \leq 5$,
 then after possibly rescaling the self-duality,
$X$ is a $G2(d)$ algebra in $\cC$.
\end{proposition}
\begin{proof}
WLOG we take $v = \zeta_3$.
Since $v$ is not a $10$th root of unity, $\alpha$ is non-zero.  Also $b+[3]\alpha = b \neq 0$.  Equation \eqref{eq:QECross-alpha'} says that the braiding is symmetric.  

Then \eqref{eq:QESq-alpha'} becomes
$$\symcross+\frac{\sqrt{-3}}{2\alpha} \left(\; \drawI + \drawH \; \right) -\frac{1}{2} \left(\; \twostrandid + \cupcap \; \right) = 0.$$

If we rescale the self-duality to make $b= 1-d$, we then have
$\alpha = \frac{[5] b}{d+\{8\}} = -\sqrt{-3}$ and get exactly the
classical $G2$ relation from \S \ref{sec:classical-G2F4}.
\end{proof}

In particular, it follows from \S \ref{sec:classical-G2F4} that $\cC$ has a quotient that is $SO(3)$ or $G_2$.

\begin{remark}
  To see the classical $G2$ relations as a limit of
  $\mathsf{QExc}_{v,w}$ as $v \to \zeta_3$,
  we need to rescale~$b$. If you rescale
  $b$ appropriately, then you can set $v = \zeta_3 w^{\frac{1-d}{6}}$
  and take the limit as
  $w \rightarrow 1$ to recover the classical $G2$ relations. See also
  Proposition~\ref{prop:quant-g2-spec}.
\end{remark}

\subsubsection{The case \texorpdfstring{$v$}{v} is a primitive fourth
  root of unity}
\label{sec:v=i}

Since $v$ is not a $10$th root of unity, $\alpha$ is nonzero.  We have that $d+2 = \frac{-2ib}{\alpha}$.
We split into cases based on whether $b+[3]\alpha$ vanishes.  In the case where $b+[3]\alpha = 0$,
we make the additional assumption that the $4$-box space is non-degenerate.

\begin{proposition} \label{prop:iPart1}
If $\cC$ is a trivalent ribbon category satisfying \eqref{eq:QEJac-alpha} with $v = i$
and $b+[3]\alpha \neq 0$, then after possibly rescaling the self-duality,
$X$ is an $F4(d)$-algebra object in $\cC$.
In particular, $\cC$ has a quotient that is one of the examples
from Table~\ref{tab:F4Table}.
\end{proposition}
\begin{proof}
  Since we assume $b+[3]\alpha \neq 0$, by \eqref{eq:QECross-alpha'} the crossing is
  symmetric.
  \eqref{eq:QEJac-alpha} becomes
$$\drawsymcrossX + \drawI + \drawH = \alpha i \left(\;\symcross +
  \twostrandid + \cupcap \;\right).$$
If we rescale the self-duality so that $\alpha i = 1/2$, then
$b = (d+2)/4$ and we have exactly the classical $F4$ relation.

Proposition~\ref{prop:F4consistency} then says $d$ is one of the values in
Table~\ref{tab:F4Table} or $d\in\{2,0,-2\}$. The value $d=0$ is
ruled out in the definition of a trivalent category, and at $d=-2$ the
trivalent vertex vanishes. At $d=2$, there is a quotient to
$\Rep(S_3)$, but this is excluded because it satisfies $b+[3]\alpha = 0$.
\end{proof}

Since we do not know whether the classical $F4$ relation is sufficient
for evaluating all closed diagrams, we do not know whether the
quotient above is the non-degenerate quotient or whether we need to
add more relations.

\begin{remark}
  To see the classical $F4$ relations as a limit of
  $\mathsf{QExc}_{v,w}$ as $v \to i$,
  we need to rescale~$b$. If you rescale
  $b$ appropriately, then you can set $w = i v^{\frac{2-d}{12}} = i
  v^{-\eta}$ and take the limit as $v \rightarrow 1$ to recover the
  classical $F4$ relations. See also
  Proposition~\ref{prop:quant-f4-spec}.
\end{remark}

\begin{lemma} \label{lem:lemmai}
  If $\cC$ is a trivalent ribbon category satisfying
  \eqref{eq:QEJac-alpha} with $v=i$ and $b+[3]\alpha = 0$, then
  $\alpha = -\frac{bi}{2}$, $d = 2$, $t=0$, and the following
  relations hold:
  \begin{gather*}
    \drawcrossX + \drawI + \drawH = \frac{b}{2} \biggl(\;\braidcross + \twostrandid + \cupcap \;\biggr)\\
    \fourgon = -\frac{b}{2} \left(\drawI + \drawH \;\right) + \frac{b^2}{2} \left(\;\twostrandid + \cupcap \;\right)
\end{gather*}
  If in addition $\mathcal{C}$ is symmetric, then after rescaling
  the self-duality $X$ is an $F4(2)$-algebra object in~$\mathcal{C}$.
\end{lemma}
\begin{proof}
  This follows immediately from specializing Eqs.\
  \eqref{eq:QEJac-alpha} and \eqref{eq:QESq-alpha'}, respectively.
\end{proof}

\begin{lemma}
If $\cC$ is a non-degenerate trivalent ribbon category with $\dim \cC(4) \le 5$ satisfying \eqref{eq:QEJac-alpha} with $v=i$ and $b+[3]\alpha = 0$, then 
$\cC$ is $\Rep(S_3)$.
\end{lemma}
\begin{proof}
By Corollary~\ref{cor:DiagramsSpan}, the diagrams
  \[
  \cupcap\;,\qquad\twostrandid\;,
    \qquad\drawI\;,\qquad\drawH\;,\qquad\braidcross\;,\qquad\fourgon
   \]
   span $\cC(4)$.   Thus 
 we can check whether an element is negligible by pairing it with the
 above $6$ diagrams. In particular it is easy to verify that 
 \[
 \drawI\; -\; \drawH \; -\frac{b}{d-1} \; \twostrandid \; + \frac{b}{d-1} \cupcap
 \]
is negligible, and hence by non-degeneracy must be zero. Now Corollary \ref{cor:BestIH}
says that $\cC$ must be $\Rep(S_3)$ or a quotient of $SO(3)_q$ with the usual braiding.
By Lemma \ref{lem:SO3q}, the latter only satisfies \eqref{eq:QEJac-alpha} with $v=i$ when $q=i$ 
(so $\cC$ is $SOSp(1|2)$), in which case $b+[3]\alpha = -3b \neq 0$.
\end{proof}

There are some interesting degenerate examples with $v=i$ and $b+[3]\alpha = 0$, which we will return to in future work.

\subsubsection{Symmetric trivalent ribbon categories}
We can combine previous results to get the following, essentially due
to Cvitanovi\'c \cite{MR2418111}.

\begin{proposition} \label{prop:sym-main-thm} Suppose that
  $(\cC, X, \tau)$ is a symmetric trivalent ribbon category with
  $\dim{\cC(4)} \leq 5$ and $\dim X = d$. Then after rescaling the
  self-duality appropriately one of the following holds for the
  symmetrically self-dual object $X$ with multiplication defined
  by~$\tau$:
\begin{itemize}
\item $X$ is an $\mathit{Sym}$-algebra object, and hence $\bfOne\oplus X$ is a normalized commutative Frobenius algebra object in the additive completion of $\cC$.
\item $X$ is a $G2$-algebra object, and hence $\bfOne \oplus X$ is a Hurwitz algebra object in the additive completion of~$\cC$.
\item $X$ is an $F4$-algebra object, and hence $\bfOne \oplus X$ is a cubic Cayley-Hamilton Jordan algebra object in the additive completion of~$\cC$.
\item $X$ is a Lie algebra object satisfying the classical Exceptional relations.
\end{itemize}
\end{proposition}
\begin{proof}
  By Proposition~\ref{prop:Jacobi} and Lemma~\ref{lem:betazero}, we
  satisfy \eqref{eq:QEJac-alpha} for some $v,\alpha$.
  Since $\cC$ is symmetric, looking at the trivalent twist we see
  $v^6 = \pm 1$, and hence $v$ is a $12$th root of unity. If
  $v = \pm 1$ then $X$ is a Lie algebra object satisfying the
  classical Exceptional relations by
  Corollary~\ref{cor:classical-identification}.
  If $v$ is a third or sixth root
  of unity then $X$ is a $G2$-algebra object by Proposition
  \ref{prop:G2-recognition}. If $v$ is a primitive $4$th root of unity
  and $b+[3]\alpha \neq 0$ then $X$ is an $F4$-algebra object by
  Proposition~\ref{prop:iPart1}.
  If $v$ is a primitive $4$th root of
  unity and $b+[3]\alpha = 0$ then by Lemma~\ref{lem:lemmai}, $X$ is
  again an $F4$-algebra object (using for the first time that
  $\mathcal{C}$ is symmetric in addition to $v^{12}=1$).

The only remaining case is where $v$ is a primitive $12$th root of unity. Then \eqref{eq:QECross-alpha'} implies
\[ \drawH \; - \; \drawI \; + \frac{b}{d-1}\; \twostrandid \; -
  \frac{b}{d-1}\; \cupcap = 0. \]
This relation together with symmetry implies that $X$ is an
$\mathit{Sym}$-algebra object.
\end{proof}

Note that the symmetric assumption was essential in the case where $v$ is 
a primitive $12$th root of unity, see Remark \ref{rem:SymtBadCase} for 
more details.

\begin{remark}
The above cases are not mutually exclusive. Indeed both $\Rep(S_3)$ and
$(SOSp(1|2),\allowbreak V_1^-)$ appear in both the $F4$ and 
$\mathit{Sym}(t)$ families, and $SO(3)$ appears in both the $G2$ and Exceptional families. 
In these overlapping cases some 
care must be taken in applying the above result, because the appropriate rescaling of $\tau$ might be different in 
the two families! This occurs in the case $SOSp(1|2)$ where the Frobenius algebra structure from Remark \ref{rem:Sym0SOSp} and the cubic Cayley-Hamilton Jordan algebra structure from Example \ref{examp:Jordan-superalg}
are not isomorphic.
\end{remark}

\begin{remark}
  Proposition~\ref{prop:sym-main-thm} can be proved more directly. In
  a symmetric ribbon category the symmetric group $S_n$ acts on
  $\cC(n)$. Since symmetric group representations are easier to study
  than braid group representations this can be analyzed directly.
  First, $\cC(3)$ is $1$-dimensional so $S_3$ either acts by the
  trivial or sign representation. Since $\dim \cC(4) \leq 5$, the six
  diagrams
\[ \twostrandid\;, \qquad \cupcap \; , \qquad \symcross \; , \qquad \drawI \; , \qquad \drawH \; , \qquad \drawsymcrossX \]
span a representation of $S_4$ of dimension at most $5$. It is easy to see what $6$-dimensional representation of $S_4$ these
six diagrams span: the first three span the induced representation of
the trivial representation of the group
$G = (\ZZ/2)^2 \rtimes \ZZ/2$ while the latter three
span the induced representation from a certain $1$-dimensional
representation of $G$ which depends on the twist. Then we can directly
analyze all possible quotients of these two specific
$6$-dimensional representations of~$S_4$.
\end{remark}

\begin{corollary}
  Suppose that $G$ is a compact Lie group of non-zero dimension, and
  $X$ is a non-trivial symmetrically self-dual irrep $G$ of dimension
  $d$ with $\dim \Hom(X^{\otimes 3},1) =1$ and
  $\dim \Hom(X^{\otimes 4},1) \leq 5$. Then one of the following
  holds:
\begin{itemize}
\item The map $X \otimes X \rightarrow X$ is anti-symmetric, and $X \cong \mathfrak{g}$ is a Lie algebra with a compatible $G$-action.
\item The map $X \otimes X \rightarrow X$ is anti-symmetric, and $X$ is not the adjoint representation, in which case there's a Hurwitz algebra $A \cong 1 \oplus X$ with a compatible $G$-action.
\item The map $X \otimes X \rightarrow X$ is symmetric, in which case there's a cubic Cayley-Hamilton Jordan algebra $A \cong 1 \oplus X$ with a compatible $G$-action.
\end{itemize}
\end{corollary}
\begin{proof}
  This follows from Proposition~\ref{prop:sym-main-thm}, except that
  in the case where the multiplication is commutative, we need to rule
  out the possibility that $1\oplus X$ is a commutative Frobenius
  algebra object. In this case we would have a functor
  $\mathit{Sym}(t) \rightarrow \mathsf{Rep}(G)$, and since the latter is
  unitary, this functor factors through the quotient by negligible
  morphisms. Hence we get a functor $\Rep(S_n) \rightarrow \Rep(G)$
  for some~$n$,
  from which it follows that $G$ is a subgroup of $S_n$ which
  contradicts the assumption on $\dim G$.
\end{proof}

An interesting corollary here is that there must be some cubic Cayley-Hamilton
algebra attached to each Lie group in the $F4$ family. We already explicitly identified these algebras,
so this doesn't provide any new results, but it explains why we got lucky and were able to find such
Jordan algebras.

\subsection{Specializations to exceptional Lie algebras} \label{sec:Specializations}

Suppose that $V$ is a self-dual representation of a quantum group
ribbon category such that there's a unique map
$V \otimes V \rightarrow V$ and such that $V\otimes V$ has no more
than $5$ distinct summands all of which are multiplicity free. Then
Theorems~\ref{thm:Jacobi} and~\ref{thm:square-crossing} apply, and
some instance of the quantum exceptional relations must hold. In this
section we explicitly identify a number of quantum group categories as
quotients of the quantum exceptional series following this outline.

We begin by proving classical specialization, following \cite[Thm. 2.1]{MR3951762}, \cite[Thm. 5.1]{MR1403861}, and \cite[Thm. 5.3]{GSZ23:DiagrammaticsF4}.

\begin{definition}
  We say that a simple Lie (super)algebra $\mathfrak{g}$ is an \emph{Exceptional
  Lie algebra} if it appears in Table~\ref{tab:exceptional}, an
  expanded version of Table~\ref{tab:lambda-group} (excluding the
  trivial group).
  We let
  $\lambda = -6/h^\vee$ where
  $h^\vee$ is the dual Coxeter number.
  For the Lie algebras in the table (excluding $\mathfrak{osp}(1|2)$),
  let $\mathfrak{k}$ denote the compact form of $\mathfrak{g}$
 and consider the group $G_\mathfrak{k} = \mathrm{Aut}(\mathfrak{k})$ of Lie
  algebra automorphisms. Explicitly, this is the semidirect product of
  the adjoint
  form of the compact Lie group $K$ attached to $\mathfrak{k}$ by the group of Dynkin diagram automorphisms.
  Since the group of Dynkin diagrams automorphisms is discrete, the
  Lie algebra of $G_\mathfrak{k}$ is $\mathfrak{k}$.
  See
  Table~\ref{tab:exceptional} for more details about the values
  of~$\lambda$, the groups, and the representations. As in
  \S\ref{sec:LAConventions}, $\Delta$ is
  the
  ratio of the squared length of the long roots to the squared length
  of the short roots.
\end{definition}

Recall that the Killing form (and hence any of the other normalizations we use for the inner product) 
restricted to $\mathfrak{k}$ is negative definite, and that any automorphism of $\mathfrak{k}$
automatically preserves the Killing form since it is defined by a matrix trace of compositions of brackets. Hence the construction
of $G_\mathfrak{k}$ from $\mathfrak{k}$ follows exactly the procedure from \S \ref{sec:specialization-outline}.

\begin{table}
\begin{tabular}{lllllll}
  \toprule
$\mathfrak{g}$ & $G_\mathfrak{k}$ & $\lambda$ & $\Delta$ & $V$ & $v$ & $w$  \\
\midrule
$\mathfrak{e}_8$	& $E_8$					& $-1/5$& 1	& $V_{e_8}$	 & $q^5$	& $q^{-1}$\\
$\mathfrak{e}_7$	& $E_7^{\mathrm{adj}}$			& $-1/3$& 1 	& $V_{e_1}$	& $q^3$	& $q^{-1}$\\
$\mathfrak{e}_6$	& $E_6^{\mathrm{adj}} \rtimes \bbZ/2\bbZ$ & $-1/2$	& 1 	& $V_{e_2}$ 	& $q^2$ 	& $q^{-1}$\\
$\mathfrak{f}_4$       & $F_4$					& $-2/3$& 2	& $V_{(1,0,0,0)}$ & $q^3$ & $q^{-2}$\\
$\mathfrak{d}_4$ 	& $PSO(8) \rtimes S_3$			& $-1$	& 1	& $V_{(0,1,0,0)}$ & $q$	& $q^{-1}$\\
$\mathfrak{g}_2$	& $G_2$ 				& $-3/2$& 3	& $V_{(0,1)}$	& $q^2$	& $q^{-3}$ \\
$\mathfrak{a}_2$ 	& $PSU(3) \rtimes \bbZ/2\bbZ$		& $-2$	& 1	& $V_{(1,1)}$	& $q^{1/2}$ & $q^{-1}$\\
$\mathfrak{a}_1$	& $PSU(2)$				& $-3$	& 1	& $V_{(2)}$	& $q^{1/3}$ & $q^{-1}$\\
$\mathfrak{osp}(1|2)$ & $SOSp(1|2)$				& $-4$	& 4	& $V_{2}^+$	& $q$	& $q^{-4}$\\
\bottomrule
\end{tabular}
\caption{The Lie algebras and groups in the exceptional family and their relevant
  parameters.}
\label{tab:exceptional}
\end{table}
Here $\mathfrak{osp}(1|2)$ uses notation from \S\ref{sec:sosp12}, and
for other groups fundamental weights are ordered as they are in
KM's quantum groups
Mathematica package \cite{QuantumGroups}. In each case the
representation is the
adjoint representation.
Also recall from \S\ref{sec:comparing} that we choose
$q=q_S$ in our quantum group conventions.

\begin{proposition} \label{prop:ClassicalSpecializationProof}
Suppose that $\mathfrak{k}$ is the compact form of an Exceptional Lie algebra that is not
$\mathfrak{osp}(1|2)$. Then there
is a full and dominant functor $\mathsf{Exc}_{\mathbb{C},\lambda}
\rightarrow \mathrm{Rep}(G_\mathfrak{k})$. 
\end{proposition}

\begin{proof}
This is easy to check directly using the skein theoretic description of $SO(3)_q$ in the case $\mathfrak{g} = A_1$.
In all other cases $\mathfrak{g}\otimes \mathfrak{g}$ is a sum of $5$ distinct simple objects as a representation 
of $G_\mathfrak{k}$ (but not necessarily as a representation of $\mathfrak{g}$).

Using the Convenient form, we have that $\mathfrak{g}=V_\theta$ is a metric Lie
algebra object in $\mathrm{Rep}(G_\mathfrak{k})$, and thus takes a functor 
from the category of Jacobi diagrams $\mathsf{Jac}_{\mathbb{C},\dim(\mathfrak{g}),12}$.
Thus, by Proposition~\ref{prop:classical-identification-sym}, the classical exceptional relation with parameter $\lambda$
 holds for some~$\lambda$.  By Lemma \ref{lem:classical-eigenvalues}
 we see that $2\lambda$ and $2(1-\lambda)$ are the unique non-zero 
 eigenvalues of the ladder operator on $\mathfrak{g}\otimes\mathfrak{g}$
 not coming from the summands $\mathfrak{g}$ or $1$.
 But, by Corollary \ref{cor:ladderontwicelongestroot}, the eigenvalue 
 on the summand $V_{2\theta}$ is
 $-12/h^\vee$. Thus (by the symmetry exchanging $\lambda$ and
 $1-\lambda$) we may
 take $\lambda = -6/h^\vee$. Hence the functor descends to one out of $\mathsf{Exc}_{\mathbb{C},\lambda}$.  

This functor is full and dominant by Proposition \ref{prop:definite-full-dominant} and Remark \ref{rem:full-dominant-quotient}.
\end{proof}

We expect that Proposition \ref{prop:ClassicalSpecializationProof} also holds for $SOSp(1|2)$ at 
$\lambda =-4$, but we do not prove that result here since we would need the appropriate version of 
Tannaka-Krein duality and the right analogue of compactness in the proof of Proposition \ref{prop:definite-full-dominant}.

In particular, the category of representations of the corresponding group on the
right of Table~\ref{tab:lambda-group} is the Cauchy completion of a quotient of $\mathsf{Exc}_{\mathbb{C},\lambda}$.
This result is in the style of the first fundamental theorem of invariant theory; one
would also like an explicit characterization of the kernel of this quotient, which we
do not have.
If $\End(\bfOne) = \mathbb{C}$ in $\mathsf{Exc}_{\CC,\lambda}$,
 in particular if Conjecture \ref{conj:class-suffic} holds, 
then the kernel of this quotient can be characterized abstractly as the tensor ideal of negligible morphisms.

Now we turn to the quantum version, a more precise version of
Proposition~\ref{prop:quant-spec}.

\begin{proposition}\label{prop:quant-spec-2}
With $\mathfrak{g}$, $\mathfrak{k}$, $G_\mathfrak{k}$, $\lambda$, and~$\Delta$ from the exceptional
family in Table~\ref{tab:exceptional}, 
and for $q$ in some dense subset of\/ $\mathbb{C}$ (in the complex topology), there is a full
and dominant functor from
$\mathsf{QExc}_{\mathbb{C},q^{-\Delta/\lambda},q^{-\Delta}}$ 
to the category of representations of $\mathrm{Rep}(G_\mathfrak{k})_q$.
\end{proposition}

On the $\mathfrak{a}_1$ line, a fractional power of $q$ appears; any
root will work, as discussed in \S\ref{sec:special-cases}. There is
also a fractional power on the $\mathfrak{a}_2$ line, but recall that
parameters $(v,w)$ and $(-v,w)$ give the same category.

\begin{proof}
First suppose that $G_\mathfrak{k}$ is not $SOSp(1|2)$.
In addition, we have already analyzed the case $G_\mathfrak{k}$ is $PSU(2)$ in
Lemma~\ref{lem:SO3q}. Otherwise,
let $V$ be the adjoint representation of $\mathrm{Rep}(G_\mathfrak{k})_q$, let the trivalent vertex be the deformation
of the Lie bracket map $V \otimes V \rightarrow V$, and choose the
Convenient normalization.  
Our main theorem applies, so \eqref{eq:QEJac-alpha} holds
for an appropriate $v$ and~$\alpha$.
We can compute $v$ directly, as follows. The highest weight vector in
the adjoint representation
is a long root $\theta$, and thus by \S\ref{sec:comparing} and
Table~\ref{tab:bilinear-normalizations} the negative
trivalent twist is $-v^6 = -q^{\langle\theta,\theta+2\rho\rangle_S/2} =
  -q^{h^\vee\Delta} = -q^{-6\Delta/\lambda}$.
  (We can also double-check this using that the negative crossing acts
  on the cup by the positive twist, which is
  $v^{12}=q^{\langle\theta+2\rho,\theta\rangle_S} = q^{2h^\vee \Delta}.$)
The negative crossing acts on the highest weight component $V_{2\theta}$
of $V \otimes V$ by
\[q^{-\langle\theta,\theta\rangle_S} = q^{-2\Delta}.\]

We would like to apply Lemma~\ref{lem:eigenvalues-twist}.
First we need to verify that $(v,\alpha)$ as $v$ varies does not lie entirely
on an excluded
line from Lemma~\ref{lem:eigenvalues-twist}.
This can be seen by, for
instance, taking the limit of~$d$ as $v \to 1$: for the excluded lines
we get $3$ or $-2$, which is not $\dim \mathfrak{g}$.
Therefore, at all but finitely many values of
$q$ we can apply Lemma~\ref{lem:eigenvalues-twist}, to see that, up to symmetry, 
we may take
$w = q^{-\Delta}$.
To see that we can take the correct root of $v^6$ and get
$v = q^{-\Delta/\lambda}$, look near $q=1$ and use
Lemma~\ref{lem:classical-eigenvalues} and continuity; see also
\S\ref{sec:classical-quantum}.
We thus get a functor from $\mathsf{QExc}_{\CC,v,w}$ to $\mathrm{Rep}(G_\mathfrak{k})_q$.

By Corollary \ref{cor:given-by-rational} the functor 
$\mathsf{Tri}_\mathbb{C} \rightarrow \mathsf{QExc}_{\CC,v,w} \rightarrow \mathrm{Rep}(G_\mathfrak{k})_q$
is given by rational functions, and by Proposition \ref{prop:rational-implies-full-dominant} this functor is
full and dominant on a dense set.

The same result holds for $SOSp(1|2)_q$ (with $\lambda =-4$ and
$\Delta = 4$). This can be easily checked directly
using the skein theoretic description of $SOSp(1|2)_q$. Alternatively,
$SOSp(1|2)_q$ with its adjoint representation $V_2^+$ as a tensor
generator is isomorphic to $SO(3)_{iq}$ with $V_{(4)}$ as a tensor
generator, and then the result can be deduced from
Proposition~\ref{prop:quant-f4-spec} below.
\end{proof}

Note that $\Delta$ appears because we used $q=q_S$; the formulas would be slightly 
nicer if we used
$q=q_L=w^{-1}$. In addition, note that $v = q_C$, the $q$-variable
corresponding to the Convenient form.

One can instead argue that a version of the previous result holds over
$\mathbb{C}(q)$, but one needs to
be a little careful in applying our main theorem since that field is not algebraically closed.  

For $q$ in the dense set from Proposition~\ref{prop:quant-spec-2}, as before
the category of representations of the corresponding group on the
right of Table~\ref{tab:lambda-group} is the Cauchy completion of a quotient of $\mathsf{QExc}_{\mathbb{C},-q^{-\Delta/\lambda},q^{-\Delta}}$.
If $\End(\bfOne) = \mathbb{C}$ in $\mathsf{QExc}$, in particular if
Conjecture \ref{conj:quant-suffic} holds, then the kernel of this
quotient
can be characterized abstractly as the tensor ideal of negligible morphisms. 

\begin{remark}
In particular, we see that if $\mathfrak{g}$ is Exceptional and not
$\mathfrak{a}_1$, the braiding
on the $4$-box space of $\mathrm{Rep}(G_\mathfrak{k})_q$ has an eigenvector for the braiding of eigenvalue $-1$.
This eigenvector corresponds to projection onto the simple summand of $\bigwedge^2 \mathfrak{g}$ complimentary
to~$\mathfrak{g}$, named $X_2$ by Deligne.
 (In the case of $\mathfrak{a}_1$, that summand
vanishes.) From the Kontsevich integral point of view, this follows
from Lemma~\ref{lem:ladderzero}.
\end{remark}

\subsection{Specializations to the quantum \texorpdfstring{$F4$}{F4} and \texorpdfstring{$G2$}{G2} families}
\label{sec:spec-f4-g2}

Similar results hold for the $F4$ and $G2$ series using similar
arguments.
For $F_4$ itself this appears in \cite{2204.11976, GSZ23:DiagrammaticsF4}, 
but in general we have to be a bit careful about choosing the correct
group $G$ attached to the Lie algebra~$\mathfrak{g}$ following \S \ref{sec:classical-G2F4}.
For $G_2$ this is already in \cite{MR1265145, MR1403861}.

\begin{definition}\label{def:F4-family}
  For $d \in \{-2,-1,0,2,5,8,14,26\}$, in
\S\ref{sec:classical-F4} we defined an $F4$ (super)-algebra $A$ of
dimension~$d$ and associated (super)-group of
automorphisms $G_{A}$,
as shown in Table~\ref{tab:F4Table}. Set $\eta=\frac{d-2}{12}$. If
$G_{A}$ is not discrete, we say that it is an \emph{$F4$
  Lie group} and the associated Lie (super)-algebra and
defining representation $V = V_\mu$ are also shown.
For $d=-1$, $G = SOSp(1|2)$, let $\Lambda = -4/3$.
For $d=-2$, $G=SL(2)$, let $\Lambda=-3/8$. (Note the representation
$V_{(1)}^-$ is odd in this case.)
For other $F4$ Lie groups, let
$\Lambda\coloneq\frac{1}{2}\langle \mu,\mu\rangle_S$.
For $d\ne -2$
there is a
symmetric map $V \otimes V \rightarrow V$, unique up to scale. (For
$d=-2$ that map is~$0$.)
\end{definition}

\begin{table}
  \begin{tabular}{rllllllll}
    \toprule
$d$ & $A$        & $G_{A}$			& $\mathfrak{g}$	& $\eta$& $\Lambda$ & $V$ & $v$ & $w$  \\
\midrule
$-2$ & $J^{0\mid 2}$ & $SU(2)$			& $\mathfrak{a}_1$	& $-1/3$& $-3/8$& $V_{(1)}^-$& $i q^{1/8}$ & $q^{3/8}$\\
$-1$ & $D_2$       & $SOSp(1|2)$		& $\mathfrak{osp}(1|2)$	& $-1/4$& $-4/3$& $V_1^-$  & $i q^{1/3}$ & $q^{4/3}$\\
 $0$ & $\RR$       & Triv			& 			& $-1/6$\\
 $2$ & $\RR^3$     & $S_3$			& 			& $0$\\
 $5$ & $H_3(\RR)$ & $SO(3)$			& $\mathfrak{a}_1$	& $1/4$ & $4$ & $V_{(4)}$      & $i q$ & $q^{-4}$\\
 $8$ & $H_3(\CC)$ & $\mathit{PSU}(3)\rtimes \bbZ/2\bbZ$	& $\mathfrak{a}_2$	& $1/2$ & 1   & $V_{(1,1)}$     & $i q^{1/2}$ & $q^{-1}$\\
$14$ & $H_3(\bbH)$ & $\mathit{PUSp}(3)$			& $\mathfrak{c}_3$	& $1$   & 1   & $V_{(0,1,0)}$   & $i q$	& $q^{-1}$\\
$26$ & $H_3(\bbO)$ & $F_4$			& $\mathfrak{f}_4$	& $2$   & 1   & $V_{(0,0,0,1)}$ & $i q^2$ & $q^{-1}$\\
  \bottomrule
\end{tabular}
\caption{The trivalent category Lie (super)algebras and (super)groups
  in the classical $F4$ family, along with (when relevant) the related
  quantum parameters.}
\label{tab:F4Table}
\end{table}

We can extend Proposition~\ref{prop:F4specialization} to cover the
quantized case.

\begin{proposition}\label{prop:quant-f4-spec}
  Let $G$ be an $F4$ Lie group with
  corresponding representation $V_\mu$. For $q$ in some dense subset
  of\/~$\mathbb{C}$, there is a full and dominant functor from
  $\mathsf{QExc}_{\mathbb{C}, i q^{\Lambda \eta} , q^{-\Lambda}}$ to
  the category of representations of $(\mathrm{Rep}(G),V_\mu)_q$.
\end{proposition}
The relevant values of $v = iq^{\Lambda \eta}$ and $w = q^{-\Lambda}$
are in Table~\ref{tab:F4Table}.
(There are some more fractional powers in the table, coming from
the fact that we no longer always take the adjoint representation;
these should be interpreted as powers of~$s$.)

The parameter-space lines for the $F4$ family are shown graphically in
Figure~\ref{fig:F4-G2-param} (left). Note they all pass through the
classical $F4$
point $(i,1)$ with slope $-1/\eta$.

\begin{proof}
  We follow the outline of Proposition~\ref{prop:quant-spec-2}, but
  the details are different.
  First suppose $d \notin \{-2,-1\}$. In the other cases, since
  $\mathrm{Rep}(G)_q$ satisfies the
  hypotheses of our main theorem, \eqref{eq:QEJac-w} holds and we have
  a functor from $\mathsf{QExc}_{v,w}$ to $(\mathrm{Rep}(G)_q,V_\mu)$.
  Again a dimension argument
  shows we are not entirely in an
  exceptional line for Lemma~\ref{lem:lin-ind-w}, and so can apply
  Lemma~\ref{lem:eigenvalues-twist}.
  Note that $V_{2\mu}$ is a summand of $\Sym^2(V_\mu)$ with
  eigenvalue $q^{-\langle\mu,\mu\rangle_S}$ 
  for $\HTw$. Comparing with
  the eigenvalues for $\HTw$ from Lemma~\ref{lem:eigenvalues-twist},
  we see that (up to symmetries) we can take $w = q^{-\Lambda}$.

  Similarly, from the symmetric product giving $V_\mu$ as a summand of
  $\Sym^2(V_\mu)$, we deduce that
  $-v^6 = q^{\langle \mu, \mu+2\rho\rangle_S/2}$. To nail down which
  sixth root to take, observe that there is another summand $W_\nu$ of
  $\bigwedge^2 V_\mu$, for which the eigenvalue of $\HTw$ must be
  $v^2/w^2 = -q^{\langle\mu,\mu+2\rho\rangle_S -
    \langle\nu,\nu+2\rho\rangle_S/2}$, so by the limit $q \to 1$ we see
  $v^2$ must be negative of a
  power of~$q$. Up to symmetries we can thus take
  $v = iq^{\Lambda \eta}$ where
  $\eta = \frac{\langle \mu,\mu+2\rho\rangle_S}{6\langle \mu,
    \mu\rangle_S}$.
  We need to relate this formula for $\eta$ to the formula in terms of
  $\mathrm{dim}(V_\mu)$ given in Definition~\ref{def:F4-family}. For this, take the
  dimension formula in $\mathsf{QExc}_{v,w}$ (also the quantum dimension
  of $V_\mu$ in $\mathrm{Rep(G_q)}$),
  \[
    d = -\frac{\Psi_4[\lambda+5][\lambda-6]}{[\lambda][\lambda-1]},
  \]
  and take the limit as $h \to 0$ with $q=e^h$, $w = e^{\Lambda h}$,
  and $v = ie^{\eta\Lambda h}$. We get
  \[
    \dim(V_\mu) = \lim_{h \to 1} d =
    -\frac{(-2)(2i)(-2\Lambda(1+6\eta))}{(\Lambda)(-2i)} = 2(1+6\eta)
  \]
  giving $\eta = (\dim(V_\mu)-2)/12$ as desired.

  Fullness on a dense subset of $\mathbb{C}$ follows from fullness at
  $q=1$ as before.

  For $(G,V) = (SOSp(1|2),V_1^-)$, the above argument does not work. (In addition
  to a dimension drop in the $4$-box space,
  $V_{2\mu}$ is in $\bigwedge^2 V_\mu$ since $\mu$ is odd, so our
  identification of $w^2$ does not work.) We instead find the
  values of $v$ and~$w$ directly using Lemmas~\ref{lem:SOSp12-SO3}
  and~\ref{lem:SO3q}:
  \[
    (SOSp(1|2),V_1^-)_q
    \cong (SO(3),V_2)_{-iq}
    \text{ is a quotient of } \mathsf{QExc}_{\CC,iq^{1/3},iq^{-1}}
    \cong \mathsf{QExc}_{\CC,iq^{1/3},q^{4/3}}.
  \]
  
  For $(G,V) = (SL(2), V_{(1)}^-)$, more breaks down: the category
  $\mathrm{Rep}(SL(2), V_{(1)}^-)_q$ is not included in any of the
  discussions above, since it is not trivalent. Nevertheless
  $\QExc_{\bbC,iq^{1/8},q^{3/8}}$ is defined (though $b=0$) and is
  easily seen to have a quotient to the Kauffman bracket spider at
  $A=-q^{1/2}$, setting all diagrams with a
  trivalent vertex to zero. This Kauffman bracket spider is
  well-known to have a full and dominant
  functor to $\mathrm{Rep}(SL(2),\varepsilon=-1)_q$.
\end{proof}

Some special values of $d$ for the classical $F4$ family are missing from
the quantum theorem. The missing values are:
\begin{itemize}
\item $d=2$: the finite group $S_3$ with its $2$-dimensional defining
  representation. The quantum deformation, if it existed, would be the
  line $v=i$,
  but the change of variables is singular and we saw in
  \S\ref{sec:v=i} that candidate categories are severely restricted.
\item $d=0$: the trivial group, all dimensions vanish. The quantum
  deformation would be $w = \pm v^k$ for $k=-5,6$ which is a zero
  of~$d$.
\end{itemize}

\begin{figure}
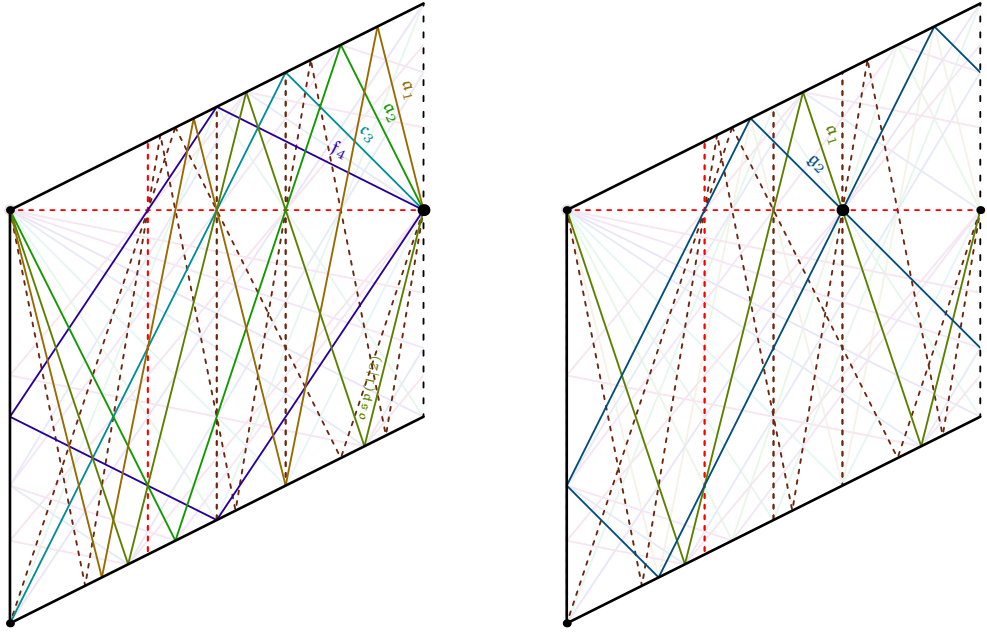

  \centering
  \includegraphics{mpdraws/param-space-1}\qquad\qquad\includegraphics{mpdraws/param-space-2}
  \caption{The quantum $F4$ (left) and $G2$ (right) families in $(v,w)$
    parameter space. The large marked points (at $(i,0)$ and $(\zeta_3,0)$,
    respectively) are the classical $F4$ and $G2$ points, where all
    lines in the respective families intersect.}
  \label{fig:F4-G2-param}
\end{figure}

\begin{remark}\label{rem:a2}
Note that $A_2$ with its adjoint representation is both an Exceptional Lie algebra and an $F4$ Lie algebra.
However, the action of $\mathit{PSU}(3)\rtimes \bbZ/2\bbZ$ on $\mathfrak{sl}(3)$ is subtly different because they
come from automorphisms of different algebra structures.  Namely the non-trivial element of $\bbZ/2\bbZ$
acts by negative transpose in the exceptional case and by transpose in the $F4$ case.
\end{remark}

\begin{definition}\label{def:G2-family}
We say that a simple Lie group and an irreducible representation
$(G, V_\mu)$ is a \emph{$G2$-Lie algebra} if $V_\mu$ is a
$G2$-algebra of dimension~$d$ as in \S\ref{sec:classical-G2} and $G$ is the full
group of automorphisms of $V_\mu$ and has positive dimension; the
possibilities are recalled in Table~\ref{tab:G2}.
Let $\Lambda = \frac{1}{2}\langle \mu,\mu\rangle_S$ be half the squared
length of the highest weight of $V$,
which in both cases is $1$, and set $\eta = (d-1)/6$.
\end{definition}
In both cases, the product
(the unique up to rescaling map $V \otimes V \rightarrow V$) is
anti-symmetric.

\begin{table}
  \begin{tabular}{rllllllll}
    \toprule
$d$ & $A$  & $G_A$ & $\mathfrak{g}$ & $\eta$ & $\Lambda$ & $V$ & $v$ & $w$  \\
\midrule
0 & $\RR$  & Triv    &			& $-1/6$\\
1 & $\CC$  & $S_2$   &			& $0$ \\
3 & $\HH$  & $SO(3)$ & $\mathfrak{a}_1$	& $1/3$	& 1	& $V_{(2)}$	& $\zeta_3 q^{1/3}$ 	& $q^{-1}$\\
7 & $\bbO$ & $G_2$   & $\mathfrak{g}_2$	& 1	& 1	& $V_{(1,0)}$	& $\zeta_3 q$ 		& $q^{-1}$\\
\bottomrule
\end{tabular}
\caption{The Hurwitz algebras and the associated groups, Lie
  algebras, and specialization parameters in the $G2$ family.}
\label{tab:G2}
\end{table}

\begin{proposition}\label{prop:quant-g2-spec}
Let $G$ be a $G2$ Lie group and $\zeta_3$ be a primitive cube root of
unity.
For $q$ in some dense subset of\/~$\mathbb{C}$, there is a full
and dominant functor from $\mathsf{QExc}_{\mathbb{C}, \zeta_3 q^{\Lambda \eta}, q^{-\Lambda}}$
to the category of representations of $\mathrm{Rep}(G)_q$.
\end{proposition}

Again, the line that specializes to $\mathrm{Rep}(G)_q$ in
Figure~\ref{fig:F4-G2-param} (right) passes through the classical $G2$ point
$(\zeta_3,1)$ at slope $-1/\eta$.

\begin{proof}
  For $\mathrm{Rep}(PSL(2), V_{(2)})_q$, this is part of Lemma~\ref{lem:SO3q}.

  The only other case is $\mathrm{Rep}(G_2)_q$, for which the $4$-box
  space is $4$-dimensional, so
  Theorems~\ref{thm:Jacobi} and~\ref{thm:square-crossing} apply, the
  change of variables is invertible, and it remains
  to identify the parameters $v$ and~$w$.
  Lemma~\ref{lem:eigenvalues-twist} does not apply (we are on the line
  $d = P_{G_2}(\zeta_3 v)$), but it follows by the same arguments that the action on
  the $4$-box space has eigenvalues that are a subset of
  $\{v^{12},-v^6,-1,w^2,v^2/w^2\}$.

  In particular, up to $w \leftrightarrow v/w$ we can take $V_{2\mu} \subset \Sym^2(V_\mu)$
  to have eigenvalue $w^2$ for $\HTw$, from which we find that
  $w = q^{-\langle \mu, \mu \rangle_S/2} = q^{-\Lambda}$. From there, if
  we set $v = \zeta_3 w^{-\eta}$, the same type of limiting dimension argument as in
  Proposition~\ref{prop:quant-f4-spec} implies that $\dim V = 1+6\eta$
  as desired.
  
  Fullness for a dense subset of~$\mathbb{C}$ follows as before.
\end{proof}

\subsection{Intersections of known families}\label{sec:intersections}
We can find the intersections of the subvarieties coming from the
known series. If the quantum
sufficiency conjecture holds, we could deduce the braided equivalences
shown in Table~\ref{tab:intersections}. Each triple $(F,\Gamma,q)$
indicates an equivariantization of the subcategory of the
(semisimplified) quantum group representation category for the Lie
algebra~$\Gamma$ in the family~$F$ at the root of unity~$q$ tensor generated by an
appropriate irreducible (as listed earlier in this section). The order
of the root of unity determines the level, but this does not
completely specify the (braided) category, which depends on the actual
value of $q$---although of course all such categories are Galois
conjugates. We only list one equivalence out of each Galois conjugacy
class.
We use again the shorthand
$\zeta_n = \exp(2 \pi i/n)$ to write~$q$. Each conjectured equivalence would take the tensor
generator to the tensor generator.

There is one ambiguity: $(F4, A_1)$ has $d=5$ in
Table~\ref{tab:F4Table}. (The copy of $A_1$ with $d=-2$ there
is not a trivalent category.)
There are also some lines that are equal (with different values of $q$). In
particular, the isomorphism $SO(3)_q \cong SOSp(1|2)_{iq}$ explained
in \S\ref{sec:sosp12} mean that $A_1$ in one family may be $SOSp(1|2)$
in another family. In this way, $(A_1,V_{(2)})$ appears as both
$(E8,A_1)$ and $(G2,A_1)$ and when quantized is isomorphic to
$(SOSp(1|2), V_1^-)$ appearing as $(F4,SOSp(1|2))$. Similarly,
$(SOSp(1|2),V_2^+)$
appears as $(E8,SOSp(1|2))$ and when quantized is isomorphic to
$(A_1,V_{(4)})$, appearing as $(F4,A_1)$. Finally, $(A_2,V_{(1,1)})$
appears as both $(E8,A_2)$ and $(F4,A_2)$; see Remark~\ref{rem:a2}.

The same data can be seen graphically already in
Figure~\ref{fig:fund-domain}, as the intersections of colored solid lines
that are not also on dashed lines, except that the excluded line $v^{10}=1$ is not
marked there.

\begin{table}[tbp]
  \small
  \[
    \begin{aligned}
(\zeta_{68},\zeta_{68}^{-7})&\colon (E8,E_8,\zeta_{68}^{7}) \cong (F4,F_4,\zeta_{34}^{13}) \\
(\zeta_{66},\zeta_{11}^{-2})&\colon (E8,E_8,\zeta_{33}^{10}) \cong (G2,G_2,\zeta_{22}^{7}) \\
(\zeta_{60},\zeta_{15}^{-2})&\colon (F4,F_4,\zeta_{15}^{2}) \cong (G2,G_2,\zeta_{20}^{7}) \\
(\zeta_{56},\zeta_{28}^{-5})&\colon (E8,F_4,\zeta_{56}^{19}) \cong (E8,E_8,\zeta_{56}^{17}) \\
(\zeta_{54},\zeta_{54}^{-11})&\colon (E8,G_2,\zeta_{27}^{7}) \cong (E8,E_8,\zeta_{54}^{11}) \\
(\zeta_{52},\zeta_{52}^{-7})&\colon (E8,F_4,\zeta_{52}^{9}) \cong (F4,F_4,\zeta_{52}^{7}) \\
(\zeta_{52},\zeta_{13}^{-1})&\colon (E8,E_8,\zeta_{52}^{21}) \cong (F4,A_1,\zeta_{26}^{7}) \\
(\zeta_{46},\zeta_{23}^{-4})&\colon (E8,E_7,\zeta_{23}^{4}) \cong (E8,E_8,\zeta_{23}^{7}) \\
(\zeta_{44},\zeta_{44}^{-9})&\colon (E8,E_8,\zeta_{44}^{9}) \cong (F4,C_3,\zeta_{11}^{3}) \\
(\zeta_{44},\zeta_{22}^{-3})&\colon (E8,E_7,\zeta_{44}^{15}) \cong (F4,F_4,\zeta_{22}^{3}) \\
(\zeta_{44},\zeta_{11}^{-1})&\colon (F4,A_1,\zeta_{11}^{3}) \cong (F4,F_4,\zeta_{44}^{17}) \\
(\zeta_{42},\zeta_{21}^{-4})&\colon (E8,G_2,\zeta_{42}^{11}) \cong (G2,G_2,\zeta_{42}^{13}) \\
(\zeta_{42},\zeta_{14}^{-1})&\colon (E8,A_1,\zeta_{14}) \cong (E8,E_8,\zeta_{42}^{17}) \\
(\zeta_{38},\zeta_{38}^{-7})&\colon (E8,G_2,\zeta_{19}^{5}) \cong (E8,F_4,\zeta_{38}^{13}) \\
(\zeta_{38},\zeta_{19}^{-2})&\colon (E8,E_8,\zeta_{19}^{2}) \cong (F4,A_1,\zeta_{76}^{21}) \\
(\zeta_{36},\zeta_{36}^{-7})&\colon (F4,C_3,\zeta_{18}^{5}) \cong (G2,G_2,\zeta_{36}^{11}) \\
(\zeta_{36},\zeta_{9}^{-1})&\colon (F4,A_1,\zeta_{18}^{5}) \cong (G2,G_2,\zeta_{36}^{13}) \\
(\zeta_{36},\zeta_{12}^{-1})&\colon (E8,A_1,\zeta_{12}) \cong (F4,F_4,\zeta_{18}^{7}) \\
(\zeta_{34},\zeta_{34}^{-7})&\colon (E8,E_6,\zeta_{34}^{9}) \cong (E8,E_8,\zeta_{34}^{7}) \\
(\zeta_{34},\zeta_{17}^{-3})&\colon (E8,G_2,\zeta_{34}^{9}) \cong (E8,E_7,\zeta_{17}^{3}) \\
(\zeta_{34},\zeta_{17}^{-2})&\colon (E8,F_4,\zeta_{17}^{3}) \cong (F4,A_1,\zeta_{68}^{19}) \\
(\zeta_{32},\zeta_{16}^{-3})&\colon (E8,F_4,\zeta_{32}^{11}) \cong (F4,C_3,\zeta_{32}^{9}) \\
(\zeta_{32},\zeta_{8}^{-1})&\colon (E8,E_7,\zeta_{32}^{11}) \cong (F4,A_1,\zeta_{32}^{9}) \\
(\zeta_{32},\zeta_{16}^{-1})&\colon (E8,A_2,\zeta_{16}) \cong (E8,E_8,\zeta_{32}^{13}) \\
(\zeta_{30},\zeta_{5}^{-1})&\colon (E8,E_6,\zeta_{15}^{4}) \cong (G2,G_2,\zeta_{10}^{3}) \\
\end{aligned}\qquad
\begin{aligned}
(\zeta_{30},\zeta_{10}^{-1})&\colon (E8,A_1,\zeta_{10}) \cong (G2,G_2,\zeta_{30}^{11}) \\
(\zeta_{28},\zeta_{28}^{-3})&\colon (E8,A_1,\zeta_{28}^{3}) \cong (E8,F_4,\zeta_{28}^{5})\cong (E8,E_8,\zeta_{28}^{3}) \\
(\zeta_{28},\zeta_{14}^{-1})&\colon (E8,A_2,\zeta_{14}) \cong (F4,F_4,\zeta_{28}^{11}) \\
(\zeta_{26},\zeta_{13}^{-2})&\colon (E8,G_2,\zeta_{26}^{7}) \cong (E8,E_8,\zeta_{13}^{4}) \cong (F4,A_1,\zeta_{52}^{15}) \\
(\zeta_{26},\zeta_{26}^{-3})&\colon (E8,A_1,\zeta_{26}^{3}) \cong (E8,E_7,\zeta_{26}^{9}) \\
(\zeta_{24},\zeta_{6}^{-1})&\colon (F4,A_1,\zeta_{24}^{7}) \cong (F4,C_3,\zeta_{24}^{7}) \\
(\zeta_{24},\zeta_{48}^{-7})&\colon (E8,G_2,\zeta_{48}^{13}) \cong (F4,F_4,\zeta_{48}^{7}) \\
(\zeta_{24},\zeta_{12}^{-1})&\colon (E8,A_2,\zeta_{12}) \cong (G2,G_2,\zeta_{8}^{3}) \\
(\zeta_{22},\zeta_{11}^{-2})&\colon (E8,E_6,\zeta_{11}^{3}) \cong (E8,E_7,\zeta_{11}^{2}) \cong (F4,A_1,\zeta_{44}^{13}) \\
(\zeta_{22},\zeta_{22}^{-3})&\colon (E8,A_1,\zeta_{22}^{3}) \cong (E8,G_2,\zeta_{11}^{3}) \\
(\zeta_{22},\zeta_{11}^{-1})&\colon (E8,A_2,\zeta_{11}) \cong (E8,F_4,\zeta_{11}^{2}) \\
(\zeta_{22},\zeta_{22}^{-1})&\colon (E8,D_4,\zeta_{22}) \cong (E8,E_8,\zeta_{22}^{9}) \\
(\zeta_{20},\zeta_{20}^{-3})&\colon (E8,A_1,\zeta_{20}^{3}) \cong (F4,F_4,\zeta_{20}^{3}) \cong (F4,C_3,\zeta_{10}^{3}) \\
(\zeta_{20},\zeta_{20}^{-1})&\colon (E8,D_4,\zeta_{20}) \cong (F4,F_4,\zeta_{5}^{2}) \\
(\zeta_{18},\zeta_{9}^{-1})&\colon (E8,A_2,\zeta_{9}) \cong (E8,G_2,\zeta_{18}^{5}) \cong (E8,E_8,\zeta_{9}) \\
(\zeta_{18},\zeta_{18}^{-1})&\colon (E8,D_4,\zeta_{18}) \cong (G2,G_2,\zeta_{18}^{7}) \\
(\zeta_{16},\zeta_{24}^{-5})&\colon (E8,F_4,\zeta_{48}^{17}) \cong (G2,G_2,\zeta_{48}^{13}) \\
(\zeta_{16},\zeta_{32}^{-5})&\colon (E8,E_6,\zeta_{32}^{9}) \cong (F4,F_4,\zeta_{32}^{5}) \\
(\zeta_{16},\zeta_{32}^{-3})&\colon (E8,G_2,\zeta_{32}) \cong (E8,G_2,\zeta_{32}^{9}) \\
(\zeta_{14},\zeta_{21}^{-4})&\colon (E8,E_7,\zeta_{21}^{4}) \cong (G2,G_2,\zeta_{42}^{11}) \\
(\zeta_{14},\zeta_{7}^{-1})&\colon (E8,A_2,\zeta_{7}) \cong (E8,E_6,\zeta_{7}^{2}) \cong (F4,A_1,\zeta_{28}^{9}) \\
(\zeta_{14},\zeta_{35}^{-4})&\colon (E8,E_8,\zeta_{35}^{4}) \cong (E8,E_8,\zeta_{35}^{11}) \\
(\zeta_{14},\zeta_{28}^{-3})&\colon (E8,G_2,\zeta_{28}) \cong (F4,C_3,\zeta_{28}^{9}) \\
(\zeta_{14},\zeta_{21}^{-1})&\colon (E8,F_4,\zeta_{42}) \cong (E8,F_4,\zeta_{21}^{4})
\end{aligned}
\]
\caption{Conjectural isomorphisms between quantum group categories.
  These are the 49 $(v,w)$ parameters
  where the lines from the exceptional (a.k.a.\ $E8$), $F4$, and~$G2$
  families intersect, up to Galois conjugation and
  excluding degenerate parameter values. Specifically we exclude values where $v$ is a
  primitive root of unity of order $1,2,3,4,5,6,8,10,12$, $w=\pm1$, $w = \pm
  v^k$ for $k=-5,6$, or $w = \pm i v^k$ for $k=-2,3$.
}
\label{tab:intersections}
\end{table}

\begin{warning}
In this discussion, we are taking our conjectural diagrammatic categories 
and semisimplifying them at the listed root of unity. In order to translate these conjectural equivalences 
into statements about the semisimplified category of tilting modules for the Lusztig quantum group 
we would need to  know that the semisimplified diagram category and the semisimplified tilting module 
category were equivalent. Even in cases where such an equivalence is known at roots of unity, it will still
fail for a few small roots of unity. So one should be especially skeptical about the conjectures below in the
semisimplified tilting module setting if the root of unity is small.
\end{warning}
\begin{remark}\label{rem:extended-haagerup}
The case of $F_4$ at level $4$, $(E8,F_4,\zeta_{52}^9) \cong
(F4,F_4,\zeta_{52}^{7})$ (circled in Figure~\ref{fig:fund-domain})
is especially interesting. This is a categorical version of the automorphism of the fusion ring 
of $F_4$ level~$4$ from \cite[\S 3.8]{MR1887583}. Composing this equivalence with its own 
Galois conjugate gives a braided equivalence
 $((F_4)_{4},\zeta_{52}) \cong ((F_4)_4,\zeta_{52}^{25}) \cong
 ((F_4)_4,\zeta_{52}^{-1})$ (with some permutation of the representations)
 and hence a braided reversing auto-equivalence of $((F_4)_{4},\zeta_{52})$.
This braided reversing auto-equivalence may be of interest to the study of exotic subfactors. 
NS and KM noticed
that $F_4$ at level $4$ has dimensions which are related to the dimensions
of objects in the  Extended Haagerup subfactor \cite{MR2979509}, and Edie-Michell 
\cite[Question 7.2]{MR4401829} strengthened this observation to ask whether the center of 
the Extended Haagerup fusion categories is an Evans-Gannon grafting \cite{MR2837122} 
of $F_4$ level $4$ using a braiding-reversing autoequivalence.  Evans-Gannon grafting has
not been rigorously defined at the level of
ribbon categories, and moreover we need to assume
the quantum sufficiency conjecture in order to construct the autoequivalence Edie-Michell requires.
We will explain this in more detail in future work.
\end{remark}



\appendix
\section{Coefficients in QESq}
\label{app:coefficients}

Here are the explicit formulae for the coefficients in the relations Eqs. \eqref{eq:QESq-w} and \eqref{eq:QESq-alpha}. The details of the calculation can be found in
 the \texttt{arXiv} sources of this article in the Mathematica
 notebook
 \[\text{{\small\tt
       arxiv-code/AppendixCalculations.nb}}.\]
 (These calculations can also be done by hand with only a moderate
 amount of pain.)

In \eqref{eq:QESq-w}, we have
\begin{equation*}
\fourgon + \frac{-\Psi_2\Psi_6^2[\lambda][\lambda-1]}{\Psi_1^2} \braidcross + \frac{x_1}{\Psi_1} \drawI + \frac{x_2}{\Psi_1} \drawH + \frac{x_3}{\Psi_1^2} \cupcap + \frac{x_4}{\Psi_1^2} \twostrandid =0
\end{equation*}
for
\begingroup
\allowdisplaybreaks
\begin{align*}
x_1 &= w^{-2} (1-v^2-v^4)  + (v^{-4} - 2 v^{-2} + 3 v^2) + w^2 (v^{-2} -1-v^2) \\
x_2 &= w^{-2}(v^{-2} +1 - v^2) + (-3v^{-2}+2v^2-v^4) + w^2(v^{-4} + v^{-2} -1) \\
x_3 &= w^{-4}(v^{-2}-1) + w^{-2}(-2v^{-4}-v^{-2}+2-v^4) + (v^{-6}+4v^{-4} -2v^{-2}-1+v^2+v^4) \\&+ w^2 (-2v^{-6}-v^{-4}+2v^{-2}-v^2) +w^4(v^{-6}-v^{-4}) \\
x_4 &= w^{-4}(-v^4 + v^6) + w^{-2}(-v^{-2}+2v^2-v^4-2v^6) + (v^{-4} + v^{-2} -1-2v^2+4v^4+v^6) \\& +w^2 (-v^{-4} +2 -v^2 - 2 v^4) + w^4 (-1+v^2) \\
\end{align*}
\endgroup

In \eqref{eq:QESq-alpha}, we have
\begin{equation*} 
\fourgon + \frac{\alpha \Psi_6 (b+[3]\alpha)}{\Psi_1 \Psi_3 \Psi_4}  \braidcross + \frac{y_1}{\Psi_3 \Psi_4}  \drawI + \frac{y_2}{\Psi_3 \Psi_4}  \drawH + \frac{\alpha y_3}{\Psi_1 \Psi_3 \Psi_4} \cupcap + \frac{\alpha y_4}{\Psi_1 \Psi_3 \Psi_4} \twostrandid =0
\end{equation*}
for
\begin{align*}
y_1 &= -b + (-v^{-5} + v^{-3} + 2 v + v^3 + 2 v^5 + v^7)\alpha \\
y_2 &= -b + (-v^{-7} -2 v^{-5} -v^{-3} - 2 v^{-1} -v^3 +v^5)\alpha\\
y_3 &= b+ (-v^{-9}+v^{-7}-v^{-5}+v^{-3}-v^{-1}+v)\alpha\\
y_4 &= b + (-v^{-1} + v - v^3 + v^5 - v^7 +v^{9})\alpha.
\end{align*}

\section{The \texorpdfstring{$E_8$}{E\_8} knot polynomial for the non-algebraic knot  \texorpdfstring{$8_{18}$}{8\_18}} \label{app:E8}

As a sample calculation, we present (unconditionally) the $E_8$
polynomial for the first non-algebraic knot, $8_{18}$, a computation
that was previously out of reach, including by the work of
Mironov-Morozov \cite{MR3475991} described in
Section~\ref{sec:other-work}.

\begin{multline*}
\frac{\left(q^4+1\right)^2 \left(q^{16}-q^{12}+q^8-q^4+1\right)}{q^{306}} 
   \times{} \\ \displaybreak[2]
\quad \times
\Big(
q^{92}+q^{90}+q^{84}+q^{82}+q^{80}+q^{78}+q^{76}+q^{74}+q^{72}+q^{70}+2 q^{68}+2q^{66}+q^{64}+q^{62}+2 q^{60}+2 q^{58} \\\qquad
+2 q^{56}+2 q^{54}+2 q^{52}+2 q^{50}+2 q^{48}+2 q^{46}+2 q^{44}+2 q^{42}+2 q^{40}+2 q^{38}+2 q^{36}+2 q^{34}+2q^{32} \\\qquad
+q^{30}+q^{28}+2 q^{26}+2 q^{24}+q^{22}+q^{20}+q^{18}+q^{16}+q^{14}+q^{12}+q^{10}+q^8+q^2+1\Big) \times{} \\ \displaybreak[1]
\quad \times 
\Big(
q^{496}-4 q^{494}+6 q^{492}-4 q^{490}+q^{488}-4q^{484}+16 q^{482}-24 q^{480}+16 q^{478}-8 q^{476}+16 q^{474}-18 q^{472} \\\qquad
-8 q^{470}+32 q^{468}-24 q^{466}+18 q^{464}-52 q^{462}+80 q^{460}-44 q^{458}+2 q^{456}-24q^{454}+44 q^{452}+20 q^{450} \\\qquad
-93 q^{448}+64 q^{446}-18 q^{444}+92 q^{442}-183 q^{440}+116 q^{438}-12 q^{436}+85 q^{434}-193 q^{432}+67 q^{430} \\\displaybreak[2]\qquad
+153 q^{428}-132q^{426}+20 q^{424}-211 q^{422}+517 q^{420}-433 q^{418}+126 q^{416}-191 q^{414}+461 q^{412}-237 q^{410} \\\displaybreak[2]\qquad
-333 q^{408}+409 q^{406}-85 q^{404}+302 q^{402}-959 q^{400}+946q^{398}-274 q^{396}+215 q^{394}-823 q^{392}+678 q^{390} \\\displaybreak[2]\qquad
+360 q^{388}-685 q^{386}-89 q^{382}+1327 q^{380}-1728 q^{378}+681 q^{376}-368 q^{374}+1555 q^{372}-1822q^{370} \\\displaybreak[2]\qquad
+143 q^{368}+927 q^{366}+14 q^{364}-256 q^{362}-1762 q^{360}+3103 q^{358}-1709 q^{356}+575 q^{354}-2107 q^{352} \\\displaybreak[2]\qquad
+3153 q^{350}-948 q^{348}-1332 q^{346}+320q^{344}+836 q^{342}+1501 q^{340}-4105 q^{338}+2762 q^{336}-556 q^{334} \\\displaybreak[2]\qquad
+2184 q^{332}-4575 q^{330}+2573 q^{328}+973 q^{326}-291 q^{324}-1979 q^{322}-276 q^{320}+4567q^{318}-4155 q^{316} \\\displaybreak[2]\qquad
+1041 q^{314}-2351 q^{312}+6237 q^{310}-5072 q^{308}+117 q^{306}+459 q^{304}+2619 q^{302}-1072 q^{300}-4676 q^{298} \\\displaybreak[2]\qquad
+5789 q^{296}-1865 q^{294}+1886q^{292}-6637 q^{290}+6941 q^{288}-1317 q^{286}-768 q^{284}-2887 q^{282}+2954 q^{280} \\\displaybreak[2]\qquad
+3203 q^{278}-6066 q^{276}+2166 q^{274}-890 q^{272}+6141 q^{270}-8532 q^{268}+3377q^{266}+87 q^{264}+3441 q^{262} \\\displaybreak[2]\qquad
-5066 q^{260}-965 q^{258}+5826 q^{256}-2922 q^{254}+494 q^{252}-5390 q^{250}+9551 q^{248}-5390 q^{246}+494 q^{244} \\\displaybreak[2]\qquad
-2922 q^{242}+5826q^{240}-965 q^{238}-5066 q^{236}+3441 q^{234}+87 q^{232}+3377 q^{230}-8532 q^{228}+6141 q^{226} \\\displaybreak[2]\qquad
-890 q^{224}+2166 q^{222}-6066 q^{220}+3203 q^{218}+2954 q^{216}-2887q^{214}-768 q^{212}-1317 q^{210}+6941 q^{208} \\\displaybreak[2]\qquad
-6637 q^{206}+1886 q^{204}-1865 q^{202}+5789 q^{200}-4676 q^{198}-1072 q^{196}+2619 q^{194}+459 q^{192}+117 q^{190} \\\displaybreak[2]\qquad
-5072q^{188}+6237 q^{186}-2351 q^{184}+1041 q^{182}-4155 q^{180}+4567 q^{178}-276 q^{176}-1979 q^{174}-291 q^{172} \\\displaybreak[2]\qquad
+973 q^{170}+2573 q^{168}-4575 q^{166}+2184 q^{164}-556q^{162}+2762 q^{160}-4105 q^{158}+1501 q^{156}+836 q^{154} \\\displaybreak[2]\qquad
+320 q^{152}-1332 q^{150}-948 q^{148}+3153 q^{146}-2107 q^{144}+575 q^{142}-1709 q^{140}+3103 q^{138}-1762q^{136} \\\displaybreak[2]\qquad
-256 q^{134}+14 q^{132}+927 q^{130}+143 q^{128}-1822 q^{126}+1555 q^{124}-368 q^{122}+681 q^{120}-1728 q^{118} \\\displaybreak[2]\qquad
+1327 q^{116}-89 q^{114}-685 q^{110}+360q^{108}+678 q^{106}-823 q^{104}+215 q^{102}-274 q^{100}+946 q^{98}-959 q^{96} \\\displaybreak[2]\qquad
+302 q^{94}-85 q^{92}+409 q^{90}-333 q^{88}-237 q^{86}+461 q^{84}-191 q^{82}+126q^{80}-433 q^{78}+517 q^{76}-211 q^{74} \\\displaybreak[2]\qquad
+20 q^{72}-132 q^{70}+153 q^{68}+67 q^{66}-193 q^{64}+85 q^{62}-12 q^{60}+116 q^{58}-183 q^{56}+92 q^{54}-18 q^{52} \\\displaybreak[2]\qquad
+64 q^{50}-93q^{48}+20 q^{46}+44 q^{44}-24 q^{42}+2 q^{40}-44 q^{38}+80 q^{36}-52 q^{34}+18 q^{32}-24 q^{30}+32 q^{28} \\\displaybreak[2]\qquad
-8 q^{26}-18 q^{24}+16 q^{22}-8 q^{20}+16 q^{18}-24 q^{16}+16q^{14}-4 q^{12}+q^8-4 q^6+6 q^4-4 q^2+1\Big)
\end{multline*}


\bibliographystyle{hamsalpha}
\bibliography{exceptional}

\end{document}